\documentclass[12pt,twoside]{book}
\usepackage{amsfonts}
\usepackage{amssymb}
\usepackage{amsthm}
\usepackage{amsmath}
\usepackage{newlfont}
\usepackage{t1enc}
\usepackage{bbm}
\usepackage{fontsmpl}
\usepackage{graphics}
\usepackage{graphicx}
\usepackage{color}
\usepackage{psfrag}
\usepackage{a4wide}
\usepackage{amssymb}
\usepackage{enumerate}
\usepackage{hyperref}

\theoremstyle{remark}

\theoremstyle{definition}
\newtheorem{teo}{Theorem}[chapter]

\newtheorem{ejem}[teo]{Examples}
\newtheorem{lema}[teo]{Lemma}
\newtheorem{prop}[teo]{Proposition}
\newtheorem{cor}[teo]{Corollary}
\newtheorem{obs}[teo]{Remark}
\newtheorem{defi}[teo]{Definition}
\newtheorem{assump}[teo]{Assumptions}
\newtheorem{notat}[teo]{Notation}
\newtheorem{cond}[teo]{Conditions}

\newcommand{\ve}{\varepsilon}
\newcommand{\R}{{\mathbb R}}
\newcommand{\C}{{\mathcal C}}
\newcommand{\Z}{{\mathbb{Z}}}

\newcommand{\N}{{\mathbb N}}

\newcommand{\F}{{\mathcal F}}
\newcommand{\G}{{\mathcal G}}

\newcommand{\E}{{\mathbb E}}

\newcommand{\B}{{\mathcal B}}

\newcommand{\0}{{\mathbf 0}}

\def\M{\mathcal M}
\def\A{\mathcal A}

\def\W{\mathcal W}
\def\F{\mathcal F}

\newcommand{\p}{\partial}
\newcommand{\eps}{\varepsilon}
\def\1{{\mathbf 1}}

\def\R{\mathbb{R}}
\def\N{\mathbb{N}}
\def\E{\mathbb{E}}

\def\M{\mathcal M}
\def\A{\mathcal A}

\def\W{\mathcal W}
\def\F{\mathcal F}

\numberwithin{equation}{chapter}

\begin{document}

\parskip 5pt

\setcounter{chapter}{0}

\pagestyle{headings}

\thispagestyle{empty}

\begin {center}

\includegraphics[scale=.3]{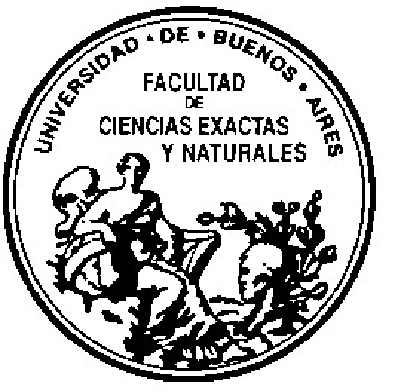}

\medskip
UNIVERSIDAD DE BUENOS AIRES

Facultad de Ciencias Exactas y Naturales

Departamento de Matem\'atica

\vspace{3cm}

\textbf{\large Metaestabilidad para una EDP con blow-up y la dinámica FFG en modelos diluidos}

\vspace{2cm}

Tesis presentada para optar al t\'\i tulo de Doctor de la Universidad de Buenos Aires en el \'area Ciencias Matem\'aticas

\vspace{2cm}

\textbf{Santiago Saglietti}

\end {center}

\vspace{1.5cm}

\noindent Director de tesis: Pablo Groisman

\noindent Consejero de estudios: Pablo Groisman

\vspace{1cm}

\noindent Buenos Aires, 2014

\noindent Fecha de defensa : 27 de Junio del 2014


\chapter*{}

\begin{center}
{\Large \bf Metaestabilidad para una EDP con blow-up y la dinámica FFG en modelos diluidos}
\end{center}

\bigskip

\bigskip

{\centerline {\bf Resumen}}

Esta tesis consiste de dos partes, en cada una estudiamos la estabilidad bajo pequeñas perturbaciones de ciertos modelos probabilísticos en diferentes contextos. En la primera parte, estudiamos pequeñas perturbaciones \textit{aleatorias} de un sistema dinámico determinístico y mostramos que las mismas son inestables, en el sentido de que los sistemas perturbados tienen un comportamiento cualitativo diferente al del sistema original. Más precisamente, dado $p > 1$ estudiamos soluciones de la ecuación en derivadas parciales estocástica
$$
\p_t U = \p^2_{xx} U + U|U|^{p-1} + \varepsilon \dot{W}
$$ con condiciones de frontera de Dirichlet homogéneas y mostramos que para $\varepsilon > 0$ pequeños éstas presentan una forma particular de inestabilidad conocida como \textit{metaestabilidad}.
En la segunda parte nos situamos dentro del contexto de la mecánica estadística, donde estudiamos la estabilidad de medidas de equilibrio en volumen infinito bajo ciertas perturbaciones \textit{determinísticas} en los parámetros del modelo. Más precisamente, mostramos que las medidas de Gibbs para una cierta clase general de sistemas son continuas con respecto a cambios en la interacción y/o en la densidad de partículas y, por lo tanto, estables bajo pequeñas perturbaciones de las mismas. También estudiamos bajo qué condiciones ciertas configuraciones típicas de estos sistemas permanecen estables en el límite de temperatura cero $T \to 0$. La herramienta principal que utilizamos para nuestro estudio es la realización de estas medidas de equilibrio como distribuciones invariantes de las dinámicas introducidas en \cite{FFG1}. Referimos al comienzo de cada una de las partes para una introducción de mayor profundidad sobre cada uno de los temas.

\bigskip

\bigskip

{\it Palabras claves:} ecuaciones en derivadas parciales estocásticas, metaestabilidad, blow-up, medidas de Gibbs, procesos estocásticos, redes de pérdida, Pirogov-Sinai.

\chapter*{}

\begin{center}
{\Large \bf Metastability for a PDE with blow-up and the FFG dynamics in diluted models}
\end{center}

\bigskip

\bigskip

{\centerline {\bf Abstract}}

This thesis consists of two separate parts: in each we study the stability under small perturbations of certain probability models in different contexts. In the first, we study small \textit{random} perturbations of a deterministic dynamical system and show that these are unstable, in the sense that the perturbed systems have a different qualitative behavior than that of the original system. More precisely, given $p > 1$ we study solutions to the stochastic partial differential equation
$$
\p_t U = \p^2_{xx} U + U|U|^{p-1} + \varepsilon \dot{W}
$$ with homogeneous Dirichlet boundary conditions and show that for small $\varepsilon > 0$ these present a rather particular form of unstability known as \textit{metastability}. In the second part we situate ourselves in the context of statistical mechanics, where we study the stability of equilibrium infinite-volume measures under small \textit{deterministic} perturbations in the parameters of the model. More precisely, we show that Gibbs measures for a general class of systems are continuous with respect to changes in the interaction and/or density of particles and, hence, stable under small perturbations of them. We also study under which conditions do certain typical configurations of these systems remain stable in the zero-temperature limit $T \to 0$. The main tool we use for our study is the realization of these equilibrium measures as invariant distributions of the dynamics introduced in \cite{FFG1}. \mbox{We refer} to the beginning of each part for a deeper introduction on each of the subjects.

\bigskip

\bigskip

{\it Key words}: stochastic partial differential equations, \mbox{metastability}, stochastic \mbox{processes,} blow-up, \mbox{Gibbs} measures, loss networks, Pirogov-Sinai.

\chapter*{Agradecimientos}

Un gran número de personas han contribuido, de alguna manera u otra, con la realización de este trabajo. Me gustaría agradecer:
\begin{enumerate}
\item [$\diamond$] A mi director, Pablo Groisman, por todo. Por su constante apoyo y \mbox{enorme ayuda} durante la elaboración de esta Tesis. Por su infinita paciencia y gran \mbox{predisposición.} Por estar siempre para atender mis inquietudes, y por recibirme todas y cada una de las veces con una sonrisa y la mejor onda. Por compartir conmigo su manera de concebir y hacer matemática, lo que ha tenido un gran impacto en mi formación como matemático y es, para mí, de un valor incalculable. Por su amistad. Por todo.
\item [$\diamond$] A Pablo Ferrari, por estar siempre dispuesto a darme una mano y a discutir sobre matemática conmigo. Considero realmente un privilegio haber tenido la oportunidad de pensar problemas juntos y entrar en contacto con su forma de ver la matemática. Son muchísimas las cosas que he aprendido de él en estos últimos cuatro años, y es por ello que le voy a estar siempre inmensamente agradecido.
\item [$\diamond$] A los jurados de esta Tesis: Pablo De Napoli, Mariela Sued y Aernout Van Enter. Por leerla y darme sus sugerencias, con todo el esfuerzo y tiempo que ello requiere.
\item [$\diamond$] A Nicolas Saintier, por su entusiasmo en mi trabajo y su colaboración en esta Tesis.
\item [$\diamond$] A Roberto Fernández y Siamak Taati, por la productiva estadía en Utrecht.
\item [$\diamond$] A Inés, Matt y Leo. Por creer siempre en mí y enseñarme algo nuevo todos los días.
\item [$\diamond$] A Maru, por las muchas tardes de clase, estudio, charlas, chismes y chocolate.
\item [$\diamond$] A mis hermanitos académicos: Anita, Nico, Nahuel, Sergio L., Sergio Y. y Julián.
Por todos los momentos compartidos, tanto de estudio como de amistad.
\item [$\diamond$] A Marto S., Pablo V., Caro N. y los chicos de la 2105 (los de ahora y los de antes).
\item [$\diamond$] A Adlivun, por los buenos momentos y la buena música.
\item [$\diamond$] A la (auténtica) banda del Gol, por ser los amigos incondicionales que son.
\item [$\diamond$] A Ale, por haber estado siempre, en las buenas y (sobre todo) en las malas.
\item [$\diamond$] A mi familia, por ser mi eterno sostén y apoyo.
\end{enumerate}
$$
\text{Gracias!}
$$

\chapter*{Errata}

We have spotted a mistake in the proof of the second assertion in Proposition A.2. In fact, this assertion is false in general. For simplicity, we give below a counterexample for the case in which the source $g \equiv 0$ (a similar counterexample can be devised with a little more effort for the original source term $g(u):=u|u|^{p-1}$). Indeed, if one takes $B:=\{u_k : k \in \N\}$
where
$$
u_k(x)=\sin(k\pi x)
$$ then the set $B$ is bounded but one can check that the analogous of the second assertion in Proposition A.2 for the heat equation is not satisfied, since each solution $U^{u_k}$ here is given by 
$$
U^{u_k}:= u_k(x) e^{-(k\pi)^2 t},
$$ so that 
$$
\|U^{u_k}(t,\cdot) - u_k \|_\infty = 1 - e^{-(k\pi)^2 t}
$$ which implies that for all $t > 0$
$$
\sup_{u \in B} \|U^{u_k}(t,\cdot) - u_k \|_\infty = 1.
$$ 

Unfortunately, we used Proposition A.2 many times throughout our analysis, which forces us to make several modifications to our proofs in order to circumvent this problem. In the following, we detail what type of problems arise by the lack of Proposition A.2 and how to fix them.

\section*{A weaker version of Proposition A.2}

Even though we cannot expect Proposition A.2 to be true in general, there is a weaker version of it which is always true. It is contained in the following proposition.

\begin{prop}\label{G.10}Given a bounded set $B \subseteq C_D([0,1])$ there exists
	$t_B > 0$ such that $\tau^u > t_B$ for every $u \in B$ and
	$$
	\sup_{u \in B, t \in [0,t_B]} \|U^u(t,\cdot)\|_\infty < +\infty.
	$$
\end{prop}

On many occasions in our original analysis we used Proposition A.2 to show that \eqref{3} below holds for certain specific functions $f$. It turns out that in many of these occasions one only needs Proposition \ref{G.10} to show this, and not the full extent of Proposition A.2. Thus, in these cases our original arguments go through without any (real) modifications. We refer to \cite{GSS} for a proof of Proposition \ref{G.10} and for further details on where it is needed. 

\section*{Proposition A.7 and the proof of Lemma 4.3}

The proof of Proposition A.7 relied heavily on Proposition A.2. Thus, Proposition A.7 must now be replaced by the following proposition.

\begin{prop}\label{G.20} Given $\delta > 0$ and $N \in \N$, for any $u \in C_D([0,1])$ let us define 
	$$
	\mathcal{H}^{(N),\delta,u}:= \inf \{ t \geq 1 : \| U^u(t,\cdot)\|_\infty > N \text{ or } d(U^u(t,\cdot) , z^{(n)}) < \delta \text{ for some }n \in \Z\}.
	$$ Then, for any bounded set $B \subseteq C_D([0,1])$ we have that 
	$$
	\sup_{u \in B} \mathcal{H}^{(N),\delta,u} < +\infty.
	$$ 
\end{prop}
The proof of this result can be carried out as in \cite[Proposition A.3]{B1} so we omit further details. This replacement affects the proof of Lemma 4.3 in the Thesis. Indeed, its proof should now go as follows:
\begin{itemize}
	\item [i.] Take $T_0 < 1$ and $r < \min\{\frac{1}{4},c\}$ sufficiently small together with $H > 0$ large enough so as to guarantee that:
	\begin{itemize}
		\item [$\bullet$] $\sup_{t \in [0,T_0]} \| U^u(t,\cdot) - U^v(t,\cdot) \|_\infty \leq 2 \|u-v\|_\infty$ for all $u,v \in B_{n_0+1}$,
		\item [$\bullet$] $\|U^u(T_0,\cdot)\|_{W^2_2([0,1])} \leq H$ for all $u \in B_{n_0+2}$,
		\item [$\bullet$] Any linear interpolation between elements $v,v' \in B_{n_0+2}$ with $W^2_2([0,1])$-norm bounded by $2H$ and at a distance smaller than $2r$ has rate less than $\frac{\delta}{9}$. 
	\end{itemize}
	\item [ii.] Given $u \in G$, construct $\varphi^u$ as follows:
	\begin{itemize}
		\item [$\bullet$] First, follow the deterministic flow until time $h^u:=\mathcal{H}^{(n_0+1),r,u} \geq 1$.
		\item [$\bullet$] If $\varphi^u(h^u,\cdot) \in \partial B_{n_0+1}$, stop the construction.  
		\item [$\bullet$] If instead $d(\varphi^u(h^u,\cdot),z^{(n)}) \leq r$ for some $z^{(n)} \in B_{n_0+1}$ different from $\mathbf{0}$, continue $\varphi^u$ by a linear interpolation between $v:=\varphi^u(h^u,\cdot)$ and $v':=(1+r)z^{(n)}$. Note that both $v$ and $v'$ lie inside $B_{n_0+2}$ and have $W^2_2([0,1])$-norm bounded by $2H$, since $v=U^{\tilde{u}}(T_0,\cdot)$ for $\tilde{u}:=U^u(h^u-T_0,\cdot) \in B_{n_0+1}$ and $z^{(n)}=U^{z^{(n)}}(T_0,\cdot)$. Afterwards, follow the construction as it is outlined in the Thesis.
		\item [$\bullet$] If $d(\varphi^u(h^u,\cdot),\mathbf{0}) \leq r$, follow the the rest of the construction as it is outlined in the Thesis.
	\end{itemize} 
\end{itemize}
Since we have that $\sup_{u \in G} h^u < +\infty$ by Proposition \ref{G.20}, with this minor modification now the rest of the proof can be carried out as in the Thesis.

\section*{Uniformity only over $\mathcal{D}$-compactifiable sets}

Our results are no longer valid uniformly over bounded sets away from the boundary (of the domain $\mathcal{D}$ in question, either $\mathcal{D}_{\mathbf{0}}$ or $\mathcal{D}_e$). Instead, we are only able to show them uniformly over ``$\mathcal{D}$-compactifiable'' sets, which is a particular class of subsets of $C_D([0,1])$ we introduce in the Main Results section of \cite{GSS}. Essentially, one could think of $\mathcal{D}$-compactifiable sets as small neighborhoods of compact sets which are contained in $\mathcal{D}$. The crucial property of these $\mathcal{D}$-compactifiable sets is the following. 

\begin{prop} If $\mathcal{D} \subseteq C_D([0,1])$ and $f : C_D([0,1]) \rightarrow \R \cup \{\pm \infty\}$ is a function satisfying:
	\begin{itemize}
		\item [i.] $f$ is finite and upper semicontinuous on $\mathcal{D}$,
		\item [ii.] for any bounded set $B \subseteq C_D([0,1])$ and every $t > 0$ sufficiently small (depending on $B$) there exists a constant $c_{t,B} > 0$ such that
		\begin{equation}\label{3}
		\sup_{u \in B} f(u) \leq c_{t,B} + \sup_{u \in B} f(U^u(t,\cdot)),
		\end{equation}
	\end{itemize}
	then 
	\begin{equation}\label{1}
	\sup_{u \in \mathcal{K}} f(u) <+\infty.
	\end{equation} for all $\mathcal{D}$-compactifiable sets $\mathcal{K}$. 
\end{prop}

Originally, we used Proposition A.2 in the Thesis to show that \eqref{1} was satisfied for all bounded sets $\mathcal{K} \subseteq \mathcal{D}$ at a positive distance from $\partial \mathcal{D}$. Now, we must restrict ourselves to considering only $\mathcal{D}$-compactifiable sets every time we need \eqref{1} to hold for such an $f$. We refer to \cite{GSS} for further details.

\section*{Minimizing the potential $S$ and quasipotential $V$}

On many occasions throughout our analysis, i.e. (3.4) and Section 4.3.1, we see ourselves considering some closed bounded domain $B \subseteq C_D([0,1])$ with $\min_{u \in B} S(u) < +\infty$ and wanting to show that for any $\delta > 0$ sufficiently small
\begin{equation}\label{2}
\inf \{ S(u) : u \in B - (m_S(B))_{(\delta)} \} > \min_{u \in B} S(u),
\end{equation} where $m_S(B):=\{ u \in B : S(u) = \min_{v \in B} S(v)\}$ is the set of minimizers of $S$ in $B$ and 
$$
(m_S(B))_{(\delta)} = \{ v \in C_D([0,1]) : d(v,m_S(B)) \leq \delta\}.
$$ In the Thesis we showed this to  be true with the aid of Proposition A.2. Now, we will do it with the help of the following result.

\begin{lema}\label{G.4} If $B' \subseteq C_D([0,1])$ is a bounded set which satisfies $\sup_{u \in B'} S(u) < +\infty$ then it has compact closure. 
\end{lema}

\begin{proof} It follows from the hypotheses on $B'$ that each $u \in B'$ is absolutely continuous and, furthermore, that
	$$
	\sup_{u \in B'} \|\partial_x u\|_{L^2} < +\infty.
	$$ From this, using Cauchy-Schwarz inequality we can show that $B'$ is also equicontinuous, so that by Arzela-Ascoli we deduce that $B'$ has compact closure.  
\end{proof}

Using Lemma \ref{G.4}, let us show that \eqref{2} holds. Indeed, if this were not the case then there would exist a sequence $(u_k)_{k \in \N} \subseteq B-(m_S(B))_{(\delta)}$ such that 
$$
\lim_{k \rightarrow +\infty} S(u_k) = \min_{u \in B} S(u) <+\infty.
$$ In particular, without loss of generality we can assume that $S(u_k) < +\infty$ for all $k \in \N$. Thus, by Lemma \ref{G.4} we have that the set $B':=\{ u_k : k \in \N\}$ has compact closure so that there exists a subsequence $(u_{k_j})_{j \in \N}$ which converges to some $u_\infty \in B$ (here we use that $B$ is closed). By the lower semicontinuity of $S$ we have that $S(u_\infty) \leq \min_{u \in B} S(u)$ which, since $u_\infty \in B$, implies that $u_\infty \in m_S(B)$. This contradicts the fact that $d(u_k,m_S(B)) > \delta$ for all $k \in \N$, from which we conclude that \eqref{2} must hold.

Similarly, on several occasions throughout the article we will need to consider closed bounded regions $B \subseteq C_D([0,1])$ with $V(\mathbf{0},B) < +\infty$ and show that, given $\delta_1 > 0$ small, there exists $\delta_2 > 0$ such that
$$
\inf \left\{ V(\mathbf{0},u) : u \in \left( B - \left(m_V(B)\right)_{(\delta_1)}\right)_{(\delta_2)}\right\} > V(\mathbf{0},B)
$$ where 
$$
m_V(B):=\left\{ u \in B : V(\mathbf{0},u) = \min_{v \in B} V(\mathbf{0},v)\right\}.
$$ This will follow from an argument similar to the one given above to show \eqref{2}, provided one can show the following lemma. 

\begin{lema} If $B' \subseteq C_D([0,1])$ is a bounded set which satisfies $\sup_{u \in B'} V(\mathbf{0},u) < +\infty$ then it has compact closure.
\end{lema}

\begin{proof} It follows from Proposition 3.6 in the Thesis that for any $u \in B'$
	$$
	2(S(u) -S(\mathbf{0})) \leq V(\mathbf{0},u),
	$$ so that $\sup_{u \in B'} S(u) <+\infty$ and thus the lemma now follows from Lemma \ref{G.4}.
\end{proof} 

\section*{The construction of the domain $G$}

The construction of the domain $G$ must also be corrected. We refer to \cite{GSS} for the correct construction, using the ideas from the previous section.

\section*{The lower bound for $\tau^u_\varepsilon$}

The proof of the convergence to zero of the second term in the right-hand side of (4.4) must also be corrected. This can be done by following the proof of \cite[Proposition 4.2]{B1}.

\tableofcontents

\part{Metastability for a PDE with blow-up}

\chapter*{Introducción a la Parte I}

Las ecuaciones diferenciales han probado ser de gran utilidad para modelar un amplio rango de fenómenos físicos, químicos y biológicos. Por ejemplo, una vasta clase de ecuaciones de evolución, conocidas como ecuaciones en derivadas parciales parabólicas \mbox{semilineales,} surgen naturalmente en el estudio de fenómenos tan diversos como la difusión de un fluido a través de un material poroso, el transporte en un semiconductor, las reacciones químicas acopladas con difusión espacial y la genética de poblaciones. En todos estos casos, la ecuación representa un modelo aproximado del fenómeno y por lo tanto es de interés entender cómo su descripción puede cambiar si es sujeta a pequeñas perturbaciones aleatorias. Nos interesa estudiar ecuaciones del tipo
\begin{equation}\label{intro0esp}
\p_t U = \p^2_{xx} U + f(U)
\end{equation} con condiciones de frontera de Dirichlet en $[0,1]$, donde $f: \R \rightarrow \R$ es una fuente localmente Lipschitz. Dependiendo del dato inicial, es posible que las soluciones de esta ecuación no se encuentren definidas para todo tiempo. Decimos entonces que estamos ante la presencia del fenómeno de \textit{blow-up} o \textit{explosión}, i.e. existe $\tau > 0$ tal que la solución $U$ se encuentra definida para todo tiempo $t < \tau$ y además satisface $\lim_{t \rightarrow \tau^-} \| U(t,\cdot)\|_\infty = +\infty$. Agregando una pequeña perturbación aleatoria al sistema se obtiene la ecuación en derivadas parciales estocástica
\begin{equation}\label{intro1esp}
\p_t U = \p^2_{xx} U + f(U) + \varepsilon \dot{W}
\end{equation} donde $\varepsilon > 0$ es un parámetro pequeño y $\dot{W}$ es ruido blanco espacio-temporal. Uno puede preguntarse entonces si existen diferencias cualitativas en comportamiento entre el sistema determinístico \eqref{intro0esp} y su perturbación estocástica. Para tiempos cortos ambos sistemas deberían comportarse de manera similar, ya que en este caso el ruido será típicamente de un orden mucho menor que los términos restantes en el miembro derecho de \eqref{intro1esp}. Sin embargo, debido a los incrementos independientes y normalmente distribuidos del ruido uno espera que, si es dado el tiempo suficiente, éste eventualmente alcanzará valores suficientemente grandes como para inducir un cambio de comportamiento significativo en \eqref{intro1esp}. Estamos interesados en entender qué cambios pueden ocurrir en el fenómeno de blow-up debido a esta situación y, más precisamente, cuáles son las propiedades asintóticas cuando $\varepsilon \rightarrow 0$ del tiempo de explosión de \eqref{intro1esp} para los diferentes datos iniciales. En particular, para sistemas como en \eqref{intro0esp}
con un único equilibrio estable $\phi$, uno espera el siguiente panorama:

\begin{enumerate}
\item [i.] Para datos iniciales en el dominio de atracción del equilibrio estable, el sistema estocástico es inmediatamente atraído hacia el equilibrio. Una vez cerca de éste, los términos en el miembro derecho de \eqref{intro0} se vuelven despreciables de manera tal que el proceso puede ser luego empujado lejos del equilibrio por acción del ruido. Estando lejos de $\phi$, el ruido vuelve a ser superado por los términos restantes en el miembro derecho de \eqref{intro1esp} y esto permite que el patrón anterior se repita: un gran número de intentos de escapar del equilibrio, seguidos de una fuerte atracción hacia el mismo.
\item [ii.] Eventualmente, luego de muchos intentos frustrados, el proceso logra escaparse del dominio de atracción de $\phi$ y alcanza el dominio de explosión, aquel conjunto de datos iniciales para los cuales la solución de \eqref{intro0esp} explota en tiempo finito. Como la probabilidad de un evento tal es muy baja, esperamos que este tiempo de escape sea exponencialmente grande. Más aún, debido al gran número de intentos que fueron necesarios, esperamos que este tiempo muestre escasa memoria del dato inicial.
\item [iii.] Una vez dentro del dominio de explosión, el sistema estocástico es forzado a explotar por la fuente $f$, que se convierte en el término dominante.
\end{enumerate}

Este tipo de fenómeno se conoce como \textit{metaestabilidad}: el sistema se comporta por un tiempo muy largo como si estuviera bajo equilibrio, para luego realizar una transición abrupta hacia el equilibrio real (en nuestro caso, hacia infinito). La descripción anterior fue probada rigurosamente en \cite{GOV} en el contexto finito-dimensional para sistemas del tipo
$$
\dot{U}=- \nabla S(U)
$$ donde $U$ es un potencial de doble pozo. En este contexto, el comportamiento metaestable es observado en la manera en que el sistema estocástico viaja desde cualquiera de los pozos hacia el otro. Luego, en \cite{MOS} y \cite{B1},
el problema análogo infinito-dimensional fue investigado, obteniendo resultados similares.

El enfoque general sugerido en \cite{GOV} para establecer el comportamiento metaestable en este tipo de sistemas es estudiar el escape de un dominio acotado $G$ satisfaciendo:

\begin{enumerate}
\item [$\bullet$] $G$ contiene al equilibrio estable $\phi$ y a los equilibrios inestables de mínima energía.
\item [$\bullet \bullet$] Existe una región $\p^*$ en la frontera de $G$ tal que:
\begin{enumerate}
\item [i.] El ``costo'' para el sistema de alcanzar $\p^*$ comenzando desde $\phi$ es el mismo que el costo de alcanzar cualquiera de los equilibrios inestables de mínima energía.
\item [ii.] Con probabilidad que tiende a uno cuando $\varepsilon \rightarrow 0^+$ el sistema estocástico comenzando en $\p^*$ alcanza el verdadero equilibrio antes de un tiempo acotado $\tau^*$ independiente de $\varepsilon$.
\end{enumerate}
\end{enumerate}

La construcción de este dominio para el potencial de doble pozo finito-dimensional fue llevada a cabo en \cite{GOV}. En el marco infinito-dimensional, sin embargo, este tipo de resultados fueron probados sin seguir estrictamente
el enfoque de \cite{GOV}: la pérdida de memoria asintótica fue lograda en \cite{MOS} sin acudir a ningún dominio auxiliar, mientras que el restante panorama fue establecido en \cite{B1} considerando un dominio
que tiene a los equilibrios inestables de mínima energía en su frontera y por lo tanto no cumple (ii). Un dominio de tales características no puede ser utilizado como sugiere \cite{GOV} para obtener la pérdida de memoria asintótica, pero es utilizado en \cite{B1} de todas maneras puesto a que dicha pérdida de memoria ya había sido probada en \cite{MOS} por otros métodos. Nosotros hemos decidido aferrarnos al enfoque general sugerido en \cite{GOV} para estudiar nuestro sistema ya que quizás éste sea el más sencillo de seguir y, además, ya que provee un único marco general sobre el cual se pueden probar todos los resultados que nos interesan. Más aún, para seguirlo deberemos introducir herramientas que son también útiles para estudiar otros tipos de problemas, como el escape de un dominio con un único equilibrio.

En nuestro trabajo también consideraremos ecuaciones de tipo gradiente, pero la situación en nuestro contexto es más delicada que en la del modelo del potencial de doble pozo. En efecto, la construcción del dominio $G$ dependerá en gran medida de la geometría del potencial asociado a la ecuación, la cual en general será más complicada que la dada por el potencial de doble pozo. Además, (ii) en la descripción del dominio $G$ dada arriba es usualmente una consecuencia directa de las estimaciones de grandes desvíos disponibles para el sistema estocástico. No obstante, la validez de estas estimaciones en todos los casos depende de un control apropiado sobre el crecimiento de las soluciones de \eqref{intro0esp}. Como estaremos enfocándonos específicamente en trayectorias que explotan en un tiempo finito, está claro que para esta última parte un nuevo enfoque será necesario en nuestro problema, uno que involucre un estudio cuidadoso del fenómeno de blow-up. Desafortunadamente, cuando se trata con perturbaciones de ecuaciones diferenciales con blow-up, entender cómo se modifica el comportamiento del tiempo de explosión o incluso mostrar la existencia del fenómeno de blow-up mismo no es para nada una tarea fácil en la mayoría de los casos. No existen resultados generales al respecto, ni siquiera para perturbaciones no aleatorias. Esta es la razón por la cual el enfoque usual a este tipo de problemas es considerar modelos particulares.

En esta primera parte estudiamos el comportamiento metaestable de la ecuación con blow-up
\begin{equation}\label{intro2}
\p_t U = \p^2_{xx} U  + U|U|^{p-1}
\end{equation} con condiciones de frontera de Dirichlet homogéneas en $[0,1]$, para un cierto parámetro \mbox{$p > 1$.} Hemos elegido esta ecuación particular ya que ha sido tomada como problema modelo por la comunidad de EDP,
dado que exhibe las principales características de interés que aparecen en la presencia de blow-up (ver por ejemplo los libros \cite{QS,SGKM} o las notas \cite{BB,GV}). También, trabajamos con una variable espacial unidimensional dado que no existen soluciones de \eqref{intro1esp} para dimensiones más altas en el sentido tradicional.

La Parte I está organizada de la siguiente manera. En el Capítulo 1 damos las definiciones necesarias y los resultados preliminares para ayudarnos a tratar nuestro problema, como también así detallamos los resultados principales que hemos obtenido. El Capítulo 2 se enfoca en el tiempo de explosión para el sistema estocástico para datos iniciales en el dominio de explosión. La construcción del dominio auxiliar $G$ en nuestro contexto es llevada a cabo en el Capítulo 3, mientras que estudiamos el escape de $G$ en el capítulo siguiente. En el Capítulo 5 establecemos el comportamiento metaestable para soluciones con datos iniciales en el dominio de atracción del equilibrio estable. En el Capítulo 6 estudiamos una variante finito-dimensional de nuestro problema original e investigamos qué resultados pueden obtenerse en este marco simplificado. Finalmente, incluimos al final un apéndice con algunos resultados auxiliares a ser utilizados durante nuestro análisis.

\chapter*{Introduction to Part I}

Differential equations have proven to be of great utility to model a wide range of \mbox{physical}, chemical and biological phenomena. For example, a broad class of evolution \mbox{equations}, known as semilinear parabolic partial differential equations, naturally arise in the study of \mbox{phenomena} as diverse as diffusion of a fluid through a porous material, transport in a semiconductor, coupled chemical reactions with spatial diffusion and population genetics. In all these cases, the equation represents an approximated model of the phenomenon and thus it is of interest to understand how its description might change if subject to small random perturbations. We are concerned with studying equations of the sort
\begin{equation}\label{intro0}
\p_t U = \p^2_{xx} U + f(U)
\end{equation} with homogeneous Dirichlet boundary conditions on $[0,1]$, where $f: \R \rightarrow \R$ is a locally Lipschitz source. Depending on the initial datum, it is possible that solutions to this equation are not defined for all times. We then say we are in the presence of a \textit{blow-up} phenomenon, i.e. there exists $\tau > 0$ such that the solution $U$ is defined for all times $t < \tau$ and verifies $\lim_{t \rightarrow \tau^-} \| U(t,\cdot)\|_\infty = +\infty$. Adding a small random perturbation to the system yields the stochastic partial differential equation
\begin{equation}\label{intro1}
\p_t U = \p^2_{xx} U + f(U) + \varepsilon \dot{W}
\end{equation} where $\varepsilon > 0$ is a small parameter and $\dot{W}$ is space-time white noise. One can then wonder if there are any qualitative differences in behavior between the deterministic system \eqref{intro0} and its stochastic perturbation. For short times both systems should behave similarly, since in this case the noise term will be typically of much smaller order than the remaining terms in the right hand side of \eqref{intro1}. However, due to the independent and normally distributed increments of the perturbation, one expects that when given enough time the noise term will eventually reach sufficiently large values so as to induce a significant change of behavior in \eqref{intro1}. We are interested in understanding what changes might occur in the blow-up phenomenon due to this situation and, more precisely, which are the asymptotic properties as $\varepsilon \rightarrow 0$ of the explosion time of \eqref{intro1} for the different \mbox{initial data.}
In particular, for systems as in \eqref{intro0} with a unique stable equilibrium $\phi$, one expects \mbox{the following scenario:}
\begin{enumerate}
\item [i.] For initial data in the domain of attraction of the stable equilibrium, the stochastic system is immediately attracted towards this equilibrium. Once near it, the terms in the right hand side of \eqref{intro0} become negligible and so the process is then pushed away from the equilibrium by noise. Being away from $\phi$, the noise becomes overpowered by the remaining terms in the right hand side of \eqref{intro1} and this allows for the previous pattern to repeat itself: a large number of attempts to escape from the \mbox{equilibrium,} followed by a strong attraction towards it.
\newpage
 \item [ii.] Eventually, after many frustrated attempts, the process succeeds in escaping the domain of attraction of $\phi$ and reaches the domain of explosion, i.e. the set of initial data for which \eqref{intro0} blows up in finite time. Since the probability of such an event is very small, we expect this escape time to be exponentially large.
Furthermore, due to the large number of attempts that are necessary, we expect this time to show little memory of the initial data.
\item [iii.] Once inside the domain of explosion, the stochastic system is forced to explode by the dominating source term $f$.
\end{enumerate}
This type of phenomenon is known as \textit{metastability}: the system behaves for a very long time as if it were under equilibrium, but then performs an abrupt transition towards the real equilibrium (in our case, towards infinity). The former description was proved rigorously in \cite{GOV} in the finite-dimensional setting for systems of the sort
$$
\dot{U}=- \nabla S(U)
$$ where $U$ is a double-well potential. In their context, metastable behavior is observed in the way in which the stochastic system travels from one of the wells to the other. Later, in \cite{MOS} and \cite{B1}, the analogous infinite-dimensional problem was investigated, obtaining similar results.

The general approach suggested in \cite{GOV} to establish metastable behavior in these kind of systems is to study the escape from a bounded domain $G$ satisfying the following:
\begin{enumerate}
\item [$\bullet$] $G$ contains the stable equilibrium $\phi$ and all the unstable equilibria of minimal energy.
\item [$\bullet \bullet$] There exists a region $\p^*$ in the boundary of $G$ such that:
\begin{enumerate}
\item [i.] The ``cost'' for the system to reach $\p^*$ starting from $\phi$ is the same as the cost to reach any of the unstable equilibria of minimal energy.
\item [ii.] With overwhelming probability as $\varepsilon \rightarrow 0^+$ the stochastic system \mbox{starting in $\p^*$} arrives at the real equilibrium before a bounded time $\tau^*$ independent of $\varepsilon$.
\end{enumerate}
\end{enumerate}

The construction of this domain for the finite-dimensional double-well potential was carried out in \cite{GOV}. In the infinite-dimensional setting, however, these type of results were proved without strictly following this approach: the asymptotic loss of memory was achieved in \cite{MOS} without resorting to any auxiliary domain, while the remaining parts of the picture were settled in \cite{B1} by considering a domain which has the unstable equilibria of minimal energy in its boundary and hence does not satisfy (ii). Such a domain cannot be used as suggested in \cite{GOV} to obtain the asymptotic loss of memory, but it is used nonetheless in \cite{B1} since this loss of memory had already been established in \cite{MOS} by different methods. We have decided to hold on to this general approach introduced in \cite{GOV} to study our system since it is perhaps the easiest one to follow and, also, since it provides with a unique general framework on which to prove all results of interest. Furthermore, in order to follow it we will need to introduce tools which are also useful for treating other type of problems, such as the escape from a domain with only one equilibrium.

\newpage
In our work we shall also consider gradient-type equations, but the situation in our context is more delicate than in the double-well potential model. Indeed, the construction of the domain $G$ will clearly rely on the geometry of the potential associated to the equation, which in general, will be more complicated than the one given by the double-well potential. Furthermore,
(ii) in the description of the domain $G$ above is usually a direct consequence of the large deviations estimates available for the stochastic system.
 The validity of these estimates always relies, however, on a proper control of the growth of solutions to \eqref{intro0}. Since we will be focusing specifically on trajectories which blow up in finite time, it is clear that for this last part a new approach is needed in our setting, one that involves a careful study of the blow-up phenomenon. Unfortunately, when dealing with perturbations of differential equations with blow-up, understanding
how the behavior of the blow-up time is modified or even showing existence of the blow-up phenomenon itself is by no means an easy task in most cases. There are no general results addressing this matter, not even for nonrandom perturbations. This is why the usual approach to this kind of problems is to consider particular models.

In this first part we study metastable behavior for the following equation with blow-up:
\begin{equation}\label{intro2}
\p_t U = \p^2_{xx} U  + U|U|^{p-1}
\end{equation} with homogeneous Dirichlet boundary conditions on $[0,1]$, for some fixed parameter $p > 1$. We chose this particular equation since it has been taken as a model problem for the PDE community as it exhibits some of the essential interesting features which appear in the presence of blow-up (see the books \cite{QS,SGKM} or the surveys \cite{BB,GV}). Also, we work with a one-dimensional space variable since there are no solutions to \eqref{intro1} for higher dimensions in the traditional sense.

Part I is organized as follows. In Chapter 1 we give the necessary definitions and preliminary results to help us address our problem, as well as detail the main results we have obtained. Chapter 2 focuses on the explosion time of the stochastic system for initial data in the domain of explosion. The construction of the auxiliary domain $G$ in our context is performed in Chapter 3, we study the escape from $G$ in the following chapter. In Chapter 5 we establish metastable behavior for solutions with initial data in the domain of attraction of the stable equilibrium. In Chapter 6 we study a finite-dimensional variant of our original problem and investigate which results can be obtained for this simplified setting. Finally, we include at the end an appendix with some auxiliary results to be used throughout our analysis.

\chapter{Preliminaries}

\section{The deterministic PDE}

Consider the partial differential equation
\begin{equation}\label{MainPDE}
\left\{\begin{array}{rll}
\p_t U &= \p^2_{xx}U + g(U) & \quad t>0 \,,\, 0<x<1 \\
U(t,0)& =0 & \quad t>0 \\
U(t,1) & = 0 & \quad t>0 \\
U(0,x) &=u(x) & \quad 0<x<1
\end{array}\right.
\end{equation}
where $g : \R \rightarrow \R$ is given by $g(u)=u|u|^{p-1}$ for a fixed $p > 1$ and $u$ belongs to the space of continuous functions defined on $[0,1]$ with homogeneous Dirichlet boundary conditions
$$
C_{D}([0,1])= \{ v \in C([0,1]) : v(0)=v(1)=0 \}.
$$
Equation \eqref{MainPDE} can be reformulated as
\begin{equation}\label{formalPDE}
\p_t U = - \frac{\partial S}{\partial \varphi} (U)
\end{equation} where the \textit{potential} $S$ is the functional on $C_D([0,1])$ given by
$$
S(v) = \left\{ \begin{array}{ll} \displaystyle{\int_0^1 \left[\frac{1}{2} \left(\frac{dv}{dx}\right)^2 - \frac{|v|^{p+1}}{p+1}\right]}
              & \text{ if $v \in H^1_0((0,1))$}
                   \\ \\ +\infty & \text{ otherwise.}\end{array}\right.
$$ Here $H^1_0((0,1))$ denotes the Sobolev space of square-integrable functions defined on $[0,1]$ with square-integrable weak derivative which vanish at the boundary $\{0,1\}$. Recall that $H^1_0((0,1))$ can be embedded into $C_D([0,1])$ so that the potential is indeed well defined.  We refer the reader to the Appendix for a review of some of the main properties of $S$ which shall be required throughout our work.

The formulation on \eqref{formalPDE} is interpreted as the validity of
$$
\int_0^1 \p_t U(t,x) \varphi(x)dx = \lim_{h \rightarrow 0} \frac{S(U + h\varphi) - S(U)}{h}
$$ for any $\varphi \in C^1([0,1])$ with $\varphi(0)=\varphi(1)=0$. It is known that for any $u \in C_D([0,1])$ there exists a unique solution $U^{u}$ to equation \eqref{MainPDE} defined on some maximal time interval $[0,\tau^{u})$ where $0 < \tau^{u} \leq +\infty$ is called the \textit{explosion time} of $U^u$ (see \cite{QS} for further details). In general this solution will belong to the space
$$
C_D([0,\tau^u) \times [0,1]) = \{ v \in C( [0,\tau^{u}) \times [0,1]) : v(\cdot,0)=v(\cdot,1) \equiv 0 \}.
$$ However, whenever we wish to make its initial datum $u$ explicit we will do so by saying that the solution belongs to the space
$$
C_{D_{u}}([0,\tau^{u}) \times [0,1]) = \{ v \in C( [0,\tau^{u}) \times [0,1]) : v(0,\cdot)=u \text{ and }v(\cdot,0)=v(\cdot,1) \equiv 0 \}.
$$ The origin $\mathbf{0} \in C_D([0,1])$ is the unique stable equilibrium of the system and is in fact asymptotically stable. It corresponds to the unique local minimum of the potential $S$. There is also a family of unstable equilibria of the system corresponding to the remaining critical points of the potential $S$, all of which are saddle points. Among these unstable equilibria there exists only one of them which is nonnegative, which we shall denote by $z$. It can be shown that this equilibrium $z$ is in fact strictly positive for $x \in (0,1)$, symmetric with respect to the axis $x=\frac{1}{2}$ (i.e. $z(x)=z(1-x)$ for every $x \in [0,1]$) and that is of both minimal potential and minimal norm among the unstable equilibria. More precisely, one has the following characterization of the unstable equilibria.

\begin{prop}\label{equilibrios} A function $w \in C_D([0,1]) $ is an equilibrium of the system \mbox{if and only if} there exists $n \in \Z$ such that $w = z^{(n)}$, where for each $n \in \N$ we define $z^{(n)} \in C_D([0,1])$ by the formula
$$
z^{(n)}(x)= \left\{ \begin{array}{rl} n^{\frac{2}{p-1}}z( nx - [nx] ) & \text{ if $[nx]$ is even} \\ \\ - n^{\frac{2}{p-1}}z( nx - [nx] )& \text{ if $[nx]$ is odd}\end{array}\right.
$$ and also define $z^{(-n)}:= - z^{(n)}$ and $z^{(0)}:= \mathbf{0}$. Furthermore, for each $n \in \Z$ we have
$$
\| z^{(n)} \|_\infty = |n|^{\frac{2}{p-1}}\|z\|_\infty \hspace{2cm}\text{ and }\hspace{2cm} S(z^{(n)}) = |n|^{2 \left(\frac{p+1}{p-1}\right)} S(z).
$$
\end{prop}

\begin{proof} It is simple to verify that for each $n \in \Z$ the function $z^{(n)}$ is an equilibrium of the system and that each $z^{(n)}$ satisfies both $\| z^{(n)} \|_\infty = |n|^{\frac{2}{p-1}}\|z\|_\infty$ and $S(z^{(n)}) = |n|^{2 \left(\frac{p+1}{p-1}\right)} S(z)$. Therefore, we must only check that for any equilibrium of the system $w \in C_D([0,1]) - \{ \mathbf{0}\}$ there exists $n \in \N$ such that $w$ coincides with either $z^{(n)}$ or $-z^{(n)}$.

Thus, for a given equilibrium $w \in C_D([0,1]) - \{ \mathbf{0}\}$ let us define the sets
$$
G^+ = \{ x \in (0,1) : w(x) > 0 \} \hspace{2cm}\text{ and }\hspace{2cm}G^- = \{x \in (0,1) : w(x) < 0 \}.
$$ Since $w \neq \mathbf{0}$ at least one of these sets must be nonempty. On the other hand, if only one of them is nonempty then, since $z$ is the unique nonnegative equilibrium different from $\mathbf{0}$, we must have either $w=z$ or $w=-z$. Therefore, we may assume that both $G^+$ and $G^-$ are nonempty. Notice that since $G^+$ and $G^-$ are open sets we may write them as
$$
G^+ = \bigcup_{k \in \N} I^{+}_k \hspace{2cm}\text{ and }\hspace{2cm}G^- = \bigcup_{k \in \N} I^-_k
$$ where the unions are disjoint and each $I^{\pm}_k$ is a (possibly empty) open interval.

Our first task now will be to show that each union is in fact finite. For this purpose, let us take $k \in \N$ and suppose that we can write
$I^+_k = (a_k, b_k)$ for some $0\leq a_k < b_k \leq 1$. It is easy to check that $\tilde{w}_k : [0,1] \rightarrow \R$ given by
$$
\tilde{w}_k (x) = (b_k - a_k)^{\frac{2}{p-1}} w( a_k + (b_k-a_k) x)
$$ is a nonnegative equilibrium of the system different from $\mathbf{0}$ and thus it must be $\tilde{w}_k = z$. This, in particular, implies that $\| w \|_\infty \geq (b_k - a_k)^{- \frac{2}{p-1}} \| \tilde{w}_k \|_\infty = (b_k - a_k)^{- \frac{2}{p-1}} \| z \|_\infty$ from where we see that an infinite number of nonempty $I^+_k$ would contradict the fact that $\| w \|_\infty < +\infty$. Therefore, we conclude that $G^+$ is a finite union of open intervals and that, by an analogous argument, the same holds for $G^-$.

Now, by Hopf's Lemma (see \cite[p.~330]{E}) we obtain that $\partial_x z(0^+) > 0$ and $\partial_x z(1^-) < 0$. In particular, this tells us that for each $I_k^+$ we must have $d( I^+_k , G^+ - I^+_k ) > 0$, i.e. no two plus intervals lie next to each other, since that would contradict the differentiability of $w$. Furthermore, we must also have $d(I^+_k, G^-) = 0$, i.e. any plus interval lies next to a minus interval, since otherwise we would have a plus interval lying next to an interval in which $w$ is constantly zero, a fact which again contradicts the differentiability of $w$. Therefore, from all this we conclude that plus and minus intervals must be presented in alternating order, and that their closures must cover all of the interval $[0,1]$.

Finally, since $z$ is symmetric with respect to $x=\frac{1}{2}$ we obtain that $\partial_x z(0^+) = - \partial_x z(1^-)$. This implies that all intervals must have the same length, otherwise we would once again contradict the differentiability of $w$. Since the measures of the intervals must add up to one, we see that their length must be $l=\frac{1}{n}$ where $n$ denotes the total amount of intervals. This concludes the proof.
\end{proof}

Regarding the behavior of solutions to the equation \eqref{MainPDE} we have the following result, whose proof was given in \cite{CE1}.
\begin{teo}\label{descomp1} Let $U^u$ be the solution to equation \eqref{MainPDE} with initial datum $u \in C_D([0,1])$. Then one of these two possibilities must hold:
\begin{enumerate}
\item [i.] $\tau^{u} < +\infty$ and $U^u$ blows up as $t \rightarrow \tau^{u}$, i.e. $\lim_{t \rightarrow \tau^{u}} \|U^u(t,\cdot)\|_\infty = +\infty$
\item [ii.] $\tau^{u} = +\infty$ and $U^u$ converges (in the $\| \cdot \|_\infty$ norm) to a stationary solution as $t \rightarrow +\infty$, i.e. a critical point of the potential $S$.
\end{enumerate}
\end{teo}
Theorem \ref{descomp1} is used to decompose the space $C_D([0,1])$ of initial data into three parts:
\begin{equation}\label{decomp12}
C_D([0,1]) = \mathcal{D}_{\mathbf{0}} \cup \mathcal{W} \cup \mathcal{D}_e
\end{equation} where $\mathcal{D}_{\mathbf{0}}$ denotes the stable manifold of the origin $\mathbf{0}$, $\mathcal{W}$ is the union of all stable manifolds of the unstable equilibria
and $\mathcal{D}_e$ constitutes the domain of explosion of the system, i.e. the set of all initial data for which the system explodes in finite time. It can be seen that both $\mathcal{D}_{\mathbf{0}}$ and $\mathcal{D}_e$ are open sets and that $\mathcal{W}$ is the common boundary separating them. The following proposition gives a useful characterization of the domain of explosion $\mathcal{D}_e$. Its proof is can be found on \cite[Theorem~17.6]{QS}.

\begin{prop}\label{caract} Let $U^u$ denote the solution to \eqref{MainPDE} with initial datum $u \in C_D([0,1])$. Then
$$
\mathcal{D}_e = \{ u \in C_D([0,1]) : S( U^u (t, \cdot) ) < 0 \text{ for some }0 \leq t < \tau^u \}.
$$ Furthermore, we have $\lim_{t \rightarrow (\tau^u)^-} S( U^u(t,\cdot) ) = -\infty$.
\end{prop}

As a consequence of these results one can obtain a precise description of the domains $\mathcal{D}_{\mathbf{0}}$ and $\mathcal{D}_e$ in the region of nonnegative data. The following theorem can be found on \cite{CE2}.

\begin{teo}\label{descomp2} $\,$
\begin{enumerate}
\item [i.] Assume $u \in C_D([0,1])$ is nonnegative and such that $U^u$ is globally defined and converges to $z$ as $t \rightarrow +\infty$. Then for $v \in C_D([0,1])$ we have that
\begin{enumerate}
\item [$\bullet$] $\mathbf{0} \lneq v \lneq u \Longrightarrow U^v$ is globally defined and converges to $\mathbf{0}$ as $t \rightarrow +\infty$.
\item [$\bullet$] $u \lneq v \Longrightarrow U^v$ explodes in finite time.
\end{enumerate}
\item [ii.] For every nonnegative $u \in C_D([0,1])$ there exists $\lambda_c^u > 0$ such that for every $\lambda > 0$
\begin{enumerate}
\item [$\bullet$] $0 < \lambda < \lambda_c^u \Longrightarrow U^{\lambda u}$ is globally defined and converges to $\mathbf{0}$ as $t \rightarrow +\infty$.
\item [$\bullet$] $\lambda = \lambda_c^u \Longrightarrow U^{\lambda u}$ is globally defined and converges to $z$ as $t \rightarrow +\infty$.
\item [$\bullet$] $\lambda > \lambda_c^u \Longrightarrow U^{\lambda u}$ explodes in finite time.
\end{enumerate}
\end{enumerate}
\end{teo}

From this result we obtain the existence of an unstable manifold of the saddle point $z$ which is contained in the region of nonnegative initial data and shall be denoted by $\mathcal{W}^z_u$. It is $1$-dimensional, has nonempty intersection with both $\mathcal{D}_{\mathbf{0}}$ and $\mathcal{D}_e$ and joins $z$ with $\mathbf{0}$. By symmetry, a similar description also holds for the opposite unstable equilibrium $-z$. Figure \ref{fig1} depicts the decomposition \eqref{decomp12} together with the unstable manifolds $\mathcal{W}^{\pm z}_u$. \mbox{By exploiting the} structure of the remaining unstable equilibria given by Proposition \ref{equilibrios} one can verify for each of them the analogue of (ii) in Theorem \ref{descomp2}. This is detailed in the following proposition.

\psfrag{U1}{\vspace{-155pt}$U\equiv 0$}
\psfrag{U2}{\vspace{-45pt}$U \equiv \1$}
\psfrag{De}{\hspace{60pt}\vspace{30pt}$\mathcal{D}_e$}
\psfrag{D0}{$\mathcal{D}_0$}
\psfrag{Wu}{$\mathcal W_1^u$}
\psfrag{Ws}{$\mathcal W_1^s$}
\begin{figure}
	\centering
	\includegraphics[width=8cm]{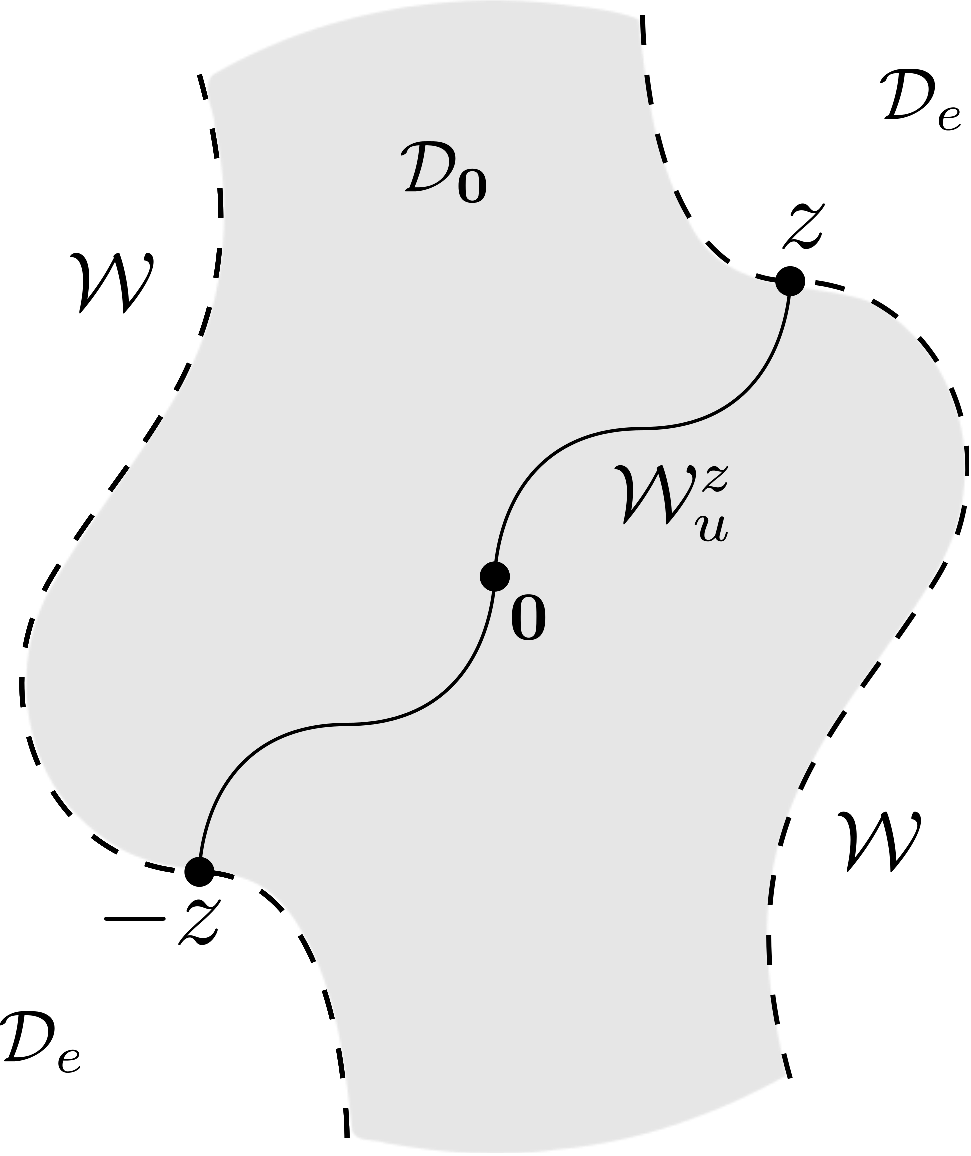}
	\caption{The phase diagram of equation \eqref{MainPDE}.}
	\label{fig1}
\end{figure}

\begin{prop}\label{descomp3} If $w \in C_D([0,1]) - \{\mathbf{0}\}$ is an equilibrium of the system then for every $\lambda > 0$ we have that
\begin{enumerate}
\item [$\bullet$] $0 < \lambda < 1 \Longrightarrow U^{\lambda w}$ is globally defined and converges to $\mathbf{0}$ as $t \rightarrow +\infty$.
\item [$\bullet$] $\lambda = 1 \Longrightarrow U^{\lambda w}$ is globally defined and satisfies $U^{\lambda w} \equiv w$.
\item [$\bullet$] $\lambda > 1 \Longrightarrow U^{\lambda w}$ explodes in finite time.
\end{enumerate}
\end{prop}

\begin{proof} Let us suppose that $w \equiv z^{(n)}$ for some $n \in \Z - \{0\}$. Then for any $\lambda > 0$ the solution to \eqref{MainPDE} with initial datum $\lambda w$ is given by the formula
$$
U^{\lambda w}(t,x) = \left\{ \begin{array}{rl} |n|^{\frac{2}{p-1}}U^{\text{sg}(n) \lambda z} (|n|^2t,|n|x - [|n|x] ) & \text{ if $[|n|x]$ is even} \\ \\ - |n|^{\frac{2}{p-1}}U^{\text{sg}(n) \lambda z}(|n|^2t,|n|x - [|n|x])& \text{ if $[|n|x]$ is odd}\end{array}\right.
$$ where $\text{sg}(n):=\frac{n}{|n|}$ and $U^{\pm \lambda z}$ is the solution to \eqref{MainPDE} with initial datum $\pm \lambda z$, respectively.

\noindent That is, $U^{\lambda w}$ is obtained from $U^{\lambda z}$ by performing the analogous procedure to the one explained in Proposition \ref{equilibrios} to obtain $w$ from $z$. Indeed, this follows from an argument similar in spirit to the one given for Proposition \ref{equilibrios} which exploits the facts that $U^{\lambda z}$ is symmetric, it verifies $U^{-\lambda z}=-U^{\lambda z}$ and also that it vanishes on the boundary of $[0,1]$. Having this formula for $U^{\lambda w}$, now the result follows at once from Theorem \ref{descomp2}.
\end{proof}

\section{Brownian sheet}

Throughout our work we consider perturbations of \eqref{MainPDE} given by additive white noise. This noise term can be regarded as the formal time derivative of a Brownian sheet process.
We say that a stochastic process $W=\{W(t,x) : (t,x) \in \R^+ \times [0,1]\}$ is a \textit{Brownian sheet} if it satisfies the following properties:
\begin{enumerate}
\item [i.] $W$ has continuous paths, i.e. $(t,x) \mapsto W{(t,x)}(\omega)$ is continuous for every $\omega \in \Omega$.
\item [ii.] $W$ is a centered Gaussian process with covariance structure given by
$$
\text{Cov}( W{(t,x)} , W{(s,y)} ) = (t \wedge s)(x \wedge y)
$$ for every $(t,x),(s,y) \in \R^+ \times [0,1]$.
\end{enumerate}
We refer to \cite{P,RY} for the construction of a such a process and a list of its basic properties, as well as the fundamentals of the theory of stochastic integration with respect to it.

\section{Definition of solution for the SPDE}

In this first part we study stochastic partial differential equations of the form
\begin{equation}\label{MainSPDE}
\left\{\begin{array}{rll}
\p_t X &= \p^2_{xx}X + f(X) + \varepsilon \dot{W}& \quad t>0 \,,\, 0<x<1 \\
X(t,0)&=X(t,1)=0 & \quad t>0 \\
X(0,x) &=u(x)
\end{array}\right.
\end{equation}where $\varepsilon > 0$ is some parameter, $u \in C_D([0,1])$ and $f: \R \rightarrow \R$ is a locally Lipschitz source. It is possible that such equations do not admit
strong solutions in the usual sense as these may not be globally defined but instead defined \textit{up to an explosion time}. In the following we review the usual definition of solution when the source is globally Lipschitz as well as formalize the idea of explosion and properly define the concept of solutions in the case of locally Lipschitz sources.

\subsection{Definition of strong solution for globally Lipschitz sources}

We begin by fixing a probability space $(\Omega,\F,P)$ in which we have defined a \mbox{Brownian sheet} $\{ W{(t,x)} : (t,x) \in \R^+ \times [0,1]\}$. For every $t \geq 0$ we define
$$
\G_t = \sigma( W{(s,x)} : 0 \leq s \leq t , x \in [0,1])
$$ and denote its augmentation by $\F_t$.\footnote{This means that $\F_t = \sigma( \G_t \cup \mathcal{N})$ where $\mathcal{N}$ denotes the class of all $P$-null sets of $\G_\infty = \sigma( \G_t : t \in \R^+)$.} The family $(\F_t)_{t \geq 0}$ constitutes a filtration on $(\Omega,\F)$.
A \textit{strong solution} of the equation \eqref{MainSPDE} on the probability space $(\Omega,\F,P)$ with \mbox{respect to} the Brownian sheet $W$
is a stochastic process
$$
X = \{ X{(t,x)} : (t,x) \in \R^+ \times [0,1]\}
$$ satisfying the following properties:

\begin{itemize}
\item [i.] $X$ has continuous paths taking values in $\R$.
\item [ii.] $X$ is adapted to the filtration $(\F_t)_{t \geq 0}$, i.e. for every $t \geq 0$ the mapping
$$
(\omega,x) \mapsto X{(t,x)}(\omega)
$$ is $\F_t \otimes \B([0,1])$-measurable.
\item [iii.] If $\Phi$ denotes the fundamental solution of the heat equation on the interval $[0,1]$ with homogeneous Dirichlet boundary conditions, which is given by the formula
$$
\Phi(t,x,y) = \frac{1}{\sqrt{4\pi t}} \sum_{n \in \Z} \left[ \exp\left( - \frac{(2n+y -x)^2}{4t} \right) - \exp\left( - \frac{(2n+y +x)^2}{4t} \right)\right],
$$ then $P$-almost surely we have
$$
\int_0^1 \int_{0}^{t} |\Phi(t -s,x,y) f(X(s,y))|dsdy  < +\infty \hspace{0,5cm} \,\forall\,\,\, 0\leq t < + \infty
$$
and
$$
X(t,x) = I_H(t,x)+ I_N(t,x)\hspace{0,5cm}\,\forall\,\,\, (t,x) \in \R^+ \times [0,1],
$$
where $I_H$ and $I_N$ are respectively defined by the formulas
$$
I_H(t,x) = \int_0^1 \Phi(t,x,y)u(y)dy
$$ and
$$
I_N(t,x) = \int_0^{t} \int_0^1 \Phi(t-s,x,y) \left(f(X(s,y))dyds + \varepsilon dW(s,y)\right).
$$
\end{itemize}
It is well known that if $f$ satisfies a global Lipschitz condition then for any initial datum $u \in C_D([0,1])$ there exists a unique strong solution to the equation \eqref{MainSPDE} on $(\Omega,\F,P)$. Furthermore, this strong solution satisfies the strong Markov property and also behaves as a weak solution in the sense described in the following lemma. See \cite{W} for details.

\begin{lema}\label{weaksol} If $X$ is a strong solution to \eqref{MainSPDE} with initial datum $u \in C_D([0,1])$ then for every $\varphi \in C^2((0,1)) \cap C_D([0,1])$ we have $P$-almost surely
$$
\int_0^1 X(t,x)\varphi(x)dx = I_H^\varphi(t) + I_N^\varphi(t) \hspace{0,5cm} \,\forall\,\,\, 0\leq t < + \infty
$$ where for each $t \geq 0$
$$
I_H^\varphi(t) = \int_0^1 u(x)\varphi(x)dx
$$ and
$$
I_N^\varphi(t) = \int_0^t\int_0^1 \left( \left(X(s,x)\varphi''(x) + f(X(s,x))\varphi(x)\right)dxds + \varepsilon \varphi(s,x)dW(s,x)\right).
$$
\end{lema}

\subsection{Solutions up to an explosion time}

Just as in the previous section we begin by fixing a probability space $(\Omega,\F,P)$ in which we have defined a \mbox{Brownian sheet} $\{ W{(t,x)} : (t,x) \in \R^+ \times [0,1]\}$ and consider its augmented generated filtration $(\F_t)_{t \geq 0}$. A \textit{solution up to an explosion time} of the equation \eqref{MainSPDE} on $(\Omega,\F,P)$ with respect to $W$
is a stochastic process $X = \{ X{(t,x)} : (t,x) \in \R^+ \times [0,1]\}$ satisfying the following properties:

\noindent \begin{itemize}
\item [i.] $X$ has continuous paths taking values in $\overline{\R}:=\R \cup \{\pm \infty\}$.
\item [ii.] $X$ is adapted to the filtration $(\F_t)_{t \geq 0}$.
\item [iii.] If we define $\tau^{(n)} := \inf\{t>0 : \|X{(t,\cdot)}\|_{\infty}=n\}$ then for every $n \in \N$ we have $P$-a.s.
$$
\int_0^1\int_{0}^{t\wedge{\tau^{(n)}}} |\Phi(t\wedge \tau^{(n)} -s,x,y)f(X(s,y))|dsdy  < +\infty \hspace{0,5cm} \,\forall\,\,\, 0\leq t < + \infty
$$
and
$$
X(t \wedge \tau^{(n)},x) = I_H^{(n)}(t,x)+ \eps I_N^{(n)}(t,x)\hspace{0,5cm}\,\forall\,\,\, (t,x) \in \R^+ \times [0,1],
$$
where
$$
I_H^{(n)}(t,x) = \int_0^1 \Phi(t \wedge \tau^{(n)},x,y)u(y)dy
$$ and
$$
I_N^{(n)}(t,x) =  \int_0^{t} \int_0^1 \mathbbm{1}_{\{s \leq \tau^{(n)}\}} \Phi(t\wedge \tau^{(n)} -s,x,y) \left( f(X(s,y))dyds + \varepsilon dW(s,y)\right)
$$ with $\Phi$ being the fundamental solution of the heat equation as before.
\end{itemize}

We call $\tau:=\lim_{n \rightarrow +\infty} \tau^{(n)}$ the {\em explosion time} for $X$. Let us notice that the assumption of \mbox{continuity of $X$} over $\overline{\R}$ implies that
\begin{enumerate}
\item [$\bullet$] $\tau = \inf \{ t > 0 : \|X{(t,\cdot)}\|_\infty =+\infty\}$
\item [$\bullet$] $\|X{(\tau^-,\cdot)}\|_\infty = \|X{(\tau,\cdot)}\|_\infty =+\infty\,\,\mbox{ on }\,\,\{ \tau < +\infty\}.$
\end{enumerate}
We stipulate that $X{(t,\cdot)}\equiv X(\tau,\cdot)$ for $t \geq \tau$ whenever $\tau < +\infty$ but we shall not assume that $\lim_{t \to +\infty} X{(t,\cdot)}$ exists if $\tau=+\infty$.
Furthermore, observe that since any initial datum $u \in C_D([0,1])$ verifies $\|u\|_\infty < +\infty$ we always have $P( \tau > 0) = 1$ and also that if $P(\tau = +\infty)=1$ then we are left with the usual definition of strong solution.

In can be shown that for $f \in C^1(\R)$ there exists a unique solution $X$ of \eqref{MainSPDE} up to an explosion time. Furthermore, if $f$ is globally Lipschitz then the solution is globally defined in the sense that $P(\tau = +\infty)=1$. Finally, it is possible to prove that this solution $X$ maintains the strong Markov property, i.e. if $\tilde \tau$ is a stopping time of $X$ then, conditional on $\tilde\tau<\tau$ and $X{(\tilde\tau,\cdot)}=w$, the future $\{ X{(t + \tilde \tau,\cdot)} \colon 0 < t<\tau-\tilde\tau\}$ is independent of the past $\{ X{(s,\cdot)} \colon 0 \leq s\le \tilde\tau \}$ and identical in law to the solution of \eqref{MainSPDE} with \mbox{initial datum $w$.} We refer to \cite{IM} for details.

\section{Freidlin-Wentzell estimates}\label{secLDP}

One of the main tools we shall use to study the solutions to \eqref{MainSPDE} are the large deviations estimates we briefly describe next. We refer to \cite{FJL,B1,SOW} for further details.

Let $X^{u,\ve}$ be the solution to the SPDE
\begin{equation}\label{MainSPDE2}
\left\{\begin{array}{rll}
\p_t X^{u,\ve} &= \p^2_{xx}X^{u,\ve} + f(X^{u,\ve}) + \varepsilon \dot{W}& \quad t>0 \,,\, 0<x<1 \\
X^{u,\ve}(t,0)&=X^{u,\ve}(t,1)=0 & \quad t>0 \\
X^{u,\ve}(0,x) &=u(x)
\end{array}\right.
\end{equation} where $u \in C_D([0,1])$ and $f: \R \to \R$ is bounded and satisfies a \textit{global} Lipschitz condition.
Let us also consider $X^{u}$ the unique solution to the deterministic equation
\begin{equation}\label{MainPDE2}
\left\{\begin{array}{rll}
\p_t X^u &= \p^2_{xx}X^u + f(X^u) & \quad t>0 \,,\, 0<x<1 \\
X^u(t,0)&=X^u(t,1)=0 & \quad t>0 \\
X^u(0,x) &=u(x).
\end{array}\right.
\end{equation}
Given $u \in C_D([0,1])$ and $T > 0$, we consider the metric space of continuous functions
$$
C_{D_u}([0,T] \times [0,1]) = \{ v \in C([0,T]\times[0,1]) : v(0,\cdot)=u \text{ and }v(\cdot,0)=v(\cdot,1)\equiv 0 \}
$$ with the distance $d_T$ induced by the supremum norm, i.e. for $v,w \in C_{D_u}([0,T]\times[0,1])$
$$
d_T(v,w) := \sup_{(t,x) \in [0,T]\times [0,1]} | v(t,x) - w(t,x) |,
$$ and define the rate function $I^u_T : C_{D_u}([0,T]\times [0,1]) \rightarrow [0,+\infty]$ by the formula
$$
I^u_T (\varphi) = \left\{ \begin{array}{ll} \displaystyle{\frac{1}{2} \int_0^T \int_0^1 |\p_t \varphi - \p_{xx} \varphi - f(\varphi)|^2} & \text{ if }\varphi \in W^{1,2}_2([0,T]\times[0,1]) \,,\,\varphi(0,\cdot) = u \\ \\ +\infty & \text{otherwise.}\end{array}\right.
$$
Here $W^{1,2}_2([0,T]\times[0,1])$ is the closure of $C^\infty([0,T] \times [0,1])$ with respect to the norm
$$
\| \varphi \|_{W^{1,2}_2} = \left( \int_0^T \int_0^1 \left[ |\varphi|^2 + |\p_t \varphi|^2 + |\p_x \varphi|^2 + |\p_{xx} \varphi|^2\right]\right)^\frac{1}{2},
$$ i.e. the Sobolev space of square-integrable functions defined on $[0,T]\times [0,1]$ with one square-integrable weak \mbox{time derivative} and two square-integrable weak space derivatives.

\begin{teo} The following estimates hold:
\begin{enumerate}
\item [i.] For any $\delta > 0$, $h > 0$ there exists $\varepsilon_0 > 0$ such that
\begin{equation}\label{LDP1}
P\left( d_T ( X^{u,\varepsilon}, \varphi ) < \delta \right) \geq e^{- \frac{ I^u_T(\varphi) + h }{\varepsilon^2}}
\end{equation} for all $0 < \varepsilon < \varepsilon_0$, $u \in C_D([0,1])$ and $\varphi \in C_{D_u}([0,T]\times [0,1])$.
\item [ii.] For any $\delta > 0$, $h > 0$, $s_0 > 0$ there exists $\varepsilon_0 > 0$ such that
\begin{equation}\label{LDP2}
\sup_{u \in C_D([0,1])} P\left( d_T ( X^{u,\varepsilon}, J^u_T(s)) \geq \delta  \right) \leq e^{-\frac{s-h}{\varepsilon^2}}
\end{equation} for all $0 < \varepsilon < \varepsilon_0$ and $0 < s \leq s_0$, where
$$
J^u_T(s) = \{ \varphi \in C_{D_u}([0,T]\times [0,1]) : I^u_T(\varphi) \leq s \}.
$$
\item [iii.] For any $\delta > 0$ there exist $\varepsilon_0 > 0$ and $C > 0$ such that
\begin{equation}
\label{grandes1}
\sup_{u \in C_D([0,1])} P\left( d_T \left( X^{u,\varepsilon},X^{u}\right) > \delta \right) \leq e^{-\frac{C}{\varepsilon^2}}
\end{equation}for all $0 < \varepsilon < \varepsilon_0$.
\end{enumerate}
\end{teo}

The first and second estimates are equivalent to those obtained in \cite{FJL}, except for the uniformity in the initial datum. This uniformity can be obtained as in \cite{B1} by exploiting the fact that $f$ is bounded and Lipschitz. On the other hand, the last estimate is in fact implied by the second one. Indeed, if $V^{\mathbf{0},\varepsilon}$ and $V^{\mathbf{0}}$ respectively denote the solutions to \eqref{MainSPDE2} and \eqref{MainPDE2} with initial datum $\mathbf{0}$ and source term $f \equiv 0$, then (iii) is obtained from (ii) upon noticing that there exists $K > 0$ depending on $f$ such that for any $u \in C_D([0,1])$
\begin{equation}\label{LDP3}
d_T \left( X^{u,\varepsilon},X^{u}\right) \leq e^{KT} d_T ( V^{\mathbf{0},\varepsilon}, V^{\mathbf{0}} )
\end{equation} and that given $\delta > 0$ there exists $s_0 > 0$ such that
\begin{equation}\label{LDP4}
\{ d_T ( V^{\mathbf{0},\varepsilon}, V^{\mathbf{0}} ) > \delta \} \subseteq \left\{ d_T ( V^{\mathbf{0},\varepsilon}, \tilde{J}_T(s_0)) > \frac{\delta}{2}\right\}
\end{equation} where for $s \geq 0$ we set
$$
\tilde{J}^{\mathbf{0}}_T(s) = \{ \varphi \in C_{D_{\mathbf{0}}}([0,T]\times [0,1]) : \tilde{I}^{\mathbf{0}}_T(\varphi) \leq s \}
$$ and $\tilde{I}^{\mathbf{0}}_T$ is the rate function obtained by setting $u=\mathbf{0}$ and $f \equiv 0$ in the definition above.
The estimate in \eqref{LDP3} is obtained as in \cite{B1} whereas the inclusion in \eqref{LDP4} follows from the fact that the level sets $\tilde{J}^{\mathbf{0}}_T(s)$ are compact for all $s \geq 0$ and also that the rate function $\tilde{I}^{\mathbf{0}}_T$ vanishes only at $V^{\mathbf{0}}$.

\section{Truncations of the potential and localization}\label{trunca}

The large deviations estimates given on Section \ref{secLDP} demand a global Lipschitz condition on the source term $f$ which is unfortunately not satisfied for our model. Even though large
deviations estimates have been obtained for systems with locally Lipschitz sources (see for example \cite{FJL,A}), these always rely on some sort of a priori control on the growth of solutions. Hence, we cannot hope to obtain similar results for our system in the study of the explosion time. Nonetheless, the use of localization techniques will help us solve this problem and allow us to take advantage of the estimates on Section \ref{secLDP}.
\mbox{In the next lines} we give details about the localization procedure to be employed in the study of our system.

For every $n \in \N$ let $G^{(n)} : \R \longrightarrow \R$ be a smooth function such that
\begin{equation*}\label{gtruncada}
G^{(n)}(u) = \left\{\begin{array}{ll}
\frac{|u|^{p+1}}{p+1} &\,\,\text{if}\,\,|u| \leq n\\
0 &\,\,\text{if}\,\,|u| \geq 2n
\end{array}\right.
\end{equation*}
and consider the potential $S^{(n)}$ given by the formula
\begin{equation*}
S^{(n)}(v)= \left\{\begin{array}{ll}\displaystyle{\int_0^1 \left[\frac{1}{2} \left(\frac{dv}{dx}\right)^2 - G^{(n)}(v)\right]} & \text{ if }v \in H^1_0 ((0,1)) \\ \\ +\infty & \text{ otherwise.}\end{array}\right.
\end{equation*} For every $u \in C_D([0,1])$ there exists a unique solution $U^{(n),u}$ to the partial differential equation
\begin{equation*}
\p_t U = - \frac{\partial S^{(n)}}{\partial \varphi}(U)
\end{equation*} with initial datum $u$. Since the source $g_n:=\left(G^{(n)}\right)'$ is globally Lipschitz, this solution $U^{(n),u}$ is globally defined
and describes the same trajectory as the solution to \eqref{MainPDE} starting at $u$
until $\tau^{(n),u}$, the escape time from the ball
$$
B_n :=\{ v \in C_D([0,1]) : \| v \|_\infty \leq n \}.
$$
In the same way, for each $\varepsilon > 0$ there exists a
unique solution $U^{(n),u,\varepsilon}$ to the stochastic partial differential equation
\begin{equation}\label{eqtruncada}
\p_t U = - \frac{\partial S^{(n)}}{\partial \varphi}(U) + \varepsilon \dot{W}
\end{equation} with initial datum $u$ and it is globally defined. Moreover, since for $n \leq m$ the functions $G_n$ and $G_m$ coincide on $B_n$ by uniqueness of the solution we have that $U^{(n),u,\varepsilon}$ and $U^{(m),u,\varepsilon}$ coincide until the escape from $B_n$. Therefore, if we write
\begin{equation*}
\label{tiempostau}
\tau^{(n), u}_\varepsilon = \inf \{ t \geq 0 : \|U^{(n),u,\varepsilon}(t,\cdot)\|_\infty
\geq n \}, \qquad \tau_\varepsilon^u:=\lim_{n \rightarrow +\infty} \tau^{(n),u}_\varepsilon,
\end{equation*}
then for $t < \tau_\varepsilon^u$ we have that $U^{u,\varepsilon}(t) := \lim_{n \rightarrow +\infty} U^{(n),\,u,\,\varepsilon}(t)$ is well defined and constitutes the solution to \eqref{MainSPDE} until the explosion time $\tau^u_\varepsilon$ with initial datum $u$. Let us observe that for each $n \in \N$ this solution $U^{u,\,\varepsilon}$ coincides with $U^{(n),\,u,\,\varepsilon}$ until the escape from $B_n$. Furthermore, each $U^{(n),u,\varepsilon}$ is a positive recurrent Markov process which almost surely hits any open set in $C_D([0,1])$ in a finite time. Finally, since each $g_n$ is bounded and Lipschitz we have that for every $n \in \N$ the family $\left(U^{(n),u,\varepsilon}\right)_{\varepsilon > 0}$ satisfies the large deviations estimates given in Section \ref{secLDP}.
Hereafter, whenever we refer to the solution of \eqref{MainSPDE} we shall mean the solution constructed in this particular manner.

\section{Main results}

Our purpose in this first part of the thesis is to study the asymptotic behavior as $\varepsilon \rightarrow 0$ of $U^{u,\varepsilon}$, the solution to the equation \eqref{MainSPDE}, for the different initial data $u \in C_{D}([0,1])$. \mbox{We present throughout this section} the main results we have obtained in this regard. From now onwards we shall write $P_u$ to denote the law of the stochastic process $U^{u,\varepsilon}$. Whenever the initial datum is made clear in this way we shall often choose to drop the superscript $u$ from the remaining notation for simplicity purposes.

Our first result is concerned with the continuity of the explosion time for initial data in the domain of explosion $\mathcal{D}_e$. In this case one expects the stochastic and deterministic systems to both exhibit a similar behavior for $\varepsilon > 0$ sufficiently small, since then the noise will not be able to grow fast enough so as to overpower the quickly exploding source term.
We show this to be truly the case for $u \in \mathcal{D}_e$ such that \mbox{$U^u$ remains bounded from one side.}

\bigskip
\noindent \textbf{Theorem I}. Let $\mathcal{D}^*_e$ be the set of those $u \in \mathcal{D}_e$ such that $U^u$ explodes only through one side, i.e. $U^{u}$ remains bounded either from below or above until its explosion time $\tau^u$. Then given $\delta > 0$ and a bounded set $\mathcal{K} \subseteq \mathcal{D}_e^*$ at a positive distance from $\p \mathcal{D}^*_e$ there exists a constant $C > 0$ such that
$$
\sup_{u \in \mathcal{K}} P_u ( |\tau_\varepsilon - \tau| > \delta ) \leq e^{-\frac{C}{\varepsilon^2}}.
$$

The main differences in behavior between the stochastic and deterministic systems appear for initial data in $\mathcal{D}_\0$, where metastable behavior is observed. According to the characterization of metastability for stochastic processes given in the articles \cite{CGOV} and \cite{GOV}, metastable behavior is given by two facts: the time averages of the process remain stable until an abrupt transition occurs and then a different value is attained; furthermore, the time of this transition is unpredictable in the sense that, when suitably rescaled, it should have an exponential distribution. We manage to establish this description rigorously for our system whenever $1 < p < 5$, where $p$ is the parameter in the source term of \eqref{MainSPDE}. This rigorous description is contained in the remaining results.
We begin by defining for each $\varepsilon > 0$ the scaling coefficient
\begin{equation}\label{defibeta}
\beta_{\varepsilon}= \inf \{ t \geq 0 : P_{\mathbf{0}} ( \tau_\ve > t ) \leq e^{-1} \}
\end{equation} and show that the family $(\beta_\varepsilon)_{\varepsilon > 0}$ verifies $\lim_{\varepsilon \rightarrow 0} \varepsilon ^{2}\log\beta_{\varepsilon} = \Delta$, where $\Delta := 2(S(z) - S(\mathbf{0}))$. In fact, we shall prove the following stronger statement which details the asymptotic order of magnitude of $\tau^u_\varepsilon$ for initial data $u \in \mathcal{D}_\0$.

\bigskip
\noindent \textbf{Theorem II}. Given $\delta > 0$ and a bounded set $\mathcal{K} \subseteq \mathcal{D}_{\mathbf{0}}$ at a positive distance from $\p \mathcal{D}_{\mathbf{0}}$ we have
$$
\lim_{\varepsilon \rightarrow 0} \left[ \sup_{u \in \mathcal{K}} \left| P_u \left( e^{\frac{\Delta - \delta}{\varepsilon^2}} < \tau_\varepsilon < e^{\frac{\Delta + \delta}{\varepsilon^2}}\right)-1\right|\right]=0.
$$

Next we show the asymptotic loss of memory of $\tau^u_\varepsilon$ for initial data $u \in \mathcal{D}_\0$.

\bigskip
\noindent \textbf{Theorem III}. Given $\delta > 0$ and a bounded set $\mathcal{K} \subseteq \mathcal{D}_{\mathbf{0}}$ at a positive distance from $\p \mathcal{D}_{\mathbf{0}}$ we have for any $t > 0$
$$
\lim_{\varepsilon \rightarrow 0} \left[ \sup_{u \in \mathcal{K}} \left| P_{u} (\tau_{\varepsilon} > t\beta_{\varepsilon}) - e^{-t} \right| \right] = 0.
$$

Finally, we show the stability of time averages of continuous functions evaluated along paths of the process starting in $\mathcal{D}_{\mathbf{0}}$, i.e. they remain close to the value of the \mbox{function at $\mathbf{0}$.} These time averages are taken along intervals of length going to infinity and times may be taken as being almost (in a suitable scale) the explosion time. This tells us that, up until the explosion time, the system spends most of its time in a small \mbox{neighborhood of $\mathbf{0}$.}

\bigskip
\noindent \textbf{Theorem IV}. There exists a sequence $(R_\varepsilon)_{\varepsilon > 0}$ with $\lim_{\varepsilon \rightarrow 0} R_\varepsilon = +\infty$ and $\lim_{\varepsilon \rightarrow 0} \frac{R_\varepsilon}{\beta_\varepsilon} = 0$ such that given $\delta > 0$ for any bounded set $\mathcal{K} \subseteq \mathcal{D}_{\mathbf{0}}$ at a positive \mbox{distance from $\mathcal{W}$} we have
$$
\lim_{\varepsilon \rightarrow 0} \left[ \sup_{u \in \mathcal{K}} P_u \left( \sup_{0 \leq t \leq \tau_\varepsilon - 3R_\varepsilon}\left| \frac{1}{R_\varepsilon}\int_t^{t+R_\varepsilon} f(U^{\varepsilon}(s,\cdot))ds - f(\mathbf{0})\right| > \delta \right) \right] = 0
$$ for any bounded continuous function $f: C_D([0,1]) \rightarrow \R$.

Theorem I is proved in Chapter 2, the remaining results are proved in \mbox{Chapters 4 and 5.}
Perhaps the proof of Theorem I is where one can find the most differences with other works in the literature dealing with similar problems. In these works, the analogue of \mbox{Theorem I} can be obtained as a direct consequence of the large deviations estimates for the system. However, since in our case Theorem I particularly focuses on trajectories of the process as it escapes any bounded domain, the estimates on \mbox{Section \ref{secLDP}} will not be of any use for the proof. Thus, a new approach is needed, one which is different from previous approaches in the literature and does not rely on large deviations estimates. The remaining results were established in \cite{B1,MOS} for the tunneling time in an infinite-dimensional double-well potential model, i.e. the time the system takes to go from one well to the bottom of the other one. Our proofs are similar to the ones found in these references, although we have the additional difficulty of dealing with solutions which are not globally defined.

\section{Resumen del Capítulo 1}

En este primer capítulo introducimos las nociones y conceptos preliminares para poder estudiar nuestro problema. La EDP con blow-up que vamos a considerar es
$$
\left\{\begin{array}{rll}
\p_t U &= \p^2_{xx}U + g(U) & \quad t>0 \,,\, 0<x<1 \\
U(t,0)& =0 & \quad t>0 \\
U(t,1) & = 0 & \quad t>0 \\
U(0,x) &=u(x) & \quad 0<x<1
\end{array}\right.
$$
donde $g : \R \rightarrow \R$ viene dada por $g(u)=u|u|^{p-1}$ para $p > 1$ y $u$ pertenece al espacio
$$
C_{D}([0,1])= \{ v \in C([0,1]) : v(0)=v(1)=0 \}.
$$ Dicha ecuación puede reformularse como
$$
\p_t U = - \frac{\partial S}{\partial \varphi} (U)
$$ donde el potencial $S$ es el funcional en $C_D([0,1])$ dado por
$$
S(v) = \left\{ \begin{array}{ll} \displaystyle{\int_0^1 \left[\frac{1}{2} \left(\frac{dv}{dx}\right)^2 - \frac{|v|^{p+1}}{p+1}\right]}
              & \text{ si $v \in H^1_0((0,1))$}
                   \\ \\ +\infty & \text{ en caso contrario.}\end{array}\right.
$$ El origen $\mathbf{0} \in C_D([0,1])$ es el único equilibrio estable del sistema y es, de hecho, asintóticamente estable. Corresponde al único mínimo local del potencial $S$. Existe también una familia de equilibrios inestables del potencial $S$, todos ellos puntos de ensilladura. Entre estos equilibrios inestables existe un único equilibrio que es no negativo, $z$. Puede mostrarse que $z$ es de hecho estrictamente positivo en $(0,1)$ y tanto de mínima energía como norma entre los equilibrios inestables. Además, $C_D([0,1])$ puede descomponerse en tres partes:
$$
C_D([0,1]) = \mathcal{D}_{\mathbf{0}} \cup \mathcal{W} \cup \mathcal{D}_e
$$ donde $\mathcal{D}_{\mathbf{0}}$ denota la variedad estable del origen, $\mathcal{W}$ es la unión de todas las variedades estables de los equilibrios inestables y
$\mathcal{D}_e$ constituye el dominio de explosión del sistema, i.e. el conjunto de todos aquellos datos iniciales $u$ para los cuales el sistema explota en un tiempo finito $\tau^u$. Puede verse que tanto $\mathcal{D}_{\mathbf{0}}$ como $\mathcal{D}_e$ son conjuntos abiertos y que $\mathcal{W}$ es la frontera común que los separa. Además, existe una variedad inestable $\mathcal{W}^z_u$ del punto de ensilladura $z$ contenida en la región de datos no negativos. La misma es $1$-dimensional, tiene intersección no vacía tanto con $\mathcal{D}_{\mathbf{0}}$ como con $\mathcal{D}_e$ y une a $z$ con $\mathbf{0}$. Por simetría, una descripción análoga también vale para el equilibrio inestable opuesto $-z$. La Figura \ref{fig1} describe esta descomposición.

Las perturbaciones estocásticas que consideramos son de la forma
\begin{equation}\label{formalSPDEresumen}
\p_t U^{\ve} = - \nabla S + \varepsilon \dot{W}
\end{equation} donde $W$ es una sábana Browniana. Definimos formalmente el concepto de solución a una ecuación de este tipo, lo cual excede el marco tradicional ya que las mismas podrían no estar definidas globalmente sino hasta un tiempo de explosión $\tau_\varepsilon$ finito. Estudiamos además dos propiedades importantes de las soluciones a este tipo de ecuaciones: la propiedad fuerte de Markov y el principio de grandes desvíos para los sistemas truncados asociados.

Por último, terminamos el capítulo presentando los resultados que habremos de probar en los capítulos siguientes. Incluimos una breve descripción de los mismos aquí.

Nuestro primer resultado es con respecto a la continuidad del tiempo de explosión para datos iniciales en $\mathcal{D}_e$. En este caso uno espera que que los sistemas estocástico y determinístico exhiban ambos un comportamiento similar para $\varepsilon > 0$ suficientemente pequeño, ya que entonces el ruido no tendrá el tiempo suficiente como para crecer lo necesario para sobrepasar al término de la fuente que está explotando. Mostramos que esto es en efecto así para los casos en que $u \in \mathcal{D}_e$ es tal que la solución $U^u$ de \eqref{MainPDE} con dato inicial $u$ permanece acotada por un lado.

\noindent \textbf{Teorema I}. Sea $\mathcal{D}^*_e$ el conjunto de aquellos $u \in \mathcal{D}_e$ tales que $U^u$ explota sólo por un lado, i.e. $U^{u}$ permanece acotada ya sea inferior o superiormente hasta su tiempo de explosión $\tau^u$. Entonces dado $\delta > 0$ y un conjunto acotado $\mathcal{K} \subseteq \mathcal{D}_e^*$ a una distancia positiva de $\p \mathcal{D}^*_e$ existe $C > 0$ tal que
$$
\sup_{u \in \mathcal{K}} P ( |\tau_\varepsilon^u - \tau_0^u| > \delta ) \leq e^{-\frac{C}{\varepsilon^2}}.
$$ donde $\tau^u_\varepsilon$ denota el tiempo de explosión de $U^{u,\varepsilon}$, la solución de \eqref{formalSPDEresumen} con dato inicial $u$.

Las principales diferencias en comportamiento entre ambos sistemas surgen para datos iniciales en $\mathcal{D}_\0$, donde se presenta el fenómeno de metaestabilidad. De acuerdo con \cite{GOV}, el comportamiento metaestable viene dado por dos hechos: los promedios temporales del proceso permanecen estables hasta que ocurre una transición abrupta y luego un valor diferente se obtiene; más aún, el tiempo en que ocurre esta transición es impredecible en el sentido de que, bajo una normalización apropiada, debería tener una distribución exponencial. Logramos establecer esta descripción rigurosamente para nuestro sistema para los casos en que $1 < p < 5$, donde $p$ es el parámetro en el término no lineal de la fuente en \eqref{MainPDE}. Esta descripción rigurosa abarca los restantes resultados.

\noindent \textbf{Teorema II}. Dado $\delta > 0$ y un conjunto acotado $\mathcal{K} \subseteq \mathcal{D}_{\mathbf{0}}$ a una distancia positiva de $\p \mathcal{D}_{\mathbf{0}}$ tenemos
$$
\lim_{\varepsilon \rightarrow 0} \left[ \sup_{u \in \mathcal{K}} \left| P \left( e^{\frac{\Delta - \delta}{\varepsilon^2}} < \tau^u_\varepsilon < e^{\frac{\Delta + \delta}{\varepsilon^2}}\right)-1\right|\right]=0.
$$

\noindent \textbf{Teorema III}. Dado $\delta > 0$ y un conjunto acotado $\mathcal{K} \subseteq \mathcal{D}_{\mathbf{0}}$ a una distancia positiva de $\p \mathcal{D}_{\mathbf{0}}$ tenemos para cualquier $t > 0$
$$
\lim_{\varepsilon \rightarrow 0} \left[ \sup_{u \in \mathcal{K}} \left| P (\tau_{\varepsilon}^u > t\beta_{\varepsilon}) - e^{-t} \right| \right] = 0.
$$ donde para cada $\varepsilon > 0$ definimos el coeficiente de normalización $\beta_\varepsilon$ como
$$
\beta_{\varepsilon}= \inf \{ t \geq 0 : P_{\mathbf{0}} ( \tau_\ve > t ) \leq e^{-1} \}.
$$

\noindent \textbf{Teorema IV}. Existe una sucesión $(R_\varepsilon)_{\varepsilon > 0}$ con $\lim_{\varepsilon \rightarrow 0} R_\varepsilon = +\infty$ y $\lim_{\varepsilon \rightarrow 0} \frac{R_\varepsilon}{\beta_\varepsilon} = 0$ tal que dado $\delta > 0$ para cualquier conjunto acotado $\mathcal{K} \subseteq \mathcal{D}_{\mathbf{0}}$ a una distancia positiva de $\mathcal{W}$
$$
\lim_{\varepsilon \rightarrow 0} \left[ \sup_{u \in \mathcal{K}} P_u \left( \sup_{0 \leq t \leq \tau_\varepsilon - 3R_\varepsilon}\left| \frac{1}{R_\varepsilon}\int_t^{t+R_\varepsilon} f(U^{\varepsilon}(s,\cdot))ds - f(\mathbf{0})\right| > \delta \right) \right] = 0
$$ para cualquier función continua $f: C_D([0,1]) \rightarrow \R$.

\chapter{Asymptotic behavior of $\tau_\varepsilon^u$ for $u \in \mathcal{D}_e$}

In this chapter we investigate the continuity properties of the explosion time $\tau_\varepsilon^u$ for initial data in the domain of explosion $\mathcal{D}_e$.
Our purpose is to show that under suitable conditions on the initial datum $u \in \mathcal{D}_e$ the explosion time $\tau_\varepsilon^u$ of the stochastic system converges in probability to the deterministic explosion time $\tau^u$. To make these conditions more precise, let us consider the sets of initial data in $\mathcal{D}_e$ which explode only through $+\infty$ or $-\infty$, i.e.
$$
\mathcal{D}_e^+ = \left\{ u \in \mathcal{D}_e : \inf_{(t,x) \in [0,\tau^u) \times [0,1]}  U^u(t,x) > -\infty \right\}
$$
and
$$
\mathcal{D}_e^- = \left\{ u \in \mathcal{D}_e : \sup_{(t,x) \in [0,\tau^u) \times [0,1]} U^u(t,x) < +\infty \right\}.
$$ Notice that $\mathcal{D}_e^+$ and $\mathcal{D}_e^-$ are disjoint and also that they satisfy the relation $\mathcal{D}_e^- = - \mathcal{D}_e^+$. Furthermore, we shall see below that $\mathcal{D}_e^+$ is an open set. Let us write $\mathcal{D}_e^*:= \mathcal{D}_e^+ \cup \mathcal{D}_e^-$. The result we are to prove is the following.

\begin{teo}\label{contexp} For any bounded set $\mathcal{K} \subseteq \mathcal{D}_e^*$ at a positive distance from $\p \mathcal{D}_e^*$ and $\delta > 0$ there exists a constant $C > 0$ such that
$$
\sup_{u \in \mathcal{K}} P_u ( |\tau_\varepsilon - \tau| > \delta ) \leq e^{- \frac{C}{\varepsilon^2}}.
$$
\end{teo} We shall split the proof of Theorem \ref{contexp} into two parts: proving first a lower bound and then an upper bound for $\tau_\varepsilon$. The first one is a consequence of the continuity of solutions to \eqref{MainSPDE} with respect to $\ve$ on intervals where the deterministic solution \mbox{remains bounded.} The precise estimate is contained in the following proposition.

\begin{prop}\label{convergenciainferior0} For any bounded set $\mathcal{K} \subseteq \mathcal{D}_e$ and $\delta > 0$ there exists a constant $C > 0$ such that
\begin{equation}\label{convergenciainferior}
\sup_{u \in \mathcal{K}} P_u ( \tau_\varepsilon < \tau - \delta ) \leq e^{- \frac{C}{\varepsilon^2}}.
\end{equation}
\end{prop}

\begin{proof} Let us observe that by Proposition \ref{G.2} we have that $\inf_{u \in \mathcal{K}} \tau^u > 0$ so that we may assume without loss of generality that $\tau^u > \delta$ for all $u \in \mathcal{K}$.
Now, for each $u \in \mathcal{D}_e$ let us define the quantity
$$
M_u := \sup_{0 \leq t \leq \max\{0, \tau^u - \delta \}} \|U^{u}(t,\cdot)\|_\infty.
$$ By resorting to Proposition \ref{G.2} once again, we obtain that the application $u \mapsto M_u$ is both upper semicontinuous and finite on $\mathcal{D}_e$ and hence, with the aid of Propositions \ref{G.1} and \ref{A.2}, we conclude that $M:= \sup_{u \in \mathcal{K}} M_u < +\infty$. Similarly, since the mapping $u \mapsto \tau^u$ is continuous and finite on $\mathcal{D}_e$ (see Corollary \ref{contdetexp} below for proof of this fact) we also obtain that $\mathcal{T}:= \sup_{u \in \mathcal{K}} \tau^u < +\infty$. Hence, for $u \in \mathcal{K}$ we get
$$
P_u (\tau^u_\varepsilon < \tau^u - \delta ) \leq P_u \left( d_{\tau^u - \delta}\left(U^{M_u+1,\varepsilon},U^{M_u+1}\right) > \frac{1}{2}\right) \leq P_u \left( d_{\mathcal{T} - \delta}\left(U^{M+1,\varepsilon},U^{M+1}\right) > \frac{1}{2}\right).
$$ By the estimate \eqref{grandes1} we conclude \eqref{convergenciainferior}.
\end{proof}

To establish the upper bound we consider for each $u \in \mathcal{D}_e^+$ the process
$$
Z^{u,\varepsilon} := U^{u,\varepsilon} - V^{\mathbf{0},\varepsilon}
$$ where $U^{u,\varepsilon}$ is the solution of \eqref{MainSPDE} with initial datum $u$ and $V^{\mathbf{0},\varepsilon}$ is the solution of \eqref{MainSPDE2} with source term $f \equiv 0$ and initial datum $\mathbf{0}$ constructed from the same Brownian sheet as $U^{u,\varepsilon}$. Let us observe that $Z^{u,\varepsilon}$ satisfies the random partial differential equation
\begin{equation}\label{randomPDE}
\left\{\begin{array}{rll}
\p_t Z^{u,\varepsilon} &= \p^2_{xx}Z^{u,\varepsilon} + g(Z^{u,\varepsilon} - V^{\mathbf{0},\varepsilon}) & \quad t>0 \,,\, 0<x<1 \\
Z^{u,\varepsilon}(t,0)&=Z^{u,\varepsilon}(t,1)=0 & \quad t>0 \\
Z^{u,\varepsilon}(0,x) &=u(x).
\end{array}\right.
\end{equation}Furthermore, since $V^{\mathbf{0},\varepsilon}$ is globally defined and remains bounded on finite time intervals, we have that $Z^{u,\varepsilon}$ and $U^{u,\varepsilon}$ share the same explosion time. Hence, to  obtain the desired upper bound on $\tau^u_\varepsilon$ we may study the behavior of $Z^{u,\varepsilon}$. The advantage of this approach is that, in general, the behavior of $Z^{u,\varepsilon}$ will be easier to understand than that of $U^{u,\varepsilon}$. Indeed, each realization of $Z^{u,\varepsilon}$ is the solution of a partial differential equation which one can handle by resorting to standard arguments in PDE theory.

Now, a straightforward calculation using the mean value theorem shows that whenever $\| V^{\mathbf{0},\varepsilon} \|_\infty < 1$ the process $Z^{u,\varepsilon}$ satisfies the inequality
\begin{equation}\label{eqcomparacion}
\p_t Z^{u,\varepsilon} \geq \p^2_{xx} Z^{u,\varepsilon} + g(Z^{u,\varepsilon}) - h|Z^{u,\varepsilon}|^{p-1} - h
\end{equation} where $h := p2^{p-1}\| V^{\mathbf{0},\varepsilon} \|_\infty > 0$. Therefore, in order to establish the upper bound on $\tau_\varepsilon^u$ one may consider for $h > 0$
the solution $\underline{Z}^{(h),u}$ to the equation
\begin{equation}\label{randomPDE2}
\left\{\begin{array}{rll}
\p_t \underline{Z}^{(h),u} &= \p^2_{xx}\underline{Z}^{(h),u} + g(\underline{Z}^{(h),u}) - h|\underline{Z}^{(h),u}|^{p-1} - h  & \quad t>0 \,,\, 0<x<1 \\
\underline{Z}^{(h),u}(t,0)&=\underline{Z}^{(h),u}(t,1)=0 & \quad t>0 \\
\underline{Z}^{(h),u}(0,x) &=u(x).
\end{array}\right.
\end{equation} and obtain a convenient upper bound for the explosion time of this new process valid for every $h$ sufficiently small. If we also manage to show that for $h$ suitably small the process $\underline{Z}^{(h),u}$ explodes through $+\infty$, then the fact that $Z^{u,\varepsilon}$ is a supersolution to \eqref{randomPDE2} will yield the desired upper bound on the explosion time of $Z^{u,\varepsilon}$, if $\| V^{\mathbf{0},\varepsilon} \|_\infty$ remains small enough. To show this, however, we will need to impose the additional condition that $u \in \mathcal{D}_e^+$. Lemma \ref{expestimate} below contains the proper estimate on $\underline{\tau}^{(h),u}$, the explosion time of $\underline{Z}^{(h),u}$.

\begin{defi} For each $h \geq 0$ we define the potential $\underline{S}^{(h)}$ on $C_D([0,1])$ associated to the equation \eqref{randomPDE2} by the formula
$$
\underline{S}^{(h)}(v)\left\{ \begin{array}{ll} \displaystyle{\int_0^1 \left[\frac{1}{2} \left(\frac{dv}{dx}\right)^2 - \frac{|v|^{p+1}}{p+1} + hg(v) + hv\right]} & \text{ if $v \in C_D \cap H^1_0([0,1])$ }\\ \\ +\infty & \text{ otherwise.}\end{array}\right.
$$ Notice that $\underline{S}^{(0)}$ coincides with our original potential $S$. Moreover, it is easy to check that for all $h \geq 0$ the potential $\underline{S}^{(h)}$ satisfies all properties established for $S$ in the Appendix.
\end{defi}

\begin{lema}\label{expestimate} Given $\delta > 0$ there exists $M > 0$ such that:
\begin{enumerate}
\item [i.] For every $0 \leq h < 1$ any $u \in C_D([0,1])$ such that $\underline{S}^{(h)}(u) \leq - \frac{M}{2}$ verifies $\underline{\tau}^{(h),u} < \frac{\delta}{2}$.
\item [ii.] Given $K > 0$ there exist constants $\rho_{M,K}, h_{M,K} > 0$ depending only on $M$ and $K$ such that any $u \in C_D([0,1])$ satisfying $S(u) \leq - M$ and $\| u\|_\infty \leq K$ verifies
$$
\sup_{v \in B_{\rho_{M,K}}(u)} \underline{\tau}^{(h),v} < \delta
$$ for all $0 \leq h < h_{M,K}$.
\end{enumerate}
\end{lema}

\begin{proof} Given $\delta > 0$ let us begin by showing that (i) holds for an appropriate choice of $M$. Thus, for fixed $M > 0$ and $0 \leq h < 1$, let $u \in C_D([0,1])$ be such that $\underline{S}^{(h)}(u) \leq - \frac{M}{2}$ and consider the application $\phi^{(h),u}: [0, \tau^{(h),u}) \rightarrow \R^+$ given by the formula
$$
\phi^{(h),u}(t) = \int_0^1 \left(\underline{Z}^{(h),u}(t,x)\right)^2dx.
$$ It is simple to verify that $\phi^{(h),u}$ is continuous and that for any $t_0 \in (0,\tau^{(h),u})$ it satisfies
\begin{equation}\label{eqpoten1}
\frac{d\phi^{(h),u}}{dt}(t_0) \geq - 4\underline{S}^{(h)}(u^{(h)}_{t_0}) + 2 \int_0^1 \left[ \left(\frac{p-1}{p+1}\right) |u^{(h)}_{t_0}|^{p+1} - h \left(\frac{p+2}{p}\right)|u^{(h)}_{t_0}|^p - h|u^{(h)}_{t_0}|\right]
\end{equation} where we write $u^{(h)}_{t_0}:=\underline{Z}^{(h),u}(t_0,\cdot)$ for convenience. Hölder's inequality reduces \eqref{eqpoten1} to
\begin{equation}\label{eqpoten2}
\frac{d\phi^{(h),u}}{dt}(t_0) \geq - 4\underline{S}^{(h)}(u^{(h)}_{t_0}) + 2\left[ \left(\frac{p-1}{p+1}\right) \|u^{(h)}_{t_0}\|_{L^{p+1}}^{p+1} - h\left(\frac{p+2}{p}\right)\|u^{(h)}_{t_0}\|_{L^{p+1}}^{p} - h\|u^{(h)}_{t_0}\|_{L^{p+1}}\right].
\end{equation} Observe that, by definition of $\underline{S}^{(h)}$ and the fact that the map $t \mapsto \underline{S}^{(h)}(u^{(h)}_t)$ is decreasing, we obtain the inequalities
$$
\frac{M}{2} \leq -\underline{S}^{(h)}(u^{(h)}_{t_0}) \leq \frac{1}{p+1}\|u^{(h)}_{t_0}\|_{L^{p+1}}^{p+1} + h\|u^{(h)}_{t_0}\|_{L^{p+1}}^p + h\|u^{(h)}_{t_0}\|_{L^{p+1}}
$$ from which we deduce that by taking $M$ sufficiently large one can force $\|u^{(h)}_{t_0}\|_{L^{p+1}}$ to be large enough so as to guarantee that
$$
\left(\frac{p-1}{p+1}\right) \|u^{(h)}_{t_0}\|_{L^{p+1}}^{p+1} - h \left(\frac{p+2}{p}\right)\|u^{(h)}_{t_0}\|_{L^{p+1}}^p - h\|u^{(h)}_{t_0}\|_{L^{p+1}} \geq \frac{1}{2}\left(\frac{p-1}{p+1}\right) \|u^{(h)}_{t_0}\|_{L^{p+1}}^{p+1}
$$ is satisfied for any $0 \leq h < 1$. Therefore, we see that if $M$ sufficiently large then for all $0 \leq h < 1$ the application $\phi^{(h),u}$ satisfies
\begin{equation}\label{eqpoten3}
\frac{d\phi^{(h),u}}{dt}(t_0) \geq 2M + \left(\frac{p-1}{p+1}\right) \left(\phi^{(h),u}(t_0)\right)^{\frac{p+1}{2}}
\end{equation} for every $t_0 \in (0,\tau^{(h),u})$, where to obtain \eqref{eqpoten3} we have once again used Hölder's inequality and the fact that the map $t \mapsto \underline{S}^{(h)}(u^{(h)}_t)$ is decreasing. Now, it is not hard to show that the solution $y$ of the ordinary differential equation
$$
\left\{\begin{array}{l} \dot{y} = 2M + \left(\frac{p-1}{p+1}\right) y^{\frac{p+1}{2}} \\ y(0) \geq 0 \end{array}\right.
$$ explodes before time $$
T = \frac{\delta}{4} + \frac{2^{\frac{p+1}{2}}(p+1)}{(p-1)^2(M\delta)^{\frac{p-1}{2}}}.
$$ Indeed, either $y$ explodes before time $\frac{\delta}{4}$ or $\tilde{y}:= y( \cdot + \frac{\delta}{4})$ satisfies
$$
\left\{\begin{array}{l} \dot{\tilde{y}} \geq \left(\frac{p-1}{p+1}\right) \tilde{y}^{\frac{p+1}{2}} \\ \tilde{y}(0) \geq \frac{M\delta}{2} \end{array}\right.
$$ which can be seen to explode before time
$$
\tilde{T}=\frac{2^{\frac{p+1}{2}}(p+1)}{(p-1)^2(M\delta)^{\frac{p-1}{2}}}
$$ by performing the standard integration method. If $M$ is taken sufficiently large then $T$ can be made strictly smaller than $\frac{\delta}{2}$ which, by \eqref{eqpoten3}, implies that $\tau^{(h),u} < \frac{\delta}{2}$ as desired.

Now let us show statement (ii). Given $K > 0$ let us take $M > 0$ as above and consider $u \in C_D([0,1])$ satisfying $S(u) \leq -M$ and $\|u\|_\infty \leq K$. Using Propositions \ref{S.1} and \ref{Lyapunov} adapted to the system \eqref{randomPDE2} we may find $\rho_{M,K} > 0$ sufficiently small so as to guarantee that for some small $0 < t_{u} < \frac{\delta}{2}$ any $v \in B_{\rho_{M,K}}(u)$ satisfies
$$
\underline{S}^{(h)}(\underline{Z}^{(h),v}(t_{u},\cdot)) \leq \underline{S}^{(h)}(u) +\frac{M}{4}
$$ for all $0 \leq h < 1$. Notice that this is possible since the constants appearing in Propositions \ref{S.1} adapted to this context are independent from $h$ provided that $h$ remains bounded. These constants still depend on $\| u \|_\infty$ though, so that the choice of $\rho_{M,K}$ will inevitably depend on both $M$ and $K$. Next, let us take $0 < h_{M,K} < 1$ so as to guarantee that $\underline{S}^{(h)}(u) \leq - \frac{3M}{4}$ for every $0 \leq h < h_{M,K}$. Notice that, since $\underline{S}^{(h)}(u) \leq S(u) + h(K^p + K),$ it is possible to choose $h_{M,K}$ depending only on $M$ and $K$. Thus, for any $v \in B_{\rho_{M,K}}(u)$ we obtain $\underline{S}^{(h)}(\underline{Z}^{(h),v}(t_{u},\cdot)) \leq - \frac{M}{2}$ which, by the choice of $M$, implies that $\tau^{(h),v} < t_u + \frac{\delta}{2} < \delta$. This concludes the proof.
\end{proof}

Let us observe that the system $\overline{Z}^{(0),u}$ coincides with $U^u$ for every $u \in C_D([0,1])$. Thus, by the previous lemma we obtain the following corollary.

\begin{cor}\label{contdetexp} The application $u \mapsto \tau^u$ is continuous on $\mathcal{D}_e$.
\end{cor}

\begin{proof} Given $u \in \mathcal{D}_e$ and $\delta > 0$ we show that there exists $\rho > 0$ such that for all $v \in B_\rho(u)$ we have
$$
-\delta + \tau^u < \tau^v < \tau^u + \delta.
$$ To see this we first notice that by Proposition \ref{G.2} there exists $\rho_1 > 0$ such that $-\delta + \tau^u < \tau^v$ for any $v \in B_{\rho_1}(u)$. On the other hand, by (i) in Lemma \ref{expestimate} we may take $M, \tilde{\rho_2} > 0$ such that $\tau^{\tilde{v}} < \delta$ for any $\tilde{v} \in B_{\tilde{\rho_2}}(\tilde{u})$ with $\tilde{u} \in C_D([0,1])$ such that $S(\tilde{u}) \leq - M$. For this choice of $M$ by Proposition \ref{caract} we may find some $0 <t_M < t^u$ such that $S( U^u(t_M,\cdot) ) \leq -M$ and using Proposition \ref{G.2} we may take $\rho_2 > 0$ such that $U^v(t_M,\cdot) \in B_{\tilde{\rho_2}}(U^u(t_M,\cdot))$ for any $v \in B_{\rho_2}(u)$. This implies that $\tau^v < t_M + \delta < t^u + \delta$ for all $v \in B_{\rho_2}(u)$ and thus by taking $\rho = \min\{ \rho_1,\rho_2\}$ we obtain the result.
\end{proof}

The following two lemmas provide the necessary tools to obtain the uniformity in the upper bound claimed in Theorem \ref{contexp}.

\begin{lema}\label{lemacontsup} Given $M > 0$ and $u \in \mathcal{D}_e$ let us define the quantities
$$
\mathcal{T}_M^u = \inf\{ t \in [0,\tau^u) : S( U^u(t,\cdot) ) < - M \} \hspace{1cm}\text{ and }\hspace{1cm}\mathcal{R}_M^u = \sup_{0 \leq t \leq \mathcal{T}_M^u} \| U^u(t,\cdot)\|_\infty.
$$ Then the applications $u \mapsto \mathcal{T}_M^u$ and $u \mapsto \mathcal{R}_M^u$ are both upper semicontinuous on $\mathcal{D}_e$.
\end{lema}

\begin{proof} We must see that the sets $\{ \mathcal{T}_M < \alpha\}$ and $\{ \mathcal{R}_M < \alpha \}$ are open in $\mathcal{D}_e$ for all $\alpha > 0$.
But the fact that $\{ \mathcal{T}_M < \alpha\}$ is open follows at once from Proposition \ref{S.1} and $\{ \mathcal{R}_M < \alpha\}$ is open by Proposition \ref{G.2}.
\end{proof}

\begin{lema}\label{lemacontsup2} For each $u \in \mathcal{D}_e^+$ let us define the quantity
$$
\mathcal{I}^u:= \inf_{(t,x) \in [0,\tau^u) \times [0,1]}  U^u(t,x).
$$ Then the application  $u \mapsto I^u$ is lower semicontinuous on $\mathcal{D}_e^+$.
\end{lema}

\begin{proof} Notice that $\mathcal{I}^u \geq 0$ for any $u \in \mathcal{D}_e^+$ since $U^u(t,0)=U^u(t,1)=0$ for all $t \in [0,\tau^u)$. Therefore, it will suffice to show that the sets $\{ \alpha < \mathcal{I} \}$ are open in $\mathcal{D}^+_e$ for every $\alpha < 0$. With this purpose in mind, given $\alpha < 0$ and $u \in \mathcal{D}_e^+$ such that $\alpha < \mathcal{I}^u$, take $\beta_1,\beta_2 < 0$ such that $\alpha < \beta_1 < \beta_2 < \mathcal{I}^u$ and let $y$ be the solution to the ordinary differential equation
\begin{equation}\label{contsup2eq}
\left\{\begin{array}{l} \dot{y} = - |y|^p \\ y(0) = \beta_2.\end{array}\right.
\end{equation} Define $t_\beta := \inf \{ t \in [0,t_{max}^y) : y(t) < \beta_1 \}$, where $t_{max}^y$ denotes the explosion time of $y$. Notice that by the lower semicontinuity of $S$ for any $M > 0$ we have
$S( U^u( \mathcal{T}^u_M, \cdot) ) \leq -M$ and thus, by Lemma \ref{expestimate}, we may choose $M$ such that
\begin{equation}\label{contsupeq4}
\sup_{v \in B_\rho( U^u( \mathcal{T}^u_M,\cdot) )} \tau^v < t_\beta
\end{equation} for some small $\rho > 0$. Moreover, if $\rho < \mathcal{I}^u - \beta_2$ then every $v \in B_\rho( U^u( \mathcal{T}^u_M,\cdot) )$ satisfies $\inf_{x \in [0,1]} v(x) \geq \beta_2$ so that $U^v$ is in fact a supersolution to the equation \eqref{contsup2eq}. By \eqref{contsupeq4} this implies that $v \in \mathcal{D}_e^+$ and $\mathcal{I}^v \geq \beta_1 > \alpha$.
On the other hand, by Proposition \ref{G.2} we may take $\delta > 0$ sufficiently small so that for every $w \in B_\delta(u)$ we have $\mathcal{T}^u_M < \tau^w$ and
$$
\sup_{t \in [0,\mathcal{T}^u_M]} \| U^w(t,\cdot) - U^u(t,\cdot)\|_\infty < \rho.
$$ Combined with the previous argument, this yields the inclusion $B_\delta(u) \subseteq \mathcal{D}_e^+ \cap \{ \alpha < \mathcal{I} \}$. In particular, this shows that $\{ \alpha < \mathcal{I} \}$ is open and thus concludes the proof.
\end{proof}

\begin{obs} The preceding proof shows, in particular, that the set $\mathcal{D}_e^+$ is open.
\end{obs}
The conclusion of the proof of Theorem \ref{contexp} is contained in the next proposition.

\begin{prop}\label{convsup}
For any bounded set $\mathcal{K} \subseteq \mathcal{D}_e^*$ at a positive distance from $\p \mathcal{D}^*_e$ and $\delta > 0$ there exists a constant $C > 0$ such that
\begin{equation}\label{convergenciasuperior}
\sup_{u \in \mathcal{K}} P_u ( \tau_\varepsilon > \tau + \delta ) \leq e^{- \frac{C}{\varepsilon^2}}.
\end{equation}
\end{prop}

\begin{proof} Since $\mathcal{D}_e^- = - \mathcal{D}_e^+$ and $U^{-u}= - U^u$ for $u \in C_D([0,1])$, without any loss of generality we may assume that $\mathcal{K}$ is contained in $\mathcal{D}_e^+$.
Let us begin by noticing that for any $M > 0$
$$
\mathcal{T}_M := \sup_{u \in \mathcal{K}} \mathcal{T}_M^u < +\infty  \hspace{1cm}\text{ and }\hspace{1cm}\mathcal{R}_M:= \sup_{u \in \mathcal{K}} \mathcal{R}^u_M < +\infty.
$$ Indeed, by Propositions \ref{G.1} and \ref{A.2} we may choose $t_0 > 0$ sufficiently small so that the orbits $\{ U^{u}(t,\cdot) : 0 \leq t \leq t_0 , u \in \mathcal{K} \}$ remain uniformly bounded and the family $\{ U^{u}(t_0,\cdot) : u \in \mathcal{K} \}$ is contained in a compact set $\mathcal{K}' \subseteq \mathcal{D}_e^+$ at a positive distance \mbox{from $\p \mathcal{D}^+_e$.} But then we have
$$
\mathcal{T}_M \leq t_0 + \sup_{u \in \mathcal{K}'} \mathcal{T}_M^u \hspace{1cm}\text{ and }\hspace{1cm}\mathcal{R}_M \leq \sup_{0 \leq t \leq t_0, u \in \mathcal{K}} \| U^u(t,\cdot) \|_\infty + \sup_{u \in \mathcal{K}'} \mathcal{R}^u_M $$ and both right hand sides are finite due to Lemma \ref{lemacontsup} and the fact that $\mathcal{T}_M^u$ and $\mathcal{R}_M$ are both finite for each $u \in \mathcal{D}_e$ by Proposition \ref{caract}. Similarly, by Lemma \ref{lemacontsup2} we also have
$$
\mathcal{I}_{\mathcal{K}}:= \inf_{u \in \mathcal{K}} \mathcal{I}^u > - \infty.
$$
Now, for each $u \in \mathcal{K}$ and $\varepsilon > 0$ by the Markov property we have for any $\rho > 0$
\begin{equation}\label{descompexp}
P_u ( \tau_\varepsilon > \tau + \delta ) \leq P( d_{\mathcal{T}_M}( U^{(\mathcal{R}_M+1),u,\varepsilon}, U^{(\mathcal{R}_M+1),u}) > \rho ) + \sup_{v \in B_{\rho}(U^u( \mathcal{T}^u_M, \cdot))} P_v ( \tau_\varepsilon > \delta).
\end{equation} The first term on the right hand side is taken care of by \eqref{grandes1} so that in order to show \eqref{convergenciasuperior} it only remains to deal with the second term by choosing $M$ and $\rho$ appropriately. The argument given to deal with this term is similar to that of the proof of \mbox{Lemma \ref{lemacontsup2}.} Let $y$ be the solution to the ordinary differential equation
\begin{equation}\label{convsupeq2}
\left\{\begin{array}{l} \dot{y} = - |y|^p - |y|^{p-1} - 1\\ y(0) = \mathcal{I}_{\mathcal{K}} - \frac{1}{2}.\end{array}\right.
\end{equation} Define $t_{\mathcal{I}} := \inf \{ t \in [0,t_{max}^y) : y(t) < \mathcal{I}_{\mathcal{K}} - 1 \}$, where $t_{max}^y$ denotes \mbox{the explosion time of $y$.}
By Lemma \ref{expestimate}, we may choose $M$ such that
\begin{equation}\label{convsupeq4}
\sup_{v \in B_{\rho_M}( U^u( \mathcal{T}^u_M,\cdot) )} \tau^{(h),v} < \min\{\delta, t_\mathcal{I}\}
\end{equation} for all $0 \leq h < h_M$, where $\rho_M > 0$ and $h_M > 0$ are suitable constants. The key observation here is that, since $\mathcal{R}_M < +\infty$, we may choose these constants so as not to depend on $u$ but rather on $M$ and $\mathcal{R}_M$ themselves. Moreover, if $\rho_M < \frac{1}{2}$ then every $v \in B_{\rho_M}( U^u( \mathcal{T}^u_M,\cdot) )$ satisfies $\inf_{x \in [0,1]} v(x) \geq \mathcal{I}_{\mathcal{K}} - \frac{1}{2}$ so that $\underline{Z}^{(h),v}$ is in fact a supersolution to the equation \eqref{convsupeq2} for all $0 \leq h < \min\{h_M,1\}$. By \eqref{convsupeq4} the former implies that $\underline{Z}^{(h),v}$ explodes through $+\infty$ and that it remains bounded from below by $\mathcal{I}_{\mathcal{K}} - 1$ until its explosion time which, by \eqref{convsupeq4}, is smaller than $\delta$. In particular, we see that if $\|V^{\mathbf{0},\varepsilon}\|_\infty < \min \{ 1, \frac{h_M}{p2^{p-1}}\}$ then $Z^{v,\varepsilon}$ explodes before $\underline{Z}^{(h),v}$ does, so that we
have that $\tau_\varepsilon < \delta$ under such conditions. Hence, we conclude that
$$
\sup_{v \in B_{\rho_M}(U^u( \mathcal{T}^u_M, \cdot))} P_v ( \tau_\varepsilon > \delta) \leq P \left( \sup_{t \in [0,\delta]} \|V^{\mathbf{0},\varepsilon}(t,\cdot) \|_\infty \leq \min \left\{ 1, \frac{h_M}{p2^{p-1}} \right\} \right)
$$ which, by recalling the estimate \eqref{grandes1}, gives the desired control on the second term in the right hand side of \eqref{descompexp}. Thus, by taking $\rho$ equal to $\rho_M$ in \eqref{descompexp}, we obtain the result.

\end{proof}

This last proposition in fact shows that for $\delta > 0$ and a given bounded set $\mathcal{K} \subseteq \mathcal{D}_e^*$ at a positive distance from $\p \mathcal{D}^*_e$ there exist constants $M,C > 0$ such that
$$
\sup_{u \in \mathcal{K}} P_u ( \tau_\varepsilon > \mathcal{T}_M^u + \delta ) \leq e^{- \frac{C}{\varepsilon^2}}.
$$ By exploiting the fact $\mathcal{T}_M < +\infty$ for every $M > 0$ we obtain the following useful corollary.

\begin{cor}\label{exploacot} For any bounded set $\mathcal{K} \subseteq \mathcal{D}_e^*$ at a positive distance from $\p \mathcal{D}_e^*$ there exist constants $\tau^*, C > 0$ such that
$$
\sup_{u \in \mathcal{K}} P_u ( \tau_\varepsilon > \tau^*) \leq e^{- \frac{C}{\varepsilon^2}}.
$$
\end{cor}

\newpage
\section{Resumen del Capítulo 2}

En este capítulo investigamos la continuidad del tiempo de explosión $\tau_\varepsilon^u$ para datos iniciales en el dominio de $\mathcal{D}_e$. Mostramos que el tiempo de explosión $\tau_\varepsilon^u$ del sistema estocástico converge en probabilidad al tiempo de explosión determinístico $\tau^u$ uniformemente sobre compactos de $\mathcal{D}_e^*$, el conjunto de aquellos datos iniciales $u \in \mathcal{D}_e$ para los cuales la solución $U^{u}$ explota por un único lado, i.e. permanece acotada en alguna dirección (inferior o superiormente) hasta el tiempo de explosión. Más precisamente, tenemos el siguiente resultado.

\noindent \textbf{Teorema}. Para cualquier conjunto acotado $\mathcal{K} \subseteq \mathcal{D}_e^*$ a una distancia positiva de $\p \mathcal{D}_e^*$ y $\delta > 0$ existe una constante $C > 0$ tal que
$$
\sup_{u \in \mathcal{K}} P_u ( |\tau_\varepsilon - \tau| > \delta ) \leq e^{- \frac{C}{\varepsilon^2}}.
$$

Dividimos la demostración de este resultado en dos partes, las cotas \eqref{convergenciainferior} y \eqref{convergenciasuperior}.

La cota inferior \eqref{convergenciainferior} es una consecuencia directa de la estimación de grandes desvíos \eqref{grandes1}. En efecto, como la solución $U^u$ permanece acotada en $[0,\tau^u - \delta)$, la condición $\tau^u_\varepsilon < \tau^u - \delta$ implica que los sistemas estocástico y determinístico deben necesariamente separarse antes de tiempo $\tau^u - \delta$ y, por lo tanto, que lo mismo debe suceder para los sistemas truncados para los que se tiene \eqref{grandes1}.

La demostración de la cota superior \eqref{convergenciasuperior} consiste en estudiar el proceso $Z^{u,\varepsilon}$ definido como la solución del problema \eqref{randomPDE}. Dicho proceso posee el mismo tiempo de explosión que $U^{u,\varepsilon}$ pero es más sencillo de tratar debido a que cada realización del mismo resuelve una ecuación diferencial en derivadas parciales que puede estudiarse mediante técnicas usuales. Además, dado $0 < h < 1$ es posible mostrar que para cada realización $\omega$ en un conjunto $\Omega_h^\varepsilon$ con probabilidad que tiende a uno cuando $\varepsilon \rightarrow 0$ la trayectoria $Z^{u,\varepsilon}(\omega)$ es supersolución de \eqref{randomPDE2}.
A partir de esto, la estrategia que adoptamos para probar \eqref{convergenciasuperior} si $u \in D_e^+$ (ver \eqref{Deplus}) es mostrar que para un conjunto de realizaciones $\Omega^{u,\varepsilon} \subseteq \Omega_h^\varepsilon$ con probabilidad que tiende a uno (exponencialmente rápido en $\frac{1}{\varepsilon^2}$) cuando $\varepsilon \rightarrow 0$ suceden dos cosas:
\begin{enumerate}
\item [i.] La solución de \eqref{randomPDE2} explota antes de tiempo $\tau^u + \delta$.
\item [ii.] La solución de \eqref{randomPDE2} explota por $+\infty$, i.e. permanece acotada inferiormente hasta el tiempo de explosión.
\end{enumerate}
Se sigue de la descripción anterior que para toda realización en $\Omega^{u,\varepsilon}$ el tiempo de explosión de $Z^{u,\varepsilon}$ (y por lo tanto $\tau^u_\varepsilon$) es menor a $\tau^u + \delta$, lo cual \eqref{grandes1} implica \eqref{convergenciasuperior} a partir de \eqref{grandes1}. Para mostrar (i) utilizamos técnicas de ecuaciones similares a las que figuran en \cite{CE1} y (ii) se deduce del hecho de que $u \in \mathcal{D}_e^+$. Por simetría se obtienen los mismos resultados para $\mathcal{D}_e^-$. Por último, la uniformidad sobre compactos se obtiene a partir de los resultados en el Apéndice.

\chapter{Construction of an auxiliary domain}\label{dominioauxiliar}

As suggested in \cite{GOV}, to study the behavior of the explosion time for initial data in $\mathcal{D}_\mathbf{0}$ it is convenient to consider an auxiliary bounded domain $G$ satisfying the conditions stated in the Introduction. By doing so we can then reduce our problem to a \mbox{simpler one:} characterizing the escape from this domain. This is simpler because for it we may assume that the source term $g$ in \eqref{MainPDE} is Lipschitz, as the escape only depends on the behavior of the system while it remains inside a bounded region. It is then that the large deviations estimates of Section \ref{secLDP} can be applied. To succeed in the construction of such a domain, we must first understand the behavior of exploding trajectories in the stochastic system. This is the purpose behind the following results.

\begin{lema}\label{compacidad} If $p < 5$ then for any $a > 0$ the sets $\{ u \in \overline{\mathcal{D}_{\mathbf{0}}} : 0 \leq S(u) \leq a \}$ are bounded.
\end{lema}

\begin{proof} Let $a > 0$ and for $v \in \{ u \in \overline{\mathcal{D}_{\mathbf{0}}} : 0 \leq S(u) \leq a \}$ consider $\psi : \R_{\geq 0} \rightarrow \R_{\geq 0}$ given by the formula
$$
\psi(t):= \int_0^1 (U^v(t,\cdot))^2.
$$ A direct computation shows that for every $t_0 > 0$ the function $\psi$ satisfies
$$
\frac{d\psi(t_0)}{dt} = - 4 S( U^v(t,\cdot) ) + 2\left(\frac{p-1}{p+1}\right) \int_0^1 |U^v(t,\cdot)|^{p+1}.
$$ By Proposition \ref{Lyapunov} and Hölder's inequality we then obtain
$$
\frac{d\psi(t_0)}{dt} \geq - 4 a + 2\left(\frac{p-1}{p+1}\right) (\psi(t_0))^{\frac{p+1}{2}}
$$ which implies that $\psi(0) \leq B:= \left[2a\left(\frac{p+1}{p-1}\right)\right]^{\frac{2}{p+1}}$ since otherwise $\psi$ (and therefore $U^v$) would explode in finite time. Now, by the Gagliardo-Niremberg interpolation inequality (recall that $v$ is absolutely continuous since $S(v) < +\infty$)
$$
\| v \|_\infty^2 \leq C_{G-N} \| v \|_{L^2} \| \p_x v \|_{L^2},
$$ we obtain
$$
\int_0^1 |v|^{p+1} \leq \| v \|_{L^2}^2 \|v \|_{\infty}^{p-1} \leq C_{G-N} B^{\frac{p+3}{4}} \| \p_x v \|_{L^2}^{\frac{p-1}{2}} \leq C_{G-N} B^{\frac{p+3}{4}} (2a + \int_0^1 |v|^{p+1})^{\frac{p-1}{4}}
$$ which for $p < 5$ implies the bound
\begin{equation}\label{bound1}
\int_0^1 |v|^{p+1} \leq B':=\max \left\{ 2a , \left[C_{G-N} B^{\frac{p+3}{4}}2^{\frac{p-1}{4}}\right]^{\frac{4}{5-p}}\right\}.
\end{equation} Since $S(v) \leq a$ we see that \eqref{bound1} implies the bound $\| \p_x v \|_{L^2} \leq \sqrt{2B'}$ and therefore we conclude
$$
\| v \|_\infty \leq \sqrt{ C_{G-N} \sqrt{2BB'}}
$$ which shows that $\{ u \in \overline{\mathcal{D}_{\mathbf{0}}} : 0 \leq S(u) \leq a \}$ is bounded.
\end{proof}

\begin{obs} The proof of Lemma \ref{compacidad} is the only instance throughout our entire work in which the assumption $p < 5$ is used. If $p \geq 5$ is such that the sets $\{ u \in \overline{\mathcal{D}_{\mathbf{0}}} : 0 \leq S(u) \leq a \}$ remain bounded for every $a > 0$, then all of our results remain valid for this choice of $p$. As a matter of fact, we shall only require the weaker condition that there exists $\alpha > 0$ such that the set $\{ u \in \overline{\mathcal{D}_{\mathbf{0}}} : 0 \leq S(u) \leq S(z) + \alpha \}$ is bounded. However, determining the validity of this condition for arbitrary $p > 1$ does not seem to be an easy problem.
\end{obs}

\begin{prop}\label{dominio1} The potential $S$ satisfies $\displaystyle{\lim_{n \rightarrow +\infty} \left[ \inf_{u \in \partial B_n \cap \mathcal{D}_{\mathbf{0}}} S(u) \right] = +\infty.}$
\end{prop}

\begin{proof} Given $M > 0$, by Lemma \ref{compacidad} we may take $N \in \N$ such that for every $n \geq N$ the set $\{ u \in \mathcal{D}_{\mathbf{0}} : 0 \leq S(u) \leq M \}$ is contained in $B_n^\circ$. Now, let us consider $u \in \partial B_n \cap \mathcal{D}_\mathbf{0}$. Since $u \in \mathcal{D}_\mathbf{0}$, we know that $S(u) \geq 0$ by Proposition \ref{Lyapunov}. Therefore, if $u \in \partial B_n$ for $n \geq N$, in particular we have that $u \notin \{ u \in \mathcal{D}_{\mathbf{0}} : 0 \leq S(u) \leq M \}$ and thus it must be $S(u) > M$. This proves the claim.
\end{proof}

\begin{defi} Given $T > 0$ and $\varphi \in C_{D}([0,T] \times [0,1])$ we define the \textit{rate} $I(\varphi)$ of $\varphi$ by the formula
$$
I(\varphi) := I^{(n),\varphi(0,\cdot)}_T(\varphi)
$$ for any $n \in \N$ larger than $\| \varphi \|_{\infty}$, where $I^{(n),\varphi(0,\cdot)}_T$ denotes the rate function associated to the system \eqref{MainSPDE2} with $f=g_n$. Notice that $I(\varphi)$ does not depend on the choice of $n$.
\end{defi}

\begin{defi} We say that a function $\varphi \in C_D([0,T]\times [0,1])$ is \textit{regular} if both derivatives $\p_t \varphi$ and $\p^2_{xx} \varphi$ exist and belong to $C_D ([0,T]\times [0,1])$.
\end{defi}

\begin{prop}\label{costo} Given $T > 0$ for any $\varphi \in C_{D} \cap W^{1,2}_2 ([0,T] \times [0,1])$ such that $\p^2_{xx}\varphi(0,\cdot)$ exists and belongs to $C_D([0,1])$ we have that
\begin{equation}\label{cotainftasa}
I(\varphi) \geq 2 \left[\sup_{0 \leq T' \leq T} \left( S(\varphi(T',\cdot))-S(\varphi(0,\cdot))\right)\right].
\end{equation}
\end{prop}

\begin{proof}Let us begin by assuming that $\varphi$ is regular and take $N \in \N$ larger than $\|\varphi\|_\infty$. Using the identity $(x-y)^2 = (x+y)^2 -4xy$ for $x,y \in \R$ and $0 \leq T' \leq T$ we obtain that
\begin{align*}
I(\varphi) = I^{(N),\varphi(0,\cdot)}_T(\varphi) & = \frac{1}{2} \int_0^T \int_0^1 |\p_t \varphi - \p_{xx}^2 \varphi - g_N(\varphi)|^2 \geq \frac{1}{2} \int_0^{T'} \int_0^1 |\p_t \varphi - \p_{xx}^2 \varphi - g_N(\varphi)|^2 \\
\\
& = \frac{1}{2} \int_0^{T'} \int_0^1 \left[|\p_t \varphi + \p_{xx}^2 \varphi + g_N(\varphi)|^2 - 2\left( \p_{xx}^2 \varphi + g_N(\varphi)\right)\p_t \varphi\right]\\
\\
& =  \frac{1}{2} \int_0^{T'} \left[ \left( \int_0^1 |\p_t \varphi + \p_{xx}^2 \varphi + g_N(\varphi)|^2\right) + 2 \frac{d S^{(N)}( \varphi(t,\cdot) )}{dt}\right]\\
\\
& \geq 2 \left(S^{(N)}(\varphi(T',\cdot)) - S^{(N)}(\varphi(0,\cdot))\right) =2\left( S(\varphi(T',\cdot))-S(\varphi(0,\cdot))\right)
\end{align*} where the last equality follows from the fact that both $S^{(N)}$ and $S$ coincide inside $B_N$. Taking supremum on $T'$ yields the result in this particular case.

Now, if $\varphi$ is not necessarily regular then by \cite[Theorem~6.9]{FJL} we may take a sequence $(\varphi_n)_{n \in \N}$ of regular functions converging to $\varphi$ on $C_{D_{\varphi(0,\cdot)}}([0,T]\times[0,1])$ and also such that $\lim_{n \rightarrow +\infty} I(\varphi_n) = I(\varphi)$. The result in the general case then follows from the validity of \eqref{cotainftasa} for regular functions and the lower semicontinuity of $S$.
\end{proof}

In order to properly interpret the content of Proposition \ref{costo} we need to introduce the concept of \textit{quasipotential} for our system. We do so in the following definitions.

\begin{defi} Given $u,v \in C_D([0,1])$ a \textit{path from $u$ to $v$} is a continuous function $\varphi \in C_{D}([0,T] \times [0,1])$ for some $T > 0$ such that $\varphi(0,\cdot)=u$ and $\varphi(T,\cdot)=v$.
\end{defi}

\begin{defi} Given $u,v \in C_D([0,1])$ we define the \textit{quasipotential} $V(u,v)$ \mbox{from $u$ to $v$} by the formula
$$
V(u,v)= \inf \{ I(\varphi) : \varphi \text{ path from $u$ to $v$}\}.
$$ Furthermore, given a subset $B \subseteq C_D([0,1])$ we define the quasipotential from $u$ to $B$ as
$$
V(u,B):= \inf \{ V(u,v) : v \in B \}.
$$ We refer the reader to the Appendix for a review of some of the main properties of $V$ which shall be required throughout our work.
\end{defi}

In a limiting sense, made rigorous through the large deviations principle established in Section \ref{secLDP}, the quasipotential $V(u,v)$ represents the energy cost for the stochastic system to travel from $u$ to (an arbitrarily small neighborhood of) $v$. In light of all these definitions we see that the energy cost for the stochastic system starting from $\mathbf{0}$ to explode in a finite time while remaining inside $\mathcal{D}_\mathbf{0}$ is infinite. Indeed, combining Propositions \ref{dominio1}, \ref{costo} and \ref{A.4} we see that $\lim_{n \rightarrow +\infty} V(\mathbf{0},\partial B_{n} \cap \mathcal{D}_{\mathbf{0}})=+\infty$ which implies that a path from $\mathbf{0}$ to infinity lying inside $\mathcal{D}_0$ should have, at least formally, an infinite rate. Thus, were the stochastic system starting from $\mathbf{0}$ to explode, it would have to do so by stepping outside $\mathcal{D}_{\mathbf{0}}$ and \mbox{crossing $\mathcal{W}$.} In view of Proposition \ref{costo}, the system will typically wish
to cross $\mathcal{W}$ through $\pm z$ since the energy cost for performing such a feat is the lowest there. Hence, if we wish the problem of escaping the domain $G$ to capture the essential characteristics of the explosion phenomenon in the stochastic system (at least when starting from $\mathbf{0}$) then it is important to guarantee that the escape from this domain occurs by passing through (an arbitrarily small neighborhood of) $\pm z$. Not only this, but we also require that once the system escapes this domain $G$ then it explodes with overwhelming probability in a quick fashion, i.e. before a certain time $\tau^*$ which does not depend on $\varepsilon$. More precisely, we wish to consider a bounded domain $G \subseteq C_D([0,1])$ verifying the following properties:
\begin{cond}\label{assumpg}$\,$
\begin{enumerate}
\item [i.] There exists $r_{\mathbf{0}}>0$ such that $B_{2r_\mathbf{0}} \subseteq \mathcal{D}_{\mathbf{0}} \cap G$.
\item [ii.] There exists $c > 0$ such that $B_c \subseteq B_{r_\mathbf{0}}$ and for all
$v \in B_c$ the solution $U^{v}$ to \eqref{MainPDE} with initial datum $v$ is globally defined and converges to $\mathbf{0}$ without
escaping $B_{r_\mathbf{0}}$.
\item [iii.] There exists a closed subset $\p^{\pm z}$ of the boundary $\partial G$ which satisfies
\begin{enumerate}
\item [$\bullet$] $V(\mathbf{0},\partial G - \partial^{\pm z} ) >  V(\mathbf{0},\partial^{\pm z}) = V( \mathbf{0}, \pm z)$.
\item [$\bullet \bullet$] $\p^{\pm z}$ is contained in $\mathcal{D}_e^*$ and at a positive distance from its boundary.
\end{enumerate}
\end{enumerate}
\end{cond}

In principle, we have seen that such a domain is useful to study the behavior of the explosion time whenever the initial datum of the stochastic system is (close to) the origin. Nevertheless, as we shall later see, when starting inside $\mathcal{D}_\mathbf{0}$ the system will typically visit a small neighborhood of the origin before crossing $\mathcal{W}$ and thus such a choice of $G$ will also be suitable to study the explosion time for arbitrary initial data in $\mathcal{D}_\mathbf{0}$.

The construction of the domain $G$ is done as follows. Since $\mathcal{D}_{\mathbf{0}}$ is open we may choose $r_{\mathbf{0}} > 0$ such that $B_{3r_{\mathbf{0}}}$ is contained in $\mathcal{D}_{\mathbf{0}}$. Moreover, by the asymptotic stability of $\mathbf{0}$ we may choose $c > 0$ verifying (ii) in Conditions \ref{assumpg}. Now, given $\zeta_1 > 0$ by Lemma \ref{compacidad} we may take $n_0 \in \N$ such that $n_0 > 3r_{\mathbf{0}}$ and the set $\{ u \in \overline{\mathcal{D}_{\mathbf{0}}} : 0 \leq S(u) \leq S(z) + \zeta_1 \}$ is contained in the interior of the ball $B_{n_0-1}$. We then define the pre-domain $\tilde{G}$ as
\begin{equation}\label{predomain}
\tilde{G}:= B_{n_0} \cap \overline{\mathcal{D}_{\mathbf{0}}}.
\end{equation}
Notice that since both $B_{n_0}$ and $\overline{\mathcal{D}_\mathbf{0}}$ are closed sets we have that
$$
\partial \tilde{G} = \left( \mathcal{W} \cap B_{n_0}\right) \cup \left( \partial B_{n_0} \cap \mathcal{D}_\mathbf{0}\right)
$$ which, by the particular choice of $n_0$ and Proposition \ref{Lyapunov}, implies $\min_{u \in \p \tilde{G}} S(u) =S(z)$. By Propositions \ref{costo} and \ref{A.4} we thus obtain $V(\mathbf{0},\p \tilde{G}) \geq \Delta$. Next, if for $u \in C_D([0,1])$ we let $u^-$ denote the negative part of $u$, i.e. $u^- = \max \{ - u, 0 \}$, then since $z^- = \mathbf{0}$ we may find $\tilde{r}_z > 0$ such that $u^- \in \mathcal{D}_{\mathbf{0}}$ for any $u \in B_{\tilde{r}_z}(z)$.
Finally, if for $r > 0$ we write $B_{r}(\pm z) := B_{r}(z) \cup B_{r}(-z)$ and take $r_{z} > 0$ such that $r_z < \frac{\tilde{r}_z}{2}$, $B_{2r_z}(\pm z)$ is contained in the interior of $B_{n_0}$ and $z$ is the unique equilibrium point of the system lying inside $B_{r_z}(z)$, then we define our final domain $G$ as
$$
G= \tilde{G} \cup B_{r_z}(\pm z).
$$ Let us now check that this domain satisfies all the required conditions. We begin by noticing that (i) and (ii) in Conditions \ref{assumpg} are immediately satisfied by the \mbox{choice of $n_0$.}
Now, let us also observe that for any $r > 0$
\begin{equation}\label{lejosdelminimo}
\inf\{ S(u) : u \in \p \tilde{G} - B_{r}(\pm z)\} > S(z).
\end{equation} Indeed, if this were not the case then there would exist $(u_k)_{k \in \N} \subseteq \left[\mathcal{W}\cap B_{n_0} - B_{r}(\pm z)\right]$ such that $\lim_{k \rightarrow +\infty} S(u_k) = S(z)$. By Proposition \ref{G.1} we have that there exists $t_0 > 0$ sufficiently small satisfying
$$
\sup_{k \in \N} \left[\sup_{t \in [0,t_0]} \| U^{u_k}(t,\cdot) \|_\infty\right] < +\infty \hspace{1cm}\text{ and }\hspace{1cm} \inf_{k \in \N} \| U^{u_k}(t_0,\cdot) - (\pm z) \|_\infty > \frac{r}{2}
$$ and therefore by Proposition \ref{A.2} we may conclude that there exists a subsequence $(u_{k_j})_{j \in \N}$ such that $U^{u_{k_j}}(t_0,\cdot)$ converges to a limit $u_{\infty} \in C_D([0,1])$ as $j \rightarrow +\infty$. Since the potential $S$ is lower semicontinuous and $\mathcal{W}$ is both closed and invariant under the deterministic flow, by Proposition \ref{Lyapunov} we conclude that $u_\infty = \pm z$ which contradicts the fact that the sequence $(U^{u_{k_j}}(t_0,\cdot))_{j \in \N}$ is at a positive distance from these equilibriums. Hence, we obtain the validity of \eqref{lejosdelminimo}.
In particular, this implies that $V(\mathbf{0}, \p \tilde{G} - B_{r}(\pm z)) > \Delta$ for any $r > 0$. Let us then take $\zeta_2 > 0$ such that $\Delta + \zeta_2 < V(\mathbf{0},\p \tilde{G} - B_{\frac{r_z}{2}}(\pm z))$ and define
$$
\tilde{\p}^z:= \{ u \in \p B_{r_z}(z) \cap \overline{\mathcal{D}_e} : V(\mathbf{0},u) \leq \Delta + \zeta_2 \}.
$$
Notice that $d(\tilde{\p}^z, \mathcal{W}) > 0$. Indeed, if this were not the case we would have sequences $(u_k)_{k \in \N} \subseteq \mathcal{W}$ and $(v_k)_{k \in \N} \subseteq \tilde{\p}^z$ such that $\lim_{k \rightarrow +\infty} d(u_k,v_k)=0$. The growth estimates on the Appendix section then imply that there exists $t_1 > 0$ sufficiently small such that
$$
\lim_{k \rightarrow +\infty} d(U^{u_k}(t_1,\cdot),U^{v_k}(t_1,\cdot))=0\hspace{0.2cm}\text{and}\hspace{0.2cm} \frac{r_z}{2} < \inf_{k \in \N} d(U^{v_k}(t_1,\cdot),z) \leq \sup_{k \in \N} d(U^{v_k}(t_1,\cdot),z) <  2r_z.
$$ By Proposition \ref{A.2} we obtain that for some appropriate subsequence we have
$$
\lim_{j \rightarrow +\infty} U^{u_{k_j}}(t_1,\cdot) =\lim_{j \rightarrow +\infty} U^{v_{k_j}}(t_1,\cdot) = v_\infty.
$$ Observe that $v_\infty \in \mathcal{W} \cap B_{n_0} - B_{\frac{r_z}{2}}(\pm z)$ and thus that $v_\infty \in \p \tilde{G}- B_{\frac{r_z}{2}}(\pm z)$. Furthermore, by the lower semicontinuity of $V(\mathbf{0},\cdot)$ and the fact that the mapping $t \mapsto V(\mathbf{0},U^u(t,\cdot))$ is monotone decreasing for any $u \in C_D([0,1])$ (see the Appendix section for details), we obtain that $V(\mathbf{0},v_\infty) \leq \Delta + \zeta_2$ which, together with the previous observation, implies the contradiction $\Delta + \zeta_2 \geq V(\mathbf{0},\p \tilde{G} - B_{\frac{r_z}{2}}(\pm z))$. Hence, we see that $d(\tilde{\p}^z, \mathcal{W}) > 0$ and thus we may define
$$
\p^z = \left\{ u \in \p B_{r_z}(z) \cap \overline{\mathcal{D}_e} : d(u,\mathcal{W}) \geq \frac{d(\tilde{\p}^z,\mathcal{W})}{2} \right\}
$$ and set $\p^{\pm z}:= \p^z \cup (-\p^z)$. Since one can easily check that
$$
\p G = [ \p \tilde{G} - B_{r_z}(\pm z) ] \cup [\p B_{r_z}(\pm z) \cap \overline{D_e} ]
$$ we conclude that $V(\mathbf{0},\p G - \p^{\pm z}) \geq \Delta + \zeta_2$. On the other hand, by proceeding similarly to the proof of Lemma \ref{cotasuplema0} below, one can show that $V(\mathbf{0},\p^z)=V(\mathbf{0},\tilde{\p}^z)= V(\mathbf{0},\pm z)=\Delta$, from which one obtains that
$$
V(\mathbf{0},\p G - \p^{\pm z}) > V(\mathbf{0},\p^z) = V(\mathbf{0},\pm z).
$$ Furthermore, by the comparison principle and the choice of $\tilde{r}_z$ we have \mbox{$B_{\tilde{r}_z}(\pm z) \cap \mathcal{D}_e \subseteq \mathcal{D}_e^*$.} Therefore, since we clearly have $d(\p^{\pm z}, \mathcal{W}) > 0$ by definition of $\p^{\pm z}$, upon recalling that $\p^{\pm z} \subseteq \mathcal{D}_e$ and $r_z < \frac{\tilde{r}_z}{2}$ we see that $\p^{\pm z} \subseteq \mathcal{D}_e^+$ and $d(\p^{\pm z}, \p \mathcal{D}_e^*) \geq \min\{d(\p^{\pm z}, \mathcal{W}), \frac{\tilde{r}_z}{2}\} > 0$, so that condition (iii) also holds. See Figure \ref{fig2}.

\begin{figure}
	\centering
	\includegraphics[width=8cm]{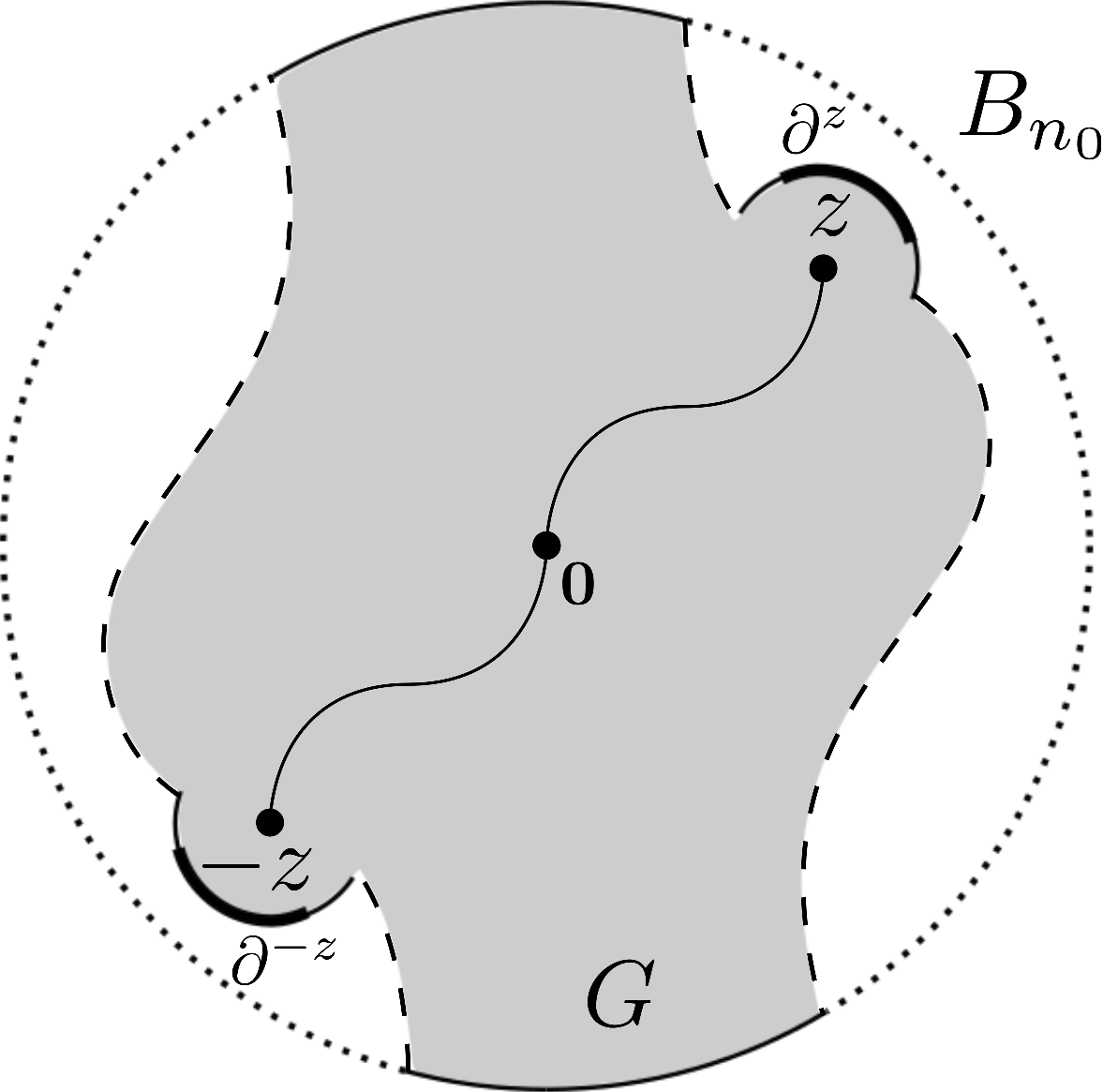}
	\caption{The auxiliary domain $G$}
	\label{fig2}
\end{figure}

\begin{obs}\label{obsequivG} Let us notice that, by Corollary \ref{exploacot}, ($\bullet \bullet$) in Conditions \ref{assumpg} implies that there exist constants $\tau^*,C > 0$ such that
$$
\sup_{u \in \p^{\pm z}} P_u (\tau_\varepsilon > \tau^* ) \leq e^{-\frac{C}{\varepsilon^2}}
$$ for all $\varepsilon > 0$ sufficiently small. Since ($\bullet$) guarantees that the escape from $G$ will typically take place through $\p^{\pm z}$, this tells us that both $\tau_\varepsilon$ and $\tau_\varepsilon(\p G)$ are asymptotically equivalent, so that it will suffice to study the escape from $G$ in order to establish each of our results.
\end{obs}

\newpage

\section{Resumen del Capítulo 3}

En este capítulo damos la construcción del dominio auxiliar acotado con las características discutidas en la Introducción. La razón por la cual llevamos a cabo tal construcción es porque nos permite reducir nuestro problema original al de estudiar cómo el sistema estocástico se escapa de dicho dominio. La ventaja de esto reside en que para estudiar este nuevo problema podemos asumir que la fuente es globalmente Lipschitz, dado que el escape depende únicamente del comportamiento del sistema mientras se encuentra en una región acotada. Es entonces que se pueden aplicar las estimaciones de grandes desvíos de la Sección \ref{secLDP}.

Para poder precisar qué condiciones debe cumplir nuestro dominio, definimos primero el quasipotencial $V$ siguiendo la Definición \ref{defipotentialV}. Dados $u,v \in C_D([0,1])$, el quasipotencial
$$
V(u,v)= \inf\{ I^u_T(\varphi) : \varphi \in C_D([0,T] \times [0,1]), \varphi(0,\cdot)= u, \varphi(T,\cdot) = v \}
$$ representa el costo para el sistema estocástico (en términos de las estimaciones de la Sección \ref{secLDP}) de ir desde $u$ hasta (un entorno arbitrariamente pequeño de) $v$. Asimismo, dado un conjunto $B \subseteq C_D([0,1])$, el quasipotencial
$$
V(u,B)= \inf \{ V(u,v) : v \in B \}
$$ representa el costo para el sistema de ir desde $u$ hasta $B$. Luego, el dominio $G$ de interés debe cumplir con las siguientes características:
\begin{enumerate}
\item [i.] Existe $r_{\mathbf{0}}>0$ tal que $B_{2r_\mathbf{0}} \subseteq \mathcal{D}_{\mathbf{0}} \cap G$.
\item [ii.] Existe $c > 0$ tal que $B_c \subseteq B_{r_\mathbf{0}}$ y para todo
$v \in B_c$ la solución $U^{v}$ de \eqref{MainPDE} con dato inicial $v$ está globalmente definida y converge a $\mathbf{0}$ sin escapar de $B_{r_\mathbf{0}}$.
\item [iii.] Existe un subconjunto cerrado $\p^{\pm z}$ de la frontera $\partial G$ que satisface
\begin{enumerate}
\item [$\bullet$] $V(\mathbf{0},\partial G - \partial^{\pm z} ) >  V(\mathbf{0},\partial^{\pm z}) = V( \mathbf{0}, \pm z)$.
\item [$\bullet \bullet$] $\p^{\pm z}$ está contenido en $\mathcal{D}_e^*$ y a una distancia positiva de su frontera.
\end{enumerate}
\end{enumerate}
La estabilidad asintótica del origen $\mathbf{0}$ garantiza que cualquier dominio que contenga a un entorno del origen va a satisfacer (i) y (ii). Por otro lado, mostramos que si el parámetro $p$ en la fuente satisface $1 < p < 5$ entonces existe un dominio acotado que además satisface (iii). Dicho dominio será la porción de $\mathcal{D}_0$ contenida en una bola de centro en $\mathbf{0}$ y radio apropiadamente grande, unida a entornos pequeños de los equilibrios inestables de mínima energía, $\pm z$ (ver Figura \ref{fig2}).
La principal dificultad a la hora de demostrar que el dominio así construido cumple con las características buscadas yace en la falta de compacidad del espacio $C_D([0,1])$. Sin embargo, nos fue posible lidiar con este problema apelando a algunos de los resultados contenidos en el Apéndice. Por último, la restricción $p < 5$ surge de limitaciones geométricas impuestas por el potencial $S$. Esperamos que una construcción similar sea posible aún en el caso $p \geq 5$, aunque no tenemos una prueba.

\chapter{The escape from $G$}\label{secescapedeg}

The behavior of the explosion time for initial data $u\in \mathcal{D}_{\mathbf{0}}$ is proved by
showing that, with overwhelming probability as $\varepsilon \to 0$, the
stochastic system describes the following path:
\begin{enumerate}
\item [(i)] The system enters the neighborhood of the origin $B_c$ before some finite time $T$ which does not depend on $\ve$.
\item [(ii)] Once inside $B_c$ the system remains in $G$ for a time of order
$e^{\frac{\Delta}{\ve^2}}$ and then escapes from $G$ through $\partial^{\pm z}$ since the
barrier imposed by the potential is the lowest there.
\item [(iii)] After escaping $G$ through $\partial^{\pm z}$ the system explodes
before some finite time $\tau^*$ which does not depend on $\ve$.
\end{enumerate}

The fact that the domain $G$ is bounded allows us to assume that the source term $g$ is
globally Lipschitz if we wish to study the behavior of our system while it
\mbox{remains in $G$.} Indeed, we may consider $n_0 \in \N$ from the definition of $G$ and study the behavior of the solution to \eqref{eqtruncada} for $n=n_0+1$ since
it coincides with our process until the \mbox{escape from $G$.} For this reason, in the following we shall often drop the superscript $(n_0+1)$ in the notation $U^{(n_0+1),u,\varepsilon}$ unless it is completely necessary.

Our aim in this section is to obtain a complete and precise understanding of (ii) in the description above, since by (iii) the explosion time will inherit all the \mbox{asymptotic properties} of the escape time from $G$. In particular, we are interested in the asymptotic magnitude and distribution of this escape time, as well as in a good understanding of which are the typical paths that lead the stochastic system outside of our bounded domain $G$.
\mbox{The problem of escaping} a bounded domain with these characteristics was first studied in \cite{GOV} for a finite-dimensional double-well potential, and later investigated in \cite{B1} for its infinite-dimensional analogue. The results we present in this chapter are an adaptation to our setting of the results featured in these references. Other references dealing with similar problems include \cite{MOS,B2,FW}.

Hereafter, $B_c$ will denote the neighborhood of the origin highlighted in Conditions \ref{assumpg}. Also, for a given closed set $\Gamma \subseteq C_D([0,1])$ we write
$$
\tau_\varepsilon^u (\Gamma) := \inf\{ t \geq 0 : U^{u,\varepsilon}(t,\cdot) \in \Gamma \}.
$$

\section{Asymptotic magnitude of $\tau_\varepsilon(\p G)$}

We begin our study of the escape from $G$ by studying its asymptotic magnitude as $\varepsilon \rightarrow 0$. The precise result we are to show is the following.

\begin{teo}\label{ecotsuplema1} Given $\delta > 0$ we have
$$
\lim_{\varepsilon \rightarrow 0 } \left[\sup_{u \in B_c} \left| P_{u}
\left( e^{\frac{\Delta - \delta}{\varepsilon^{2}}} < \tau_{\varepsilon}(\partial
G) < e^{\frac{\Delta + \delta}{\varepsilon^{2}}}\right)-1 \right|\right] = 0.
$$
\end{teo}

We shall split the proof of this result into two parts over Sections \ref{secupperbound} and \ref{seclowerbound}, the first of them dealing with the upper bound and the second one with the lower bound.

\subsection{Upper bound on $\tau_\varepsilon(\partial G)$}\label{secupperbound}

Our first goal is to establish the upper bound for the escape time from $G$ contained in the following theorem.

\begin{teo}\label{cotsuplema1} For any $\delta > 0$ we have
\begin{equation}\label{cotsuplema1eq1}
\lim_{\varepsilon \rightarrow 0} \left[\sup_{u \in G} P_{u} \left( \tau_{\varepsilon}(\partial G) > e^{\frac{\Delta + \delta}{\varepsilon^{2}}}\right) \right] = 0.
\end{equation}
\end{teo}

The main idea behind the proof is to show that there exist paths escaping the domain $G$ with rates arbitrarily close to $\Delta$, so that the typical time one must wait for any of these paths to be described by the stochastic system is of lesser asymptotic order than $e^{\frac{\Delta + \delta}{\varepsilon^{2}}}$. The precise estimate is contained in the following lemma.

\begin{lema}\label{cotasuplema0} Given $\delta > 0$ there exists $T^{(\delta)} > 0$ such that for each $u \in G$ there exists a set of paths $\mathcal{E}_{u,T^{(\delta)}} \subseteq C_{D_u}([0,T^{(\delta)}] \times [0,1])$ satisfying
\begin{enumerate}
\item [i.] Every path in $\mathcal{E}_{u,T^{(\delta)}}$ escapes $G$ before time $T^{(\delta)}$.
\item [ii.] For any $\varepsilon > 0$ sufficiently small we have $\inf_{u \in G} P_u ( U^\varepsilon \in \mathcal{E}_{u,T^{(\delta)}} ) \geq  T^{(\delta)}e^{-\frac{\Delta + \frac{\delta}{2}}{\varepsilon^2}}.$
\end{enumerate}
\end{lema}

\begin{proof} We will prove the lemma with the aid of the large deviations principle established for our system. The idea of the proof is to show that for each $u \in G$ there exists a path $\varphi^u \in C_{D_u}([0,T^{(\delta)}]\times [0,1])$ starting at $u$ with rate less than $\Delta + \frac{\delta}{3}$ and such that not only does $\varphi^u$ itself escape from $G$ before time $T^{(\delta)}$, but also any path sufficiently close to $\varphi^u$ does so as well. We construct $\varphi^{u}$ explicitly for each $u \in G$. Each path $\varphi^{u}$ will consist of several pieces, each of which must either follow the trajectory described by the deterministic system (sometimes in the opposite direction) or be a linear interpolation between nearby elements of $C_D \cap W^2_2([0,1])$. In view of this last possibility, we first need to establish some control on the contribution of these linear interpolations to \mbox{the rate of $\varphi^u$.} Thus, let us consider a velocity one linear interpolation $s$ between two
points $u,w \in B_{n_0+1}$, i.e. $s: [0,\|w-v\|_\infty] \times [0,1] \to \R$ given by
$$
s(t,x)= u(x) + t \cdot \left(\frac{w(x)-u(x)}{\|w-u\|_\infty}\right)\text{ for $(t,x) \in [0,\|w-u\|_\infty] \times [0,1]$},
$$ and suppose that both $u$ and $w$ have $W^2_2([0,1])$-norm bounded by some constant $M > 0$. Then we have
$$
I(s) = \frac{1}{2}\int_0^{\|w-u\|_\infty} \int_0^1 | \p_t s - \p^2_{xx} s - g_{n_0}(s) |^2 \leq  \int_0^{\|w-u\|_\infty} \int_0^1 \left[ | \p_t s |^2 + |\p^2_{xx} s + g_{n_0}(s) |^2 \right].
$$
Since $\| \p_t s \|_\infty = 1$ by construction and $\p^2_{xx} s (t,\cdot) = \partial^2_{xx}u + t \cdot \left(\frac{\p^2_{xx}w-\p^2_{xx}u}{\|w-u\|_\infty}\right)$ we obtain that
$$
I(s) \leq C_{n_0,M} \|w-v\|_\infty
$$ where $C_{n_0,M} > 0$ is a constant depending only on $\| g_{n_0+1}\|_\infty$ and $M$.

Taking this into consideration, by using Propositions \ref{G.1}, \ref{G.2} and \ref{A.1} in the Appendix we may take a time $T_0 > 0$ sufficiently small such that
$$
\sup_{t \in [0,T_0]} \| U^u(t,\cdot) - u \|_\infty < \frac{1}{2} \hspace{1.2cm}\text{ and }\hspace{1cm} \sup_{t \in [0,T_0]} \| U^u(T_0,\cdot) - U^v(T_0,\cdot) \|_\infty \leq 2 \| u - v\|_\infty
$$ for any $u,v \in B_{n_0+1}$, a constant $H > 0$ such that $\| U^{u}(T_0,\cdot)\|_{W^2_2([0,1])} \leq H$ for any $u \in B_{n_0+1}$ and a distance $r > 0$ such that any linear interpolation between elements in $B_{n_0+1}$ with $W^2_2([0,1])$-norm bounded by $2H$ and at a distance smaller than $2r$ has rate less than $\frac{\delta}{9}$. We may assume that both $T_0$ and $r$ are sufficiently small, e.g. $T_0 < 1$ and $r < \min\{ \frac{1}{4}, c\}$. We then define the set
$$
\mathcal{W}_{(r^-)} = \{ u \in C_D([0,1]) : d(u, \mathcal{W}) < r\}.
$$ The construction of $\varphi^{u}$ is done as follows. The first step is to follow the deterministic flow for a time period of length $T_0$. The remainder of the construction will vary according to where $v_0:=U^{u}(T_0,\cdot)$ is located. We describe the different scenarios below.
\begin{itemize}
\item If $v_0 \in B_{n_0+r}^c$ then $\varphi^u$ has already escaped $G$ and reached a distance from $G$ which is greater than $r$.
The construction in this scenario ends here.
\item If $v_0 \in B_{n_0+r} \cap \mathcal{W}_{(r^-)}$ then:

\begin{enumerate}
\item [$\diamond$] We choose $v \in \mathcal{W}$ such that $d(v_0,v) < r$ and first let $\varphi^u$ follow once again the deterministic flow for a time period of length $T_0$ and afterwards describe the linear interpolation between the points $U^{v_0}(T_0, \cdot)$ and $v_1:= U^{v}(T_0,\cdot)$ in time $T_1^u = d(U^{v_0}(T_0, \cdot),v_1)$. Notice that $U^{v_0}(T_0, \cdot)$ and $v_1$ lie inside $B_{n_0 +1}$, are at a distance smaller than $2r$ from each other and both have $W^2_2([0,1])$-norm less than $H$ by the choice of $T_0$ and the fact that both $v_0$ and $v$ lie inside $B_{n_0 + \frac{1}{2}}$.
\item [$\diamond$] From there we let the path $\varphi^u$ follow the deterministic flow $U^{v_1}$ until the time $T_2^u = \inf \{ t \geq T_0 : d( U^{v_1}(t,\cdot), z^{(n)}) \leq r \}$ for some $n \in \Z - \{ 0 \}$.
\item [$\diamond$] If for some $t \in [0,T_2^u]$ we have $U^{v_1}(t,\cdot) \notin B_{n_0+r}$ then once again we have that $\varphi^u$ has already escaped $G$ and reached a distance from $G$ greater than $r$, in which case we end the construction here.
\item [$\diamond$] If this is not the case then we continue $\varphi^u$ by describing the linear interpolation between $v_2:=U^{v_1}(T_2^u,\cdot)$ and $v_3:=(1+r)z^{(n)}$ in time $T_3^u= d(v_2,v_3) \leq 2r$. Notice that $v_2$ and $v_3$ lie inside $B_{n_0 +1}$ and both have $W^2_2([0,1])$-norm less than $2H$ since we have that $v_2 = U^{U^{v_1}(T_2^u - T_0,\cdot)}(T_0,\cdot)$ and $U^{z^{(n)}}(T_0,\cdot)=z^{(n)}$.
\item [$\diamond$] Finally, we follow once again the deterministic flow $U^{v_3}$ until we reach a distance from $G$ greater than $r$ in a finite time $T^u_4$ which depends only on $r$ and $z^{(n)}$. Notice that this is possible due to the fact that $v_3 \in \mathcal{D}_e$ by Proposition \ref{descomp3}.
    \end{enumerate}
\item If $v_0 \in B_{n_0 + r} \cap \mathcal{D}_{\mathbf{0}} \cap \mathcal{W}_{(r^-)}^c$ then:
\begin{enumerate}
\item [$\diamond$] From there we let the path $\varphi^u$ follow the deterministic flow $U^{v_0}$ until the time $T_5^u = \inf \{ t \geq T_0 : U^{u}(t,\cdot) \in B_r\}$.
\item [$\diamond$] Next we fix $u^* \in W^z_u \cap \p B_r$ and consider $T^*= \inf \{ t \geq T_0 : U^{u^*}(t,\cdot) \in B_r \}$, a time which only depends on the choice of $u^*$. We then continue $\varphi^u$ by
describing the linear interpolation between $v_4:= U^{u}(T_5^u,\cdot)$ and $v_5:= U^{u^*}(T^*,\cdot)$ in time $T_6^u = d(v_4,v_5) \leq 2r$. Notice that $v_4$ and $v_5$ lie inside $G$ since $r < c$ and both have $W^2_2([0,1])$-norm less than $H$ by a similar argument to the one given above.
\item [$\diamond$] Once on $\mathcal{W}^z_u$ we let $\varphi^u$ follow the reverse deterministic flow until the time $T_7^u = \inf \{ t \geq 0 : U^{v_5}(-t,\cdot) \in G \cap \mathcal{W}_{(r^-)}\}$ which does not depend on $u$, but rather on the choice of $u^*$ instead.
\item [$\diamond$] We can then continue as in the second scenario, by noticing that $U^{v_5}(-T_7^u,\cdot)$ belongs to $B_{n_0}$ and has $W^2_2([0,1])$-norm less than $H$ since $\mathcal{W}^z_u \cap \mathcal{D}_{\mathbf{0}} \subseteq G$ by the mere construction of $G$.
\end{enumerate}
\item If $v_0 \in B_{n_0 + r} \cap \mathcal{D}_{e} \cap \mathcal{W}_{(r^-)}^c$ then we let the path $\varphi^u$ follow the deterministic flow $U^{v_0}$ until we reach a distance from $G$ greater than $r$ in a finite time $T^u_8$.
\end{itemize}
If built in this way, the path $\varphi^u$ verifies all the required properties. Indeed, we have that:
\begin{itemize}
\item [$\bullet$] Each $\varphi^u$ belongs to $C_{D_u} \cap W^{1,2}_2 ([0,T^u]\times [0,1])$ for some $T_u > 0$ (the sum of the corresponding $T^u_i$) since by construction $\varphi^u$ is piecewise differentiable.
\item [$\bullet$] The total time length $T^u$ of the path $\varphi^u$ is uniformly bounded in $u \in G$. Indeed, the total time which $\varphi^u$ spends following the deterministic flow is uniformly bounded by Proposition \ref{A.3} since there are only finitely many equilibrium \mbox{points lying in $B_{n_0+1}$.} On the other hand, there are at most three linear interpolations in $\varphi^u$, each of which lasts a time of length less than one. Finally, the time spent by $\varphi^u$ following the reverse deterministic flow is finite and does not depend on $u$.
\item Each $\varphi^u$ has total rate less than $\Delta + \frac{\delta}{3}$. Indeed, its total rate can be computed as the sum of the rate of each of its pieces.
We have already seen that each linear interpolation has rate less than $\frac{\delta}{9}$ and, since in any case there at most three of them, their total contribution is less than $\frac{\delta}{3}$. On the other hand, the pieces in which $\varphi^u$ follows the deterministic flow have zero rate. Finally, if $\varphi^u$ follows the reverse deterministic flow (i.e. $\p_t \varphi^u = - (\p^2_{xx} \varphi^u + g(\varphi^u)) )$ during the time interval $[t_1,t_2]$ then, similarly to Proposition \ref{Lyapunov}, we have
$$
\frac{1}{2}\int_{t_1}^{t_2} \int_0^1 |\p_t \varphi^u - (\p^2_{xx} \varphi^u + g(\varphi^u))|^{2} = 2 \int_{t_1}^{t_2} \frac{d S(\varphi^u(t,\cdot))}{dt} = 2 (S(\varphi^u(t_2)) -  S(\varphi^u(t_1)))
$$ from where, upon recalling that $\varphi^u(t_2),\varphi^u(t_1) \in \mathcal{W}^z_u$, we obtain that the rate of this last piece is less than $\Delta$.

\item Any path at a distance strictly less than $r$ from $\varphi^u$ in the supremum norm must also escape from $G$ before $T^u$, since by this time $\varphi^u$ reaches a distance $r$ from $G$.
\end{itemize}

Let us notice that for each $u \in G$ we have built a path $\varphi^u$ on the time interval $[0,T^u]$, but we wish all constructed paths to be defined on a same time interval.
For this reason, we consider $T^{(\delta)}:= \sup_{u \in G} T^u < +\infty$ and extend all $\varphi^u$ to the time interval $[0,T^{(\delta)}]$ by following the deterministic flow. It is easy to check that these extended paths
maintain the aforementioned properties. We then define the set $\mathcal{E}_{u,T^{(\delta)}}$ for each $u \in G$ as
$$
\mathcal{E}_{u,T^{(\delta)}}:=\{ \psi \in C_{D_u}([0,T^{(\delta)}] \times [0,1]) : d_{T^{(\delta)}}(\psi , \varphi^u ) < r\}.
$$ It is clear that each $\mathcal{E}_{u,T^{(\delta)}}$ verifies condition (i) by construction, whereas (ii) follows from the large deviations estimate (i) in Section \ref{secLDP}.
\end{proof}

Now, if we write $T^+_\varepsilon = e^{\frac{\Delta + \delta}{\varepsilon^2}}$ and split the interval $[0,T^+_\varepsilon]$ into subintervals of length $T^{(\delta)}$ given by Lemma \ref{cotasuplema0}, then by the Markov property for the solution of \eqref{MainSPDE} we obtain
\begin{equation}\label{boundonguniform}
P_u ( \tau_{\varepsilon}(\partial G) > T^+_\varepsilon) \leq P_u \left( \bigcap_{k=1}^{m_\varepsilon} \{ U^\varepsilon (t,\cdot) \in G \text{ for all } t \in [(k-1)T^{(\delta)} , kT^{(\delta)}] \}\right) \leq (1-\alpha_\varepsilon)^{m_\varepsilon}
\end{equation} where $\alpha_\varepsilon:=T^{(\delta)}e^{-\frac{\Delta + \frac{\delta}{2}}{\varepsilon^2}}$, $m_\varepsilon:=\left\lfloor\frac{T^+_\varepsilon}{T^{(\delta)}}\right\rfloor$ and for the second inequality we used that for $u \in G$
$$
\{ U^{u,\varepsilon}(t,\cdot) \in G \text{ for all }t \in [0,T^{(\delta)}]\} \subseteq \{ U^{u,\varepsilon} \notin \mathcal{E}_{u,T^{(\delta)}}\}.
$$ Since the bound \eqref{boundonguniform} is uniform on $G$, by taking $\varepsilon \rightarrow 0$ a direct \mbox{computation yields \eqref{cotsuplema1eq1}.}

\subsection{Lower bound on $\tau_\varepsilon(\partial G)$}\label{seclowerbound}

Our next purpose is to establish the lower bound on $\tau_\varepsilon(\p G)$ contained in the \mbox{theorem below.}

\begin{teo}There exists $r > 0$ sufficiently small such that
\begin{equation}\label{cotainferioreq0}
\lim_{\varepsilon \rightarrow 0} \left[ \sup_{u \in \p B_{r}} P_u \left( \tau_\varepsilon (\p G) < e^{\frac{\Delta - \delta}{\varepsilon^2}}\right) \right] = 0.
\end{equation}
\end{teo}

The key to establishing this lower bound is to observe that for initial data in a small neighborhood of the origin the path described by the stochastic system while it remains inside $G$ will typically consist of several failed attempts to reach $\p G$ followed by one last attempt which is successful and thus leads to the escape from $G$. Each of these failed attempts is an excursion drifting away from the origin which, having failed to reach $\p G$, later returns to (a small neighborhood of) the origin. Hence, the desired lower bound on the time the process needs in order to escape from $G$ can be obtained upon giving suitable bounds on the number and length of these excursions.

To accomplish this we consider, given constants $r_1,r_2 > 0$ such \mbox{that $r_1 < \frac{r_2}{2}$ and $r_2 < c$,} for each $u \in \p B_{r_1}$ and $\varepsilon > 0$ the increasing sequence of stopping times
$$
\left\{\begin{array}{l} \eta_0 = 0\\
\\
\sigma_0 = \inf \{ t \geq 0 : U^{u,\varepsilon}_t \in \p B_{r_2}\}
\end{array}\right.
$$
and for $n \in \N_0$
$$
\left\{\begin{array}{l} \eta_{n+1} = \inf \{ t > \sigma_n : U^{u,\varepsilon}_t \in \p B_{r_1} \cup (\partial \tilde{G})_{(d)}\}\\
\\
\sigma_{n+1} = \inf \{ t > \eta_{n+1} : U^{u,\varepsilon}_t \in \p B_{r_2}\}
\end{array}\right.
$$ where $\p \tilde{G}$ is the pre-domain defined in \eqref{predomain} and $d > 0$ is taken such that $B_{r_2}$ and $(\p \tilde{G})_{(d)}$ are at positive distance from each other, where
$$
(\partial \tilde{G})_{(d)}:=\{ u \in C_D([0,1]) : d(u,\partial \tilde{G}) \leq d \}.
$$ These positive constants $r_1,r_2$ and $d$ will be later taken conveniently small \mbox{during the proof.} Also, whenever any of the sets involved is empty we take the corresponding time as $+\infty$. We then define the Markov chain $(Z^{u,\varepsilon}_n)_{n \in \N}$ by the formula
$$
Z^{u,\varepsilon}_n=U^{u,\varepsilon}_{\eta_n}
$$ for each $n \in \N_0$, and set $\vartheta^u_\varepsilon := \min\{ n \in \N : Z^{u,\varepsilon}_n \in (\partial \tilde{G})_{(d)}\}$. Since the process $U^{u,\varepsilon}$ escapes $\tilde{G}$ in a finite time almost surely and we are only interested in events occurring before $\tau_\varepsilon(\p \tilde{G})$, we need not worry about the possibility of $Z_n^\varepsilon$ not being well defined.

Let us notice then that, since for any $u \in \p B_{r_1}$ we have $\tau^u_\varepsilon (\p \tilde{G}) \leq \tau^u_\varepsilon (\p G)$ by the continuity of the trajectories of $U^u$, for each $\delta > 0$ we obtain the bound
\begin{equation}\label{cotainferioreq1}
\sup_{u \in \p B_{r_1}} P_u \left( \tau_\varepsilon (\p G) < e^{\frac{\Delta - \delta}{\varepsilon^2}}\right) \leq \sup_{u \in \p B_{r_1}} P_u( \nu_\varepsilon \leq k_\varepsilon ) + \sup_{u \in \p B_{r_1}} P_u \left( \eta_{k_\varepsilon} < e^{\frac{\Delta - \delta}{\varepsilon^2}} \right)
\end{equation} where $k_\varepsilon \in \N$ is to be determined next for each $\varepsilon > 0$. Thus, we see that to obtain \eqref{cotainferioreq0} it suffices to show that, for a suitable choice of $(k_\varepsilon)_{\varepsilon > 0}$, both terms in the right hand side of \eqref{cotainferioreq1} vanish as $\varepsilon \rightarrow 0$. We will do this with the aid of the following two lemmas, which are slight modifications of two results originally appearing in \cite{B1}.

\begin{lema}\label{lemab1} Let $F,B \subseteq C_D([0,1])$ be bounded sets and suppose $\psi \in C_D([0,1])$ is such that $d(\psi, F) > 3r$ for some $r > 0$. Then for any $h,T > 0$ there exists $r^* > 0$ such that
$$
\sup_{u \in B_{\rho}(\psi)} P_u ( \tau_\varepsilon( F_{(r)}) \leq  \min \{ T , \tau_\varepsilon( B^c )\} ) \leq e^{-\frac{V(\psi,F_{(2r)}) - h}{\varepsilon^2}}
$$ for any $0 < \rho < r^*$.
\end{lema}

\begin{lema}\label{lemab2} Let $B \subseteq C_D([0,1])$ be bounded and closed. If for $e > 0$ we consider
$$
C_e := \bigcup_{n \in \Z } B_e(z^{(n)}).
$$ then given $K > 0$ there exists $T > 0$ such that
$$
\sup_{u \in B} P_u \left( \min\{ \tau(C_e), \tau( B^c )\} > T \right) \leq e^{- \frac{K}{\varepsilon^2}}.
$$
\end{lema}

Now, in order to deal with the first term in the right hand side of \eqref{cotainferioreq1} it will suffice to establish the bound
\begin{equation}\label{cotainferioreq2}
\sup_{u \in \p B_{r_1}} P_u ( \vartheta_\varepsilon = 1 ) \leq e^{- \frac{\Delta - \frac{\delta}{4}}{\varepsilon^2}}
\end{equation} for $\varepsilon > 0$ sufficiently small. Indeed, if we prove \eqref{cotainferioreq2} then by the Markov property of $Z^\varepsilon$ we obtain
$$
\inf_{u \in \p B_{r_1}} P_u( \vartheta_\varepsilon > n ) \geq \left(1 - e^{- \frac{\Delta - \frac{\delta}{4}}{\varepsilon^2}}\right)^n
$$ for every $n \in \N$ and $\varepsilon > 0$ sufficiently small, which implies the inequality
\begin{equation}\label{cotainferioreq5}
\sup_{u \in \p B_{r_1}} P_u( \vartheta_\varepsilon \leq k_\varepsilon ) \leq 1 - \left(1 - e^{- \frac{\Delta - \frac{\delta}{4}}{\varepsilon^2}}\right)^{k_\varepsilon}
\end{equation} whose right hand side vanishes as $\varepsilon \rightarrow 0$ if we set for example $k_\varepsilon := \left[\exp\left({\frac{\Delta - \frac{\delta}{2}}{\varepsilon^2}}\right)\right]+1$.
Thus, let us check that \eqref{cotainferioreq2} holds. Notice that the strong Markov property for $\sigma_0$ yields
$$
\sup_{u \in \p B_{r_1}} P_u( \vartheta_\varepsilon = 1 ) \leq \sup_{u \in \p B_{r_2}} P_u \left( \tau_\varepsilon ((\partial \tilde{G})_{(d)}) = \tau_\varepsilon (\p B_{r_1} \cup (\partial \tilde{G})_{(d)})\right)
$$ so that to check \eqref{cotainferioreq2} it will suffice to find $T > 0$ such that
\begin{equation}\label{cotainferioreq3}
\sup_{u \in \p B_{r_2}} P_u \left( \tau_\varepsilon ((\partial \tilde{G})_{(d)}) \leq T\right) \leq e^{- \frac{ \Delta - \frac{\delta}{8}}{\varepsilon^2}}
\end{equation}
and
\begin{equation}\label{cotainferioreq4}
\sup_{u \in \p B_{r_2}} P_u \left( \tau_\varepsilon (\p B_{r_1} \cup (\partial \tilde{G})_{(d)}) > T\right) \leq e^{- \frac{ \Delta - \frac{\delta}{8}}{\varepsilon^2}}
\end{equation}
for every $\varepsilon > 0$ sufficiently small.

Now, let us observe that, since the potential $S$ attains in $\pm z$ its minimum on $\partial \tilde{G}$, by proceeding as in the proof of Lemma \ref{cotasuplema0} one can show that $V(\mathbf{0},\p \tilde{G})=\Delta$.
Therefore, by conducting a similar argument to the one given in the construction of $G$, it follows that the set
$$
\mathcal{V} := \left\{ u \in C_D([0,1]) : V(\mathbf{0},u) \leq \Delta - \frac{\delta}{16} \right\}
$$ is contained in $\tilde{G}$ and satisfies $d( \mathcal{V},\mathcal{W} ) > 0$. Furthermore, since $\mathcal{V}$ is contained in $\overline{D_\mathbf{0}}$ and the set $\{ u \in \overline{D_\mathbf{0}} : 0 \leq S(u) \leq S(z) \}$ is contained in $B_{n_0 -1}$, then Proposition \ref{costo} implies that $d( \mathcal{V}, \p B_{n_0}) > 0$ and thus we obtain that $d(\mathcal{V}, \p \tilde{G}) \geq \min\{ d( \mathcal{V},\mathcal{W} ) , d( \mathcal{V}, \p B_{n_0})\} > 0$. In particular, if we take $d < \frac{d(\mathcal{V}, \p \tilde{G})}{3}$ then we have $V(\mathbf{0}, (\partial \tilde{G})_{(2d)}) \geq \Delta - \frac{\delta}{16}$ and therefore by Lemma \ref{lemab1} we conclude that for a fixed $T > 0$ and $r_2 > 0$ sufficiently small
$$
\sup_{u \in \p B_{r_2}} P_u ( \tau_\varepsilon ( (\p \tilde{G})_{(d)}) \leq T ) \leq e^{- \frac{ \Delta - \frac{\delta}{8}}{\varepsilon^2}}
$$ provided that $\varepsilon > 0$ is sufficiently small, which yields \eqref{cotainferioreq3}.

On the other hand, if we take $0 < e < \min\{ r_1 , d\}$ then we have that $B_e( z^{(n)} ) \subseteq (\p \tilde{G})_{(d)}$ for any $n \in \Z - \{ 0 \}$ such that $z^{(n)} \in \tilde{G}$. In particular, we see that
\begin{equation}\label{cotainferioreq6}
\sup_{u \in \p B_{r_2}} P_u \left( \tau_\varepsilon (\p B_{r_1} \cup (\partial \tilde{G})_{(d)}) > T\right) \leq \sup_{u \in \p B_{r_2}} P_u \left( \min\{ \tau(C_e), \tau( (\partial \tilde{G})_{(d)} )\} > T \right).
\end{equation} It then follows from Lemma \ref{lemab2} that $T > 0$ can be taken sufficiently large so that \eqref{cotainferioreq4} holds. Together with \eqref{cotainferioreq3}, this yields \eqref{cotainferioreq2} and establishes the convergence to zero of the first term in the right hand side of \eqref{cotainferioreq1}. Thus, it only remains to establish the same convergence for the second term.

Notice that by Proposition \ref{G.1} there exists some $T_{r_2} > 0$ such that for any $u \in \p B_{r_2}$ the system $U^u$ spends a time of length at least $T_{r_2}$ before reaching $B_{\frac{r_2}{2}}$. Furthermore, we may assume that $r_2$ is sufficiently small so that the path described by $U^u$ until time $T_{r_2}$ is at a distance greater than $\frac{r_2}{2}$ from $(\p \tilde{G})_{(d)}$. By the strong Markov property and \eqref{grandes1} we conclude that for $\varepsilon > 0$ sufficiently small
$$
\inf_{u \in \p B_{r_1}} P_u ( \eta_1 \geq T_{r_2} ) \geq \inf_{u \in \p B_{r_2}} P_u \left( d_{T_{r_2}}\left( U^{(n_0),\varepsilon},U^{(n_0)} \right) < \frac{r_2}{2} - r_1 \right) \geq \frac{2}{3}.
$$ Let us observe that since for any $k \in \N$ we have the inequality $\eta_k \geq \sum_{i=1}^k T_{r_2} \mathbbm{1}_{\{\eta_i - \eta_{i-1} \geq T_{r_2}\}}$, by definition of $k_\varepsilon$ we obtain that for $\varepsilon > 0$ sufficiently small
$$
\sup_{u \in \p B_{r_1}} P_u \left( \eta_{k_\varepsilon} < e^{\frac{\Delta - \delta}{\varepsilon^2}} \right) \leq \sup_{u \in \p B_{r_1}} P_u \left( \frac{\eta_{k_\varepsilon}}{k_\varepsilon} < e^{- \frac{\delta}{2\varepsilon^2}} \right) \leq P \left( \frac{1}{k_\varepsilon}\sum_{i=1}^{k_\varepsilon} X_i < \frac{e^{- \frac{\delta}{2\varepsilon^2}}}{T_{r_2}}\right)
$$ where $(X_i)_{i \in \N}$ is a sequence of i.i.d. Bernoulli random variables with parameter $p=\frac{2}{3}$. The result now follows at once from the law of large numbers.

\section{The escape route}

We are now interested in characterizing the typical route that the stochastic system describes to escape from $G$. By the considerations made at the beginning of this Section we expect this typical path to escape $G$ by going through the region of the boundary with the lowest quasipotential, namely $\p^{\pm z}$. More precisely, we wish to prove the \mbox{following theorem.}
\begin{teo}\label{teoescape0} If $c > 0$ is given by (ii) in Conditions \ref{assumpg} then
\begin{equation}\label{escape0}
\lim_{\varepsilon \rightarrow 0} \left[\sup_{u \in B_c} P_u \left( U^{\varepsilon}(\tau_\varepsilon(\partial G),\cdot) \notin \partial^{\pm z}\right)\right] = 0.
\end{equation}
\end{teo}
To prove this result we shall need to establish the following two crucial facts:
\begin{enumerate}
\item [i.] For each $r > 0$ strictly smaller than $c$
$$
\lim_{\varepsilon \rightarrow 0} \left[\sup_{u \in B_c}
P_u\left(\tau_\varepsilon(\partial G) < \tau\left(B_r\right)\right)\right]=0.
$$
\item [ii.] For any $r > 0$ sufficiently small
\begin{equation}\label{escape1}
   \lim_{\varepsilon \rightarrow 0} \left[\sup_{u \in \partial B_r} P_u \left( U^{\varepsilon}(\tau_\varepsilon(\partial G),\cdot) \notin \partial^{\pm z} \right) \right] = 0.
\end{equation}
\end{enumerate}
Indeed, \eqref{escape0} follows immediately from (i) and (ii) by applying the strong \mbox{Markov property.} Hence, it will suffice to establish (i) and (ii). Assertion (i) is shown in the following lemma.

\begin{lema}\label{escapelemai}
For each $r > 0$ strictly smaller than $c$ we have
$$
\lim_{\varepsilon \rightarrow 0} \left[\sup_{u \in B_c}
P_u\left(\tau_\varepsilon(\partial G) < \tau\left(B_r\right)\right)\right]=0.
$$
\end{lema}

\begin{proof}
Notice that by choice of $c$ we have that any $u \in B_c$ reaches the neighborhood $B_{\frac{r}{2}}$ in a finite time $T^u_r$ while remaining at a distance greater than $r_{\mathbf{0}}$ in \mbox{Conditions \ref{assumpg} from both} $\partial G$ and $\mathcal{W}$. By Proposition \ref{A.3} we therefore conclude that $T_r = \sup_{u \in B_c} T^u_r$ must be finite. Hence, for $u \in B_c$ we obtain
$$
P_u \left( \tau_\varepsilon(\partial G) < \tau_\varepsilon\left(B_r\right) \right) \leq \sup_{v \in B_c} P_v \left( d_{T_r}(U^{(n_0),\varepsilon},U^{(n_0)}) \geq r_{\mathbf{0}}\wedge \frac{r}{2}\right)
$$ which, by \eqref{grandes1} and the uniformity of the bound in $u \in B_c$, implies the result at once.
\end{proof}

In order to prove assertion (ii) we first show that it suffices to study the path described by the stochastic system since its last visit to (a small neighborhood of) the origin. \mbox{We shall do} this by resorting to a Markov chain similar to the one introduced in the preceding section to establish the lower bound. More precisely, given constants $r_1,r_2 > 0$ such that $r_1 < \frac{r_2}{2}$ and $r_2 < c$, for each $u \in \p B_{r_1}$ and $\varepsilon > 0$ we consider the increasing sequence of stopping times
$$
\left\{\begin{array}{l} \eta_0 = 0\\
\\
\sigma_0 = \inf \{ t \geq 0 : U^{u,\varepsilon}(t,\cdot) \in \p B_{r_2}\}
\end{array}\right.
$$
and for $n \in \N_0$
$$
\left\{\begin{array}{l} \eta_{n+1} = \inf \{ t > \sigma_n : U^{u,\varepsilon}(t,\cdot) \in \p B_{r_1} \cup \partial G\}\\
\\
\sigma_{n+1} = \inf \{ t > \eta_{n+1} : U^{u,\varepsilon}(t,\cdot) \in \p B_{r_2}\}
\end{array}\right.
$$ with the convention that $\inf \emptyset = +\infty$. We then define the Markov chain $(Z^{u,\varepsilon}_n)_{n \in \N}$ as
$$
Z^{u,\varepsilon}_n:=U^{u,\varepsilon}(\eta_n,\cdot)
$$ for each $n \in \N_0$ and set $\vartheta^u_\varepsilon := \min \{ n \in \N : Z^{u,\varepsilon}_n \in \p G\}$. Just as in the previous section, for our purposes we will not need to worry about the possibility of $Z_n^\varepsilon$ not being well defined. Also, the constants $r_1$ and $r_2$ will be later taken conveniently small \mbox{throughout the proof.}
It is not hard to see that for any $u \in \p B_{r_1}$ the strong Markov property yields
\begin{equation}\label{cociente}
P_u \left( U^{\varepsilon}(\tau_\varepsilon(\partial G),\cdot) \notin \partial^{\pm z} \right) \leq \sup_{v \in \p B_{r_1}} \frac{ P_v ( Z_1^\varepsilon \in \p G - \p^{\pm z} )}{P_v ( Z_1^\varepsilon \in \p G)}
\end{equation} from which we conclude that in order to show (ii) it will suffice to give a lower bound for the denominator and an upper bound for the numerator such that the quotient of these bounds goes to zero with $\varepsilon$. The following lemma, whose proof can be found in \cite{B1}, provides the desired lower bound.

\begin{lema}\label{lemab0} Let us suppose that $r_1$ is sufficiently small so as to guarantee that for any $u \in \p B_{r_1}$ the deterministic orbit $\{ U^u(t,\cdot) : t \geq 0 \}$ does not intersect $\p B_{\frac{r_2}{2}}$. Then for all $\varepsilon > 0$ sufficiently small
$$
\inf_{u \in \p B_{r_1}} P_u( Z_1^\varepsilon \in \p G) \geq e^{- \frac{\Delta + k r_2}{\varepsilon^2} },
$$ where $k > 0$ is a constant which does not depend on the choice of $r_1$ and $r_2$.
\end{lema}

The upper bound on the numerator in \eqref{cociente} is more involved and requieres several steps. We start by considering, for fixed constants $d,e > 0$ such that $e < r_z$ and $d < \frac{r_z - e}{2}$ (which will be later made conveniently small), the stopping times
$$
\tau_\varepsilon^0 = \tau_\varepsilon \left( (\p G - \p^{\pm z})_{(d)} \right) \hspace{2cm}\text{ and }\hspace{2cm} \tau_\varepsilon^2 = \tau_\varepsilon \left( B_e( \pm z) \right)
$$ where
$$
(\p G - \p^{\pm z})_{(d)} := \{ u \in C_D([0,1]) : d(u, \p G - \p^{\pm z}) \leq d \}.
$$ where the initial datum is implicit. Notice that $d$ is such that $(\p G - \p^{\pm z})_{(d)}$ and $B_e( \pm z)$ are disjoint. Then we can decompose the event in the numerator into three disjoint parts:
\begin{equation}\label{decomp}
A:=\{ Z_1^{u,\varepsilon} \in \p G - \p^{\pm z}\} = \bigcup_{i=0}^2 [A \cap \{ \tau_\varepsilon^i = \min\{ \tau_\varepsilon^0, \tau_\varepsilon^1, \tau_\varepsilon^2, \tau_\varepsilon^3 \} \}]
\end{equation} where, for notational convenience, we have set $\tau_\varepsilon^1$ as some fixed time $T > 0$ which is to be conveniently determined later and $\tau_\varepsilon^3$ as the escape time $\tau_\varepsilon(\p G)$ from our domain $G$. Observe that the set $A \cap \{\tau_\varepsilon^3 = \min\{ \tau_\varepsilon^0, \tau_\varepsilon^1, \tau_\varepsilon^2, \tau_\varepsilon^3 \} \}$ is empty and is therefore left out of the decomposition. Thus, in order to provide an upper bound for the numerator we see that it will suffice to estimate the probabilities of each of the sets in the decomposition.

\subsection{Upper bound on $P(A \cap \{ \tau_\varepsilon^0 = \min\{ \tau_\varepsilon^0, \tau_\varepsilon^1, \tau_\varepsilon^2, \tau_\varepsilon^3 \} \})$}

To bound the probability of the first set we observe that if $d > 0$ is taken sufficiently small then
\begin{equation}\label{cotacuasipotencial}
V(\mathbf{0},(\p G - \p^{\pm z})_{(2d)}) > \Delta.
\end{equation} Indeed, since
$$
(\p G - \p^{\pm z})_{(2d)} = ( \p \tilde{G} - B_{r_z}(\pm z))_{(2d)} \cup ( [\p B_{r_z}(\pm z) \cap \overline{\mathcal{D}_e}] - \p^{\pm z})_{(2d)}
$$ to establish \eqref{cotacuasipotencial} it will suffice to show that
\begin{equation}\label{cotacuasipotencial2}
V(\mathbf{0},( \p \tilde{G} - B_{r_z}(\pm z))_{(2d)}) > \Delta\hspace{0.7cm}\text{ and }\hspace{0.7cm}V(\mathbf{0},( [\p B_{r_z}(\pm z) \cap \overline{\mathcal{D}_e}] - \p^{\pm z})_{(2d)}) > \Delta.
\end{equation} From the definition of $\p^{\pm z}$ it easily follows that if we take $2d < \frac{d( \tilde{\p}^z, \mathcal{W})}{2}$ then the sets $\tilde{\p}^z$ and $( [\p B_{r_z}(z) \cap \overline{\mathcal{D}_e}] - \p^{z})_{(2d)}$ are disjoint so that the second inequality in \eqref{cotacuasipotencial2} is settled. To obtain the first inequality we begin by noticing that by an argument analogous to the one employed in the construction of $G$ there exists $\alpha > 0$ such that
$$
\inf\{ S(u) : u\in [\mathcal{W} \cap B_{n_0+1}] - B_{\frac{r_z}{2}}(\pm z) \} > \alpha > S(z).
$$ Hence, to establish the first inequality in \eqref{cotacuasipotencial2} it will suffice to consider the set
$$
\mathcal{V}'= \{ u \in C_D([0,1]) : V(\mathbf{0},u) \leq 2 \min\{\alpha, S(z) + \zeta_1\} \}
$$ and show that $d( \mathcal{V}', \p \tilde{G} - B_{r_z}(\pm z) ) > 0$. But, if this were not so, then there would exist sequences $(u_k)_{k \in \N} \subseteq \mathcal{V}'$ and $(v_k)_{k \in \N} \subseteq \p \tilde{G} - B_{r_z}(\pm z)$ such that $d(u_k,v_k) \rightarrow 0$ as $k \rightarrow +\infty$. By Propositions \ref{G.2} and \ref{A.2} we may find subsequences $(u_{k_j})_{j \in \N},(v_{k_j})_{j \in \N}$ and a time $t > 0$ such that
$$
\lim_{j \rightarrow +\infty} U^{u_{k_j}}(t,\cdot) =\lim_{j \rightarrow +\infty} U^{v_{k_j}}(t,\cdot) = v_\infty
$$ for some limit $v_\infty \in C_D([0,1])$. Let us observe that, since we have $v_\infty=\lim_{j \rightarrow +\infty} U^{u_{k_j}}(t,\cdot)$, the lower semicontinuity of $V(\mathbf{0},\cdot)$ and the fact that the mapping $t \mapsto V(\mathbf{0},U^u(t,\cdot))$ is monotone decreasing for any $u \in C_D([0,1])$ together imply that
\begin{equation}\label{cotacuasipotencial3}
V(\mathbf{0},v_\infty) \leq 2 \min\{\alpha, S(z) + \zeta_1\}.
\end{equation} At least one of the following possibilities must then occur:
\begin{enumerate}
\item [$\bullet$] If $v_{k} \in \p B_{n_0} \cap \mathcal{D}_{\mathbf{0}}$ for infinitely many $k \in \N$, then $(k_j)_{j \in \N}$ and $t > 0$ can be taken so as to guarantee that the condition $U^{v_{k_j}}(t,\cdot) \notin B_{n_0 -1}$ is satisfied for every $j \in \N$. Since $\overline{\mathcal{D}_{\mathbf{0}}}$ is closed and invariant under the deterministic flow we therefore conclude that $v_\infty \in \overline{\mathcal{D}_{\mathbf{0}}} - B_{n_0 -1}^\circ$ and thus that $S(v_\infty) > S(z)+\zeta_1$. In particular, we obtain that $V(\mathbf{0},v_\infty) > 2(S(z)+\zeta_1)$, a fact which contradicts \eqref{cotacuasipotencial3}.
\item [$\bullet$] If $v_{k} \in \mathcal{W} \cap B_{n_0}$ for infinitely many $k \in \N$, then $(k_j)_{j \in \N}$ and $t > 0$ can be taken so as to guarantee that $U^{v_{k_j}}(t,\cdot) \in B_{n_0+1} - B_{\frac{2}{3}r_z}(\pm z)$ is satisfied for every $j \in \N$. Since $\mathcal{W}$ is closed and invariant under the deterministic flow we then conclude that $v_\infty \in [\mathcal{W} \cap B_{n_0+1}] - B_{\frac{r_z}{2}}(\pm z)$ and thus that $S(v_\infty) > \alpha$. In particular, we obtain that $V(\mathbf{0},v_\infty) > 2\alpha$, which again contradicts \eqref{cotacuasipotencial3}.
\end{enumerate}

Now that we have shown \eqref{cotacuasipotencial}, Lemma \ref{lemab1} guarantees that if $r_1$ is sufficiently small then there exists $h > 0$ such that
$$
\sup_{u \in \p B_{r_1}} P_u (A \cap \{ \tau_\varepsilon^0 = \min\{ \tau_\varepsilon^0, \tau_\varepsilon^1, \tau_\varepsilon^2, \tau_\varepsilon^3 \} \}) \leq \sup_{u \in \p B_{r_1}} P_u ( \tau_\varepsilon^0 \leq \min\{ T, \tau_\varepsilon(\p G) \} ) \leq e^{- \frac{\Delta + h }{\varepsilon^2}}
$$ which provides the desired upper bound. Indeed, notice that if $T > 0$ is fixed and $r_2 > 0$ is taken sufficiently small then the upper bound obtained and Lemma \ref{lemab0} together yield
\begin{equation}\label{decomp1}
\lim_{\varepsilon \rightarrow 0} \left[\sup_{u \in \p B_{r_1}} \frac{P_u (A \cap \{ \tau_\varepsilon^0 = \min\{ \tau_\varepsilon^0, \tau_\varepsilon^1, \tau_\varepsilon^2, \tau_\varepsilon^3 \} \})}{P_u ( Z_1^\varepsilon \in \p G ) }\right] = 0.
\end{equation}

\subsection{Upper bound on $P(A \cap \{ \tau_\varepsilon^1 = \min\{ \tau_\varepsilon^0, \tau_\varepsilon^1, \tau_\varepsilon^2, \tau_\varepsilon^3 \} \})$}

To bound the probability of the second set we observe that any unstable equilibrium of the system lying inside $G$ and different from $\pm z$ must necessarily belong to $\p G - \p^{\pm z}$. This implies that if we take $e < \min\{r_1,d\}$ then
$$
\sup_{u \in \p B_{r_1}} P_u(A \cap \{ \tau_\varepsilon^1 = \min\{ \tau_\varepsilon^0, \tau_\varepsilon^1, \tau_\varepsilon^2, \tau_\varepsilon^3 \} \}) \leq  \sup_{u \in \p B_{r_1}} P_u ( T \leq \min\{ \tau_\varepsilon(C_e), \tau_\varepsilon(\p G)\} )
$$ where $C_e$ is defined as in Lemma \ref{lemab2}. Hence, by the same lemma we obtain for $T > 0$ sufficiently large the upper bound
$$
\sup_{u \in \p B_{r_1}} P_u(A \cap \{ \tau_\varepsilon^1 = \min\{ \tau_\varepsilon^0, \tau_\varepsilon^1, \tau_\varepsilon^2, \tau_\varepsilon^3 \} \}) \leq  e^{-\frac{ \Delta + 1}{\varepsilon^2}}.
$$ Together with Lemma \ref{lemab0} this upper bound yields
\begin{equation}\label{decomp2}
\lim_{\varepsilon \rightarrow 0} \left[\sup_{u \in \p B_{r_1}} \frac{P_u (A \cap \{ \tau_\varepsilon^1 = \min\{ \tau_\varepsilon^0, \tau_\varepsilon^1, \tau_\varepsilon^2, \tau_\varepsilon^3 \} \})}{P_u ( Z_1^\varepsilon \in \p G ) }\right] = 0.
\end{equation}

\subsection{Upper bound on $P(A \cap \{ \tau_\varepsilon^2 = \min\{ \tau_\varepsilon^0, \tau_\varepsilon^1, \tau_\varepsilon^2, \tau_\varepsilon^3 \} \})$}

To conclude the proof it only remains to give an upper bound on the probability of the third set in the right hand side of \eqref{decomp}. If we write $D^2 = \{ \tau_\varepsilon^2 = \min\{ \tau_\varepsilon^0, \tau_\varepsilon^1, \tau_\varepsilon^2, \tau_\varepsilon^3 \} \}$ then notice that by the strong Markov property one has
$$
P_u (A \cap D^2 ) \leq \E( \mathbbm{1}_{D^2} P_{U^{u,\varepsilon}(\tau_\varepsilon^2,\cdot)} ( U^\varepsilon(\tau_\varepsilon(\p G),\cdot) \in \p G - \p^{\pm z} , \tau_\varepsilon( B_{r_1} ) > \tau_\varepsilon(\p G) )).
$$ The next lemma provides a suitable upper bound on the probability inside the expectation.

\begin{lema} There exists a constant $C > 0$ such that for every $e > 0$ sufficiently small one has
\begin{equation}\label{eqlemab3}
\sup_{v \in B_e(\pm z)} \frac{ P_v ( U^\varepsilon(\tau_\varepsilon(\p G),\cdot) \in \p G - \p^{\pm z} , \tau_\varepsilon( B_{r_1} ) > \tau_\varepsilon(\p G) )}{P_v (\tau_\varepsilon( B_{r_1} ) > \tau_\varepsilon(\p G))} \leq e^{-\frac{C}{\varepsilon^2}}
\end{equation} for every $\varepsilon > 0$ sufficiently small.
\end{lema}

\begin{proof} The idea is to consider once again a suitable embedded Markov chain. Fix $f > 0$ such that $B_f(\pm z)$ is at a positive distance from $(\p G - \p^{\pm z})_{(d)}$ and assume that $e < \frac{f}{2}$. Consider then the stopping times
$$
\left\{\begin{array}{l} \tilde{\eta}_0 = 0\\
\\
\tilde{\sigma}_0 = \inf \{ t \geq 0 : U^{u,\varepsilon}(t,\cdot) \in \p B_{f}(\pm z) \}
\end{array}\right.
$$
and for $n \in \N_0$
$$
\left\{\begin{array}{l} \tilde{\eta}_{n+1} = \inf \{ t > \tilde{\sigma}_n : U^{u,\varepsilon}(t,\cdot) \in \p B_{e}(\pm z) \cup \partial G\}\\
\\
\tilde{\sigma}_{n+1} = \inf \{ t > \tilde{\eta}_{n+1} : U^{u,\varepsilon}(t,\cdot) \in \p B_{f}(\pm z)\}
\end{array}\right.
$$ with the convention that $\inf \emptyset = +\infty$. We then define the Markov chain $(W^{u,\varepsilon}_n)_{n \in \N}$ as
$$
W^{u,\varepsilon}_n:=U^{u,\varepsilon}({\tilde{\eta}_n},\cdot)
$$ for $n \in \N_0$ and set $\tilde{\vartheta}^u_\varepsilon := \min \{ n \in \N : W^{u,\varepsilon}_n \in \p G\}$. To show \eqref{eqlemab3} it \mbox{suffices to check that}
$$
\sup_{v \in B_e(\pm z)} \frac{ P_v ( W_1^\varepsilon \in \p G - \p^{\pm z} , \tau_\varepsilon( B_{r_1} ) > \tau_\varepsilon(\p G), \tilde{\vartheta}_\varepsilon=1 ) }{ P_v( \tau_\varepsilon( B_{r_1} ) > \tau_\varepsilon(\p G), \tilde{\vartheta}_\varepsilon = 1)} \leq e^{- \frac{C}{\varepsilon^2}}
$$ holds for all $\varepsilon > 0$ sufficiently small provided $e > 0$ is chosen adequately. To see this, \mbox{let us observe} that $V( \pm z , \p G)= 0$ and thus, by Lemma \ref{lemab0}, we obtain the lower bound $e^{- \frac{K f}{\varepsilon^2}}$ for the denominator, where $K > 0$ does not depend on the choice of both $e$ and $f$. On the other hand, since
$$
V(\mathbf{0}, (\p G - \p^{\pm z})_{(2d)} ) \leq V(\mathbf{0}, \pm z) + V( \pm z , (\p G - \p^{\pm z})_{(2d)})
$$ we see that $V(\pm,(\p G - \p^{\pm z})_{(2d)}) > 0$ and thus, with the aid of Lemmas \ref{lemab1} and \ref{lemab2}, \mbox{one concludes the upper} bound $e^{- \frac{h}{\varepsilon^2}}$ for the numerator, for some small constant $h > 0$. The lemma follows at once by taking $f$ sufficiently small.
\end{proof}

We are now ready to finish the proof of Theorem \ref{teoescape0}. Indeed, by \eqref{eqlemab3} we obtain
$$
P_u (A \cap D^2 ) \leq e^{-\frac{C}{\varepsilon^2}}\E( \mathbbm{1}_{D^2} P_{U^{u,\varepsilon}(\tau_\varepsilon^2,\cdot)} (\tau_\varepsilon( B_{r_1} ) > \tau_\varepsilon(\p G) )) \leq e^{-\frac{C}{\varepsilon^2}}P_u ( Z_1^\varepsilon \in \p G)
$$ for every $u \in \p B_{r_1}$ provided that $e > 0$ is taken sufficiently small. This implies that
\begin{equation}\label{decomp3}
\lim_{\varepsilon \rightarrow 0} \left[\sup_{u \in \p B_{r_1}} \frac{P_u (A \cap \{ \tau_\varepsilon^2 = \min\{ \tau_\varepsilon^0, \tau_\varepsilon^1, \tau_\varepsilon^2, \tau_\varepsilon^3 \} \})}{P_u ( Z_1^\varepsilon \in \p G ) }\right] = 0.
\end{equation} and thus concludes the proof.

\begin{obs} By using \eqref{cotacuasipotencial2}, the same argument given here to prove Theorem \ref{teoescape0} can be used to show that for any $\delta > 0$
$$
\lim_{\varepsilon \rightarrow 0} \left[ \sup_{u \in B_c} P_u \left( U^\varepsilon \left( \tau_\varepsilon(\p \tilde{G}),\cdot \right) \notin B_\delta(\pm z) \right) \right] = 0.
$$ This result tells us, perhaps in a more explicit manner than Theorem \ref{teoescape0} does, that for sufficiently small $\varepsilon > 0$ the escape from $\tilde{G}$ of the stochastic system and thus its route towards explosion typically involves passing through the unstable equilibria with minimal potential, namely $\pm z$, at least whenever the initial datum is close enough to the origin. By the argument to be used in the proof of Theorem \ref{taumagnitude} below, this implies that the same fate holds for arbitrary initial data in $\mathcal{D}_{\mathbf{0}}$.
\end{obs}

\section{Asymptotic loss of memory of $\tau_\varepsilon(\p G)$}\label{sec7asymp}

Our next goal is to show the asymptotic loss of memory as $\varepsilon \rightarrow 0$ of $\tau_\varepsilon^u (\p G)$ for $u \in G$. To this end for each
$\varepsilon > 0$ we define the normalization coefficient $\gamma_\varepsilon > 0$ by the relation
$$
P_{\mathbf{0}}( \tau_\varepsilon(\partial G) > \gamma_\varepsilon ) = e^{-1}.
$$ Notice that $\gamma_\varepsilon$ is well defined since $\tau_\varepsilon^{\mathbf{0}}$ is a continuous almost surely finite random variable, with a strictly increasing distribution function.
The result we aim to prove reads as follows.

\begin{teo}\label{escapeteo1} For every $\varepsilon > 0$ consider the function $\nu_\varepsilon : \R_{\geq 0} \rightarrow [0,1]$ given by
$$
\nu_\varepsilon (t) = P_{\mathbf{0}} (\tau_\varepsilon (\partial G) > t\gamma_\varepsilon).
$$ Then
\begin{enumerate}
\item [i.] There exists $(\delta_\varepsilon)_{\varepsilon > 0} \subseteq \R_{> 0}$ satisfying $\lim_{\varepsilon \rightarrow 0} \delta_\varepsilon = 0$ and such that for any $s,t > 0$
    \begin{equation}
    \label{1.49}\nu_\varepsilon(s +
\delta_\varepsilon)\nu_\varepsilon(t) -
\psi_\varepsilon(s,t) \leq \nu_\varepsilon (s+t) \leq
\nu_\varepsilon (s) \nu_\varepsilon (t -
\delta_\varepsilon) + \psi_\varepsilon(s,t)
\end{equation} where $\psi_\varepsilon(s,t)$ is a function which for any fixed $t_0 > 0$ verifies
    \begin{equation}
    \label{1.50}\lim_{\varepsilon \rightarrow 0}\left[\sup_{s \geq0 \,,\,t\geq t_0}\psi_\varepsilon(s,t)\right]= 0.
    \end{equation}
\item [ii.] There exists $\rho > 0$ such that for every $t \geq 0$
\begin{equation}\label{convunie}
\lim_{\varepsilon \rightarrow 0} \left[ \sup_{u \in B_\rho} |P_u( \tau_\varepsilon(\partial G) > t\gamma_\varepsilon) - e^{-t}|\right] = 0.
\end{equation}
\end{enumerate}
\end{teo}

\newpage

\subsection{Coupling of solutions with small initial data}

The key element in the proof of Theorem \ref{escapeteo1} is the fact that, uniformly over any pair $u,v$ of initial data in a small neighborhood of the origin the corresponding escape times $\tau_\varepsilon^u(\p G)$ and $\tau_\varepsilon^v(\p G)$ possess the same asymptotic distribution. We shall establish this fact rigorously on Lemma \ref{escapelema1} below with the aid of a suitable coupling between \mbox{$U^{u,\varepsilon}$ and $U^{v,\varepsilon}$.} More precisely, for $n_0 \in \N$ as in the definition of $G$ the result we require is the following.

\begin{teo}\label{coupling0} There exists $\rho > 0$ such that for any pair of initial data $u,v \in B_\rho$ and $\varepsilon > 0$ sufficiently small there exists a coupling of $U^{(n_0+1),u,\varepsilon}$ and $U^{(n_0+1),v,\varepsilon}$ satisfying
\begin{equation}\label{coupling2}
U^{(n_0+1),u,\varepsilon}(t,\cdot) \equiv U^{(n_0+1),v,\varepsilon}(t,\cdot) \hspace{0.5cm}\text{ for all }t \geq \eta^{\varepsilon}_{u,v},
\end{equation} where
$$
\eta^\varepsilon_{u,v} := \inf\{ t \geq 0 : U^{(n_0+1),u,\varepsilon}(t,\cdot) \equiv U^{(n_0+1),v,\varepsilon}(t,\cdot)\}.
$$ Furthermore, $\eta^\varepsilon_{u,v}$ satisfies
\begin{equation}\label{coupling1}
\lim_{\varepsilon \rightarrow 0} \left[\sup_{u,v \in B_\rho} P \left( \eta^\varepsilon_{u,v} \geq \frac{1}{\varepsilon^3} \right)\right]=0.
\end{equation}
\end{teo}

The existence of a coupling fulfilling these characteristics was first established in \cite{M} for a class of stochastic differential equations with periodic boundary \mbox{conditions of the sort}
\begin{equation}\label{MainSPDEMueller}
\left\{\begin{array}{rll}
\p_t U^u &= \p^2_{xx}U^u - \alpha U^u + a(U^u) + b(U^u)\dot{W} & \quad t>0 \,,\, x \in S^1 \\
U^u(0,x) &= u(x) &\quad x \in S^1
\end{array}\right.
\end{equation}
where $\alpha > 0$ is fixed parameter, $a: \R \rightarrow \R$ is a Lipschitz nonincreasing function and $b: \R \rightarrow \R$ is a positive, Lipschitz function bounded away from both zero and infinity. In this work, Mueller considers an arbitrary pair of continuous functions $u,v \in C(S^1)$ as initial data and shows that the coupling time $\eta^\varepsilon_{u,v}$ is almost surely finite. Later this very same result was adapted in \cite{B2} to a particular system verifying the assumptions in \cite{M} but including Dirichlet boundary conditions instead of periodic ones. Furthermore, in this second work Brassesco shows the asymptotic estimate \eqref{coupling1} for the coupling time $\eta^\varepsilon_{u,v}$. Unfortunately, for our system the coefficient $a$ fails to be increasing so that the results on \cite{B2} cannot be directly applied. Nevertheless, it is still possible to obtain Theorem \ref{coupling0} by performing some minor adjustments to the proof given there, although this comes at the expense of
losing the almost sure finiteness of the coupling time and also a certain freedom in the choice of initial data. In what follows we present a brief summary of the \mbox{proof of Theorem \ref{coupling0}}, highlighting the main differences with \cite{B2} and explaining how to deal with each of them. We refer the reader to \cite{M} and \cite{B2} for the remaining details.

Let us begin by observing that it will suffice to show the coupling for initial data $u,v$ such that $u \geq v$. Indeed, it follows from the proof in this case that if $w:= \max\{u,v\}$ then we can construct the three solutions $U^{(n_0+1),u,\varepsilon}$, $U^{(n_0+1),v,\varepsilon}$ and $U^{(n_0+1),w,\varepsilon}$ in the same probability space so that $U^{(n_0+1),u,\varepsilon}$ and $U^{(n_0+1),w,\varepsilon}$ are identical after the time $\eta^\varepsilon_{u,w}$ and also $U^{(n_0+1),u,\varepsilon}$ and $U^{(n_0+1),w,\varepsilon}$ are identical after the time $\eta^\varepsilon_{v,w}$. It then follows that $U^{(n_0+1),u,\varepsilon}$ and $U^{(n_0+1),v,\varepsilon}$ must become identical after the $\eta^\varepsilon_{u,v} \leq \max\{ \eta^\varepsilon_{u,w}, \eta^\varepsilon_{v,w}\}$, so that it suffices to estimate the desired probability in \eqref{coupling1} for the simpler case $u \geq v$.

Thus, we assume that $u \geq v$ and given two independent Brownian sheets $W_1$ and $W_2$ we consider the pair $(U^{(n_0+1),u,\varepsilon},U^{(n_0+1),v,\varepsilon})$ which satisfies
$$
\p_t U^{(n_0+1),u,\varepsilon} = - \frac{\partial S^{(n_0+1)}}{\partial \varphi}(U^{(n_0+1),u,\varepsilon}) + \eps \dot{W_1}
$$ and
$$
\p_t U^{(n_0+1),v,\varepsilon} = - \frac{\partial S^{(n_0+1)}}{\partial \varphi}(U^{(n_0+1),v,\varepsilon}) + \eps \left( \left( \sqrt{1- \min\{|E|, 1\}} \dot{W_1} + \sqrt{\min\{|E|,1\}} \dot{W_2} \right)\right)
$$
with initial data $u$ and $v$ respectively and where
$$
E:=U^{(n_0+1),u,\varepsilon} - U^{(n_0+1),v,\varepsilon}.
$$ Thus, both $U^{(n_0+1),u,\varepsilon}$ and $U^{(n_0+1),v,\varepsilon}$ are solutions \eqref{eqtruncada} with the appropriate initial data, constructed in the same probability space but with respect to different white noises.
Following Lemma 3.1 of \cite{M} it is possible to construct the pair $(U^{(n_0+1),u,\varepsilon},U^{(n_0+1),v,\varepsilon})$ in such a way that the process $E$ is nonnegative almost surely for $\varepsilon > 0$ sufficiently small. Then, using the weak formulation of solutions to \eqref{MainSPDE} available on Lemma \ref{weaksol} one can immediately see as in \cite{B2} that if we write
$$
U(t) = \int_0^1 E(t,x) \sin (\pi x) dx
$$ then $U$ satisfies
\begin{equation}\label{EcuaU}
U(t)= U(t) + \int_0^t C(s)ds + M_t
\end{equation}
where
\begin{equation}\label{drift}
C(s) = \int_0^1 \left( g_{n_0+1}\left(U^{(n_0+1),u,\varepsilon}(s,x)\right)-g_{n_0+1}\left(U^{(n_0+1),v,\varepsilon}(s,x)\right) - \pi^2 E(s,x)\right) \sin(\pi x) dx.
\end{equation}
and $M$ is a continuous martingale with respect to the filtration generated by $W_1$ and $W_2$ satisfying
\begin{equation}\label{compensator}
\langle M \rangle(t) = 2\eps^2 \int_0^t\int_0^1 \frac{\min\{ E(s,x), 1\}}{1 + \sqrt{1-\min\{E(s,x),1\}}}dxds.
\end{equation} Notice that in order for \eqref{EcuaU} to hold, it was necessary to introduce a $C^\infty$ function in $C_D([0,1])$ in the definition of $U$.
We selected $\sin(\pi x)$ as in \cite{B2} but the same reasoning also holds for other nonnegative $C^\infty$ functions. Now, from \eqref{compensator} we easily obtain
$$
\frac{\langle M \rangle (t)}{dt} \geq \varepsilon^2 \int_0^1 \sin^2 (\pi x) \min\{ E(t,x), 1\}dx.
$$ Using Hölder's inequality we obtain that
$$
\int_0^1 \sin(\pi x) \min\{E(t,x),1\}dx \leq \left(\int_0^1 \sin^2(\pi x) \min\{E(t,x),1\}dx\right)^{\frac{5}{8}} \left(\int_0^1 \sin^{-\frac{2}{3}}(\pi x)dx\right)^{\frac{3}{8}}
$$ which implies the estimate
\begin{align*}
\frac{\langle M \rangle (t)}{dt}& \geq K\varepsilon \left[\int_0^1 \sin(\pi x) \min\{E(t,x),1\}dx\right]^{\frac{8}{5}}\\
\\
& \geq K\varepsilon \left[\int_0^1 \sin(\pi x) \frac{E(t,x)}{\max\{E(t,x),1\}}dx\right]^{\frac{8}{5}}\\
\\
& \geq \frac{K\varepsilon^2}{\left[\sup_{x \in [0,1]} \max\{E(t,x),1\}\right]^\frac{8}{5}}\left(U(t)\right)^{\frac{8}{5}}
\end{align*} where
$$
K := \left(\int_0^1 \sin^{-\frac{2}{3}}(\pi x)dx\right)^{-\frac{3}{5}}.
$$ Thus, we conclude that there exists an adapted process $D$ such that for all $t \geq 0$
$$
\frac{\langle M \rangle (t)}{dt} = \left( U(t) \right)^{\frac{8}{5}}D(t)
$$ and
\begin{equation}\label{boundtimechange}
D(t) \geq \frac{K\varepsilon^2}{\left[\sup_{x \in [0,1]} \max\{E(t,x),1\}\right]^\frac{8}{5}}.
\end{equation} Next, we introduce the time change
$$
\varphi(t) = \int_0^t D(s)ds
$$ and consider the time-changed process
$$
X(t) := U( \varphi^{-1}(t)).
$$ In \cite{M} it is shown that whenever the condition
\begin{equation}\label{conditiontimechange}
\max\left\{ \sup_{t \geq 0} \E\left( \sup_{x \in [0,1]} U^{(n_0+1),u,\varepsilon}(t,x) \right) , \sup_{t \geq 0} \E \left( \sup_{x \in [0,1]} U^{(n_0+1),v,\varepsilon}(t,x) \right) \right\} < +\infty
\end{equation}is met then $\lim_{t \rightarrow +\infty} \varphi(t) = +\infty$ so that the process $X$ is globally defined. Unfortunately, in this article condition \eqref{conditiontimechange} is seen to hold only under the presence of the linear term in \eqref{MainSPDEMueller}, i.e. $\alpha > 0$, which is missing in our system. However, since we are interested in achieving the coupling between the solutions before they escape the domain $G$, \mbox{one can} modify the source term $g_{n_0 +1}$ outside $B_{n_0 +1}$ in such a way that \eqref{conditiontimechange} is satisfied without it affecting our plans. Hence, by \eqref{EcuaU} and the definition of $\varphi$, it can be seen that $X$ satisfies
$$
X(t) = U(0) + \int_0^t \tilde{C}(s) ds + \int_0^t \left(X(t)\right)^{\frac{8}{10}} dB_t
$$ for a certain Brownian motion $B$ and $\tilde{C}$ given by the formula
$$
\tilde C(t):=C(\varphi^{-1}(t)) \frac{1}{\varphi'(\varphi^{-1}(t))}.
$$ Now, Itô's formula yields that, up until its arrival time at zero, the process $Y: = 5 X^{\frac{1}{5}}$ satisfies
$$
Y(t) = 5\left(U(0)\right)^{\frac{1}{5}} +\int_0^t \left(\frac{\tilde C(s)}{\left(Y(s)\right)^4}-\frac{2}{5Y(s)}\right)ds + B_t
$$ In both \cite{M} and \cite{B2} the corresponding term $\tilde{C}$ is nonpositive, so that $Y$ is guaranteed to hit zero before the time the Brownian motion $B$ takes to reach $-5\left(U(0)\right)^{\frac{1}{5}}$. However, in our case the term $\tilde{C}$ may eventually take positive values so that some additional work is needed in order to arrive at the same conclusion. Notice that a straightforward calculation using the definition of $\varphi$ and the bound \eqref{boundtimechange} shows that if for some $h> 0$ the term $C(t)$ is nonpositive for all $t \leq e^{\frac{h}{\varepsilon^2}}$ then $\tilde{C}(t)$ is also nonpositive but only for all $t \leq e^{\frac{h}{2\varepsilon^2}}$. Hence, if we set
$$
\gamma:= \inf \{ t \geq 0 : B_t = -5\left(U(0)\right)^{\frac{1}{5}}\}
$$ then $Y$ is guaranteed to hit zero before $\gamma$ provided that there exists $h > 0$ such that $\gamma \leq e^{\frac{h}{2\varepsilon^2}}$ and the term $C(t)$ is nonpositive for all $t \leq e^{\frac{h}{\varepsilon^2}}$. Now, $Y(t)=0$ implies that $U(\varphi^{-1}(t))=0$ and ultimately that $E( \varphi^{-1}(t), x ) \equiv \mathbf{0}$ which means that both solutions $U^{(n_0+1),u,\varepsilon}$ and $U^{(n_0+1),v,\varepsilon}$ coincide at time $t$. But, since the process $E$ is governed by the differential equation
$$
\p_t E = \p^{2}_{xx} E + \left(g_{n_0+1}(U^{(n_0+1),u,\varepsilon}) - g_{n_0+1}(U^{(n_0+1),v,\varepsilon})\right) + 2\varepsilon^2 \frac{\min\{ E(s,x), 1\}}{1 + \sqrt{1-\min\{E(s,x),1\}}}\dot{W},
$$ we see that once the solutions meet each other, they remain identical forever afterwards and thus the coupling is achieved. Hence, we conclude that the pair $(U^{(n_0+1),u,\varepsilon},U^{(n_0+1),v,\varepsilon})$ satisfies \eqref{coupling2} and furthermore that, if there exists $h > 0$ such that $\gamma \leq e^{\frac{h}{2\varepsilon^2}}$ and the term $C(t)$ is nonpositive for all $t \leq e^{\frac{h}{\varepsilon^2}}$, then $\eta^\varepsilon_{u,v}$ is bounded from above by $\varphi^{-1}(\gamma)$. Our goal now is then to find $h > 0$ such that both these conditions are satisfied with overwhelming probability as $\varepsilon > 0$ tends to zero for all $u,v$ in a small neighborhood of the origin.

Notice that, since $g'_{n_0 + 1}$ is continuous and $g'_{n_0+1}(0)= g_{n_0+1}(0) = 0$, there exists $\delta > 0$ sufficiently small such that $\sup_{|y|\leq \delta} g_{n_0+1}(y) \leq \pi^2$. Therefore, if for some $t \geq 0$ we have that $U^{(n_0+1),u,\varepsilon}(t,\cdot)$ and $U^{(n_0+1),v,\varepsilon}(t,\cdot)$ both belong to $B_\delta$ then by \eqref{drift} we obtain
$$
C(t) \leq \int_0^1 \left[\sup_{|y| \leq \delta} g'_{n_0+1}(y) - \pi^2\right] E(t,x) \sin(\pi x) dx \leq 0.
$$
Since one can show as in the proof of \eqref{lejosdelminimo} that for every $r > 0$ sufficiently small one has
$$
\inf_{u \in \p B_r} S(u) > S(\mathbf{0})=0
$$ then by the methods applied in Section \ref{seclowerbound} we conclude that there exists $0 < \rho < \delta$ sufficiently small and $h > 0$ such that
$$
\lim_{\varepsilon \rightarrow 0} \left[\sup_{u \in B_\rho} P_u \left( \tau_\varepsilon (\p B_\delta) \leq e^{\frac{h}{\varepsilon^2}} \right) \right] = 0.
$$ Thus, by all these considerations we see that if $\delta < \frac{1}{2}$ then for $\varepsilon > 0$ sufficiently small
\begin{equation}\label{cotacoupling0}
P\left( \eta^\varepsilon_{u,v} \geq \frac{1}{\varepsilon^3} \right) \leq 2 \left( P\left( \gamma > \frac{K}{\varepsilon} \right) +  \sup_{u \in B_\rho} P_u \left( \tau_\varepsilon (\p B_\delta) \leq e^{\frac{h}{\varepsilon^2}} \right) \right)
\end{equation} where we have used the fact that if $U^{(n_0+1),u,\varepsilon}(t,\cdot)$ and $U^{(n_0+1),v,\varepsilon}(t,\cdot)$ both belong \mbox{to $B_\delta$} for all $t \in [0,\frac{1}{\varepsilon^3}]$ and $\delta < \frac{1}{2}$ then $\sup_{x \in [0,1]} |E(t,x)| < 1$ for all $t \in [0,\frac{1}{\varepsilon^3}]$ \mbox{so that $\varphi(\frac{1}{\varepsilon^3}) \geq \frac{K}{\varepsilon}$.} Since the bound obtained in \eqref{cotacoupling0} holds for all pairs $u,v \in B_{\rho}$ and $\gamma$ is \mbox{almost surely finite,} we conclude \eqref{coupling1} and so Theorem \ref{coupling0} is proved.

As a direct consequence of Theorem \ref{coupling0} we now obtain the following lemma which establishes the claim in the beginning of the section regarding the asymptotic distribution
of the escape time $\tau_{\varepsilon}(\p G)$ for initial data in a small neighborhood of the origin.

\begin{lema}\label{escapelema1} There exists $\rho > 0$ such that for every $t_0 > 0$
\begin{equation}
\lim_{\varepsilon \rightarrow 0}\left[ \sup_{u,v \in B_{\rho}} \left[ \sup_{t > t_0} |P_u(\tau_{\varepsilon}(\partial G)  > t\gamma_{\varepsilon}) - P_{v}(\tau_{\varepsilon}(\partial G) > t\gamma_{\varepsilon})|\right]\right] = 0.
\end{equation}
\end{lema}

\begin{proof} Let $\rho > 0$ be as in Theorem \ref{coupling0} and given a pair of initial data $u,v \in B_\rho$ let us consider the coupling $\left( U^{(n_0+1),u,\varepsilon}, U^{(n_0+1),v,\varepsilon}\right)$ constructed in the aforementioned theorem. Since by the results established in Section \ref{seclowerbound} for any given $t_0 > 0$ there exists $\varepsilon_0 > 0$ such that $t_0 \gamma_\varepsilon > \frac{1}{\varepsilon^3}$ for all $0 < \varepsilon < \varepsilon_0$, then for all $t \geq t_0$ we have
$$
|P_u(\tau_{\varepsilon}(\partial G)  > t\gamma_{\varepsilon}) - P_{v}(\tau_{\varepsilon}(\partial G) > t\gamma_{\varepsilon})| \leq P ( \eta^\varepsilon_{u,v} \geq t\gamma_\varepsilon ) \leq P\left( \eta^\varepsilon_{u,v} \geq \frac{1}{\varepsilon^2}\right)
$$ for all $0 < \varepsilon < \varepsilon_0$, so that the result follows at once by \eqref{coupling1}.
\end{proof}

\begin{obs} In \cite{MOS} the authors study the asymptotic distribution of the tunneling in a double-well potential model, i.e. the time needed for the stochastic system to go from one well to a small neighborhood of the bottom of the other one. They show that, under proper normalization, the tunneling time converges in distribution to an exponential random variable, as it also happens with $\tau^\varepsilon(\p G)$ in our case. To do this they show an analogue of Lemma \ref{escapelema1} but using a different technique, which is based on an exponential loss of memory of the initial datum. The joining of trajectories due to the attractive drift in the final part of the motion plays an essential role in their argument, and so the reasoning no longer works, for example, when studying the exit time from a bounded region containing only the attractor. However, we point out that the approach we introduce here relying on the coupling of solutions does not have the same limitation, and so it can also be used to study
these other type of problems. This is shown in detail in \cite{B2}.
\end{obs}

\subsection{Proof of Theorem \ref{escapeteo1}}

We shall need the results contained in following lemma for the proof of Theorem \ref{escapeteo1}.

\begin{lema}\label{escapelema2} Let us consider $ 0 < \alpha < \Delta$ and define $\eta_\varepsilon := e^{\frac{\alpha}{\varepsilon^2}}$. Then
\begin{enumerate}
\item [i.] $\lim_{\varepsilon \rightarrow 0}\frac{\eta_\varepsilon}{\gamma_\varepsilon} = 0$
\item [ii.] $\lim_{\varepsilon \rightarrow 0} \left[ \displaystyle{\sup_{u \in G} P_u ( \tau_{\varepsilon}(\partial G) > \eta_\varepsilon \,,\, \tau_{\varepsilon}(B_{\rho}) > \eta_\varepsilon)}\right] = 0$ for any $\rho > 0$.
\end{enumerate}
\end{lema}

\newpage
\begin{proof} Let us notice that by the bounds established for $\tau_\varepsilon(\p G)$ in Sections \ref{secupperbound} and \ref{seclowerbound} we have that
$$
\lim_{\varepsilon \rightarrow 0} \varepsilon^2 \log \gamma_\varepsilon = \Delta
$$ from where (i) immediately follows. Next, we establish (ii) with the aid from the large deviations principle as in Lemma \ref{cotasuplema0}. We must show that there exists a time $T > 0$ such that for each $u \in G$ there exists a set of paths $\mathcal{E}_{u,T} \subseteq C_{D_u}([0,T]\times [0,1])$ satisfying
\begin{enumerate}
\item [$\bullet$] Every path in $\mathcal{E}_{u, T}$ reaches $\partial G \cap B_\rho$ before times $T$.
\item [$\bullet$] $\inf_{x \in G} P_u (U^{\varepsilon} \in \mathcal{E}_{u, T}) \geq \tilde{\alpha}_{\varepsilon}$, where $\tilde{\alpha}_{\varepsilon}:= Te^{-\frac{\alpha}{2\varepsilon^{2}}}.$
\end{enumerate}
Once again, it suffices to show that for each $u \in G$ there exists $\varphi^u \in C_{D_u}([0,T]\times [0,1])$ starting at $u$ with rate less than $\frac{\alpha}{3}$ and such that not only does $\varphi^u$ reach $\partial G \cup B_\rho$ before time $T$, but also any path sufficiently close to $\varphi^u$ does so as well. The construction of such a $\varphi^u$ is similar to the one given in the proof of Lemma \ref{cotasuplema0}. The remainder of the proof follows once again from the large deviations principle valid for our system.
\end{proof}

With Lemmas \ref{escapelema1} and \ref{escapelema2} at our disposal, we are now ready to prove Theorem \ref{escapeteo1}. Given $s > 0$ let us define
$$
R^{u,s}_{\varepsilon} = \inf \{ r > s\gamma_\varepsilon : U^{u,\varepsilon}(r,\cdot) \in B_\rho\}
$$ where $\rho > 0$ is given by Lemma \ref{escapelema1}. We may then decompose $\nu_\varepsilon (t+s)$ as
$$
\nu_\varepsilon (t+s) = P_{\mathbf{0}} ( \tau_\varepsilon(\partial G) > (s+t)\gamma_\varepsilon \,,\, R^{s}_{\varepsilon} > s\gamma_\varepsilon + \eta_\varepsilon) + P_{\mathbf{0}} ( \tau_\varepsilon(\partial G) > (s+t)\gamma_\varepsilon \,,\, R^{s}_{\varepsilon} \leq s \gamma_\varepsilon + \eta_\varepsilon).
$$ Let us observe that for $u \in G$ the Markov property yields
$$
P_u ( \tau_\varepsilon(\partial G) > s\gamma_\varepsilon + \eta_\varepsilon \,,\, R^{s}_{\varepsilon} > s\gamma_\varepsilon  + \eta_\varepsilon) \leq \sup_{u \in G} P_u \left(\tau_\varepsilon(\partial G) > \eta_\varepsilon \,,\, \tau_\varepsilon(B_\rho) > \eta_\varepsilon\right).
$$
Thus, by Lemma \ref{escapelema2} we conclude that for any fixed $t_0 > 0$
$$
\lim_{\varepsilon \rightarrow 0} \left[\sup_{s\geq 0\,,\,t\geq t_0} \left[\sup_{u\in G} P_u (\tau_\varepsilon(\partial G) > (s+t)\gamma_\varepsilon \,,\, R^{s}_{\varepsilon} > s\gamma_\varepsilon + \eta_\varepsilon)\right]\right] = 0.
$$
To establish (i) it suffices then to give proper upper and lower bounds on the second term of the decomposition. But by applying the strong Markov property with respect to the stopping time $R^{\mathbf{0},s}_\varepsilon$ we obtain $$
P_\mathbf{0} (\tau_\varepsilon(\partial G) > (s+t)\gamma_\varepsilon \,,\, R^{s}_{\varepsilon} \leq s\gamma_\varepsilon + \eta_\varepsilon) \leq P_\mathbf{0}(\tau_\varepsilon(\partial G)> s\gamma_\varepsilon) \left[\sup_{u \in
B_\rho} P_u (\tau_\varepsilon(\partial G) > t\gamma_\varepsilon - \eta_\varepsilon)\right]
$$ and
$$
P_\mathbf{0} (\tau_\varepsilon(\partial G) > (s+t)\gamma_\varepsilon \,,\, R^{s}_{\varepsilon} \leq s\gamma_\varepsilon + \eta_\varepsilon) \geq P_\mathbf{0} ( R^{s}_{\varepsilon} \leq s \gamma_\varepsilon + \eta_\varepsilon)\left[\inf_{u \in B_{\rho}}P_u(\tau_\varepsilon (\partial G) > t\gamma_\varepsilon)\right].
$$ From this we immediately obtain (i) by using Lemmas \ref{escapelema1} and \ref{escapelema2}. Now, assertion (ii) will follow immediately from Lemma \ref{escapelema1} once we manage to show that for every $t > 0$ we have
\begin{equation}\label{casop}
\lim_{\varepsilon \rightarrow 0} \nu_\varepsilon(t)=e^{-t}.
\end{equation} To see this, let us first observe that by applying (i) successively we obtain
$$
\left\{ \begin{array}{l}\nu_\varepsilon (2k) \leq [\nu_\varepsilon (2 - \delta_\varepsilon)]^k + \sum_{i=1}^{k-1} \psi_{\varepsilon} (2i,2)\\
\\
\nu_\varepsilon(1)= e^{-1}\leq [\nu_\varepsilon (\frac{1}{k} - \delta_\varepsilon)]^k + \sum_{i=1}^{k-1} \psi_{\varepsilon} (\frac{i}{k},\frac{1}{k}).
\end{array}\right.
$$
Thus, given $0 < \delta < 1$ and $k \in \N$ such that $e^{-k} < \frac{\delta}{2}$ and $(1- \delta)^k < \frac{e^{-1}}{2}$, in light of \eqref{1.50} we may take
$\varepsilon_0 > 0$ such that for every $\varepsilon \leq \varepsilon_0$ the following conditions hold:
\begin{itemize}
\item$\sum_{i=1}^{k-1} \psi_{\varepsilon} (2i,2) <
    \frac{\delta}{2},$
\item$\sum_{i=1}^{k-1} \psi_{\varepsilon}
    (\frac{i}{k},\frac{1}{k}) < \frac{e^{-1}}{2},$
\item$2 - \delta_\varepsilon > 1,$
\item$\frac{1}{k} - \delta_\varepsilon > \frac{1}{2k}.$
\end{itemize}
Under these conditions it can be seen that $\nu_\varepsilon(2k) < \delta \:,\: \nu_\varepsilon (\frac{1}{2k}) > 1 - \delta$ for every $\varepsilon \leq \varepsilon_0$. In particular, this implies that any sequence $(\varepsilon_j)_{j \in \N} \subseteq \R_{ > 0}$ with $\lim_{j \rightarrow +\infty} \varepsilon_j = 0$ satisfies that the family $(\nu_{\varepsilon_j})_{j \rightarrow +\infty}$ is tight, i.e.
$$
\lim_{k \rightarrow +\infty} \left[ \inf_{j \in \N} \left[ \nu_{\varepsilon_j}(k) - \nu_{\varepsilon_j}(k^{-1}) \right] \right] = 1.
$$ Therefore, by Prohorov's theorem we see that in order to establish \eqref{casop} we must only check that any sequence $(\nu_{\varepsilon_j})_{j \in \N}$ which is weakly convergent has the mean one exponential distribution as its limit. But if we denote this limit by $\nu$, then (i) implies that $\nu$ must satisfy the memory loss property, i.e. for every $s,t > 0$
$$
\nu(s+t)=\nu(s)\nu(t)
$$ and thus it must be $\nu(t)=e^{-\lambda t}$ for some $\lambda \geq 0$. By recalling that $\nu_\varepsilon(1)=e^{-1}$ for every $\varepsilon > 0$ we see that $\lambda = 1$. This concludes the proof.

\newpage

\section{Resumen del Capítulo 4}

Este capítulo se encuentra dedicado a estudiar, para datos iniciales $u$ en un entorno pequeño de $\mathbf{0}$, el fenómeno del escape del dominio $G$ construido en el Capítulo 3 por parte del sistema estocástico $U^{u,\varepsilon}$. El problema del escape de un dominio acotado con estas características fue originalmente estudiado en \cite{GOV} para el caso de un potencial de doble pozo finito-dimensional, y luego investigado en \cite{B1} en su variante infinito-dimensional. Los resultados que presentamos en este capítulo son una adaptación a nuestro contexto de los resultados que aparecen en dichas referencias.

El primer resultado caracteriza el orden de magnitud asintótico de $\tau^u_\varepsilon(\p G)$, el tiempo de salida del dominio $G$.

\noindent \textbf{Teorema}. Dado $\delta > 0$ se tiene
$$
\lim_{\varepsilon \rightarrow 0 } \left[\sup_{u \in B_c} \left| P_{u}
\left( e^{\frac{\Delta - \delta}{\varepsilon^{2}}} < \tau_{\varepsilon}(\partial
G) < e^{\frac{\Delta + \delta}{\varepsilon^{2}}}\right)-1 \right|\right] = 0,
$$ donde $B_c$ es el entorno del origen resaltado en la construcción de $G$.

Para probar este resultado mostraremos por separado la cota superior \eqref{cotsuplema1} y la inferior \eqref{cotainferioreq0}.

La cota superior se sigue del hecho de que dado $\delta > 0$ existe $T > 0$ tal que para todo $u \in G$ existe una trayectoria $\varphi^u \in C_D([0,T] \times [0,1])$
con $\varphi^u(0)=u$ y de tasa $I^u_T(\varphi^u) < \Delta + \frac{\delta}{3}$ tal que toda trayectoria suficientemente cercana a $\phi^u$ se escapa de $G$ antes de tiempo $T$.
Usando la estimación \eqref{LDP1} podemos concluir entonces que
$$
\inf_{u \in G} P( \tau^u_\varepsilon (\p G) \leq T ) \geq e^{- \frac{\Delta + \frac{\delta}{2}}{\varepsilon^2}}
$$ de modo tal que, por la propiedad de Markov, el tiempo que $U^{u,\varepsilon}$ tarde en escapar de $G$ será típicamente menor a $T e^{\frac{\Delta + \frac{\delta}{2}}{\varepsilon^2}}$. Observando que
$T e^{\frac{\Delta + \frac{\delta}{2}}{\varepsilon^2}} \ll e^{\frac{\Delta + \delta}{\varepsilon^2}}$ cuando $\varepsilon \rightarrow 0$ se concluye el resultado.

Para la cota inferior se divide el intervalo $[0,\tau^u_\varepsilon(\p G)]$ en subintervalos disjuntos que corresponden a las excursiones que realiza el sistema $U^{u,\varepsilon}$ alejándose de $\mathbf{0}$ en busca de $\p G$. En todas estas excursiones el sistema $U^{u,\varepsilon}$ fracasa en llegar a $\p G$ exceptuando la última de ellas, donde finalmente consigue el éxito y alcanza $\p G$. Es posible mostrar que cada una de estas excursiones tiene típicamente una longitud mayor a cierto $T' > 0$ y que la probabilidad de que sea exitosa es inferior a $e^{-\frac{\Delta - \frac{\delta}{2}}{\varepsilon^2}}$, de modo tal que el tiempo que $U^{u,\varepsilon}$ tarde en escapar de $G$ será típicamente mayor a $T e^{\frac{\Delta - \frac{\delta}{2}}{\varepsilon^2}}$. Observando que
$T e^{\frac{\Delta + \frac{\delta}{2}}{\varepsilon^2}} \gg e^{\frac{\Delta - \delta}{\varepsilon^2}}$ cuando $\varepsilon \rightarrow 0$ se concluye el resultado.

El siguiente resultado obtenido en este capítulo muestra que el sistema $U^{u,\varepsilon}$ típicamente se escapa de $\p G$ por $\p^{\pm z}$.

\noindent \textbf{Teorema}. Si $B_c$ es el entorno de $\mathbf{0}$ resaltado en la construcción de $G$ entonces
$$
\lim_{\varepsilon \rightarrow 0} \left[\sup_{u \in B_c} P_u \left( U^{\varepsilon}(\tau_\varepsilon(\partial G),\cdot) \notin \partial^{\pm z}\right)\right] = 0.
$$

Para probar este resultado es necesario nuevamente dividir el intervalo $[0,\tau^u_\varepsilon(\p G)]$ en las distintas excursiones que realiza el sistema $U^{u,\varepsilon}$ alejándose de $\mathbf{0}$ en busca de $\p G$, y estimar la probabilidad de que en la última de ellas el sistema haya alcanzado $\p^{\pm z}$. Por la propiedad de Markov esto coincide con estimar la probabilidad de que la excursión inicial haya alcanzado $\p^{\pm z}$ condicionada a ser exitosa. Como $V(\mathbf{0}, \p G - \p^{\pm z}) > V(\mathbf{0}, \p^{\pm z})$, i.e. el costo para el sistema $U^{\mathbf{0},\varepsilon}$ de escapar de $G$ es menor si lo hace por $\p^{\pm z}$, con ayuda de las estimaciones de grandes desvíos es posible probar que dicha probabilidad condicional tiende a cero cuando $\varepsilon \rightarrow 0$ y se obtiene así el resultado.

El último resultado de este capítulo concierne la distribución asintótica del tiempo de escape $\tau^u_\varepsilon (\p G)$. Concretamente, mostramos que bajo una normalización adecuada, $\tau^u_\varepsilon (\p G)$ converge en distribución a una variable aleatoria exponencial.

\noindent \textbf{Teorema}. Si para cada $\varepsilon > 0$ definimos el coeficiente $\gamma_\varepsilon > 0$ mediante la relación
$$
P_{\mathbf{0}}(\tau_\varepsilon (\p G) > \gamma_\varepsilon ) = e^{-1}
$$ entonces existe $\rho > 0$ tal que para todo $t \geq 0$ se tiene
$$
\lim_{\varepsilon \rightarrow 0} \left[ \sup_{u \in B_\rho} |P_u( \tau_\varepsilon(\partial G) > t\gamma_\varepsilon) - e^{-t}|\right] = 0.
$$

Para demostrar este resultado, primero tratamos el caso $u= \mathbf{0}$ separadamente. Para probar el resultado en este caso, si definimos $\nu_\varepsilon(t) = P_{\mathbf{0}}(\tau_\varepsilon (\p G) > t\gamma_\varepsilon )$, bastará con verificar que
\begin{enumerate}
\item [i.] La familia de distribuciones $(\nu_\varepsilon)_{\varepsilon > 0}$ es asintóticamente acotada en probabilidad, i.e. uniformemente sobre $\varepsilon > 0$ suficientemente chico.
\item [ii.] Cualquier límite por subsucesiones de $\nu_\varepsilon$ cuando $\varepsilon \rightarrow 0$ es exponencial de \mbox{parámetro 1.}
\end{enumerate}
Tanto (i) como (ii) se obtienen fácilmente una vez que se demuestran las desigualdades en \eqref{1.49}. La dificultad más importante a la hora de mostrar \eqref{1.49} yace en verificar que la distribución asintótica de $\tau^u_\varepsilon(\p G)$ es la misma para datos iniciales $u$ en un entorno suficientemente pequeño de $\mathbf{0}$, i.e. existe $\rho > 0$ tal que para todo $t_0 > 0$
$$
\lim_{\varepsilon \rightarrow 0}\left[ \sup_{u,v \in B_{\rho}} \left[ \sup_{t > t_0} |P_u(\tau_{\varepsilon}(\partial G)  > t\gamma_{\varepsilon}) - P_{v}(\tau_{\varepsilon}(\partial G) > t\gamma_{\varepsilon})|\right]\right] = 0.
$$ Para probar esto recurrimos a un acoplamiento entre soluciones del sistema estocástico para datos en un entorno del origen similar al estudiado en \cite{B2}. Con este resultado, a partir del caso $u=\mathbf{0}$ se deduce inmediatamente el caso $u \in B_\rho$.

\chapter{Asymptotic behavior of $\tau^u_\varepsilon$ for $u \in \mathcal{D}_\mathbf{0}$}

\section{Asymptotic properties of $\tau_\varepsilon$ for initial data in $\mathcal{D}_\mathbf{0}$}\label{sec8}

In this section we devote ourselves to establishing the asymptotic properties as $\varepsilon \rightarrow 0$ of the explosion time $\tau_\varepsilon^u$ for arbitrary $u \in \mathcal{D}_{\mathbf{0}}$. Our first result in this direction, detailed on the following theorem, is concerned with its asymptotic magnitude.

\begin{teo}\label{taumagnitude}
For any bounded set $\mathcal{K} \subseteq \mathcal{D}_{\mathbf{0}}$ at a positive distance from $\mathcal{W}$ and $\delta > 0$
\begin{equation*}\label{taumagnitud1}
\lim_{\varepsilon \rightarrow 0} \left[ \sup_{u \in \mathcal{K}} P_u \left( e^{\frac{\Delta - \delta}{\varepsilon^2}} < \tau_\varepsilon < e^{\frac{\Delta + \delta}{\varepsilon^2}}\right)\right]=1.
\end{equation*}
\end{teo}

\begin{proof} First let us suppose that $\mathcal{K}= B_{c}$. In this case the continuity of trajectories and the strong Markov property imply that for $u \in B_{c}$ we have
$$
P_u \left( \tau_\varepsilon < e^{\frac{\Delta - \delta}{\varepsilon^2}} \right) \leq P_u \left( \tau_\varepsilon (\partial G) < e^{\frac{\Delta - \delta}{\varepsilon^2}} \right)
$$ and
$$
P_u \left( \tau_\varepsilon > e^{\frac{\Delta + \delta}{\varepsilon^2}} \right) \leq P_u \left( \tau_\varepsilon(\partial G) > e^{\frac{\Delta + \frac{\delta}{2}}{\varepsilon^2}} \right) + P_u \left( U(\tau_\varepsilon(\partial G),\cdot) \notin \p^{\pm z} \right) + \sup_{v \in \p^{\pm z}} P_v ( \tau_\varepsilon > \tau^* )
$$
for every $\varepsilon > 0$ sufficiently small, from which we can conclude the result in this case by the results in Section \ref{secescapedeg}.

Now, let us observe that for any $u \in \mathcal{D}_{\mathbf{0}}$ the system $U^u$ reaches the set $B_{\frac{c}{2}}$ in a finite time $\tau^u(B_{\frac{c}{2}})$ while remaining at all times inside the ball $B_{r^u}$ where $r^u:= \sup_{t \geq 0} \| U^u(t,\cdot) \|_\infty$. Therefore, if $\mathcal{K}$ is now any bounded set contained in $\mathcal{D}_{\mathbf{0}}$ at a
positive distance from $\mathcal{W}$ then we have that $\tau_{\mathcal{K},\frac{c}{2}} := \sup_{u \in \mathcal{K}} \tau^u(B_{\frac{c}{2}})$ and $r_{\mathcal{K}}:=\sup_{u \in \mathcal{K}} r^u$ are both finite. Indeed, \mbox{the finiteness} of $\tau_{\mathcal{K},\frac{c}{2}}$ follows at once from Proposition \ref{A.3} whereas $r_{\mathcal{K}}$ is finite since by Proposition \ref{G.1} one may find $t_0 > 0$ sufficiently small such that $\sup_{u \in \mathcal{K}} \left[ \sup_{t \in [0,t_0]} \| U^u(t,\cdot) \|_\infty \right]$ is finite. That $\sup_{u \in \mathcal{K}} \left[ \sup_{t \geq t_0} \| U^u(t,\cdot) \|_\infty \right]$ is finite then follows as in the proof of Proposition \ref{A.3} due to the fact that the mapping $u \mapsto r^u$ is both upper semicontinuous and finite on $\mathcal{D}_{\mathbf{0}}$. Using the strong Markov property we can then obtain the bounds
$$
P_u \left( \tau_\varepsilon < e^{\frac{\Delta - \delta}{\varepsilon^2}} \right) \leq P_u \left( \tau_\varepsilon(B_c) > \tau_{\mathcal{K},\frac{c}{2}} \right) + P_u \left( \tau_\varepsilon \leq \tau_{\mathcal{K},\frac{c}{2}}\right) + \sup_{v \in B_c} P_v \left( \tau_\varepsilon < e^{\frac{\Delta - \delta}{\varepsilon^2}} \right)
$$
and
$$
P_u \left( \tau_\varepsilon > e^{\frac{\Delta + \delta}{\varepsilon^2}} \right) \leq P_u \left( \tau_\varepsilon(B_c) > \tau_{\mathcal{K},\frac{c}{2}} \right) + \sup_{v \in B_c} P_v \left( \tau_\varepsilon > e^{\frac{\Delta + \frac{\delta}{2}}{\varepsilon^2}} \right)
$$ for any $u \in \mathcal{K}$ and $\varepsilon > 0$ sufficiently small. But let us observe that for $u \in \mathcal{K}$ we have
\begin{equation}\label{convunibola0}
P_u \left( \tau_\varepsilon(B_c) > \tau_{\mathcal{K},\frac{c}{2}} \right) \leq P_u\left( d_{\tau_{\mathcal{K},\frac{c}{2}}}( U^{(r_{\mathcal{K}}+1),\varepsilon}, U^{(r_{\mathcal{K}}+1)}) > \min\left\{\frac{c}{2}, \frac{1}{2}\right\} \right)
\end{equation} and
$$
P_u \left( \tau_\varepsilon \leq \tau_{\mathcal{K},\frac{c}{2}}\right) \leq P_u \left( d_{\tau_{\mathcal{K},\frac{c}{2}}}( U^{(r_{\mathcal{K}}+1),\varepsilon}, U^{(r_{\mathcal{K}}+1)}) > 1\right).
$$ Now the uniform bounds given by \eqref{grandes1} allow us to conclude the result.
\end{proof}

The next proposition shows that, for initial data in a small neighborhood of the origin, both the explosion time and the escape time from $G$ are asymptotically of the same order of magnitude. We will use this fact to conclude that the explosion time $\tau_\varepsilon$ shares the same asymptotic distribution with the escape time from $G$ and thus obtain the asymptotic loss of memory for $\tau_\varepsilon$.

\begin{prop}\label{nescapelema3} If $\tau^* > 0$ is taken as in Remark \ref{obsequivG} then
\begin{equation}
\lim_{\varepsilon \rightarrow 0} \left[ \sup_{u \in B_c} P_u ( \tau_\varepsilon > \tau_\varepsilon (\partial G) + \tau^* )\right] = 0.
\end{equation}
\end{prop}

\begin{proof} For any $u \in B_c$ we have by the strong Markov property that
$$
P_u ( \tau_\varepsilon > \tau_\varepsilon (\partial G) + \tau^* ) \leq \sup_{v \in B_c} P_v \left( U^\varepsilon (\tau_\varepsilon(\partial G), \cdot) \notin \partial^{\pm z} \right) + \sup_{v \in \p^{\pm z}} P_v ( \tau_\varepsilon > \tau^*).
$$ We may now conclude the result by Theorem \ref{teoescape0} and Remark \ref{obsequivG}.
\end{proof}

\medskip
\begin{cor}\label{nescapecor0} Let $\rho > 0$ be as in Lemma \ref{escapelema1} and for each $\varepsilon > 0$ define $\beta_\varepsilon$ as in \eqref{defibeta}. Then
\begin{enumerate}
\item [i.] $\lim_{\varepsilon \rightarrow 0} \frac
    {\beta_\varepsilon}{\gamma_\varepsilon} = 1$.
\item [ii.] $\lim_{\varepsilon \rightarrow 0} \left[ \sup_{u \in B_\rho} |P_u (\tau_\varepsilon (\partial G) > t\beta_\varepsilon ) - e^{-t}| \right]= 0.$
\end{enumerate}
\end{cor}

\medskip\begin{proof} Let us first notice that by the upper bound for the explosion time we have that $\beta_\varepsilon$ is finite for every $\varepsilon > 0$ sufficiently small. Now, the continuity of trajectories implies that
$$
P_{\mathbf{0}}(\tau_\varepsilon (\partial G) > \beta_\varepsilon ) \leq P_\mathbf{0}( \tau_\varepsilon > \beta_\varepsilon ) \leq e^{-1},
$$ from where we conclude $\gamma_\varepsilon \leq \beta_\varepsilon$ and, thus, that $\liminf_{\varepsilon \rightarrow 0} \frac {\beta_\varepsilon}{\gamma_\varepsilon} \geq 1$. Let us now suppose that
$\limsup_{\varepsilon \rightarrow 0} \frac {\beta_\varepsilon}{\gamma_\varepsilon} > 1$. Then there would exist $\lambda > 0$ and a sequence $(\varepsilon_j)_{j \in \N} \subseteq \R_{> 0}$ with $\lim_{j \rightarrow +\infty} \varepsilon_j = 0$ such that for all $j \in \N$ sufficiently large
$$
e^{-1} < P_\mathbf{0}( \tau_{\varepsilon_j} > \beta_{\varepsilon_j} - 1 ) \leq P_{\mathbf{0}} \left(\tau_{\varepsilon_j}(\partial G) > \left(1+ \lambda_0\right)\gamma_{\varepsilon_j}\right) + P_{\mathbf{0}} \left( \tau_{\varepsilon_j} > \tau_{\varepsilon_j} (\partial G) + \tau^* \right).
$$ Taking the limit on the right hand side of this inequality with $j \rightarrow +\infty$, by \mbox{Proposition \ref{nescapelema3}} we arrive at the contradiction $e^{-1} \leq e^{-(1+\lambda)}$. We thus conclude that $\limsup_{\varepsilon \rightarrow 0} \frac {\beta_\varepsilon}{\gamma_\varepsilon} \leq 1$ which implies (i).

Notice that (i) itself implies by Theorem \ref{escapeteo1} that for every $t > 0$
$$
\lim_{\varepsilon \rightarrow 0} P_{\mathbf{0}} ( \tau_\varepsilon(\partial G) > t \beta_\varepsilon ) = e^{-t}.
$$ We can now establish (ii) by following the proof of Theorem \ref{escapeteo1} since one can show as in Lemma \ref{escapelema1} that for every $t_0 > 0$
\begin{equation}
\lim_{\varepsilon \rightarrow 0}\left[ \sup_{u,v \in B_{\rho}} \left[ \sup_{t > t_0} |P_u(\tau_{\varepsilon}(\partial G)  > t\beta_{\varepsilon}) - P_{v}(\tau_{\varepsilon}(\partial G) > t\beta_{\varepsilon})|\right]\right] = 0.
\end{equation}
\end{proof}

We are now ready to show the asymptotic exponential distribution of the \mbox{explosion time} for arbitrary initial data in $\mathcal{D}_{\mathbf{0}}$. This is contained in the following theorem.

\begin{teo}\label{nteoasint} For any bounded set $\mathcal{K} \subseteq \mathcal{D}_{\mathbf{0}}$ at a positive distance from $\mathcal{W}$ we have
\begin{equation}\label{convunicompact}
\lim_{\varepsilon \rightarrow 0} \left[ \sup_{u \in \mathcal{K}} |P_u (\tau_{\varepsilon} > t\beta_{\varepsilon}) - e^{-t}| \right] = 0.
\end{equation} for every $t > 0$.
\end{teo}

\begin{proof} Let us consider the radius $\rho > 0$ given by Lemma \ref{escapelema1} and suppose $\rho \leq c$ where $c$ is taken as in Conditions \ref{assumpg}. Then from the inequalities
$$
P_u( \tau_\varepsilon(\partial G) > t\beta_\varepsilon ) \leq P_u (\tau_{\varepsilon} > t\beta_{\varepsilon}) \leq P_u( \tau_\varepsilon(\partial G)> t\beta_\varepsilon - \tau^*) + P_u ( \tau_\varepsilon > \tau_\varepsilon(\p G) + \tau^*)
$$ for $u \in B_\rho$ one can easily verify, using (ii) in Corollary \ref{nescapecor0} and Proposition \ref{nescapelema3}, that
\begin{equation}\label{convunibola}
\lim_{\varepsilon \rightarrow 0} \left[ \sup_{u \in B_\rho} | P_u (\tau_\varepsilon > t \beta_\varepsilon ) - e^{-t} | \right] = 0.
\end{equation}

Now, given a bounded set $\mathcal{K} \subseteq \mathcal{D}_{\mathbf{0}}$ at a positive distance from $\mathcal{W}$, take $\tau_{\mathcal{K},\frac{c}{2}} > 0$ as in the proof of Theorem \ref{taumagnitude}. The strong Markov property implies for \mbox{each $u \in \mathcal{K}$ the inequalities}
$$
\inf_{v \in \mathcal{K}} P_v ( \tau_\varepsilon (B_c) \leq \tau_{\mathcal{K},\frac{c}{2}}) \inf_{v \in B_c} P( \tau_\varepsilon > t\beta_\varepsilon) \leq P_u (\tau_\varepsilon > t\beta_\varepsilon)
$$ and
$$
P_u (\tau_\varepsilon > t\beta_\varepsilon) \leq \sup_{v \in \mathcal{K}}  P_v ( \tau_\varepsilon (B_c) > \tau_{\mathcal{K},\frac{c}{2}}) + \sup_{v \in B_c} P( \tau_\varepsilon > t\beta_\varepsilon - \tau_{\mathcal{K},\frac{c}{2}}).
$$ From these we may conclude \eqref{convunicompact} by recalling \eqref{convunibola} and \eqref{convunibola0}.
\end{proof}

\section{Stability of time averages}

Our purpose in this final section is to show the stability of time averages along typical paths of the stochastic system up until (almost) the explosion time. The precise statement we wish to show is that of the following theorem.

\begin{teo} There exists a sequence $(R_\varepsilon)_{\varepsilon > 0}$ with $\lim_{\varepsilon \rightarrow 0} R_\varepsilon = +\infty$ and $\lim_{\varepsilon \rightarrow 0} \frac{R_\varepsilon}{\beta_\varepsilon} = 0$ such that given $\delta > 0$ for any bounded set $\mathcal{K} \subseteq \mathcal{D}_{\mathbf{0}}$ at a positive \mbox{distance from $\mathcal{W}$} we have
\begin{equation}\label{average}
\lim_{\varepsilon \rightarrow 0} \left[ \sup_{u \in B} P_u \left( \sup_{0 \leq t \leq \tau_\varepsilon - 3R_\varepsilon}\left| \frac{1}{R_\varepsilon}\int_t^{t+R_\varepsilon} f(U^{\varepsilon}(s,\cdot))ds - f(\mathbf{0})\right| > \delta \right) \right] = 0
\end{equation} for any bounded continuous function $f: C_D([0,1]) \rightarrow \R$.
\end{teo}

This result was originally established in \cite{GOV} for the double-well potential model in the finite-dimensional setting. Later the analogous result in the infinite-dimensional setting was obtained in \cite{B1}. We present here an adaptation of those proofs to our model.

Let us observe that it suffices to show the result for the particular case $\mathcal{K}=B_c(\mathbf{0})$. Indeed, if we take $\tau_{\mathcal{K},\frac{c}{2}} > 0$ as in the proof of Theorem \ref{taumagnitude} then, since $R_\varepsilon > k$ holds for $\varepsilon > 0$ sufficiently small, the strong Markov property then implies that for every $u \in \mathcal{K}$ and bounded continuous function $f: C_D([0,1]) \rightarrow \R$ we have
$$
P_u \left( \sup_{0 \leq t \leq \tau_\varepsilon - 3R_\varepsilon}\left| \vartheta^\varepsilon_t(f) \right| > \delta \right) \leq \sup_{v \in \mathcal{K}} P( \tau_\varepsilon(B_c) > \tau_{\mathcal{K},\frac{c}{2}}) + \sup_{v \in B_c} P_v \left( \sup_{0 \leq t \leq \tau_\varepsilon - 3R_\varepsilon}\left| \vartheta^\varepsilon_t(f) \right| > \frac{\delta}{2} \right)
$$ where for $0 \leq t < \tau_\varepsilon - R_\varepsilon$ we write
$$
\vartheta^{u,\varepsilon}_t(f) = \frac{1}{R_\varepsilon}\int_t^{t+R_\varepsilon} f(U^{u,\varepsilon}(s,\cdot))ds - f(\mathbf{0}).
$$ Furthermore, by Proposition \ref{nescapelema3} we see that in fact it will suffice to show that
\begin{equation}\label{averagelocal1}
\lim_{\varepsilon \rightarrow 0} \left[ \sup_{u \in B_c} P_u \left( \sup_{0 \leq t \leq \tau_\varepsilon(\p G) - 2R_\varepsilon}\left| \vartheta^\varepsilon_t(f) \right| > \delta \right)\right] = 0.
\end{equation}This provides the advantage of only having to consider paths inside a bounded domain. Finally, let us notice that in order to obtain \eqref{averagelocal1} for any bounded continuous function $f :C_D([0,1]) \rightarrow \R$, it will be enough to show \eqref{averagelocal1} only for the class of functions $\mathbbm{1}_\theta$ with $\theta > 0$ sufficiently small, where $\mathbbm{1}_\theta$ denotes the indicator function of the ball $B_\theta$. Indeed, this follows from \mbox{the fact} that for  and $\theta > 0$ one has
$$
|\vartheta^{u,\varepsilon}_t(f)| \leq \sup_{u \in B_\theta}|f(u)-f(\mathbf{0})|- 2 \|f\|_\infty \vartheta^{u,\varepsilon}_t(\mathbbm{1}_\theta)
$$ for every $0 \leq t < \tau_\varepsilon - R_\varepsilon$. Thus, let us fix $\delta, \theta > 0$ and for each $u \in B_c$ and $l \in \N_0$ let us define the set
$$
A^{u,\varepsilon}_l:= \{ |\vartheta^{u,\varepsilon}_{lR_\varepsilon}(\mathbbm{1}_\theta)| \leq \delta \}
$$ with the convention that $|\vartheta^{u,\varepsilon}_{lR_\varepsilon}(\mathbbm{1}_\theta)|=+\infty$ whenever $l \geq l^u_\varepsilon$, where
$$
l^u_\varepsilon := \inf \{ l \in \N_0 : \tau^u_\varepsilon(\p G) \leq (l+1)R_\varepsilon \}.
$$ Let us observe that the validity of \eqref{averagelocal1} for $f=\mathbbm{1}_\theta$ will follow if we manage to show that, for $(R_\varepsilon)_{\varepsilon}$ as in the statement of the theorem, one has
\begin{equation}\label{averagelocal2}
\lim_{\varepsilon \rightarrow 0} \left[\inf_{u \in B_c} P_u\left( \left[\bigcap_{0 \leq l < l_\varepsilon} A^\varepsilon_l \right] \cap \{ l_\varepsilon > 1 \} \right)\right]=1.
\end{equation}
Now, for each $u \in B_c$ and $K_\varepsilon \geq 2$ we have
\begin{align*}
P_u\left( \left[\bigcap_{0 \leq l < l_\varepsilon} A^\varepsilon_l \right] \cap \{ l_\varepsilon > 1 \} \right)& = \sum_{L=2}^\infty P_u\left( \left[\bigcap_{0 \leq l < l_\varepsilon} A^\varepsilon_l \right] \cap \{ l_\varepsilon = L \} \right)\\
\\
& = P_u ( l_\varepsilon > 1 ) - \sum_{L=2}^\infty P_u\left( \left[\bigcup_{0 \leq l < l_\varepsilon} (A^\varepsilon_l)^c \right] \cap \{ l_\varepsilon = L \} \right)\\
\\
& \geq P_u ( K_\varepsilon \geq l_\varepsilon > 1 ) - \sum_{L=2}^{[K_\varepsilon]} P_u\left( \left[\bigcup_{0 \leq l < l_\varepsilon} (A^\varepsilon_l)^c \right] \cap \{ l_\varepsilon = L \} \right)
\end{align*} so that we may obtain \eqref{averagelocal2} provided that we can choose the sequences $(R_\varepsilon)_{\varepsilon > 0}$ and $(K_\varepsilon)_{\varepsilon}$ in such a way that:
\begin{enumerate}
\item [i.] $\lim_{\varepsilon \rightarrow 0} \left[\inf_{u \in B_c} P_u ( K_\varepsilon \geq l_\varepsilon > 1 )\right] = 1$
\item [ii.] $\lim_{\varepsilon \rightarrow 0} \left[\sup_{u \in B_c} \sum_{L=2}^{[K_\varepsilon]} P_u \left( \left[\bigcup_{0 \leq l < l_\varepsilon} (A^\varepsilon_l)^c \right] \cap \{ l_\varepsilon = L \} \right) \right] =0$.
\end{enumerate} Since by definition of $l^u_\varepsilon$ for every $u \in B_c$ we have
$$
P_u ( K_\varepsilon \geq l_\varepsilon > 1 ) = P_u ( 2R_\varepsilon < \tau_\varepsilon(\p G) \leq (K_\varepsilon + 1)R_\varepsilon),
$$ by Theorem \ref{ecotsuplema1} we see that (i) follows if for each $\varepsilon > 0$ we choose $R_\varepsilon = e^{\frac{\alpha}{\varepsilon^2}}$ with $0 < \alpha < \Delta$ and $K_\varepsilon = e^{\frac{\gamma}{\varepsilon^2}}$ with $\gamma > \Delta - \alpha$. Therefore, it only remains to check that (ii) holds for this choice of the sequences $(R_\varepsilon)_{\varepsilon > 0}$ and $(K_\varepsilon)_{\varepsilon}$. But notice that for each $u \in B_c$ we have
\begin{align*}
\sum_{L=2}^{[K_\varepsilon]} P_u \left( \left[\bigcup_{0 \leq l < l_\varepsilon} (A^\varepsilon_l)^c \right] \cap \{ l_\varepsilon = L \} \right)& \leq \sum_{L=2}^{[K_\varepsilon]} P_u \left( \left[\bigcup_{0 \leq l < l_\varepsilon} (A^\varepsilon_l)^c \right] \cap \{ l_\varepsilon > l \} \right)\\
\\
& \leq \sum_{L=2}^{[K_\varepsilon]} \sum_{l=0}^L P_u \left( (A^\varepsilon_l)^c \cap \{ l_\varepsilon > l \} \right)\\
\\
& \leq K_\varepsilon^2 \sup_{0 \leq l < K_\varepsilon} P_u\left( (A^\varepsilon_l)^c \cap \{ l_\varepsilon > l \} \right)\\
\\
& \leq K_\varepsilon^2 \sup_{u \in G} P_u \left( (A^\varepsilon_0)^c \cap \{ l_\varepsilon > 0 \} \right)
\end{align*} where in the last inequality we have used the Markov property. The fact that (ii) holds now follows from the following proposition. This concludes the proof.

\begin{prop} If $0 < \theta < c$ where $c$ is given by Conditions \ref{assumpg} then there exists $a > 0$ such that given $\delta > 0$ for $\varepsilon > 0$ sufficiently small we have
\begin{equation}\label{averagelocal3}
\sup_{u \in G} P_u \left( (A^\varepsilon_0)^c \cap \{ l_\varepsilon > 0 \} \right) \leq e^{-\frac{\delta}{16} e^{\frac{a}{\varepsilon^2}}}.
\end{equation}
\end{prop}

\begin{proof} Notice that, since $\theta < c$, we have that $B_\theta \subseteq \mathcal{D}_{\mathbf{0}}$ and thus that $V(\mathbf{0}, \p B_\theta) > 0$. Thus, by the methods in Section \ref{seclowerbound} one can show that there exist sufficiently small $0 < r < \theta$ and $0 < b < \alpha$ such that
\begin{equation}\label{aleq4}
\lim_{\varepsilon \rightarrow 0} \left[\sup_{u \in B_r} P_u \left( \tau_\varepsilon(\p B_\theta) \leq e^{\frac{b}{\varepsilon^2}} \right)\right] = 0.
\end{equation} Now, for each $\varepsilon > 0$ let us set $t_\varepsilon:=e^{\frac{b}{\varepsilon^2}}$, $N_\varepsilon:= \left[ \frac{R_\varepsilon}{t_\varepsilon}\right]$ and for $1 \leq i \leq N_\varepsilon$ and $u \in G$ define the random variable
$$
Y^{u,\varepsilon}_i = \left\{ \begin{array}{ll}0 &\text{ if $U^{u,\varepsilon}$ visits $B_d$ in $[(i-1)t_\varepsilon, (i-1)t_\varepsilon + \sqrt{t_\varepsilon})$ and then}\\ & \text{ spends the rest of the time interval $[(i-1)t_\varepsilon,it_\varepsilon)$ in $B_\theta$}\\ \\ 1 & \text{ otherwise.}\end{array}\right.
$$ Since $\lim_{\varepsilon \rightarrow 0} \frac{R_\varepsilon}{t_\varepsilon} = +\infty$ and $\lim_{\varepsilon \rightarrow 0} t_\varepsilon = +\infty$, for $\varepsilon > 0$ sufficiently small and any $u \in G$ we have that
\begin{align*}
P_u \left( (A^\varepsilon_0)^c \cap \{ l_\varepsilon > 0 \} \right) &= P_u \left( \left|\frac{1}{R_\varepsilon}\int_0^{R_\varepsilon} \mathbbm{1}_\theta( U^\varepsilon(t,\cdot))dt - \right|> \delta , \tau_\varepsilon(\p G) > R_\varepsilon \right) \\
\\
& \leq P_u \left( \left|\frac{1}{N_\varepsilon t_\varepsilon}\int_0^{N_\varepsilon t_\varepsilon} \mathbbm{1}_\theta( U^\varepsilon(t,\cdot))dt - \right|> \frac{\delta}{2} , \tau_\varepsilon(\p G) > N_\varepsilon t_\varepsilon \right) \\
\\
& \leq P_u \left( \frac{1}{N_\varepsilon} \sum_{i=1}^{N_\varepsilon} Y^\varepsilon_i > \frac{\delta}{2} - \frac{1}{\sqrt{t_\varepsilon}} , \tau_\varepsilon(\p G) > N_\varepsilon t_\varepsilon \right)\\
\\
& \leq P_u \left( \sum_{i=1}^{N_\varepsilon} Y^\varepsilon_i > \frac{\delta}{4}N_\varepsilon , \tau_\varepsilon(\p G) > N_\varepsilon t_\varepsilon \right)\\
\\
& \leq e^{-\frac{\delta}{4}N_\varepsilon} \E_u\left( \mathbbm{1}_{\{\tau_\varepsilon(\p G) > N_\varepsilon t_\varepsilon\}}e^{ \sum_{i=1}^{N_\varepsilon} Y^\varepsilon_i} \right)\\
\\
& \leq e^{-\frac{\delta}{4}N_\varepsilon} \left[\sup_{v \in G} \E_v \left( \mathbbm{1}_{\{\tau_\varepsilon(\p G) > t_\varepsilon\}}e^{Y^\varepsilon_1}\right)\right]^{N_\varepsilon}
\end{align*} where the last inequality is a consequence of the Markov property. Now
\begin{align*}
\sup_{v \in G} \E_v \left( \mathbbm{1}_{\{\tau_\varepsilon(\p G) > t_\varepsilon\}}e^{Y^\varepsilon_1}\right)& \leq \sup_{v \in G} \left[ P_v( Y^\varepsilon_1 = 0 , \tau_\varepsilon(\p G) > t_\varepsilon) + e P_v ( Y^\varepsilon_1 = 1 , \tau_\varepsilon(\p G) > t_\varepsilon) \right] \\
\\
& = \sup_{v \in G} \left[ P_v(\tau_\varepsilon(\p G) > t_\varepsilon) + (e-1) P_v ( Y^\varepsilon_1 = 1 , \tau_\varepsilon(\p G) > t_\varepsilon) \right] \\
\\
& \leq e^{(e-1) \sup_{v \in G}P_v ( Y^\varepsilon_1 = 1 , \tau_\varepsilon(\p G) > t_\varepsilon)}
\end{align*} where in the last inequality we have used the fact that $1+x \leq e^x$ is valid for all $x \geq 0$.
But let us observe that
$$
\sup_{v \in G}P_v ( Y^\varepsilon_1 = 1 , \tau_\varepsilon(\p G) > t_\varepsilon) \leq \sup_{v \in G} P( \tau_\varepsilon(B_r) > \sqrt{t_\varepsilon} , \tau_\varepsilon(\p G) > t_\varepsilon ) + \sup_{v \in B_r} P_v ( \tau_\varepsilon( \p B_\theta) \leq t_\varepsilon )
$$ where each term in the right hand side tends to zero as $\varepsilon \rightarrow 0$ by Lemma \ref{escapelema2} and \eqref{aleq4} respectively. Thus, for $\varepsilon > 0$ sufficiently small we obtain that
$$
\sup_{u \in G} P_u \left( (A^\varepsilon_0)^c \cap \{ l_\varepsilon > 0 \} \right) \leq e^{-\frac{\delta}{8} N_\varepsilon} \leq e^{-\frac{\delta}{16} \frac{R_\varepsilon}{t_\varepsilon}} = e^{-\frac{\delta}{16} e^{\frac{a}{\varepsilon^2}}}
$$ where $a = \alpha - b > 0$.
\end{proof}

\newpage

\section{Resumen del Capítulo 5}

En este capítulo damos la demostración de los Teoremas II, III y IV en la sección de resultados del Capítulo 1. Los Teoremas II y III se deducen de los resultados del Capítulo 4 para el tiempo de escape del dominio $G$ dado que $\tau^u_\varepsilon$ y $\tau^u_\varepsilon(\p G)$ son asintóticamente equivalentes. Más precisamente, existe $\tau^* > 0$ tal que
$$
\lim_{\varepsilon \rightarrow 0} \left[\sup_{u \in G} | P_u ( \tau_\varepsilon (\p G) \leq \tau_\varepsilon \leq \tau_\varepsilon (\p G) + \tau^* ) - 1 |\right] = 0.
$$ La desigualdad $\tau_\varepsilon (\p G) \leq \tau_\varepsilon$ vale siempre como consecuencia de la continuidad de las trayectorias de $U^{u,\varepsilon}$. Por otro lado, la segunda desigualdad $\tau_\varepsilon \leq \tau_\varepsilon (\p G) + \tau^*$ se deduce de los resultados del Capítulo 2, puesto que $U^{u,\varepsilon}$ se escapa de $G$ típicamente por $\p G$ y, además, $\p G$ es un subconjunto cerrado de $\mathcal{D}_e^*$ a una distancia positiva de la frontera.

Por último, como $\tau^u_\varepsilon$ y $\tau^u_\varepsilon(\p G)$ son asintóticamente equivalentes, podemos suponer que en los promedios ergódicos en el enunciado del Teorema IV figura $\tau^u_\varepsilon(\p G)$ en lugar de $\tau^u_\varepsilon$. La demostración del Teorema IV en este caso sigue los pasos de \cite{GOV} y \cite{B1}.

\chapter{A finite-dimensional problem}

In this chapter we study the asymptotic properties of the explosion time for small random perturbations of a particular ordinary differential equation with blow-up. This can be seen as a finite-dimensional version of our original problem. However, the equation we consider in this chapter is not the finite-dimensional analogue of the original equation \eqref{MainPDE}, and thus one cannot perform the exact same analysis of the previous chapters.
We decided to include this variant here for a number of reasons. First, because it serves as an example of the lack of an unified approach to treat perturbations of differential equations with blow-up: in general, different systems require different techniques to study them. We also do it to show that the finite-dimensional structure can simplify matters to some extent, allowing us to achieve more general results than for the infinite-dimensional alternative.
Finally, we do it to show that the ideas developed in this first part are not restricted to equations with Dirichlet boundary conditions. The analysis of this chapter can be found in more detail in \cite{GS}.

\section{Preliminaries}

\subsection{The deterministic system}
We consider small random perturbations of the following ODE
\begin{equation}
\label{1.1}
\left\{\begin{array}{lcll}
U'_1 &= &\frac{2}{h^2} ( -U_1 + U_2 ),\\
U'_i &= &\frac{1}{h^2} ( U_{i+1} - 2U_i + U_{i-1} ) &\,\,\, 2 \leq i \leq d-1,\\
U'_d &= &\frac{2}{h^2} ( -U_d + U_{d-1} +hg(U_d) )\\
U(0) &= &u.
\end{array}\right.
\end{equation}
Here $g\colon \R \to \R$ is a reaction term given by $g(x) = (x^+)^p -x$ for $p>1$ \mbox{and $h>0$ is fixed.} These kind of systems arise as spatial discretizations of diffusion equations with nonlinear boundary conditions of Neumann type. In fact, it is well known that as $h\to 0$ solutions to this system converge to solutions of the PDE
$$
\left \{\begin{array}{rcll}
\p_t U(t,x) & = & \p^2_{xx}U(t,x) & 0<x<1, 0\le t<T,\\
\p_x U(t,0) & = & 0 & 0\le t <T,\\
\p_x U(t,1) &= & g(U(t,1)) & 0\le t <T,\\
U(0,x) & = & U_0(x) & 0\le x \le 1.
\end{array} \right.
$$
\newpage

For details on this convergence see \cite{DER}. Equation \eqref{1.1} can be written in matrix form as
\begin{equation}
\label{A1}
dU = \Big(-AU + \frac2h g(U_d)e_d\Big)dt
\end{equation} for some positive definite $A \in \R^{d \times d}$ and where $e_d$ denotes the $d$-th canonical vector on $\R^d$.
The field $b(u):= -AU + \frac2hg(U_d)e_d$ is of gradient type, i.e. $b=-\nabla S$, with \mbox{potential $S$} given by
$$
S(u) = \frac{1}{2} \langle Au , u \rangle - \frac{2}{h}\Big(\frac{\;\;\;|u_d^+|^{p+1}}{p+1} - \frac{\,\,{u_d}^2}{2}\Big).
$$ This potential satisfies all properties shown for its \mbox{infinite-dimensional} analogue in \eqref{formalPDE}.
It has exactly two critical points, $\1:=(1,\dots,1)$ and the origin, both of them hyperbolic. The origin $0$ is the unique local minimum of the potential $S$ while $\1$ is a saddle point. Furthermore, we have a decomposition of $\R^d$ similar to \eqref{decomp12} (see \cite{AFBR,GS} for details). Indeed, we have
\begin{equation*}
\R^d = \mathcal{D}_0 \cup \mathcal{W}^s_1 \cup \mathcal{D}_e
\end{equation*} where $\mathcal{D}_0$ denotes the stable manifold of the origin, $\mathcal{W}^s_{1}$ is the stable manifold of $\1$ and $\mathcal{D}_e$ is the
domain of explosion. Once again the sets $\mathcal{D}_0$ and $\mathcal{D}_e$ are open in $\R^d$ and the origin is an asymptotically stable equilibrium of the system.
$\mathcal{W}^s_{1}$ is a manifold of \mbox{codimension one.} The saddle point $\1$ also admits an unstable manifold, $\mathcal W^u_\1$. This unstable manifold is contained in $\R^d_+$ and has dimension one. Furthermore, it has nonempty intersection with both $\mathcal{D}_0$ and $\mathcal{D}_e$ and joins $\1$ with the origin. An illustration of this decomposition is given in Figure \ref{fig:inclination}
for the $2$-dimensional case. Finally, we have the finite-dimensional analogue of Theorem \ref{descomp2}, originally proved in \cite{AFBR}.

\psfrag{U1}{\vspace{-155pt}$U\equiv 0$}
\psfrag{U2}{\vspace{-45pt}$U \equiv \1$}
\psfrag{De}{\hspace{60pt}\vspace{30pt}$\mathcal{D}_e$}
\psfrag{D0}{$\mathcal{D}_0$}
\psfrag{Wu}{$\mathcal W_1^u$}
\psfrag{Ws}{$\mathcal W_1^s$}
\begin{figure}
	\centering
	\includegraphics[width=8cm]{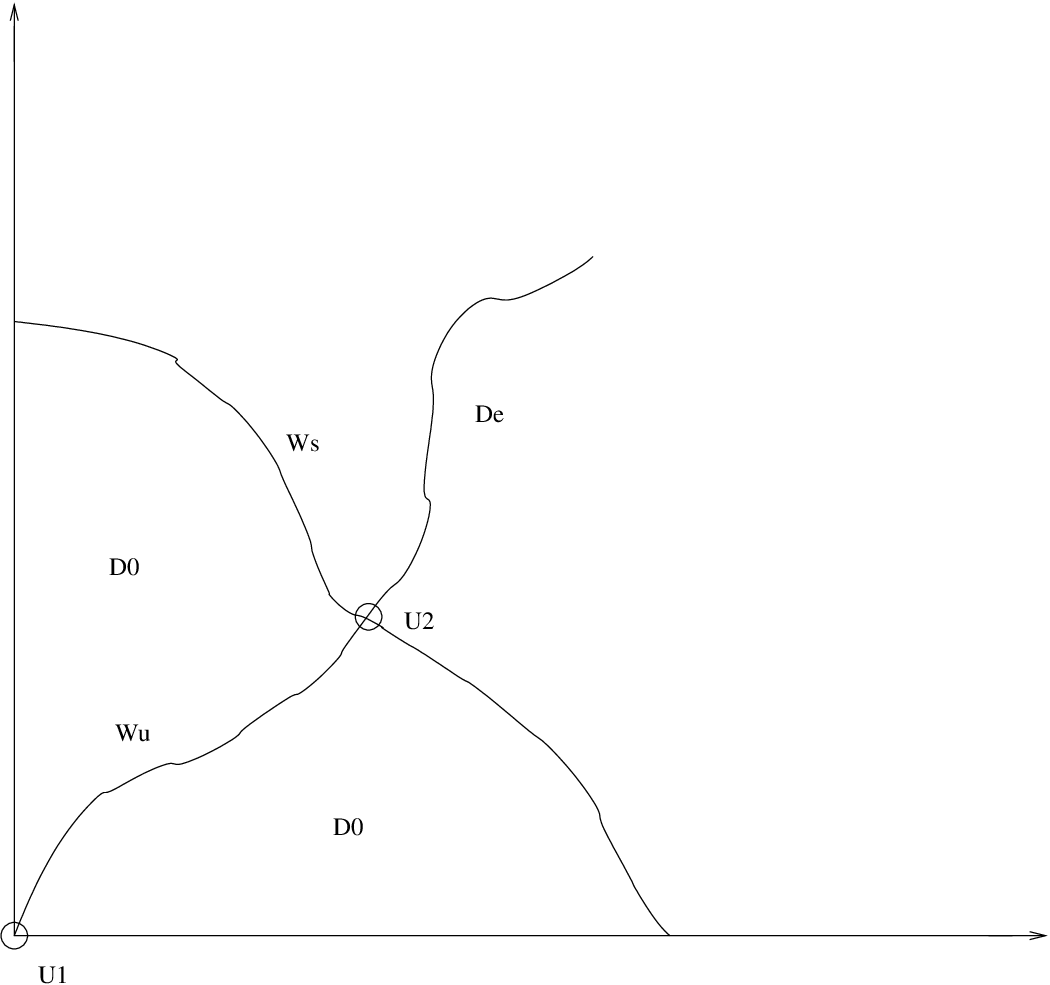}
	\caption{The phase diagram of equation \eqref{1.1}.}
	\label{fig:inclination}
\end{figure}

\subsection{The stochastic system}

We study random perturbations of \eqref{1.1} given by additive white-noise. More precisely, we consider stochastic differential equations of the form

\begin{equation}
\label{Aestoc}
dU^{\ve} = \Big(-AU^{\ve} + \frac2h g(U_d^{\ve})e_d\Big)dt + \ve dW
\end{equation}
for $\ve>0$ small and where $W=(W_1, \dots,W_d)$ a $d-$dimensional standard Brownian motion.
Given a probability space $(\Omega,\F, P)$ and a standard $d$-dimensional Brownian motion $W$, we say that a stochastic process $U^\varepsilon = (U^\varepsilon(t))_{t \geq 0}$ is a solution up to an explosion time of \eqref{Aestoc} on $(\Omega, \mathcal{F}, P)$ and with respect to $W$ if it satisfies the following:

\begin{itemize}
\item $U^\varepsilon$ has continuous paths taking values in $\R^d \cup \{ \infty\}$
\item $U^\varepsilon$ is adapted to the augmented filtration generated by $W$.
\item For every $n \in \N$ we have $P$-almost surely
$$
\int_{0}^{t\wedge{\tau^{(n)}}} |b(U^\varepsilon(s))|\,ds < +\infty \hspace{0,5cm} \,\forall\,\,\, 0\leq t < + \infty
$$
and
$$
U^\varepsilon({t\wedge \tau^{(n)}}) = U^\varepsilon(0) + \int_{0}^{t} b(U^\varepsilon(s))\mathbbm{1}_{\{s \leq \tau^{(n)}\}}\,ds
+ \varepsilon W({t \wedge \tau{(n)}}); \,\,\,\forall\,\,\, 0\leq t < + \infty.
$$ where $\tau^{(n),\varepsilon}:=\inf \{ t \geq 0 : \|U^\varepsilon(t)\|_\infty \geq n\}$.

\item $U^\varepsilon$ has the strong Markov property.
\end{itemize}We call $\tau^\varepsilon:= \lim_{n \rightarrow +\infty} \tau^{(n),\varepsilon}$ the \textit{explosion time} of $U^{\varepsilon}$.

As in the previous chapters, we shall write $U^{u,\ve}$ to denote the unique solution to
\eqref{Aestoc} with initial datum $u \in \R^d$ and also write $U^u$ to denote the corresponding solution to \eqref{1.1}.
Furthermore, for each $n \in \N$ we consider the truncation $S^{(n)}$ of the potential $S$ given by
$$
S^{(n)}(u) = \frac{1}{2}\langle Au, u \rangle - \frac{2}{h} G_n(u_d)
$$ where $G_n : \R \longrightarrow \R$ is a function of class $C^2$ satisfying
that \begin{equation*}\label{gtruncada2}
G_n(u) = \left\{\begin{array}{ll}
\frac{|u^+|^{p+1}}{p+1} - \frac{u^2}{2} &\,\,\text{if}\,\,u \leq n\\
0 &\,\,\text{if}\,\,u \geq 2n.
\end{array}\right.
\end{equation*} The unique solution $U^{(n),u}$ to the equation $\dot{U}^{(n),u} = -\nabla S^{(n)}(U^{(n),u})$ with \mbox{initial datum $u$} is globally defined and coincides with $U^u$ until the escape from the \textit{unbounded} set
$$
\Pi^n := \{ u \in \R^d : u_d < n \}.
$$ Similarly, for $\varepsilon > 0$ the unique solution $U^{(n),u,\varepsilon}$ to the equation
$$
d{U}^{(n),u,\varepsilon} = -\nabla S^{(n)}(U^{(n),u,\varepsilon}) dt + \varepsilon dW
$$ with initial datum $u$ is globally defined and coincides with $U^{u,\varepsilon}$ until the escape from $\Pi^n$. Moreover, since the field $-\nabla S^{(n)}$ is globally Lipschitz, the family of solutions $\left(U^{(n),u,\varepsilon}\right)_{\varepsilon > 0}$ satisfies the analogous large deviations estimates of Section \ref{secLDP} with rate function
$$
I^u_T(\varphi)=\left\{\begin{array}{ll}\displaystyle{\frac12\int_0^T|\dot \varphi(s) + \nabla S^{(n)}(\varphi(s))|^2 ds} & \mbox{if $\varphi$ is absolutely continuous and $\varphi(0)=u$}\\
\\
+\infty & \mbox{otherwise}
\end{array} \right.
$$ Finally, for each $\varepsilon > 0$ and $u \in \R^d$ the process $U^{(n),u,\varepsilon}$ is positive recurrent.

\subsection{Main results}

We now state the main results obtained in this finite-dimensional setting. Some of these results are more refined than their infinite-dimensional counterparts. This is due to the friendlier finite-dimensional setting and also to the convenient choice of reaction term. We maintain the notation of Chapter 1. Our first result is concerned with the almost sure existence of blow-up for arbitrary initial data and noise parameter.

\medskip
\noindent \textbf{Theorem I}. For any $u \in \R^d$ and $\varepsilon > 0$ we have $P_u(\tau_\ve < +\infty) = 1$.

\medskip
Let us notice that for the infinite-dimensional system we were only able to show that
$$
\lim_{\varepsilon \rightarrow 0} P_u ( \tau_\varepsilon < +\infty ) = 1.
$$ This is because, by the particular choice of reaction term $g$, solutions in this setting only explode in one direction, so that comparison arguments can be successfully applied.

Next, we study the asymptotic behavior of the explosion time for initial data in $\mathcal{D}_e$.

\medskip
\noindent \textbf{Theorem II}. Given $\delta > 0$ and a compact set $\mathcal{K} \subseteq \mathcal{D}_e$ there exists $C > 0$ such that
$$
\sup_{u \in \mathcal{K}} P_u ( |\tau_\varepsilon - \tau_0| > \delta ) \leq  e^{-\frac{C}{\varepsilon^2}}
$$ for every $\varepsilon > 0$ sufficiently small.

\medskip
Finally, we show metastable behavior for solutions of \eqref{Aestoc} with initial data $u \in \mathcal{D}_0$. We have the following results.

\medskip
\noindent \textbf{Theorem III}. Given $\delta > 0$ and a compact set $\mathcal{K} \subseteq \mathcal{D}_0$ we have
$$
\lim_{\varepsilon \rightarrow +\infty} \left[ \sup_{u \in \mathcal{K}} \left|P_u \left( e^{\frac{\Delta - \delta}{\varepsilon^2}} < \tau_\varepsilon < e^{\frac{\Delta + \delta}{\varepsilon^2}}\right)- 1\right|\right]=0
$$ where $\Delta:=2(S(\1)-S(0))$.

\medskip
\noindent \textbf{Theorem IV}. If for each $\varepsilon > 0$ we define the scaling coefficient
$$
\beta_{\varepsilon}= \inf \{ t \geq 0 : P_0 ( \tau_\ve > t ) \leq e^{-1} \}
$$ then $\lim_{\varepsilon \rightarrow 0} \varepsilon ^{2}\log\beta_{\varepsilon} = \Delta$ and for each compact set $\mathcal{K} \subseteq \mathcal{D}_0$ and $t > 0$ we have
$$
\lim_{\varepsilon \rightarrow 0} \left[ \sup_{u \in \mathcal{K}} \left|P_{u} (\tau_{\varepsilon} > t\beta_{\varepsilon}) - e^{-t}\right| \right]=0.
$$
\noindent \textbf{Theorem V}. There exists a sequence $(R_\varepsilon)_{\varepsilon > 0}$ with $\lim_{\varepsilon \rightarrow 0} R_\varepsilon = +\infty$ and $\lim_{\varepsilon \rightarrow 0} \frac{R_\varepsilon}{\beta_\varepsilon} = 0$ such that given $\delta > 0$ for any compact set $\mathcal{K} \subseteq \mathcal{D}_{0}$ we have
$$
\lim_{\varepsilon \rightarrow 0} \left[ \sup_{u \in \mathcal{K}} P_u \left( \sup_{0 \leq t \leq \tau_\varepsilon - 3R_\varepsilon}\left| \frac{1}{R_\varepsilon}\int_t^{t+R_\varepsilon} f(U^{\varepsilon}(s,\cdot))ds - f(0)\right| > \delta \right) \right] = 0
$$ for any bounded continuous function $f: \R \rightarrow \R$.

\newpage

With the exception of Theorem I, the proof of the remaining results follow very closely the ideas featured in the previous chapters for the infinite-dimensional problem (perhaps with even fewer technical difficulties). Thus, we include here only the parts of the analysis which differ from the ones given in the previous setting. The main differences appear on the proof of Theorem II and the construction of the auxiliary domain $G$. For the latter, we shall exploit the fact that we are in a finite-dimensional setting to obtain a different construction of $G$, one which does not rely so heavily on the structure of the potential $S$. In particular, this will allow us to obtain our results for every $p > 1$ instead of $p \in (1,5)$. Finally, the results presented here can be extended to more general systems than \eqref{1.1}. We refer the reader to \cite{GS} for details on possible extensions.

\section{Almost sure blow-up in the stochastic model}\label{section4}

\label{estoc.exp}

In this section we devote ourselves to the proof of Theorem I. The idea is to show that, conditioned on non-explosion, the
system is guaranteed to enter a specific region of space in which we can prove that explosion occurs with total probability. From this we can conclude that non-explosion must happen with zero probability.
We do this by comparison with an adequate Ornstein-Ühlenbeck process. Indeed, let $Y^{y,\,\varepsilon}$ be the solution to
\begin{equation}
\label{OU}
dY^{y,\,\varepsilon}=-\Big(AY^{y,\,\varepsilon} + \frac{2}{h} Y^{y,\,\varepsilon}_d e_d\Big)\,dt + \varepsilon dW
\end{equation}
with initial datum $y \in \R^d$. Notice that the drift term in \eqref{OU} is linear and given by a negative definite matrix. Hence, $Y^{y,\ve}$ is in fact a
$d$-dimensional Ornstein-Ühlenbeck process which admits an invariant
distribution supported in $\R^d$. Since we also have convergence to this equilibrium measure for any initial distribution, the hitting time of $Y^{y,\ve}$ of any open set is finite almost surely.

On the other hand, the drift term of \eqref{OU} is smaller or equal than
$b$ so that by the comparison principle we conclude that if $u \geq y$ then $U^{u,\ve}(t) \ge Y^{y,\ve}$ holds almost surely for as long as $U^{u,\ve}$ is finite. From here, Theorem I follows at once from the next lemma and the
strong Markov property.

\begin{lema}
If we consider the set
$$
\Theta^M:=\{ y \in \R^d : y_k \geq 0 \,\,\text{for all}\,\,0\leq k \leq d-1\,,\, y_d \geq M \},
$$
then we have
$$
\lim_{M\rightarrow \infty } \left[\inf_{y\in \Theta^M} P_y(\tau_\ve <\infty)\right] =1.
$$
\end{lema}

\begin{proof}
Consider the auxiliary process $Z^{y,\,\varepsilon}:= U^{y,\,\varepsilon} - \ve
W$. Notice that this process verifies the random differential equation
\begin{equation*}
 dZ^{y,\ve}=b(Z^{y,\ve} +\ve W)dt, \quad Z^{y,\ve}(0)=y.
\end{equation*}
Also observe that $Z^{y,\ve}$ has the same explosion time as
$U^{y,\ve}$. For each $k \in \N$ let us define the set $A_k:=\{\sup_{0\leq t
\leq 1} |W_d(t)|\le k\}$. On $A_k$ we have that $Z^{y,\ve}$ verifies the
inequality
\begin{equation}
\label{ineq.z}
\frac{dZ^{y,\,\varepsilon}}{dt} \ge -AZ^{y,\,\varepsilon} - \frac{4}{h^2}\ve k \sum e_i+ \frac{2}{h} ((Z^{y,\,\varepsilon}_d - \varepsilon k)_+^p - Z^{y,\,\varepsilon}_d - \varepsilon k) e_d.
\end{equation}
Observe that \eqref{ineq.z} can be written as
$$
\frac{dZ^{y,\,\varepsilon}}{dt} \ge QZ^{y,\,\varepsilon} + q +
(Z^{y,\,\varepsilon}_d - \varepsilon k)_+^p e_d \ge QZ^{y,\,\varepsilon} + q ,
$$
where $Q\in\R^{d\times d}$ verifies a comparison principle and $q\in\R^d$ both depend on $\ve, h$ and $k$, but
not on $M$. This allows us to conclude the inequality
$Z_{d-1}^{y,\,\varepsilon} \ge -(M + |q|){\mathrm exp}(|Q|)$ for all $0\le t \le
\min\{1,\tau_\varepsilon^y\}$. In particular, for all $0\le t \le
\min\{1,\tau_\varepsilon^y\}$ the last coordinate verifies the inequality
$$
\left\{\begin{array}{ll}
\frac{dZ_d^{y,\,\varepsilon}}{dt} \ge -\alpha_1 M + \alpha_2
Z_d^{y,\,\varepsilon} + \alpha_3(Z^{y,\,\varepsilon}_d)^p & \\
\\
Z_d^{y,\,\varepsilon}(0)\ge M
\end{array}
\right.
$$
for positive constants $\alpha_1, \alpha_2, \alpha_3$
which do not
depend on $M$. It is a
straightforward calculation to check that solutions to this one-dimensional
inequality blow up in a finite time that converges to zero as \mbox{$M\to
+\infty$.} Therefore, for each $k \in \N$ there exists $M_k$ such that
$P(A_k)\leq \inf_{y \in \Theta^{M}}  P_y(\tau_\ve <\infty)$ for all $M \geq
M_k$. Since $\lim_{k \rightarrow +\infty} P(A_k)=1$, this concludes the proof.
\end{proof}

\section{Convergence of $\tau^u_\varepsilon$ for initial data in $\mathcal{D}_e$}
\label{estoc.conv}

Our purpose in this section is to prove Theorem II. We shall only give the upper bound for the explosion time $\tau^u_\varepsilon$ since the lower bound can be obtained exactly as in \mbox{Proposition \ref{convergenciainferior0}.}
We need the following lemma, whose proof can be found on \cite{GS}.

\begin{lema}[\textbf{Maximum Principle}]\label{prinmax} If $U$ is a solution to \eqref{1.1} then for every $t \geq 0$
\begin{equation}\label{pmaximo} \max_{k=1,\dots,d}
|U_k (t)| \leq \max \{ \max_{k=1,\dots,d} |U_k(0)| , \max_{0\leq s \leq t} U_d(s)\}.
\end{equation}
\end{lema} The upper bound for the explosion time is given in the following proposition.

\begin{prop} For any $\delta > 0$ and compact set $\mathcal{K} \in \mathcal{D}_e$ there exists $C > 0$ such that
\begin{equation*}
\sup_{u \in \mathcal{K}} P_u ( \tau_\varepsilon > \tau_0 + \delta ) \leq e^{- \frac{C}{\varepsilon^2}}
\end{equation*}for every $\varepsilon > 0$ sufficiently small.
\end{prop}

\begin{proof} Given $u \in \mathcal{K}$, let $Y^{u}$ be the solution to the ordinary differential equation
$$
\dot{Y}^{u}=-\Big(AY^{u} + \frac{2}{h} Y^{u,\,\varepsilon}_d e_d\Big)
$$ with initial datum $u$. Let us notice that we have $U^{u} \geq Y^{u}$ for as long as $U^{u}$ is defined by \mbox{the comparison principle.} Now, since $Y^{u}$ is
the solution to a linear system of ordinary differential equations whose associated matrix is symmetric and negative definite, we get that there exists
$\rho_{\mathcal{K}} \in \R$ such that for all $u \in \mathcal{K}$ every coordinate of $U^{u}$ remains bounded from below by $\rho_{\mathcal{K}} + 1$ up until $\tau^u_0$.
Thus, if for $\rho \in \R$ and $M > 0$ we write
$$
\Theta_{\rho}^M:=\{ y \in \R^d : y_k \geq \rho \,\,\text{for all}\,\,0\leq k \leq d-1\,,\, y_d \geq M \}
$$ then by Lemma \ref{prinmax} we conclude that $T_u:=
\inf \{ t \geq 0 : U^{u}_t \in \Theta_{\rho_{\mathcal{K}}+1}^{M+1}\}$ is finite.
Furthermore, since $U^{M+2,\,u}$ agrees with $U^{u}$ until the escape from $\Pi^{M+2}$, we obtain the expression
$$
T_u= \inf \{ t \geq 0 : U^{M+2,\,u}_t \in \Theta_{\rho_{\mathcal{K}+1}}^{M+1}\}.
$$ Taking $T_{\mathcal{K}}:=\sup_{u \in \mathcal{K}} T_u <+\infty$ we may compute
\begin{align*}
P_u\big(\tau_{\varepsilon}(\Theta_{\rho_{\mathcal{K}}}^M) > T_u\big) &\leq  P_u\big( \pi^{M+2}_\varepsilon \wedge \tau_{\varepsilon}(\Theta_{\rho_{\mathcal{K}}}^M) > T_u \big) + P_u\big(\pi^{M+2}_\varepsilon \leq
T_u\,,\,\tau_{\varepsilon}(\Theta_{\rho_{\mathcal{K}}}^M) > T_u \big) \\
\\
& \leq 2 \sup_{v \in \R^d} P_v \Big(\sup_{0\le t \le T_{\mathcal{K}}} |U^{M+2,\,\varepsilon}(t) - U^{M+2}(t)| > 1 \Big).
\end{align*}
On the other hand, by the strong Markov property for $U^{u,\varepsilon}$ we get
\begin{equation}\label{cotaexpfinito}
P_u \big( \tau_\varepsilon > \tau_0 + \delta \big) \leq P_u \big(
\tau_\varepsilon > T_u + \delta \big) \leq \sup_{y \in
\Theta_{\rho_{\mathcal{K}}}^M} P_y ( \tau_\varepsilon > \delta) + \sup_{u \in
\mathcal{K}} P_u\big(\tau_{\varepsilon}(\Theta_{\rho_{\mathcal{K}}}^M) >
T_u\big).
\end{equation}
Thus, by \eqref{grandes1} and the previous computation, in order to finish the proof it will suffice to show that the first term on the right hand side of \eqref{cotaexpfinito} tends to zero exponentially fast in $\frac{1}{\varepsilon^2}$ as $\varepsilon \rightarrow 0$ for an adequate choice of $M$. To see this we consider for $\varepsilon > 0$ and $y \in \Theta_{\rho_{\mathcal{K}}}^M$
the processes $Y^{y,\,\varepsilon}$ and $Z^{y,\,\varepsilon}$ defined by
$$
dY^{y,\,\varepsilon}=-\Big(AY^{y,\,\varepsilon} + \frac{2}{h}
Y^{y,\,\varepsilon}_d e_d\Big)\,dt + \varepsilon dW
$$
and $Z^{y,\,\varepsilon}:= U^{y,\,\varepsilon} - Y^{y,\,\varepsilon}$,
respectively. Notice that $Y^{y,\,\varepsilon}$ is globally defined and thus that both $U^{y,\,\varepsilon}$ and $Z^{y,\,\varepsilon}$ have the same explosion
time. Furthermore, $Z^{y,\,\varepsilon}$ is the solution of the random differential equation
$$
dZ^{y,\,\varepsilon}=-\Big(AZ^{y,\,\varepsilon} + \frac{2}{h}\Big(
\Big[\big(U_d^{y,\ve}\big)^{+}\Big]^p - Z^{y,\,\varepsilon}_d\Big)e_d\Big)\,dt.
$$
The continuity of trajectories allows us to use the Fundamental Theorem of
Calculus to show that almost surely $Z^{y,\,\varepsilon}(\omega)$ is a solution
to the ordinary differential equation
\begin{equation}\label{rde1}
\dot{Z}^{y,\,\varepsilon}(t)(\omega) = -AZ^{y,\,\varepsilon}(\omega) +
\frac{2}{h}\Big( \Big[\big(U_d^{y,\ve}\big)^{+}\Big]^p(\omega) -
Z^{y,\,\varepsilon}_d(\omega)\Big)e_d.
\end{equation}
Then, for each $y \in \Theta_{\rho_{\mathcal{K}}}^M$ and $\varepsilon > 0$ let $\Omega^y_\varepsilon$ be a set of probability one in which \eqref{rde1} holds.
Notice that for every $\omega \in \Omega^y_\varepsilon$ we have the inequality
\begin{equation*}\label{rde2}
\dot{Z}^{y,\,\varepsilon}(\omega) \geq -AZ^{y,\,\varepsilon}(\omega) - \frac{2}{h}Z^{y,\,\varepsilon}_d(\omega)e_d.
\end{equation*}
By the comparison principle we conclude that $Z^{y,\,\varepsilon}(\omega)\geq 0 $ for every $\omega \in \Omega^y_\varepsilon$
and, therefore, that the inequality $U^{y,\,\varepsilon}(\omega) \geq
Y^{y,\,\varepsilon}(\omega)$ holds for as long as $U^{y,\,\varepsilon}(\omega)$
is defined.

For each $y \in \Theta_{\rho_{\mathcal{K}}}^M$ and $\varepsilon > 0$ let us
also consider the set
$$
\tilde{\Omega}^y_\varepsilon = \Big\{ \omega \in \Omega : \sup_{0\le t\le
\delta} | Y^{y,\,\varepsilon}(\omega,t)-  Y^{y}(\omega,t)| \leq 1 \,,\, \sup_{0
\leq t \leq \delta} |\varepsilon W(\omega,t)| \leq 1 \Big\}.
$$ Our goal is to show that if $M$ is appropriate then for each $y \in \Theta_{\rho_{\mathcal{K}}}^M$ \mbox{the trajectory $U^{y,\,\varepsilon}(\omega)$} explodes before time $\delta$ for all $\omega \in \Omega^y_\varepsilon \cap \tilde{\Omega}^y_\varepsilon$.
From this we get that
$$
\inf_{y \in \Theta_{\rho_{\mathcal{K}}}^M} P(\tilde{\Omega}^y_\varepsilon) =
\inf_{y \in \Theta_{\rho_{\mathcal{K}}}^M} P( \Omega^y_\varepsilon \cap
\tilde{\Omega}^y_\varepsilon) \leq \inf_{y \in \Theta_{\rho_{\mathcal{K}}}^M}
P_y (\tau_\varepsilon \leq \delta ).
$$ and so by \eqref{grandes1} we may conclude the result.

Hence, let us take $y \in  \Theta_{\rho_{\mathcal{K}}}^M$, $\omega \in
\Omega^y_\varepsilon \cap \tilde{\Omega}_\varepsilon$ and suppose that
$U^{y,\,\varepsilon}(\omega)$ is defined in $[0,\delta]$.
Notice that since $\omega \in \Omega^y_\varepsilon \cap
\tilde{\Omega}_\varepsilon$ then the $(d-1)$-th coordinate of
$Y^{y,\,\varepsilon}(\omega,t)$ is bounded from below by $\rho_{\mathcal{K}} -
1$ for all $t \in [0,\delta]$. By comparison we know that the $(d-1)$-th coordinate
of $U^{y,\,\varepsilon}_t(\omega,t)$ is bounded from below by
$\rho_{\mathcal{K}} - 1$ as well.
From here we deduce that the last coordinate of $U^{y,\varepsilon}(\omega)$
verifies the integral equation
\begin{equation*}
{U}^{y,\,\varepsilon}_d(\omega,t) \geq {U}^{y,\,\varepsilon}_d(\omega,s) + \int_s^t \frac{2}{h²}\,\Big(- U^{y,\,\varepsilon}_d(\omega,r) +\rho_{\mathcal{K}} -1  + hg\big( U^{y,\,\varepsilon}_d(\omega,r)\big) \Big)\,dr - 1
\end{equation*}for $s < t$ in the interval $[0,\delta]$. We can take $M \in \N$ sufficiently large so as to guarantee that there exists a constant $\alpha > 0$ such that for all $m\geq M$ we have
\begin{equation*}\label{lcotafea}
\frac{2}{h²}\,\big(-m +\rho_{\mathcal{K}} -1  + hg(m) \big) \geq \alpha m^p.
\end{equation*} If we recall that $U^{y,\,\varepsilon}_d(\omega,0) \geq M$ then our selection of $M$ implies that
\begin{equation*}
{U}^{y,\,\varepsilon}_d(\omega,t) \geq M-1 + \alpha\int_0^t \big( U^{y,\,\varepsilon}_d(\omega,u)\big)^p\,du
\end{equation*}
for every $t \in [0,\delta]$. But if this inequality holds and $M$ is sufficiently large, one can check that $U^{y,\,\varepsilon}(\omega)$ explodes before time
$\delta$, a fact which contradicts our assumptions. Therefore, if $y \in
\Theta_{\rho_{\mathcal{K}}}^M$ and $\omega \in \Omega^y_\varepsilon \cap
\tilde{\Omega}_\varepsilon$ then $U^{y,\,\varepsilon}(\omega)$ explodes before
time $\delta$, \mbox{which concludes the proof.}
\end{proof}

\section{Construction of an auxiliary domain}

In this final section we present the alternative construction of the auxiliary domain $G$. The reader will notice that the finite-dimensional environment plays an essential role in the construction. Despite this fact, we point out that the only other ingredient which is relevant in this alternative construction is the validity of an analogue of Theorem \ref{descomp2}, so that one may carry out the same construction in other systems with a similar description.

We wish to construct a bounded domain $G$ satisfying the following properties:
\begin{enumerate}
\item [i.] $G$ contains $\1$ and the origin.
\item [ii.] There exists $c > 0$ such that $B_c \subseteq G$ and for all
$u \in B_c$ the system $U^{u}$ is globally defined and tends to $0$ without escaping $G$.
\item [iii.] There exists a closed subset $\p^{\1}$ of the boundary $\partial G$ which satisfies:
\begin{enumerate}
\item [$\bullet$] $V(0,\partial G - \partial^{\1} ) >  V(0,\partial^{\1}) = V(0, \1)$.
\item [$\bullet$] $\p^{\1}$ is contained in $\mathcal{D}_e$ and at a positive distance from its boundary.
\end{enumerate}
\end{enumerate}

\newpage

The domain $G$ is constructed in the following manner:

Since $S: \R^d \rightarrow \R$ is continuous and $S(\1) > S(0)$, we may take $c>0$ such that $S(u)< S(\1)$ for $u \in B_c$. Then, for each $u\in \partial B_c$ consider the ray $R_u:=\{ \lambda u :
\lambda > 0\}$. Since the vector $\1$ is not tangent to $\W_\1^s$ at $\1$, we
may take a sufficiently small neighborhood $V$ of $c\cdot \1$ such that for every $u \in
V\cap \partial B_c$ the ray $R_u$ intersects $\W_\1^s \cap (\R_{> 0})^d$. For
such $V$ we may then define $\bar{\lambda}_u=\inf\{ \lambda > 0 : \lambda u \in
\W_\1^s\}$ for $u \in V\cap\partial B_c$. If we consider
$$
\eta:= \inf_{u \in \partial[V\cap\partial B_c]} \phi(\bar{\lambda}_u u) >
\phi(\1)
$$ where by $\partial[V\cap \partial B_c]$ we understand the boundary of $V\cap \partial B_c(0)$ as a $(d-1)$-dimensional manifold, then the fact that $S(U^u(t))$ is strictly decreasing allows us to shrink $V$ into a smaller neighborhood $V^*$ of $c\cdot \1$ such that $S(v)=\eta$ is satisfied for all $v \in \partial[V^*\cap \partial B_c]$. Let us also observe that since $\1$ is the only saddle point we can
take $V$ sufficiently small so as to guarantee that $\max\{ S(\lambda u) : \lambda > 0\} \geq \eta$ for all $u \in \partial B_c\setminus V^*$. Then if we take the level curve $C_\eta = \{ x \in \R^d : S(x) = \eta \}$ every ray
$R_u$ with $u \in \partial B_c\setminus V^*$ intersects $C_\eta$. With this we may define for each $u \in \partial B_c$
$$
\lambda_u^*= \left\{\begin{array}{ll}
\bar{\lambda}_u & \mbox{ if }\,\,u \in V^*\\
\\
\inf\{ \lambda > 0 : \lambda u \in C_\eta\} & \mbox{ if }\,\,u \in B_c(0)\setminus V^*
\end{array}
\right..
$$ Notice that the mapping $u \mapsto \lambda^*_u$ is continuous. Thus, if $\tilde{G}:=\{ \lambda u : 0 \leq \lambda < \lambda^*_u\,,\, u \in
\partial B_c\}$ then $\partial\tilde{G} =\{ \lambda^*_u u : u \in \partial
B_c(0)\}$. To finish the construction of our domain we must make a slight radial expansion of $\tilde{G}$, i.e. for $\alpha > 0$ consider $G$ defined by the
formula
$$
G:=\{\lambda u : 0 \leq \lambda < (1+\alpha)\lambda^*_u\,,\, u \in \partial B_c\}.
$$
Observe that the finite-dimensional analogue of Theorem \ref{descomp2} ensures that $G$ verifies (i). Since $\lambda^*_u > 1$ for all $u \in \partial B_c(0)$ then it
must also verify (ii). Furthermore, if we define $\partial^\1:=\{(1+\alpha)\lambda^*(u) u : u \in \overline{V^*}\}$ then $\partial^\1$ is closed and contained in $\mathcal{D}_e$. By taking $\alpha > 0$ sufficiently small, the continuity of $S$ implies that (iii) holds as well.
See Figure 6.2.

\psfrag{0}{$0$}
\psfrag{1}{\vspace{-45pt}$\1$}
\psfrag{W1s}{$\W_\1^s$}
\psfrag{phieta}{$C_\eta$}
\psfrag{Wu}{$\mathcal W_1^s$}
\psfrag{Ws}{$\mathcal W_1^s$}
\begin{figure}
\begin{center}
\includegraphics[width=7cm]{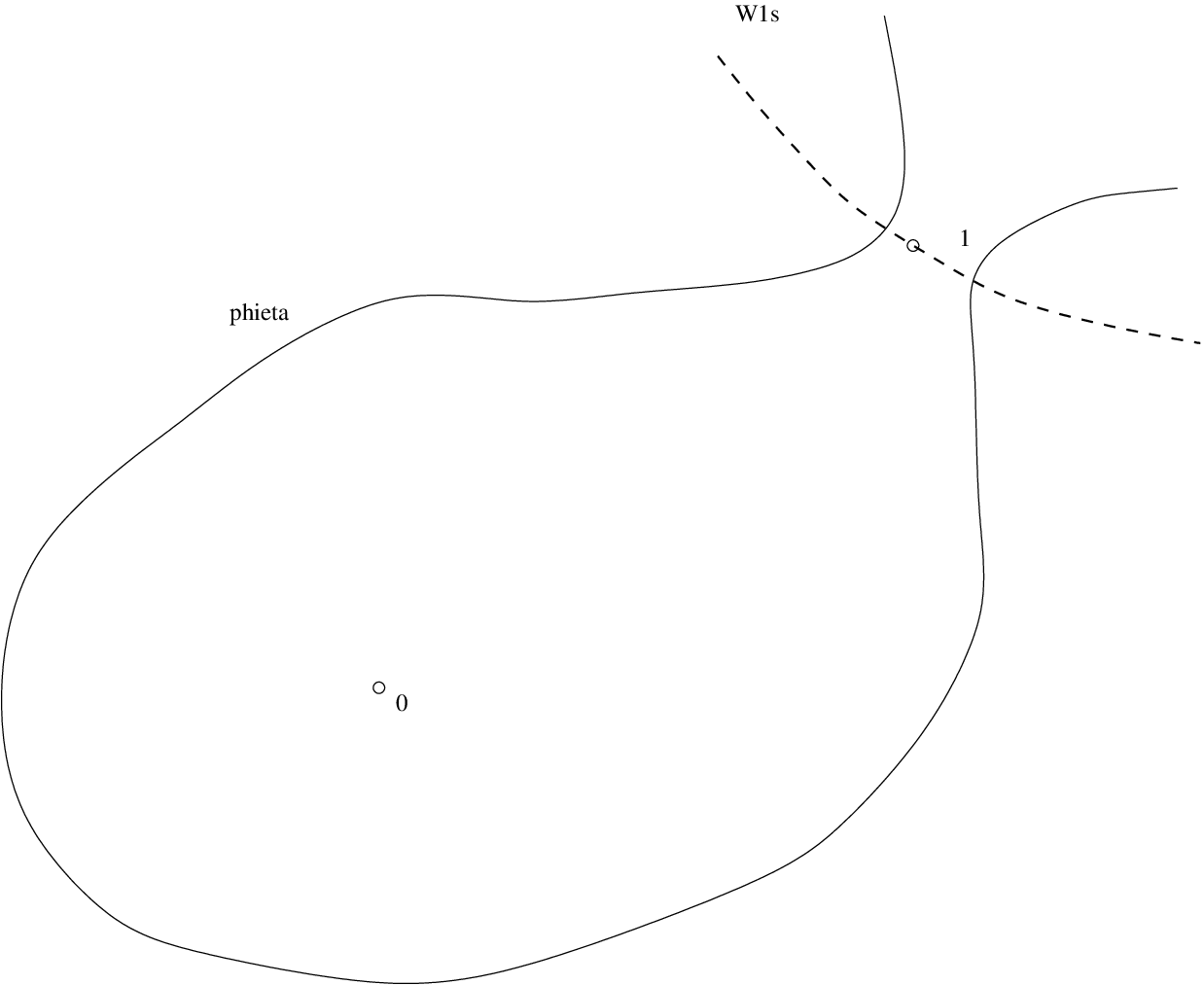}
\end{center}
\label{fig:gaux}
\caption{The level curve $C_\eta$ and the stable manifold of $\1$.}
\end{figure}

\newpage
\section{Resumen del Capítulo 6}

Estudiamos las propiedades asintóticas del tiempo de explosión para perturbaciones aleatorias por ruido blanco aditivo de la ecuación diferencial ordinaria
$$
\left\{\begin{array}{lcll}
U'_1 &= &\frac{2}{h^2} ( -U_1 + U_2 ),\\
U'_i &= &\frac{1}{h^2} ( U_{i+1} - 2U_i + U_{i-1} ) &\,\,\, 2 \leq i \leq d-1,\\
U'_d &= &\frac{2}{h^2} ( -U_d + U_{d-1} +hg(U_d) )\\
U(0) &= &u
\end{array}\right.
$$ donde $h > 0$ es un parámetro fijo y $g : \R \rightarrow \R$ es un término de reacción dado por $g(x) = (x^+)^p -x$ para $p>1$. Este tipo de sistemas surge como discretizaciones espaciales de ecuaciones de difusión con condiciones de frontera no lineales de Neumann. De hecho, puede probarse que cuando $h\to 0$ las soluciones de este sistema convergen a las solución de la EDP
$$
\left \{\begin{array}{rcll}
\p_t U(t,x) & = & \p^2_{xx}U(t,x) & 0<x<1, 0\le t<T,\\
\p_x U(t,0) & = & 0 & 0\le t <T,\\
\p_x U(t,1) &= & g(U(t,1)) & 0\le t <T,\\
U(0,x) & = & U_0(x) & 0\le x \le 1.
\end{array} \right.
$$

La ecuación diferencial ordinaria puede escribirse como en forma matricial como
$$
dU = - \nabla S,
$$ donde $S$ viene dado por
$$
S(u) = \frac{1}{2} \langle Au , u \rangle - \frac{2}{h}\Big(\frac{\;\;\;|u_d^+|^{p+1}}{p+1} - \frac{\,\,{u_d}^2}{2}\Big).
$$ para cierta matriz $A \in \R^{d \times d}$ definida positiva. Este potencial $S$ satisface las mismas propiedades que su análogo infinito-dimensional en \eqref{formalPDE}.
Tiene exactamente dos puntos críticos, $\1:=(1,\dots,1)$ y el origen, ambos ellos hiperbólicos. El origen $\mathbf{0}$ es el único mínimo local de $S$ mientras que $\1$ es un punto de ensilladura. Más aún, se tiene una descomposición de $\R^d$ análoga a la de \eqref{decomp12}.

Las perturbaciones estocásticas que vamos a considerar son de la forma
$$
dU^{\ve} = - \nabla S dt + \ve dW
$$
par $\ve>0$ pequeño y donde $W=(W_1, \dots,W_d)$ es un movimiento Browniano $d$-dimensional estándar. La solución $U^\varepsilon$ de esta EDOE conserva las propiedades de la solución de \eqref{MainSPDE}.

Los resultados que podemos probar en este contexto son esencialmente los mismos que para la EDP \eqref{MainPDE}, con la excepción de que la restricción $p < 5$ desaparece en este contexto puesto que la geometría del potencial $S$ puede manejarse con mayor facilidad al estar definido sobre un espacio finito-dimensional como lo es $\R^d$ (el potencial de \eqref{MainPDE} estaba definido $C_D([0,1])$, un espacio infinito-dimensional). Por otro lado, la elección particular del término $g$ (con término no lineal acotado inferiormente) nos permite probar además que el fenómeno de blow-up se hace presente casi seguramente.

\noindent \textbf{Teorema}. Para cualquier $u \in \R^d$ y $\varepsilon > 0$ tenemos $P(\tau^u_\ve < +\infty) = 1$.

La idea de la demostración de este resultado consiste en mostrar que, condicionado a no explotar, el sistema estocástico alcanza inexorablemente una región particular del espacio en donde uno puede probar que el fenómeno de blow-up ocurre con probabilidad total. A partir de esto se deduce inmediatamente que la ausencia de blow-up debe darse con probabilidad nula. Mostramos que el proceso alcanza esta región particular mediante comparación con un proceso de Ornstein-Ühlenbeck adecuado, mientras que la explosión casi segura ocurre en esa región se obtiene mediante técnicas de ecuaciones similares a las empleadas durante el Capítulo 2.

Con respecto a los resultados restantes, la demostración de los mismos sigue muy de cerca las ideas presentadas en los capítulos anteriores para el problema infinito-dimensional (con quizás menos dificultades técnicas). Incluimos en este capítulo únicamente las partes del análisis que difieren de aquellas dadas en el marco anterior. Estas aparecen en la demostración de lo que sería el Teorema I en este contexto y en la construcción del dominio auxiliar $G$. Con respecto al Teorema I, el potencial en este contexto tiene una estructura ligeramente diferente al infinito-dimensional, lo cual no nos permite adaptar por completo las ideas desarrolladas en el Capítulo 2 y nos obliga por lo tanto a introducir algunas variantes de las técnicas de ecuaciones empleadas durante éste. Finalmente, para la construcción de $G$ explotamos el marco finito-dimensional de este nuevo problema para proponer una nueva construcción que no impone la restricción $p < 5$. A partir de aquí, pueden obtenerse los resultados restantes para todo $p > 1$.

\part{The Fernández-Ferrari-Garcia dynamics on diluted models}

\chapter*{Introducción a la Parte II}

La mecánica estadística del equilibrio intenta explicar el comportamiento macroscópico de sistemas en equilibrio térmico en términos de la interacción microscópica entre su gran número de constituyentes.
Como un ejemplo típico, uno podría tomar un material ferromagnético como el hierro: sus constituyentes son entonces los spins de los imanes elementales en los sitios de un cierto reticulado de cristal.
O también podemos pensar en una aproximación discreta de un gas real, en cuyo caso los constituyentes son los números de partículas en las celdas elementales de cualquier partición del espacio. El objeto central en cualquiera de estos sistemas es el Hamiltoniano que describe la interacción entre los constituyentes. Éste determina la energía relativa entre configuraciones que difieren únicamente microscópicamente. Los estados de equilibrio con respecto a la interacción dada son descritos por las medidas de Gibbs asociadas. Éstas son medidas de probabilidad en el espacio de configuraciones con probabilidades condicionales dadas respecto a configuraciones fijas fuera de regiones acotadas. Dichas probabilidades condicionales son determinadas por el factor de Boltzmann: la exponencial de la temperatura inversa multiplicada por la energía relativa. Esto le permite a uno calcular, al menos en principio, esperanzas en equilibrio y funciones de correlación espacial siguiendo el formalismo de Gibbs estándar. El conjunto de medidas de Gibbs para una interacción dada es un simplex cuyos vértices llamamos medidas de Gibbs extremales.
Estas medidas extremales son de mayor importancia puesto que describen los posibles macroestados (o fases de equilibrio) de \mbox{nues}tro sistema físico. En un estado tal, las observables macroscópicas no fluctúan mientras que la correlación entre observaciones locales hechas a larga distancia entre ellas decaen a cero. Un aspecto muy importante en el estudio de cualquier sistema de la mecánica estadística del equilibrio es determinar cuando existe más de un posible macroestado para el sistema, un fenómeno conocido como transición de fase. Por la estructura de simplex del conjunto de medidas de Gibbs, la ocurrencia de transición de fase para un sistema dado es equivalente a la existencia de medidas de Gibbs múltiples (no necesariamente extremales).

Uno de los modelos más famosos y mejor entendidos de la mecánica estadística es el modelo de Ising ferromagnético estándar en el reticulado $\Z^2$. En este modelo, sobre cada sitio del reticulado se tiene una variable de spin que toma puede tomar solamente los valores $+$ y $-$. La interacción es entre vecinos más cercanos y tiende a alinear spins vecinos en la misma dirección. Mediante los argumentos ingeniosos formulados en primera instancia por Peierls en 1936,
la transición de fase en este modelo puede entenderse a través de la inspección de configuraciones típicas de contornos, i.e. líneas quebradas que separan los dominios con spins $+$ y $-$, respectivamente.
Para temperaturas bajas, existe una fase $+$ que es realizada por un mar infinito de spins $+$ con islas finitas de spins $-$ (que a su vez pueden contener lagos de spins $+$, y así sucesivamente). En términos de contornos,
este panorama equivale a que haya únicamente finitos contornos rodeando cada sitio del reticulado. También se tiene una fase $-$ que verifica la descripción simétrica. Por otro lado, por encima de cierta temperatura crítica no existe ningún camino infinito que una vecinos más cercanos con el mismo spin. Por lo tanto, para este modelo la estructura geométrica de las configuraciones típicas está bien entendida (ver \cite{PS,G} por ejemplo).
En general, sin embargo, se sabe mucho menos, y mucho menos es cierto. Aún así, ciertos aspectos de este análisis geométrico tienen amplias aplicaciones, al menos en ciertos regímenes del diagrama de fases. Estos ``ciertos regímenes'' son, por un lado, el régimen de alta temperatura (o, en el contexto de gases, baja densidad) y, al otro extremo, el comportamiento a baja temperatura (o altas densidades).

A temperaturas altas o baja densidad, todas las consideraciones termodinámicas están basadas en el hecho de que la entropía domina sobre la energía. Esto es, la interacción entre los \mbox{cons}tituyentes no es lo suficientemente efectiva
para forzar un ordenamiento macroscópico del sistema. Como resultado, los constituyentes se comportan aproximadamente al azar, no muy influenciados por otros constituyentes que se encuentran lejos.
Así, el comportamiento del sistema es casi el de un sistema libre con componentes independientes. Esto significa, en particular, que en el centro de una caja grande típicamente vamos a encontrar aproximadamente las mismas configuraciones sin importar qué condiciones de frontera sean impuestas fuera de dicha caja. A bajas temperaturas o densidades grandes (cuando la interacción es suficientemente fuerte), el panorama de arriba ya no es válido.
En realidad, las características específicas de la interacción entrarán en juego y determinarán las cualidades específicas de los macroestados. En muchos casos, el comportamiento a baja temperatura
puede ser descrito como una perturbación aleatoria de un estado fundamental, i.e. una configuración fija de energía mínima. Luego, a bajas temperaturas podemos esperar que las fases de equilibrio se realicen como una \mbox{configuración} de estado fundamental determinística, perturbada por finitas islas aleatorias en donde la configuración difiere con dicho estado fundamental. Esto significa que el patrón del estado fundamental puede percolar a través del espacio hasta el finito. Una manera prominente de confirmar este panorama general es provista por la llamada teoría de Pirogov-Sinai, descrita en detalle en \cite{Z2}.

Esencialmente, esta teoría introduce en primer lugar una noción generalizada de contornos que puede ser utilizada por una amplia gama de sistemas y luego da condiciones que garantizan cuando existen solamente finitos de estos contornos alrededor de cada sitio del reticulado. Cuando esto suceda, un panorama similar al del modelo de Ising puede obtenerse. Más allá de lo poderosa que sea como herramienta, una de las desventajas de la teoría de Pirogov-Sinai es que la mayoría de sus aplicaciones se apoyan fuertemente en la convergencia absoluta de ciertas expansiones, llamadas expansiones en aglomerados, y esta convergencia muchas veces depende de resultados combinatorios profundos. Así, la teoría de Pirogov-Sinai constituye en realidad (al menos hasta cierto grado) un enfoque más combinatorio que probabilístico para entender las fases de equilibrio de un sistema físico dado. Como el problema matemático mismo se encuentra formulado dentro de un marco probabilístico, uno puede ver que este enfoque quizás no sea el más natural posible. Más aún, debido a la absoluta convergencia de las expansiones involucradas, uno obtiene gratuitamente la analiticidad de las funciones de correlación con respecto a los parámetros del modelo (como lo son la temperatura inversa o la densidad de partículas). Aunque la analiticidad es una buena propiedad para tener, es también un síntoma de que este enfoque es quizás demasiado fuerte y no óptimo desde el punto de vista probabilístico.

Como una alternativa, en \cite{FFG1} los autores proveen un enfoque nuevo al estudio de este tipo de sistemas, uno que es puramente probabilístico. En lugar de depender de las expansiones en aglomerados para probar que existe una medida de equilibrio que satisface el panorama de mar con islas descrito arriba, ellos realizan esta medida como la distribución estacionaria de una red de pérdida que puede ser estudiada utilizando herramientas estándar y nociones de modelos probabilísticos y procesos. En este contexto, la existencia de la medida de equilibrio está relacionada con la ausencia de percolación en un proceso de percolación orientada. Más aún, muestran que la dinámica converge exponencialmente rápido a la medida buscada, de manera que este enfoque es también valioso para propósitos de simulación. Para ser precisos, en su trabajo los autores consideran únicamente el modelo de contornos de Peierls para la interacción de Ising a baja temperatura, mientras que en \cite{FFG2} discuten como algunas de estas ideas pueden ser extendidas a otros modelos.

En esta segunda parte de la tesis introducimos una familia general de sistemas, la clase de los \textit{modelos diluidos}, y mostramos que los resultados principales en \cite{FFG1} pueden extenderse a esta familia más amplia.
El marco de modelos diluidos encaja perfectamente con el rango de aplicabilidad de este nuevo enfoque: los modelos diluidos son, quizás, la familia más amplia de modelos a la cual la dinámica presentada en \cite{FFG1} pueda aplicarse. Este marco incluye tanto modelos discretos como continuos de manera unificada, pero es lo suficientemente concreto como para que aún sea posible obtener un criterio general para la existencia de la medida de equilibrio. Concretamente, en esta segunda parte vamos a desarrollar el siguiente plan:

\begin{enumerate}
\item [i.] Introducir la familia de modelos diluidos y mostrar que para cualquier elemento en esta familia podemos definir una dinámica con las características mostradas en \cite{FFG1}.
\item [ii.] Utilizar la construcción de la dinámica para obtener un criterio general para la unicidad de medidas de Gibbs en modelos diluidos y estudiar propiedades de este único equilibrio, como la propiedad de mixing exponencial.
\item [iii.] Mostrar que la ausencia de percolación en el proceso de percolación orientada implica la continuidad de las funciones de correlación con respecto a los parámetros del modelo.
\item [iv.] Explotar el marco general de los modelos diluidos y las características de sus dinámicas asociadas para mostrar que, bajo condiciones adecuadas, las medidas de equilibrio de sistemas discretos convergen, cuando son apropiadamente escaladas, a la medida de equilibrio de sistemas continuos en el régimen de alta temperatura o baja densidad.
\item [v.] Combinar las ideas y resultados previos con el marco de la teoría de Pirogov-Sinai para obtener algunos resultados fuera del rango de convergencia de las expansiones en \mbox{aglomerados.}
\item [vi.] Combinar las ideas y resultados previos con el marco de la teoría de Pirogov-Sinai para obtener un algoritmo de simulación perfecta para una clase medianamente grande de medidas de equilibrio en el régimen de baja temperatura o alta densidad.
\end{enumerate}

La Parte II está organizada de la siguiente manera. En el Capítulo 7 proveemos del marco teórico en donde se definen los modelos diluidos. El Capítulo 8 se enfoca en la definición de modelos diluidos y alguna de sus propiedades básicas. En este capítulo también adaptamos algunas nociones elementales de la mecánica estadística a este marco de trabajo. El capítulo siguiente esta destinado a mostrar lo restante de (i) y (ii) en el plan de arriba.
Los items (iii) y (iv) son establecidos en los Capítulos 10 y 11, respectivamente. Finalmente, los items (v) y (vi) se establecen en el Capítulo 12.

\chapter*{Introduction to Part II}

The purpose of equilibrium statistical mechanics is to describe macroscopic behavior of systems in thermal equilibrium in terms of the microscopic interactions among the great number of elements which constitute them. 
The most common example is the one of some ferromagnetic material like iron: its elements are then the spins of elementary magnets situated at the various sites of a given crystal lattice. Or we may also consider a lattice approximation to a real gas, in which case the elements are the particle numbers inside each of the cells of a given partition of space.
In any of these systems, the central object is the Hamiltonian which describes the microscopic interaction between its elements by determining energy between configurations which differ only microscopically. 
The equilibrium states with respect to the given interaction are specified by the so called Gibbs measures associated to the model.
These are probability measures on the space of configurations with given conditional probabilities relative to fixed configurations outside of bounded regions. These conditional probabilities are given by the Boltzmann factor, i.e. the exponential of the inverse temperature times the relative energy of the configuration. The set of Gibbs measures for a given interaction is a simplex whose vertices we call extremal Gibbs measures.
These are most important since they describe the possible macrostates (or equilibrium phases) of our physical system. In such a state, we have that macroscopic observables do not fluctuate and also that the correlation between local observations made far apart from each other decays to zero. A very important aspect in the study of any system in equilibrium statistical mechanics is determining whether there exists more than one possible macrostate for the system, a phenomenon known as phase transition. By the simplex structure of the set of Gibbs measures, the occurrence of phase transition for a given system is equivalent to the existence of multiple (not necessarily extremal) Gibbs measures.

One of the most famous and better understood models in statistical mechanics is the standard ferromagnetic Ising model on the square lattice. In this model, at each site of the lattice we have a spin variable which can take only two values, $+$ or $-$. The interaction is of nearest-neighbor and tends to align neighboring spins in the same direction. In 1936, Peierls showed that the phase
transition in this model can be understood by looking at the typical configurations of contours: finite circuits separating the domains with plus and minus spins, respectively.
For low temperatures, there exists a plus phase which is realized by an infinite sea of plus spins with finite islands of minus spins (which may further contain lakes of plus spins, and so on). In terms of contours, this picture corresponds to there being only finitely many contours surrounding any site in the lattice. One also has a minus phase verifying the symmetric description. On the other hand, above a certain critical temperature there is no infinite path joining nearest neighbors with the same spin value. Thus, for this model the geometric structure of typical configurations is well understood (see \cite{PS,G} for example). In general, however, much less is known and a similar description may not always hold. Still, certain aspects of the geometric analysis performed by Peierls have wide applications, at least in certain regimes of the phase diagram. These regimes are, on the one hand, the high-temperature (or low-density of particles in a lattice gas setting) regime and, on the other hand, the low-temperature behavior (or high-density of particles).

At high temperatures, the behavior is dictated by the fact that entropy dominates over energy. That is, the interaction between the elements of our system is not
strong enough to enforce a macroscopic ordering of it. As a consequence, the elements of our system behave more or less at random, without being much influenced by other elements which are far apart.
Thus, the behavior of the system is almost like that of a free system with independent components.
In particular, we have that deep inside a large region we will typically encounter more or less the same configurations no matter which boundary conditions are imposed outside this region.
However, at low temperatures (i.e. when the interaction is strong enough),
the scenario described above no longer holds. Instead, the particular characteristics of the interaction
will come into play and determine the specific features of the low temperature macrostates. In
many cases, the low temperature behavior can be seen as a random perturbation of
a ground state: a fixed configuration having minimal energy. Therefore, one expects, for sufficiently low temperatures, the equilibrium phases to be obtained as a deterministic ground state configuration perturbed by finite
random islands on which the configuration disagrees with the corresponding ground state. In particular, the configuration pattern imposed by the ground state percolates in space to infinity. One way in which to show this general picture is provided by Pirogov-Sinai theory, described in detail in \cite{Z2}.

Essentially, this theory first introduces a generalized notion of contours which can be used for a wide range of systems and then gives conditions which guarantee when are there only finitely many of such contours surrounding each site in the lattice. Whenever this is the case, a similar picture to the one for the Ising model can be obtained. As powerful a tool as it may be, one of the disadvantages of Pirogov-Sinai theory is that most of its applications rely heavily on the absolute convergence of certain expansions, known as the cluster expansions, and this convergence often depends on deep combinatorial results. Thus, in fact Pirogov-Sinai theory constitutes (at least to some degree) more of a combinatorial approach to understanding the equilibrium phases of a given physical system rather than a probabilistic one. Since the mathematical problem itself is posed within the probabilistic framework, one can see that this approach is perhaps not the most natural one to have. Furthermore, due to the absolute
convergence of the expansions involved, one obtains for free the analyticity of
correlation functions with respect to the parameters in the model (e.g. inverse temperature or density of particles). Though analyticity is a very nice property to have, it is also a symptom that this approach is perhaps too strong and not optimal from the probabilistic point of view.

As an alternative, in \cite{FFG1} the authors provide a fresh new approach to study this type of systems, one which is purely probabilistic. Instead of relying on cluster expansions to prove that there exists an equilibrium measure satisfying the sea with islands picture described above, they realize this measure as the stationary distribution of a loss network dynamics that can be studied using standard tools and notions from probabilistic models and processes. In this context, the existence of the equilibrium measure is related to the absence of percolation in an oriented percolation process. Furthermore, they show that the dynamics converges exponentially fast to the desired measure, so that this approach is also valuable for simulation purposes. Strictly speaking, in their work the authors only consider the model of Peierls contours for the Ising interaction at low temperature (i.e. low density of contours), while in \cite{FFG2} they discuss how some of these ideas can be extended to other models in the low density regime.

In this second part of the thesis we introduce a general family of systems called \mbox{\textit{diluted models},} and show that the main results obtained in \cite{FFG1} can be extended to this broader family.
The framework of diluted models perfectly fits the range of applicability of this new approach: diluted models are, perhaps, the broadest family of models for which the dynamics presented in \cite{FFG1} may be applied. This framework covers both discrete and continuum systems in an unified way, while remaining concise enough so that a general criterion for the existence of the equilibrium measure can still be obtained. Concretely, in this second part of the thesis we carry out the following plan:

\begin{enumerate}
\item [i.] Introduce the family of diluted models and show that for any element in this family we can define a dynamics with the characteristics shown in \cite{FFG1}.
\item [ii.] Use the construction of the dynamics to obtain a general criterion for uniqueness of Gibbs measures in diluted models and study properties of this unique equilibrium, such as exponential mixing.
\item [iii.] Show that the absence of percolation in the oriented percolation process implies the continuity of correlation functions with respect to the parameters of the model.
\item [iv.] Exploit the general framework of diluted models and the characteristics of their associated dynamics to show that, under suitable conditions, equilibrium measures of discrete systems converge, when properly rescaled, to equilibrium measures of continuum systems in the high temperature (or low density) regime.
\item [v.] Combine the previous ideas and results with the framework of Pirogov-Sinai theory to obtain results outside the range of convergence of cluster expansions.
\item [vi.] Combine the previous ideas and results with the framework of Pirogov-Sinai theory to obtain a perfect simulation algorithm for a large class of equilibrium measures in the low temperature (or high density) regime. This extends the previous results in \cite{FFG1} obtained for the high temperature (or low density) regime.
\end{enumerate}

Part II is organized as follows. In Chapter 7 we provide the theoretical setting in which diluted models are defined. Chapter 8 focuses on the definition of diluted models and some of their basic properties. In this chapter we also adapt some elementary notions from statistical mechanics to this framework. The following chapter is devoted to showing the remainder of (i) and (ii) in the plan above.
Items (iii) and (iv) are established in Chapters 10 and 11, respectively. Finally, items (v) and (vi) are settled in Chapter 12.

\chapter{Preliminaries}

Since we intend the class of diluted models to include discrete and continuum models, we are interested in adopting a general framework which a priori makes no distinction between both types of systems.
The correct framework is that of particle configurations, which we introduce now. We refer to \cite{K,DVJ2,DVJ1} for further details.

\section{Particle configurations}

Throughout this second part we fix two locally compact complete separable metric spaces: the \textit{allocation space} or \textit{lattice} $(S,d_S)$ and the \textit{animal set} or \textit{spin set} $(G,d_G)$. Typical examples of allocation spaces include $S=\Z^d$ or $S=\R^d$ for some $d \in \N$, whereas the spin set can range from finite sets such as $G=\{+,-\}$, to uncountable sets such as $G=S^{d-1}$.
The product space $S \times G$ is also a locally compact complete separable metric space if endowed with the product metric $d_S + d_G$. For convenience, we shall denote an element $(x,\gamma) \in S \times G$ simply by $\gamma_x$, which we interpret as an animal $\gamma$ positioned at location $x$.

\begin{notat} Given a metric space $(X,d)$ we write:
\begin{itemize}
\item [$\bullet$] $\B_X$ for the class of all Borel subsets of $X$.
\item [$\bullet$] $\B^0_X$ for the class of all Borel subsets of $X$ with compact closure.
\end{itemize}
\end{notat}

\begin{defi}Let $\xi$ be a measure on $(S\times G, \B_{S \times G})$.
\begin{itemize}
\item [$\bullet$] $\xi$ is said to be a \textit{Radon measure} if $\xi(B) < +\infty$ for every $B \in \B^0_{S \times G}$.
\item [$\bullet$] $\xi$ is said to be a \textit{particle configuration} if $\xi(B) \in \N_0$ for every $B \in \B^0_{S \times G}$.
\end{itemize}
\end{defi}

\begin{prop} A Radon measure $\xi$ on $(S\times G, \B_{S \times G})$ is a particle configuration if and only if there exist a countable set $Q_\xi \subseteq S \times G$ and $m_\xi: Q_\xi \to \N$ such that
\begin{equation}\label{standardrep}
\xi = \sum_{\gamma_x \in Q_\xi} m_\xi(\gamma_x) \delta_{\gamma_x}
\end{equation} with $\delta_{\gamma_x}$ being the Dirac measure centered at $\gamma_x$. We call \eqref{standardrep} the \textit{standard representation} of the particle configuration $\xi$.
\end{prop}

Thus any particle configuration on $S \times G$ may also be regarded as a locally finite point configuration on $S \times G$ where the points are allowed to have varied multiplicities. \mbox{In the following} we shall often view particle configurations in this manner if convenient. With this in mind, we define the \textit{support} of $\xi$ as
$$
[\xi]:=\{ (\gamma_x,i) \in (S\times G)\times \N : \gamma_x \in Q_\xi \text{ and }i \leq m_\xi(\gamma_x)\}
$$ which is merely the set of points that constitute $\xi$ as a point configuration counted with their respective multiplicities. If we only wish to consider the set of points in $\xi$ without regard for their multiplicities then we shall write $\langle \xi \rangle$, i.e. the projection onto $S \times G$ of $[\xi]$.

\begin{defi} A measure $\xi$ on $(S\times G, \B_{S \times G})$ is said to be of \textit{locally finite allocation} if it satisfies $\xi(\Lambda \times G) < +\infty$ for every $\Lambda \in \B^0_S$.
\end{defi}

\begin{notat}$\,$
\begin{itemize}
\item [$\bullet$] We shall write $\mathcal{N}(S\times G)$ to denote the space of all particle configurations on $S \times G$ which are of locally finite allocation.
\item [$\bullet$] Given $\Lambda \in \B^0_S$ we write $\mathcal{N}(\Lambda \times G)$ to denote the space of all particle configurations of locally finite allocation which are supported on $\Lambda \times G$.
\end{itemize}
\end{notat}

\section{The space $\mathcal{N}(S\times G)$ of particle configurations}

\subsection{Restriction and superposition of particle configurations}

\begin{defi} Given $\xi \in \mathcal{N}(S \times G)$ and $A \in \B_{S \times G}$ we define the \textit{restriction} of $\xi$ to $A$ as the particle configuration $\xi_{A}$ given for every $B \in \B_{S \times G}$ by the formula
$$
\xi_A (B) = \xi (A \cap B).
$$ Equivalently, if $\xi = \sum_{\gamma_x \in Q_\xi} m(\gamma_x) \delta_{\gamma_x}$ we define $\xi_A$ through the standard representation
$$
\xi_A = \sum_{\gamma_x \in Q_\xi \cap A} m(\gamma_x) \delta_{\gamma_x}.
$$
\end{defi}

\begin{defi} Given $\sigma,\eta \in \mathcal{N}(S \times G)$ we define their \textit{superposition} as the particle configuration $\sigma \cdot \eta$ given for every $B \in \B_{S \times G}$ by the formula
$$
\sigma \cdot \eta (B) = \sigma(B) + \eta (B).
$$
\end{defi}

\begin{obs}\label{obsiden} Given $\Lambda \in \B^0_S$ there is a natural identification between $\mathcal{N}(S \times G)$ and $\mathcal{N}(\Lambda \times G) \times \mathcal{N}(\Lambda^c \times G)$ given by the restriction and superposition operations, i.e. we have that the applications
$$
\begin{array}{ccc}
\mathcal{N}(S \times G) \overset{r}{\longrightarrow} \mathcal{N}(\Lambda \times G)\times \mathcal{N}(\Lambda^c \times G) & &\mathcal{N}(\Lambda \times G)\times \mathcal{N}(\Lambda^c \times G) \overset{s}{\longrightarrow} \mathcal{N}(S \times G)\\
\hspace{-0.35cm}\xi \longmapsto (\xi_{\Lambda \times G},\xi_{\Lambda^c \times G}) & \text{ and }& \hspace{2.5cm} (\sigma,\eta) \longmapsto \sigma\cdot \eta.
\end{array}
$$ are bijections and have each other as their respective inverse.
\end{obs}

\subsection{Measurable structure}

The space $\mathcal{N}(S \times G)$ can be endowed with a measurable space structure by considering the $\sigma$-algebra $\F$ generated by the counting events
\begin{equation}\label{salgebra}
\F = \sigma\left( \{ \xi \in \mathcal{N}(S \times G) : \xi(B) = k \} : k \in \N_0 \text{ and } B \in \B^0_{S\times G} \right).
\end{equation} Furthermore, for any $A \in \B_{S\times G}$ we define $\F_A$, the $\sigma$-\textit{algebra of events occurring in }$A$, by considering only the counting events inside $A$, i.e.
$$
\F_A = \sigma\left( \{ \xi \in \mathcal{N}(S \times G) : \xi(B) = k \} : k \in \N_0 \text{ and } B \in \B^0_A \right).
$$ Alternatively, if for every $B \in \B_{S\times G}$ we define the respective counting random variable $N_B : \mathcal{N}(S \times G) \to \N_0$ by the formula $N_B(\eta) = \eta(B)$ then for each $A \in \B_{S\times G}$ the $\sigma$-algebra $\F_A$ can also be defined as the one generated by the counting random variables inside $A$, i.e.
$$
\F_A = \sigma\left( N_B : B \in \B^0_A\right).
$$

\begin{obs} For any $\Lambda \in \B^0_s$ the identification $\mathcal{N}(S \times G)=\mathcal{N}(\Lambda \times G) \times \mathcal{N}(\Lambda^c \times G)$ on Remark \ref{obsiden} is in fact a measurable isomorphism when endowing each space with the $\sigma$-algebras $\F$ and $\F_{\Lambda \times G} \otimes \F_{\Lambda^c \times G}$, respectively.
\end{obs}

\begin{defi} $\,$
\begin{enumerate}
\item [$\bullet$] A function $f: \mathcal{N}(S\times G) \to \R$ is called a \textit{local function} if there exists $\Lambda \in \B^0_S$ such that $f$ is $\F_{\Lambda \times G}$-measurable.
\item [$\bullet$] An event $A \in \F$ is called a \textit{local event} if $\mathbbm{1}_A$ is a local function.
\item [$\bullet$] Given a function $f: \mathcal{N}(S\times G) \to \R$ we define its \textit{measurability support} as
$$
\Lambda_f = \bigcap_{\Lambda \in \mathcal{D}_f} \overline{\Lambda}
$$ where $\mathcal{D}_f = \{ \Lambda \in \B_S : f \text{ is }\F_{\Lambda \times G}\text{-measurable}\}$. That is, $\Lambda_f$ is the \mbox{smallest closed set} $\Lambda \in \B_S$ such that $f$ is $\F_{\Lambda \times G}$-measurable. Notice that if $f$ is local then $\Lambda_f \in \B^0_S$.
\end{enumerate}
\end{defi}

\begin{obs}Notice that a function $f: \mathcal{N}(S\times G) \to \R$ is $\F_{\Lambda\times G}$-measurable if and only if $f(\sigma)=f(\eta)$ whenever $\sigma,\eta \in \mathcal{N}(S \times G)$ are such that $\sigma_{\Lambda \times G}=\eta_{\Lambda \times G}$.
\end{obs}

\subsection{Topological structure}

The space $\mathcal{N}(S \times G)$ can also be endowed with a topological structure. We think of any two particle configurations $\xi,\eta$ as close to each other whenever the particles in $\xi$ lying inside some sufficiently large compact set $K$ can be matched with nearby particles of $\eta$ and viceversa. The precise definitions are given below.

\begin{defi} Given $\delta > 0$ and $\xi,\eta \in \mathcal{N}(S)$ we say that $\xi$ is $\delta$-\textit{embedded} in $\eta$ if there exists an injective application $p:[\xi] \to [\eta]$ such that $d\left( \pi_{S\times G}(x,i) , \pi_{S\times G}(p(x,i))\right) < \delta$ for each $(x,i) \in [\xi]$, where $\pi_{S\times G} : (S\times G) \times \N \to S$ is the projection onto $S \times G$ and $d=d_S + d_G$ is the metric on $S\times G$. We denote it by $\xi \preceq_{\delta} \eta$.
\end{defi}

\begin{defi} Given a particle configuration $\xi \in \mathcal{N}(S \times G)$, a compact set $K \subseteq S \times G$ and $\delta > 0$ we define the $(K,\delta)$-neighborhood of $\xi$ by the formula
$$
(\xi)_{K,\delta} = \{ \eta \in \mathcal{N}(S \times G) : \xi_K \preceq_\delta \eta \text{ and }\eta_K \preceq_\delta \xi \}.
$$
\end{defi}

\begin{defi} We define the \textit{vague topology} on $\mathcal{N}(S \times G)$ as the one generated by the basis
$$
\mathfrak{B} = \{ (\xi)_{K,\delta} : \xi \in \mathcal{N}(S \times G) , K \subseteq S \times G\text{ compact and }\delta > 0\}.
$$
\end{defi}

\begin{obs}It can be shown that $\mathcal{N}(S \times G)$ admits a metric which is consistent with the vague topology under which
it is complete and separable.
\end{obs}

\begin{obs} The $\sigma$-algebra $\F$ defined in \eqref{salgebra} is actually the Borel $\sigma$-algebra given by the vague topology on $\mathcal{N}(S \times G)$.
\end{obs}

\section{Poisson processes on $S \times G$}

We shall call any random element of $\mathcal{N}(S \times G)$ a \textit{point process }on $S \times G$. Throughout this part we shall work with many different point processes on $S \times G$, all of them related in one way or another to the Poisson point process, which we define below.

\begin{defi} Let $\nu$ be a measure on $(S \times G,\B_{S \times G})$ with locally finite allocation. \mbox{The \textit{Poisson measure with intensity}} $\nu$ is defined as the unique measure $\pi^\nu$ on $\mathcal{N}(S \times G)$ which satisfies
$$
\pi^\nu ( \{ \xi \in \mathcal{N}(S \times G) : \xi(B_i) = k_i \text{ for all }i=1,\dots,n \} ) = \prod_{i=1}^n \frac{e^{-\nu(B_i)} \left(\nu(B_i)\right)^{k_i}}{k_i!}
$$ for all $k_1,\dots,k_n \in \N_0$, disjoint $B_1,\dots,B_n \in \B^0_{S \times G}$ and $n \in \N$.
\end{defi}

\begin{defi}\label{defiPoissonasd} Let $\nu$ be a measure on $(S \times G,\B_{S \times G})$ with locally finite allocation. A point process $X$ on $S \times G$ is called a \textit{Poisson process} with intensity measure $\nu$ if it is distributed according to $\pi^\nu$, i.e. for every $n \in \N$ and all choices of disjoint sets $B_1,\dots,B_n \in \B^0_{S \times G}$ the random variables $X(B_1),\dots,X(B_n)$ are independent and have a Poisson distribution with respective means $\nu(B_1),\dots,\nu(B_n)$.
\end{defi}

It follows from Definition \ref{defiPoissonasd} that if $X$ is a \mbox{Poisson process with intensity measure $\nu$} and we consider $\Lambda \in \B^0_S$ then, conditioned on the event $\{X(\Lambda \times G) = n\}$, the locations of these $n$ particles inside $\Lambda \times G$ are independent and distributed according to $\frac{\nu}{\nu(\Lambda \times G)}$. The next proposition found in \cite[Proposition~3.1]{M2} generalizes this idea to obtain a convenient formula for the integral of functions with respect to the restricted Poisson measures.
\begin{prop} For any $\Lambda \in \B^0_S$ and any bounded nonnegative $f: \mathcal{N}(\Lambda \times G) \to \R$ we have the formula
\begin{equation}\label{poisson}
\int f(\sigma) d\pi^\nu_{\Lambda}(\sigma) = \sum_{n=0}^\infty \frac{e^{-\nu(\Lambda \times G)}}{n!} \int_{(\Lambda \times G)^n} f\left(\sum_{i=1}^n \delta_{\gamma_x^i}\right) d\nu^n(\gamma_x^1,\dots,\gamma_x^n)
\end{equation} where $\nu^n$ denotes the $n$-fold product measure of $\nu$.
\end{prop}

In our work we shall, among other things, establish limit theorems for point processes. Therefore, we shall require a notion of convergence which is appropriate for our purposes. The most familiar notion available is that of convergence in distribution.

\begin{defi}We say that a sequence $(X_n)_{n \in \N}$ of point processes on $S \times G$ converges \textit{in distribution} to a point process $X$ on $S \times G$ if
$$
\lim_{n \rightarrow +\infty}\E(f(X_n)) = \E(f(X))
$$ for every bounded continuous function $f:\mathcal{N}(S \times G) \to \R$. We denote it by $X_n \overset{d}{\longrightarrow} X_n$.
\end{defi} At some points throughout our work the use of local functions shall be much more natural than that of continuous ones. Under such circumstances we shall adopt instead the notion of local convergence for point processes, which we introduce next.

\begin{defi}\label{localconvergence} We say that a sequence $(X_n)_{n \in \N}$ of point processes on $S \times G$ converges \textit{locally} to a point process $X$ on $S \times G$ if
$$
\lim_{n \rightarrow +\infty}\E(f(X_n)) = \E(f(X))
$$ for every bounded local function $f:\mathcal{N}(S \times G) \to \R$. We denote it by $X_n \overset{loc}{\longrightarrow} X_n$.
\end{defi}

It is important to notice that in most cases local functions need not be continuous. Therefore, in general the notions of local convergence and convergence in distribution do not coincide. Nevertheless, since local functions are always dense in the space of uniformly continuous functions with the supremum norm, we get the following result.

\begin{teo}Local convergence implies convergence in distribution.
\end{teo}

\newpage

\section{Resumen del Capítulo 7}

En este primer capítulo de la segunda parte introducimos el marco teórico sobre el cuál definiremos todos los modelos que nos interesará estudiar. Todos estos modelos serán casos particulares de configuraciones de partículas aleatorias. Dado un espacio de posiciones $S$ y otro de spines $G$, una configuración de partículas $\xi$ en $S \times G$ es una medida sobre $S \times G$ que admite la representación
$$
\xi = \sum_{(x,\gamma) \in Q_\xi} m_\xi (x,\gamma) \delta_{(x,\gamma)}
$$ para cierto conjunto numerable $Q_\xi \subseteq S \times G$ sin puntos de acumulación y una función $m_\xi : Q_\xi \rightarrow \N$. Denotamos por $\mathcal{N}(S \times G)$ al espacio de aquellas configuraciones de partículas $\xi$ en $S \times G$ que satisfacen que $Q_\xi \cap (\Lambda \times G)$ es finito para todo subconjunto acotado $\Lambda$ de $S$.

El espacio $\mathcal{N}(S \times G)$ tiene una estructura de espacio medible bajo la $\sigma$-álgebra generada por los eventos de conteo, i.e.
$$
\F = \sigma\left( \{ \xi \in \mathcal{N}(S \times G) : \xi(B) = k \} : k \in \N_0 \text{ y } B \subseteq S \text{ boreliano acotado} \right).
$$ También existe una topología natural en este espacio, la topología vaga, que es la generada por la base
$$
\mathfrak{B} = \{ (\xi)_{K,\delta} : \xi \in \mathcal{N}(S \times G) , K \subseteq S \times G\text{ compacto y }\delta > 0\},
$$ donde el entorno $(\xi)_{K,\delta}$ viene dado por
$$
(\xi)_{K,\delta} = \{ \eta \in \mathcal{N}(S \times G) : \xi_K \preceq_\delta \eta \text{ and }\eta_K \preceq_\delta \xi \}.
$$ y cuando, dadas dos configuraciones de partículas $\xi$ y $\eta$, por $\xi_K \preceq_\delta \eta$ entendemos que existe una correspondencia inyectiva (teniendo en cuenta la multiplicidad) entre las partículas de $\xi$ dentro de $K$ y las de $\eta$ tal que las partículas correspondidas están a distancia menor que $\delta$ entre sí. Puede verse que esta topología es metrizable y que así $\mathcal{N}(S \times G)$ resulta completo y separable.

Provistos de este marco teórico, podemos definir el más básico de los modelos de interés que es el Proceso de Poisson. Dada una medida $\nu$ en $S \times G$ con $\nu(\Lambda \times G) < +\infty$ para todo $\Lambda \subseteq S$ acotado, decimos que una configuración de partículas aleatoria $X$ es un proceso de Poisson en $S \times G$
si
\begin{enumerate}
\item Para cada $B \subseteq S \times G$ boreliano la variable aleatoria $X(B)$ tiene distribución Poisson de parámetro $\nu(B)$
\item Si $B_1,\dots,B_n$ son borelianos disjuntos de $S \times G$ entonces las variables aleatorias $X(B_1),\dots,X(B_n)$ son independientes.
\end{enumerate} Todos los demás modelos que estudiemos en esta segunda parte se podrán obtener, de una manera u otra, a partir de un proceso de Poisson adecuado.

Por último, introducimos las nociones de convergencia en distribución y local para configuraciones de partículas aleatorias en $S \times G$. Decimos que una sucesión $(X_n)_{n \in \N}$ de configuraciones de partículas aleatorias converge a otra $X$ en distribución si
$$
\lim_{n \rightarrow +\infty}\E(f(X_n)) = \E(f(X))
$$ para toda función $f:\mathcal{N}(S \times G) \to \R$ continua y acotada, mientras que decimos que lo hace localmente si vale lo anterior para toda función acotada local (i.e., que depende del estado de la configuración sólo dentro de una región acotada) en lugar de continua. Puede verificarse que la convergencia local implica la convergencia en distribución.

\chapter{Diluted models}\label{him}

\section{Definition of a diluted model}

In this section we formally define the models which we shall study throughout this work.
Diluted models on $\mathcal{N}(S \times G)$ are always defined by specifying two characteristic elements: a measure $\nu$ on $S\times G$ called the \textit{intensity measure} and a family $H$ of measurable functions
$$
H_{\Lambda} : \mathcal{N}(\Lambda \times G) \times \mathcal{N}(S\times G) \to [0,+\infty]
$$ called the \textit{local Hamiltonians}, both satisfying the conditions on Assumptions \ref{assump} below. Essentially, the measure $\nu$ will be responsible for the way in which particles in $G$ are distributed throughout the location space $S$ while the Hamiltonians $H_{\Lambda}$ will determine how these particles interact among themselves. Also, the family $H$ shall be referred to as the \textit{Hamiltonian}.

\begin{notat}We establish the following applications:
\begin{enumerate}
\item [$\bullet$] For $\Lambda \in \B^0_S$ and $\eta \in \mathcal{N}(S\times G)$ we define $H_{\Lambda|\eta}: \mathcal{N}(\Lambda \times G) \to [0,+\infty]$ by the formula
$$
H_{\Lambda|\eta} = H_{\Lambda}( \cdot , \eta).
$$ $H_{\Lambda|\eta}$ shall be called the \textit{local Hamiltonian on $\Lambda$ with boundary condition $\eta$}.
\item [$\bullet$] For $\eta \in \mathcal{N}(S\times G)$ we define $\Delta E_{\eta} : S \times G \to [-\infty,+\infty]$ by the formula
$$
\Delta E_{\eta} (\gamma_x) = H_{\{x\}|\eta}( \eta_{\{x\}} + \delta_{\gamma_x} ) - H_{\{x\}|\eta}(\eta_{\{x\}}).\footnote{Here we adopt the convention $\infty - \infty = \infty$.}
$$ $\Delta E_{\eta}$ shall be called the \textit{energy leap function with base configuration $\eta$}. It represents the energy cost for the model to add the particle $\gamma_x$ to its current configuration whenever the latter is given by $\eta$.
\end{enumerate}
\end{notat}

\begin{assump} \label{assump} We assume that the pair $(\nu,H)$ satisfies the following:
\begin{enumerate}
\item \textit{Locally finite allocation}. For every $\Lambda \in \mathcal{B}^0_S$ the measure $\nu$ satisfies $\nu(\Lambda \times G) < +\infty$.
\item \textit{Diluteness condition}. $H_{\Lambda|\eta}( \emptyset ) = 0$ for every $\Lambda \in \B^0_S$ and $\eta \in \mathcal{N}(S \times G)$.
\item \textit{Bounded energy loss}.
$$
-\infty < \Delta E := \inf_{\substack{ \eta \in \mathcal{N}(S \times G) \\ \gamma_x \in S \times G }} \Delta E_{\eta} (\gamma_x) < +\infty.
$$

\item \textit{Integrable local interaction range}. If we define a relation $\rightharpoonup$ on $S\times G$ by setting
    $$
    \tilde{\gamma}_y \rightharpoonup \gamma_x \Longleftrightarrow \exists\,\, \eta \in \mathcal{N}(S\times G) \text{ with }\Delta E_{\eta}(\gamma_x) \neq \Delta E_{\eta + \delta_{\tilde{\gamma}_y}} (\gamma_x)
    $$ then for every $B \in \B_{S \times G}$ the \textit{interaction range} of $B$ defined as the set
$$
I(B)= \{ \tilde{\gamma}_y \in S \times G : \exists\,\, \gamma_x \in B \text{ such that } \tilde{\gamma}_y \rightharpoonup \gamma_x \}
$$ is measurable and each $\Lambda \in \mathcal{B}^0_S$ satisfies $\nu(I(\Lambda \times G)) < +\infty$.

\item \textit{Consistent Hamiltonian}.
\begin{enumerate}
\item [i.] Given $\Delta, \Lambda \in \B^0_S$ such that $\Delta \subseteq \Lambda$ and any $\eta \in \mathcal{N}(S\times G)$
$$
H_{\Lambda|\eta}(\sigma) = H_{\Delta|\sigma_{(\Lambda - \Delta) \times G} \cdot \eta_{\Lambda^c \times G}}(\sigma_{\Delta \times G}) + H_{\Lambda - \Delta| \emptyset_{\Lambda \times G}\cdot \eta_{\Lambda^c \times G}}(\sigma_{(\Lambda - \Delta) \times G})
$$ for every $\sigma \in \mathcal{N}(\Lambda \times G)$.
\item [ii.] For every $\Lambda \in \B^0_S$, $\eta \in \mathcal{N}(S \times G)$ and $\gamma_x \in S \times G$
$$
H_{\Lambda|\eta}(\sigma + \delta_{\gamma_x}) = H_{\Lambda|\eta}(\sigma) + \Delta E_{\sigma_{\Lambda \times G} \cdot \eta_{\Lambda^c \times G}} (\gamma_x).
$$
\end{enumerate}
\item \textit{Interaction measurability of the Hamiltonian}.
\begin{enumerate}
\item [i.] For every $\Lambda \in \B^0_S$ the application $H_\Lambda$ is $(\F_{\Lambda \times G} \otimes \F_{(\Lambda^c \times G) \cap I(\Lambda \times G)})$-measurable.
\item [ii.] For every $\gamma_x \in S \times G$ the application $\eta \mapsto \Delta E_{\eta} (\gamma_x)$ is $\F_{I(\{\gamma_x\})}$-measurable.
\end{enumerate}

\end{enumerate}
\end{assump}

In what follows we shall refer to the different diluted models by the pair $(\nu,H)$ which defines them. We shall say that a given diluted model is of \textit{bounded local interaction range} whenever $I(\Lambda \times G)$ is bounded for every $\Lambda \in \mathcal{B}^0_S$. Also, whenever $\tilde{\gamma}_y \rightharpoonup \gamma_x$ we shall say that $\tilde{\gamma}_y$ has an \mbox{\textit{impact} on $\gamma_x$}. Notice that this impact relation $\rightharpoonup$ need not be symmetric.

\begin{defi}\label{defibgd} Given $\Lambda \in \mathcal{B}^0_S$ and a particle configuration $\eta \in \mathcal{N}( S \times G)$ we define the corresponding \textit{Boltzmann-Gibbs distribution} $\mu_{\Lambda|\eta}$ by the formula
\begin{equation}\label{Gibbs1}
\mu_{\Lambda|\eta} = \omega_{\Lambda|\eta} \times \delta_{\eta_{\Lambda^c}}
\end{equation} where we identify $\mathcal{N}(S \times G)$ with $\mathcal{N}( \Lambda \times G) \times \mathcal{N}(\Lambda^c \times G)$ and $\omega^\eta_\Lambda$ is the probability measure on $\mathcal{N}(\Lambda \times G)$ defined through the relation
$$
d\omega_{\Lambda|\eta} = \frac{e^{-H_{\Lambda|\eta}}}{Z_{\Lambda|\eta}} d\pi^\nu_\Lambda
$$ with $\pi^\nu_\Lambda$ denoting the Poisson measure on $\mathcal{N}(\Lambda \times G)$ of intensity measure $\nu_{\Lambda \times G}$ and
$$
Z_{\Lambda|\eta}=\displaystyle{\int_{\mathcal{N}(\Lambda \times G)} e^{-H_{\Lambda|\eta}(\sigma)} d\pi^\nu_\Lambda(\sigma)}
$$ serving as a normalizing constant. Notice that due to Assumptions \ref{assump} we have
$$
Z_{\Lambda|\eta} \geq \pi^\nu ( N_{\Lambda \times G} = 0 ) = e^{- \nu(\Lambda \times G) } > 0
$$ and thus $\omega_{\Lambda|\eta}$ is well defined.
\end{defi}

\begin{obs}\label{tradition} We would like to point out that for the discrete setting, i.e. when $S=\Z^d$, this is not the standard way in which most lattice systems are defined. Traditionally, discrete systems are defined on the configuration space $G^{\Z^d}$ for a given spin set $G$, so that in each configuration all sites in the lattice are assigned exactly one spin. In this context, a model is regarded as diluted whenever the element $0$ belongs to the spin set $G$ and a site with $0$-spin is understood as an empty site, i.e. devoid of any particles. Furthermore, for any given $\Lambda \in \B^0_{\Z^d}$ and $\eta \in G^{\Z^d}$, the corresponding Boltzmann-Gibbs distribution $\mu_\Lambda^\eta$ (notice the difference in notation) in this context is usually defined by the formula
$$
\mu_{\Lambda}^\eta(\sigma) = \frac{\mathbbm{1}_{\{\sigma_{\Lambda^c} \equiv \eta_{\Lambda^c}\}}}{Z_{\Lambda}^\eta} e^{- \beta \sum_{B \subseteq \Z^d : B \cap \Lambda \neq \emptyset} \Phi_B(\sigma)}
$$ where $\beta > 0$ is a parameter known as the \textit{inverse temperature}, $\sigma_{\Lambda^c}$ indicates the restriction of $\sigma$ to the region $\Lambda^c$ and also for every $B \subseteq \Z^d$ the function $\Phi_B : G^{\Z^d} \rightarrow \overline{\R}$ depends only on the values of spins inside $B$. The family $\Phi = (\Phi_B)_{B \subseteq \Z^d}$ is known as the \textit{potential}. Nonetheless, we prefer to adopt the definition given in Definition \ref{defibgd} as it will allow us to treat discrete and continuum systems in the same manner. This will be necessary for Chapter 11, where we study the convergence of discrete systems towards continuum ones.
\end{obs}

The Boltzmann-Gibbs distribution $\mu_{\Lambda|\eta}$ is meant to describe the local behavior of the model inside the bounded set $\Lambda$ once the configuration outside $\Lambda$ has been fixed as $\eta$. It then seems natural to expect Boltzmann-Gibbs distributions to exhibit some sort of consistency among themselves. This is indeed true, as our next proposition shows.

\begin{prop} For $\Delta \subseteq \Lambda \in \mathcal{B}^0_S$ and $\eta \in \mathcal{N}(S \times G)$ we have the following consistency property:
\begin{equation}\label{consistencia}
\mu_{\Lambda|\eta}(A) = \int_{\mathcal{N}(S\times G)} \mu_{\Delta|\xi} (A) d\mu_{\Lambda|\eta} (\xi)
\end{equation}for every $A \in \F$.
\end{prop}

\begin{proof} Notice that if we identify $\mathcal{N}(S \times G)$ with $\mathcal{N}(\Delta \times G) \times \mathcal{N}((\Lambda - \Delta)\times G) \times \mathcal{N}(\Lambda^c \times G)$ and for every $\xi \in \mathcal{N}(S \times G)$ we write $\xi = (\xi^{(1)},\xi^{(2)},\xi^{(3)})$ accordingly then, since we have $\pi^\nu = \pi^\nu_{\Delta} \times \pi^\nu_{\Lambda -\Delta} \times \pi^\nu_{\Lambda^c}$, using the Fubini-Tonelli theorem and (i) in the the consistent Hamiltonian property we obtain
\begin{align*}
\mu_{\Lambda|\eta}(A) &= \frac{1}{Z_{\Lambda|\eta}} \int_{\mathcal{N}\left( (\Lambda - \Delta) \times G \right)} \int_{\mathcal{N}(\Delta \times G)} e^{-H_{\Lambda|\eta}(\xi^{(1)}\cdot\xi^{(2)})} \mathbbm{1}_{\{(\xi^{(1)},\xi^{(2)},\eta^{(3)}) \in A\}}d\pi^\nu_{\Delta}(\xi^{(1)})d\pi^\nu_{\Lambda - \Delta}(\xi^{(2)})\\
\\
& = \frac{1}{Z_{\Lambda|\eta}} \int_{\mathcal{N}\left( (\Lambda - \Delta) \times G \right)} Z_{\Delta|\xi^{(2)} \cdot \eta^{(3)}} \mu_{\Delta|{\xi^{(2)} \cdot \eta^{(3)}}} (A) e^{-H_{\Lambda - \Delta| \eta^{(3)}}(\xi^{(2)})}d\pi^\nu_{\Lambda - \Delta}(\xi^{(2)})\\
\\
& = \frac{1}{Z_{\Lambda|\eta}} \int_{\mathcal{N}\left( (\Lambda - \Delta) \times G \right)} \int_{\mathcal{N}(\Delta \times G)} e^{-H_{\Lambda|\eta}(\xi^{(1)}\cdot\xi^{(2)})} \mu_{\Delta|\xi^{(1)}\cdot \xi^{(2)} \cdot \eta^{(3)}} (A) d\pi^\nu_{\Delta}(\xi^{(1)}) d\pi^\nu_{\Lambda - \Delta}(\xi^{(2)})\\
\\
& = \int_{\mathcal{N}(S\times G)} \mu_{\Delta|\xi} (A) d\mu_{\Lambda|\eta} (\xi)
\end{align*}for every $A \in \F$ which concludes the proof.
\end{proof}

\begin{defi}\label{Gibbs2}
A probability measure $\mu$ on $\mathcal{N}( S \times G)$ is called a \textit{Gibbs measure} for the diluted model with intensity measure $\nu$ and Hamiltonian family $H$ if
$$
\mu(A) = \int_{\mathcal{N}( S \times G)} \mu_{\Lambda|\eta}(A) \,d\mu(\eta)
$$for every $\Lambda \in \mathcal{B}^0_S$ and $A \in \F$.
\end{defi}

Notice that by Definition \ref{Gibbs2} a probability measure $\mu$ is a Gibbs measure if and only if it is consistent, in the sense of equation \eqref{consistencia}, with every Boltzmann-Gibbs distribution, each of which describes the local equilibrium state of the model inside some bounded region. Hence, we may think of Gibbs measures as those representing the global equilibrium states of our model. For this reason they are sometimes referred to as \textit{infinite-volume \mbox{Boltzmann-Gibbs distributions}}. The next proposition validates this choice of terminology.

\begin{prop}\label{limitegibbs} Let $(\nu,H)$ be a diluted model of bounded local interaction range and $\mu$ be a probability measure on $\mathcal{N}(S \times G)$ such that
$$
\mu_{\Lambda|\eta} \overset{loc}{\longrightarrow} \mu
$$ as $\Lambda \nearrow S$ for some $\eta \in \mathcal{N}(S \times G)$. Then $\mu$ is a Gibbs measure for the model $(\nu,H)$.
\end{prop}

\begin{proof} Let us first notice that, by the interaction measurability of $H$ in \mbox{Assumptions \ref{assump},} for any $\Delta \in \mathcal{B}^0_S$ and $\xi \in \mathcal{N}(S\times G)$ we have
$$
H_{\Delta|\xi}=H_{\Delta|\xi_{I(\Delta \times G)}}.
$$ From this we obtain that for any local event $A \in \F$ the mapping
$$
(\sigma,\xi) \mapsto \mathbbm{1}_{\{\sigma \cdot \xi_{\Delta^c \times G} \in A\}}e^{-H_{\Delta|\xi}(\sigma)}
$$ is $\left(\F_{\Delta \times G} \otimes \F_{(\Lambda_A \times G) \cup I(\Delta \times G)}\right)$-measurable which implies that the mapping $\xi \mapsto \mu_{\Delta|\xi}(A)$ is $\F_{(\Lambda_A \times G) \cup I(\Delta \times G)}$-measurable. Since the model is of bounded local interaction range we get that for any local event $A \in \F$ the mapping $\xi \mapsto \mu_{\Delta|\xi}(A)$ is local as well. Therefore, $\mu_{\Lambda_n|\eta} \overset{loc}{\longrightarrow} \mu$ implies that for any local event $A$ we have
$$
\int \mu_{\Delta|\xi} (A)\,d\mu(\xi)= \lim_{\Lambda \nearrow S} \int \mu_{\Delta|\xi} (A)\,d\mu_{\Lambda|\eta} (\xi)= \lim_{\Lambda \nearrow S} \mu_{\Lambda|\eta}  (A)= \mu(A)
$$
where the second equality follows from the consistency property of the Boltzmann-Gibbs distributions and the fact that $\Lambda \nearrow S$. Since the class of local events is closed under intersections and it generates $\F$, we see that
$$
\mu(A) = \int_\Omega \mu_{\Delta|\xi}(A) \,d\mu(\xi)
$$ for any $A \in \F$ and $\Delta \in \mathcal{B}^0_S$ which allows us to conclude $\mu$ is a Gibbs measure.
\end{proof}

Standard compactness arguments imply that Gibbs measures for traditional lattice systems (see Remark \ref{tradition}) always exist. In the continuum setting the situation is much more delicate, although one can show that, under some reasonable additional assumptions on the pair $(\nu,H)$ (known as the almost sure Feller property), every diluted model of bounded interaction range admits at least one Gibbs measure. Whenever a diluted model admits more than one Gibbs measure we say that the model exhibits a \textit{phase transition}. In this second part of the thesis we shall be specifically interested in studying properties of Gibbs measures for diluted models in general, as well as developing tools to establish the occurrence (or absence) of phase transitions. For the latter, the property detailed on the following proposition will play an important role.

\begin{defi} Let $H$ be a Hamiltonian inducing a measurable local interaction range, i.e. $I(B)$ is measurable for every $B \in \mathcal{B}_{S \times G}$. A particle configuration $\sigma \in \mathcal{N}(S\times G)$ is of \textit{finite local interaction range} with respect to $H$ if $\sigma( I(\Lambda \times G) )< +\infty$ for every $\Lambda \in \mathcal{B}^0_{S}$. We write $\mathcal{N}_H(S \times G)$ to denote the space of all particle configurations which are of finite local \mbox{interaction range} with respect to $H$.
\end{defi}

\begin{prop}\label{lfir} If the pair $(\nu,H)$ satisfies Assumptions \ref{assump} then
\begin{enumerate}
\item [i.] $\mathcal{N}_H(S \times G)$ is a measurable subset of $\mathcal{N}(S \times G )$.
\item [ii.] Every Gibbs measure of the diluted model $(\nu,H)$ is supported on $\mathcal{N}_H(S \times G)$.
\end{enumerate}
\end{prop}

\begin{proof} If $(\Lambda_n)_{n \in \N} \subseteq \mathcal{B}^0_{S}$ is such that $\Lambda_n \nearrow S$ then we can write
$$
\mathcal{N}_H(S \times G) = \bigcap_{n \in \N} \{ \sigma \in \mathcal{N}(S\times G) : \sigma ( I(\Lambda_n \times G) ) < + \infty \}.
$$ Since for each $n \in \N$ the set $I(\Lambda_n \times G)$ is measurable by Assumptions \ref{assump} then we obtain that the sets $\{ \sigma \in \mathcal{N}(S\times G) : \sigma ( I(\Lambda_n \times G) ) < + \infty \}$ are also measurable for every $n \in \N$ and (i) follows at once from this. To establish (ii) it suffices to show that for each $n \in \N$
\begin{equation}\label{lfirsup}
\int_{\mathcal{N}(S \times G)} \sigma(I(\Lambda_n \times G)) d\mu(\sigma) = \sup_{k \in \N} \int_{\mathcal{N}(S \times G)} \sigma(I(\Lambda_n \times G) \cap \Lambda_k) d\mu(\sigma) < +\infty.
\end{equation}
But since $\mu$ is a Gibbs measure for each $n,k \in \N$ we have
\begin{align*}
\int \sigma(I(\Lambda_n \times G) \cap \Lambda_k) d\mu(\sigma)&= \int \left[\int \sigma(I(\Lambda_n \times G) \cap \Lambda_k) d\mu_{I(\Lambda_n \times G) \cap \Lambda_k | \xi}(\sigma) \right] d\mu(\xi)\\
\\
&\leq \int \left[ \int \frac{\sigma(I(\Lambda_n \times G) \cap \Lambda_k)}{Z_{I(\Lambda_n \times G) \cap \Lambda_k|\xi}} d \pi^\nu_{I(\Lambda_n \times G) \cap \Lambda_k} (\sigma) \right] d\mu(\xi)\\
\\
&\leq \frac{\displaystyle{\int \sigma(I(\Lambda_n \times G)) d\pi^\nu(\sigma)}}{ \pi^\nu ( N_{I(\Lambda_n \times G)}=0)} = e^{\nu(I(\Lambda_n \times G))}\nu( I(\Lambda_n \times G))
\end{align*} which, by Assumptions \ref{assump}, establishes \eqref{lfirsup} and thus concludes the proof.
\end{proof}

\section{Some examples of diluted models}\label{examples}

\subsection{The Widom-Rowlinson model}

The Widom-Rowlinson model is a classical hardcore interaction model in which particles located on $\R^d$ for $d \geq 2$ may be of two different types, $(+)$-particles and $(-)$-particles, and any two particles of different type are forbidden to become within a certain distance $r > 0$ of each other. In the present context of diluted models, the \mbox{Widom-Rowlinson model} is defined as the diluted model on $\mathcal{N}(\R^d \times \{+,-\})$ specified by
\begin{enumerate}
\item [$\bullet$] The intensity measure $\nu^{\lambda_+,\lambda_-}$ defined as
$$
\nu^{\lambda_+,\lambda_-} = \left( \lambda_+ \mathcal{L}^d \times \delta_+ \right) +  \left(\lambda_- \mathcal{L}^d \times \delta_-\right)
$$ where $\lambda_+,\lambda_- > 0$ are two fixed parameters known as the \textit{fugacities} of $(+)$-particles and $(-)$-particles respectively and $\mathcal{L}^d$ denotes the Lebesgue measure on $\R^d$.
\item [$\bullet$] The Hamiltonian $H$ given for each $\Lambda \in \B^0_{\R^d}$ and $\eta \in \mathcal{N}(\R^d \times \{+,-\})$ by the formula
$$
H_{\Lambda|\eta}(\sigma)= \sum_{(\gamma_x ,\tilde{\gamma}_y) \in e_{\Lambda}(\sigma|\eta)} U( \gamma_x , \tilde{\gamma}_y )
$$ where
\begin{equation}\label{wru}
U(\gamma_x,\tilde{\gamma}_y) := \left\{ \begin{array}{ll} +\infty &\text{if }\gamma \neq \tilde{\gamma}\text{ and }\|x-y\|_\infty \leq r\\ 0 &\text{otherwise}\end{array}\right.
\end{equation} and
$$
e_{\Lambda}(\sigma|\eta) := \{ (\gamma_x ,\tilde{\gamma}_y) \in \langle \sigma \cdot \eta_{\Lambda^c \times G} \rangle^2 : x \in \Lambda \}.
$$
\end{enumerate}
Thus, in this model the measure $\omega_{\Lambda|\eta}$ in \eqref{Gibbs1} becomes the distribution of a superposition of two independent homogeneous Poisson processes of respective intensities $\lambda_+$ and $\lambda_-$ conditioned on the event that no particle inside $\Lambda$ has a particle of the opposite type (including also particles in $\eta_{\Lambda^c \times \{+,-\}}$) at a distance smaller than $r$ from them.

There also exists a discrete version of the Widom-Rowlinson model, first introduced by Lebowitz and Gallavoti in \cite{L}. In the traditional setting, this model is defined on the configuration space $\{+,0,-\}^{\Z^d}$ through the Boltzmann-Gibbs distributions given for each $\Lambda \in \B^0_{\Z^d}$ and $\eta \in \{+,0,-\}^{\Z^d}$ by the formula
\begin{equation}\label{wrdbgd}
\mu_{\Lambda}^\eta(\sigma) = \frac{\mathbbm{1}_{\{\sigma_{\Lambda^c} \equiv \eta_{\Lambda^c}\}}}{Z_{\Lambda}^\eta} e^{-\sum_{B : B \cap \Lambda \neq \emptyset} \Phi_B(\sigma)}
\end{equation} where for each $B \subseteq \Z^d$ the interaction $\Phi_B$ is given by
\begin{equation}\label{wrdht}
\Phi_B(\sigma) = \left\{ \begin{array}{ll} (+\infty) \mathbbm{1}_{\{ \sigma(x) \times \sigma(y)= - \}} & \text{ if $B=\{x,y\}$ with $\|x-y\|_\infty \leq r$} \\ \\ - (\mathbbm{1}_{\{\sigma(x)= +\}} \log \lambda_+  + \mathbbm{1}_{\{\sigma(x)= -\}} \log \lambda_- ) & \text{ if $B=\{x\}$} \\ \\ 0 & \text{ otherwise.}\end{array}\right.
\end{equation}
Notice that the inverse temperature $\beta$ is missing in \eqref{wrdbgd}: since pair interactions are either $0$ or $+\infty$, it is customary to set $\beta = 1$ and vary only the fugacity parameters in the model.

\newpage
In the present setting of diluted models, this discrete version of the Widom-Rowlinson model is defined as the diluted model on $\mathcal{N}(\Z^d \times \{+,-\})$ specified by
\begin{enumerate}
\item [$\bullet$] The intensity measure $\nu^{\lambda_+,\lambda_-}$ defined as
$$
\nu^{\lambda_+,\lambda_-} = \left( \lambda_+ c^d \times \delta_+ \right) +  \left(\lambda_- c^d \times \delta_-\right)
$$ where $c^d$ denotes the counting measure on $\Z^d$.
\item [$\bullet$] The Hamiltonian $H$ given for each $\Lambda \in \B^0_{\Z^d}$ and $\eta \in \mathcal{N}(\Z^d \times \{+,-\})$ by the formula
\begin{equation}\label{wrhd}
H_{\Lambda|\eta}(\sigma)= \sum_{(\gamma_x ,\tilde{\gamma}_y) \in e_{\Lambda}(\sigma|\eta)} U( \gamma_x , \tilde{\gamma}_y ) + \sum_{x \in \Lambda} V_x(\sigma)
\end{equation} where the pair interaction $U$ is the same as in \eqref{wru} and for each $x \in \Lambda$ we set
$$
V_x (\sigma):= \left\{ \begin{array}{ll} +\infty &\text{if } \sigma( \{x\} \times \{+,-\} ) > 1 \\ 0 &\text{otherwise.}\end{array}\right.
$$
\end{enumerate}
The term $V$ is introduced to allow at most one \mbox{particle per site} as in the traditional setting. \mbox{Another possible (and perhaps more natural) way} in which to define this discrete version within the setting of diluted models is to leave the term $V$ out of \eqref{wrhd} and then avoid the possibility of multiple particles of the same type per site by considering the projected Boltzmann-Gibbs distributions $\langle \mu_{\Lambda|\eta} \rangle$ instead, i.e. by considering particle configurations without any regard for their respective particle multiplicities. Both possibilities are indeed equivalent, but the latter introduces a change in fugacities: for a choice of fugacities $\lambda_{\pm}$ in the second alternative, one recovers the traditional discrete Widom-Rowlinson model with fugacities $e^{\lambda_\pm} - 1$. For this reason, unless explicitly stated otherwise, we shall always work with the first of these alternatives. We refer the reader to \cite{GHM} for a review of the general results known
for these models.

\subsection{The Widom-Rowlinson model with generalized interactions}

There are several interesting generalizations of the Widom-Rowlinson model. One possible generalization which has been well studied is to consider a model in which nearby pairs of particles of the opposite type are not necessarily forbidden, but merely \mbox{discouraged instead.} More precisely, given a decreasing function $h: \R^+ \rightarrow [0,+\infty]$ with bounded support we define the Widom-Rowlinson model
with \textit{interspecies repulsion function} $h$ by replacing the previous pair interaction $U$ in \eqref{wru} by
$$
U_h(\gamma_x,\tilde{\gamma}_y) := h(|x-y|)\mathbbm{1}_{\{\gamma \neq \tilde{\gamma}\}} = \left\{ \begin{array}{ll} h(\|x-y\|_\infty) &\text{if }\gamma \neq \tilde{\gamma}\\ 0 &\text{otherwise.}\end{array}\right.
$$ We may also add a type-independent repulsion by taking a second decreasing function $j: \R^+ \rightarrow [0,+\infty]$ with bounded support and redefining the Hamiltonian as
$$
H_{\Lambda|\eta}(\sigma)= \sum_{\{\gamma_x ,\tilde{\gamma}_y\} \in e^u_{\Lambda}(\sigma|\eta)} h(\|x-y\|_\infty)\mathbbm{1}_{\{\gamma \neq \tilde{\gamma}\}} + j(\|x-y\|_\infty)
$$ where
$$
e^u_{\Lambda}(\sigma|\eta) = \{ \{\gamma_x, \tilde{\gamma}_y\} \subseteq \langle \sigma \cdot \eta_{\Lambda^c \times G} \rangle : \gamma_x \neq \tilde{\gamma}_y, \{x,y\} \cap \Lambda \neq \emptyset \}.
$$
Notice that for $h:= (+\infty) \mathbbm{1}_{[0,r]}$ and $j \equiv 0$ we obtain the original Widom-Rowlinson model.
We refer the reader to \cite{GH}, where these type of generalizations were investigated.

\subsection{The thin rods model}

Another possible generalization of the Widom-Rowlinson model is to consider $q \geq 3$ types of particles instead of just two. Furthermore, one could have different \textit{exclusion radii} $r_{ij}$ for the different pairs of types of particles $1 \leq i < j \leq q$. This asymmetric generalization of the original model was studied in \cite{BKL}. We introduce here a different asymmetric variant which is also featured on the former reference. Given $q \geq 3$ and $l> 0$ we consider in $\R^2$ a system consisting of rods of length $2l$ and zero width positioned anywhere throughout the plane. We assume that these rods may possess $q$ different orientations specified by some fixed angles $0 \leq \theta_1 < \dots < \theta_q < \pi$ measured with respect to the $x$-axis. Finally, the interaction between these rods is that no two rods are allowed to intersect each other. More precisely, if we set
$$
L_i = \{ t \cdot (\cos\theta_i, \sin \theta_i) : t \in [-l,l] \}
$$ then the thin rods model is defined as the diluted model on $\mathcal{N}(\R^2 \times \{1,\dots,q\})$ \mbox{specified by}
\begin{enumerate}
\item [$\bullet$] The intensity measure
$$
\nu := \sum_{i=1}^q \lambda_i \mathcal{L}^d \times \delta_{\theta_{i}}
$$ where for $1 \leq i \leq q$ the parameter $\lambda_i > 0$ is the fugacity of rods of orientation $\theta_i$.
\item [$\bullet$] The Hamiltonian
$$
H_{\Lambda|\eta}(\sigma) := \sum_{(\gamma_x ,\tilde{\gamma}_y) \in e_{\Lambda}(\sigma|\eta)} U( \gamma_x , \tilde{\gamma}_y )
$$ where
\begin{equation}\label{tru}
U(\gamma_x,\tilde{\gamma}_y) := \left\{ \begin{array}{ll} +\infty &\text{if }(L_\gamma + x) \cap (L_{\tilde{\gamma}} + y) \neq \emptyset\\ 0 &\text{otherwise.}\end{array}\right.
\end{equation}
\end{enumerate}
In general, given a probability measure $\rho$ on $S^1_*:=[0,\pi)$ we define the thin rods model with fugacity $\lambda$, rod length $2l$ and orientation measure $\rho$ as the diluted model on $\mathcal{N}(\R^2 \times S^1_*)$ specified by the intensity measure $\nu^\lambda := \lambda \mathcal{L}^2 \times \rho$ and the Hamiltonian $H$ given by \eqref{tru}. Notice that this broader definition allows for an infinite number of possible orientations depending on the measure $\rho$. Also, notice that the original model with $q$ orientations is recovered by setting $\lambda = \lambda_1 + \dots + \lambda_q$ and $\rho = \frac{1}{\lambda} \sum_{i=1}^q \lambda_i \delta_{\theta_{i}}$.

\subsection{The tolerant Widom-Rowlinson model}

Yet another variant to consider is the tolerant Widom-Rowlinson model, in which particles can tolerate up to $k \in \N$ particles of the opposite type within a distance $r > 0$ from them. In this case, the intensity measure remains unchanged while the Hamiltonian $H$ becomes
$$
H_{\Lambda|\eta}(\sigma)= \sum_{(\gamma^1_{x_1} , \dots, \gamma^{k+1}_{x_n}) \in e_{\Lambda}^k(\sigma|\eta)} U(\gamma^1_{x_1} , \dots, \gamma^{k+1}_{x_n})
$$ where
$$
U(\gamma^1_{x_1},\dots, \gamma^{k+1}_{x_n}) := \left\{ \begin{array}{ll} +\infty &\text{if }\gamma^1 \neq \gamma^2 = \dots = \gamma^{k+1}\text{ and } \max_{j=2,\dots,k+1} \|x_1-x_j\|_\infty \leq r \\ 0 & \text{otherwise}\end{array}\right.
$$ and
$$
e_{\Lambda}^k(\sigma|\eta) := \{ (\gamma^1_{x_1} , \dots, \gamma^{k+1}_{x_n}) \in \langle \sigma \cdot \eta_{\Lambda^c \times G} \rangle^{k+1} : x_1 \in \Lambda \}.
$$ Observe that, unlike all previous models, the interactions featured here are not pairwise. This fact is of interest since most continuum models featured in the literature present only pairwise interactions. A discrete analogue of this model is also available just as it was in the original setting. \mbox{Its specification} can be deduced from its continuum counterpart in the same way as before, so we omit it here.

\subsection{The symbiotic model}

This is an example of a model in which the interactions involved are of attractive type instead of repulsive. It features particles of two types, \textit{hosts} and \textit{parasites}, which interact in the following way: the hosts spread freely throughout $\R^d$ without any care for the location of parasites, whereas the parasites prefer to locate themselves near the hosts. More precisely, the symbiotic model is defined as the diluted model on $\mathcal{N}(\R^d \times \{h,p\})$ with intensity measure
$$
\nu^{\lambda_h,\lambda_p} = (\lambda_h \mathcal{L}^d \times \delta_{h}) + (\lambda_p \mathcal{L}^d \times \delta_{p})
$$ and Hamiltonian given for each $\Lambda \in \B^0_{\R^d}$ and $\eta \in \mathcal{N}(\R^d \times \{h,p\})$ by the formula
$$
H_{\Lambda|\eta}(\sigma) = \sum_{p_x \in \sigma \,:\, x \in \Lambda} J\mathbbm{1}_{\{ \sigma( B^*(x,r) \times \{h\}) = 0 \}}
$$ where $\lambda_{h},$ $\lambda_p$, $J$ and $r$ are all positive constants, and for $x \in \R^d$ we set
$$
B^*(x,r):=\{ y \in \R^d : 0 < \|x - y \|_2 < r\}.
$$Unlike all the previous examples, notice that for this model we have $\Delta E < 0$. Furthermore, this is the example of a model in which the impact relation is not symmetric: parasites do not have any impact on hosts, whereas hosts always have an impact on nearby parasites.

\subsection{An inconsistent example: the Ising contours model}\label{exampleisingc}

A very important tool in the study of phase transitions is the use of contour models.
Perhaps one of the most popular examples in statistical mechanics of a contour model is the Ising contours model (also known as Peierls contours), which arises as a geometrical representation of the Ising model on $\Z^d$ and was used originally by Peierls to establish the occurrence of phase transition at low temperatures.
The Ising model is not a diluted model itself in the sense that particles are forced to occupy \textit{every} site of the lattice.\footnote{Although there exists an equivalent representation of the Ising model as a lattice gas which falls into the category of diluted models as presented in this chapter. See \cite[p.~154]{OV}.} Nonetheless, it fits within the framework of Boltzmann-Gibbs distributions \mbox{proposed in this chapter.}

The Ising model (with zero external field) is defined on the configuration space $\{+,-\}^{\Z^d}$ through the Boltzmann-Gibbs distributions given for each $\Lambda \in \B^0_{\Z^d}$ and $\eta \in \{+,-\}^{\Z^d}$ by the formula
$$
\mu_{\Lambda}^\eta(\sigma) = \frac{\mathbbm{1}_{\{\sigma_{\Lambda^c} \equiv \eta_{\Lambda^c}\}}}{Z_{\Lambda}^\eta} e^{-\beta \sum_{B : B \cap \Lambda \neq \emptyset} \Phi_B(\sigma)}
$$ with $Z_{\Lambda}^\eta$ being the normalizing constant, and for each $B \subseteq \Z^d$
\begin{equation}\label{hising}
\Phi_B(\sigma) = \left\{ \begin{array}{ll} \mathbbm{1}_{\{ \sigma(x)\sigma(y)=-1\}} & \text{ if $B=\{x,y\}$ with $\|x-y\|_1 = 1$} \\ \\ 0 & \text{ otherwise.}\end{array}\right.
\end{equation}

Now, the contour representation arises in the study of Boltzmann-Gibbs distributions with a constant boundary condition, i.e. either $\eta(x)=+$ or $\eta(x)=-$ for all $x \in \Z^d$. To fix ideas let us consider the $(+)$-boundary condition and denote the corresponding Boltzmann-Gibbs distribution on the volume $\Lambda$ by $\mu_{\Lambda}^+.$ From \eqref{hising} we immediately see that the weight assigned by $\mu_{\Lambda}^+$ to each configuration which is positively aligned outside $\Lambda$ depends only on the amount of misaligned nearest neighbors spins in the configuration. With this in mind, one may introduce the following alternative representation of any such configuration which keeps track of misaligned spins:
\begin{enumerate}
\item [$\bullet$] Consider the edge set $e(\Z^d)$ consisting of all bonds joining nearest neighbors in $\Z^d$.
\item [$\bullet$] For each bond $e \in e(\Z^d)$ consider the plaquette $p(e)$: the unique $(d-1)$-dimensional unit cube with vertices in the dual lattice intersecting $e$ in a perpendicular manner. Recall that the dual lattice $(\Z^d)^*$ is defined as
    $$
    (\Z^d)^*:=\left\{ \left(x_1+\frac{1}{2}, \dots, x_d + \frac{1}{2}\right) : x \in \Z^d\right\}.
    $$
\item [$\bullet$] We call any collection of plaquettes a \textit{surface}. We shall say that a surface $P$ is \textit{closed} if every $(d-2)$-dimensional face of $P$ is shared by an even number of \mbox{plaquettes in $P$.}
\item [$\bullet$] We say that two plaquettes are adjacent if they share a $(d-2)$-dimensional face. \mbox{A surface $P$} is said to be \textit{connected} if for every pair of plaquettes in $P$ there exists a sequence of pairwise adjacent plaquettes joining them.
\item [$\bullet$] A \textit{contour} is then defined as a connected and closed surface. Two contours $\gamma$ and $\gamma'$ are said to be \textit{incompatible} if they share a $(d-2)$-dimensional face, in which case we denote this fact by $\gamma \not \sim \gamma'$.
\item [$\bullet$] Given a configuration $\sigma$ which is positively aligned outside $\Lambda$ we may assign to it a family $\Gamma_\sigma$ of pairwise compatible contours lying inside $\Lambda^*$, the smallest subset of the dual lattice which contains $\Lambda$.
Indeed, given any such configuration $\sigma$ we may consider the surface $P_\sigma$ consisting of those plaquettes $p(e)$ such that the bond $e$ joins misaligned spins in $\sigma$. This surface $P_\sigma$ is split into maximal connected components, each of which is a contour. If $\Gamma_\sigma$ denotes the collection of these maximal components, we immediately see that $\Gamma_\sigma$ satisfies all desired requirements.
\item [$\bullet$] If $\Lambda \in \B^0_{\Z^d}$ is simply connected\footnote{We say that $\Lambda \subseteq \Z^d$ is \textit{simply connected} if the set $\bigcup_{x \in \Lambda} (x + [-\frac{1}{2},\frac{1}{2}]^d)$ is simply connected in $\R^d$.} then for any given family $\Gamma$ of compatible contours lying inside $\Lambda^*$ there exists a unique \mbox{configuration} $\sigma^+_\Gamma$ such that $\Gamma_{\sigma^+_\Gamma} = \Gamma$. Indeed, the value of $\sigma^+_\Gamma(x)$ for $x \in \Lambda$ can be computed as $(-1)^{n_\Gamma(x)}$, where $n_\Gamma(x)$ denotes the total number of contours in $\Gamma$ around $x$. We call $\sigma^+_\Gamma$ the $(+)$-\textit{alignment} of $\Gamma$.
\end{enumerate}

From the considerations made above it is clear that for any configuration $\sigma \in \{+,-\}^{\Z^d}$ in the support of $\mu^+_\Lambda$ one has that
$$
\mu_{\Lambda}^+(\sigma) = \frac{1}{Z_{\Lambda}^+}e^{-\beta \sum_{\gamma \in \Gamma_\sigma} |\gamma|}
$$ where $|\gamma|$ denotes the total number of plaquettes in $\gamma$. Thus, if we are only interested in understanding the behavior of the system with a positively aligned boundary condition, \mbox{we may restrict} ourselves to the study of the interactions between the different contours, which give rise to a contour model. Before we can introduce it, we make some conventions.

We define the spin set $G$ as the space of possible contour shapes, i.e. without any regard for their position on $(\Z^d)^*$, and identify each contour $\gamma$ with an element $\gamma_x \in (\Z^d)^* \times G$: \mbox{$\gamma$ shall be its shape} whereas $x$ will be the location of the minimal vertex in $\gamma$ when ordered according to the dictionary order in $(\Z^d)^*$. See Figure \ref{fig3} for a possible example. Finally, we define the Ising contours model as the model on $\mathcal{N}((\Z^d)^* \times G)$ with intensity measure
\begin{figure}
	\centering
	\includegraphics[width=8cm]{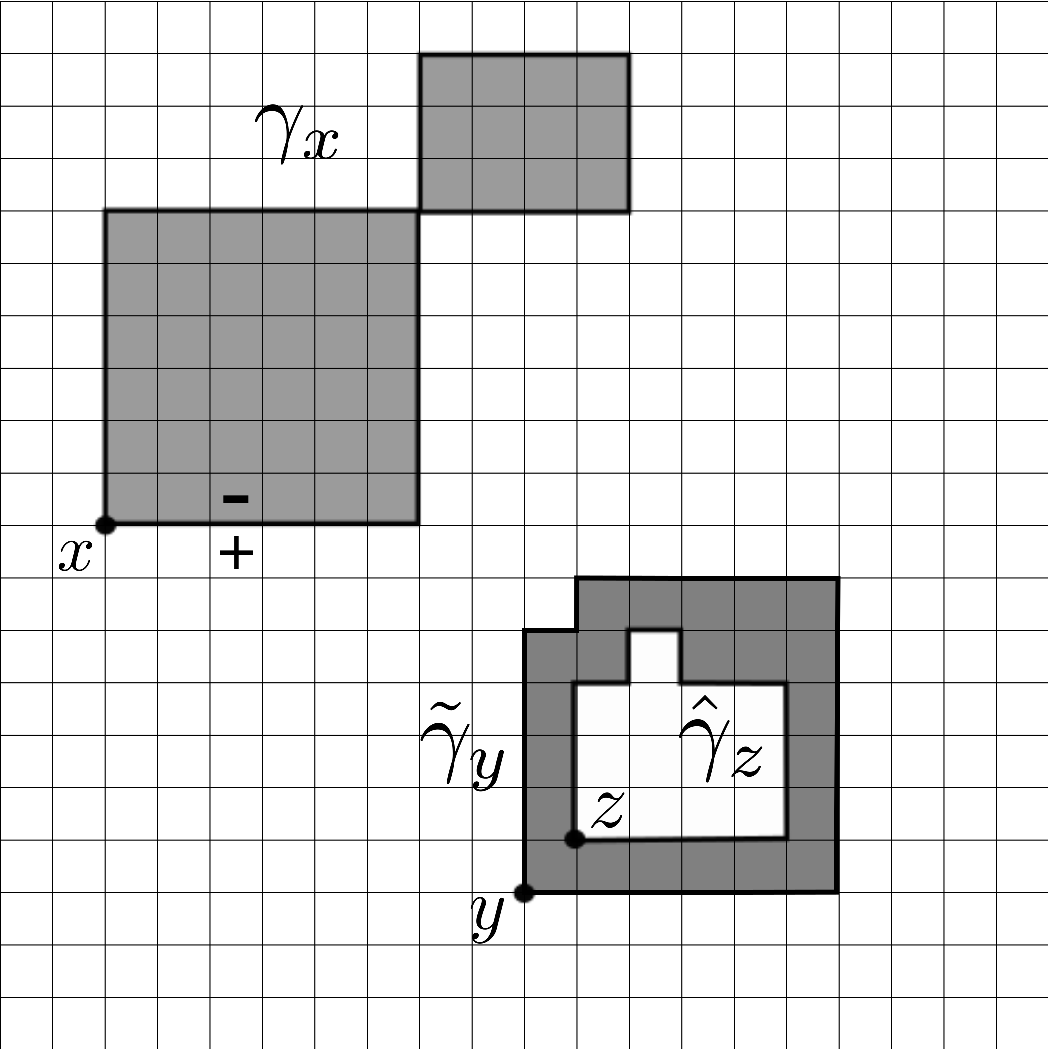}
	\caption{Ising contours in the dual lattice $(\Z^2)^*$ for the $(+)$-boundary condition.}
	\label{fig3}
\end{figure}

\begin{equation}\label{imic}
\nu^\beta(\gamma_x):= e^{-\beta |\gamma_x|}
\end{equation} and Hamiltonian
\begin{equation}\label{hic}
H_{\Lambda|\Gamma'} (\Gamma) = \left\{ \begin{array}{ll} +\infty & \text{ if either $\Gamma$ is incompatible, $\Gamma \not \sim \Gamma'_{\Lambda^c \times G}$ or $\Gamma \not \subseteq \Lambda$}\\ \\ 0 & \text{ otherwise.}\end{array}\right.
\end{equation} where $\Gamma \subseteq \Lambda$ indicates that all contours in $\Gamma$ lie entirely inside $\Lambda$. Notice that it is because of this last restriction in the Hamiltonian that the model fails to satisfy Assumptions \ref{assump}. Indeed, one has that:
\begin{enumerate}
\item [i.] $\Delta E_{\eta} \equiv +\infty$ for all $\eta \in \mathcal{N}((\Z^d)^* \times G)$ since contours contain more than one vertex. Therefore, for each $\gamma_x \in S \times G$ the quantity $\Delta_\eta(\gamma_x)$ fails to represent what it should: the energy cost for the infinite-volume system to add the particle $\gamma_x$ when the current configuration of the system is given by $\eta$. To fix this problem one considers instead the \textit{localized energy leap functions} $\Delta E_{\Lambda|\eta} : S \times G \rightarrow [-\infty,+\infty]$ defined as
    $$
    \Delta E_{\Lambda | \eta} (\gamma_x ) = \left\{ \begin{array}{ll} +\infty & \text{ if either $\gamma_x \not \sim \eta $ or $\gamma_x \not \subseteq \Lambda$}\\ \\ 0 & \text{ otherwise.}\end{array}\right.
    $$ Notice that for every $\eta \in \mathcal{N}((\Z^d)^* \times G)$ and $\gamma_x \in S \times G$ one has that
    \begin{equation}\label{elising}
    \Delta E^*_{\eta}(\gamma_x) := \lim_{\Lambda \nearrow (\Z^d)^*} \Delta E_{\Lambda|\eta} (\gamma_x) = (+\infty)\mathbbm{1}_{\{ \gamma_x \not \sim \eta \}}.
    \end{equation} Thus, for the localized energy leap functions one recovers in the limit as $\Lambda \nearrow (\Z^d)^*$ the correct notion of energy cost. Taking this into consideration, one also redefines the impact relation by incompatibility, i.e. $\tilde{\gamma}_y \rightharpoonup \gamma_x$ if and only if $\tilde{\gamma}_y \not \sim \gamma_x$.

\item [ii.] The consistent Hamiltonian property is not satisfied. As a consequence, one can easily check that the resulting Boltzmann-Gibbs distributions are also inconsistent in the sense of \eqref{consistencia}. In this context, Gibbs measures as defined in Definition \ref{Gibbs2} could fail to exist. However, it is still possible for these contour model to admit ``infinite-volume Boltzmann-Gibbs distributions'' in the sense described by Proposition \ref{limitegibbs}. These limiting measures will be of particular interest to us since, as we will see on Chapter \ref{chapterffg}, their existence implies a phase transition in the Ising model.
\end{enumerate}
Despite the fact that not all conditions on Assumptions \ref{assump} are satisfied, under the minor adjustments suggested above most of the analysis carried out in the next chapter for diluted models will also hold in this context, which is why we decided to include this model among the examples even if it is not a diluted model as we understand them. Another fact worth mentioning is that this is the only given example which is of unbounded local interaction range. Finally, we point out the following crucial relation between the Boltzmann-Gibbs distributions in the original Ising and Ising contours models: for any simply connected $\Lambda \in \B^0_{\Z^d}$ and family $\Gamma$ of compatible contours lying inside $\Lambda^*$ we have
\begin{equation}\label{dualidadising}
\mu^+_\Lambda ( \sigma^+_{\Gamma}) = \mu_{\Lambda^*|\emptyset}(\Gamma).
\end{equation} An analogous contour representation is also available for Boltzmann-Gibbs distributions with negatively aligned boundary condition.
We refer to \cite{PS} for a \mbox{thorough review} of the Ising model and the standard proof of phase transition using its contour representation.

\newpage

\section{Resumen del Capítulo 8}

Definimos en este capítulo la clase de modelos diluidos, que habremos de estudiar en lo que resta de la segunda parte. Esencialmente, un modelo diluido se define a partir de sus distribuciones de Boltzmann-Gibbs, medidas sobre el espacio de configuraciones $\mathcal{N}(S \times G)$ que describen el comportamiento local del modelo en volúmenes finitos sujeto sujeto a distintas condiciones de frontera. Concretamente, fijada una medida $\nu$ en $S \times G$,  dado $\Lambda \subseteq S$ acotado y $\eta \in \mathcal{N}(S \times G)$ se define la distribución de Boltzmann-Gibbs $\mu_{\Lambda|\eta}$ en $\Lambda$ con respecto a la condición de frontera $\eta$ mediante la fórmula
$$
\mu_{\Lambda|\eta} = \omega_{\Lambda|\eta} \times \delta_{\eta_{\Lambda^c}}
$$ donde identificamos $\mathcal{N}(S \times G) = \mathcal{N}( \Lambda \times G) \times \mathcal{N}(\Lambda^c \times G)$ y $\omega^\eta_\Lambda$ es la medida de probabilidad en $\mathcal{N}(\Lambda \times G)$ definida a través de la relación
$$
d\omega_{\Lambda|\eta} = \frac{e^{-H_{\Lambda|\eta}}}{Z_{\Lambda|\eta}} d\pi^\nu_\Lambda.
$$ Aquí $\pi^\nu_\Lambda$ denota la distribución de un proceso de Poisson en $\mathcal{N}(\Lambda \times G)$ con intensidad $\nu_{\Lambda \times G}$, cuyo rol es el de distribuir las partículas dentro de $\Lambda$, mientras que $H_{\Lambda|\eta}$ es lo que se denomina el Hamiltoniano relativo a $\Lambda$ con condición de frontera $\eta$, encargado de asignar un peso a las distintas configuraciones de acuerdo a la interacción que haya entre las partículas que la conforman. De esta manera, si $H$ denota a la familia de Hamiltonianos locales, el par $(\nu, H)$ determina por completo al modelo. Asumimos que $(\nu,H)$ cumple las condiciones dadas en \ref{assump}.

De particular interés en cada modelo son las medidas de Gibbs asociadas al mismo, es decir, las medidas $\mu$ sobre $\mathcal{N}(S \times G)$ que son compatibles con las distribuciones de Boltzmann-Gibbs para cualquier volumen $\Lambda$, i.e.
$$
\mu(\cdot) = \int_{\mathcal{N}(S \times G)} \mu_{\Lambda|\eta}(\cdot) d\mu(\eta).
$$ Las medidas de Gibbs representan los distintos posibles estados de equilibrio globales del modelo. En efecto, mostramos que bajo ciertas condiciones adicionales sobre el par $(\nu,H)$, cualquier límite local de las distribuciones de Boltzmann-Gibbs $\mu_{\Lambda|\eta}$ cuando $\Lambda \nearrow S$ es una medida de Gibbs. En lo que resta, nos interesará estudiar qué condiciones garantizan existencia y unicidad o multiplicidad de medidas de Gibbs.

Por último, culminamos el capítulo mostrando que algunos modelos clásicos, como lo son el modelo de Widom-Rowlinson (tanto en su versión continua como discreta) y el de contornos de Ising, caen dentro de la familia de modelos diluidos. Explicamos en detalle la dualidad entre el modelo de Ising original y su modelo de contornos asociado, que será de vital importancia para los desarrollos del Capítulo 12.
Agregamos además algunos otros ejemplos, como los modelos de Widom-Rowlinson tolerante y el simbiótico, para mostrar la flexibilidad de nuestras definiciones y la amplitud de nuestro marco teórico.

\chapter{The Fernández-Ferrari-Garcia dynamics}\label{chapterffg}

In this chapter we study the Fernández-Ferrari-Garcia dynamics first \mbox{introduced in \cite{FFG1}}. In their work the authors focus on the Ising contours model and show that, for a sufficiently large value of the inverse temperature $\beta$, the infinite-volume Boltzmann-Gibbs distribution of this contour model can be realized as the unique invariant measure of these dynamics. Later on \cite{FFG2}, the authors investigate the possibility of using this new approach as a perfect simulation scheme for Gibbs measures of a number of systems with exclusion in the low density or extreme temperature regime. Our purpose now is to \mbox{carry out} the same endeavor in general for the broader family of diluted models. The main ideas in this chapter are those originally featured in \cite{FFG1}: the majority of the results presented here are direct generalizations of those found there.
Nonetheless, some of the proofs we give in this chapter are different from the ones in the original article, since not all of their proofs can be adapted to continuum models.

Given an intensity measure $\nu$ and a Hamiltonian $H$ satisfying Asssumptions \ref{assump}, \mbox{one could summarize} the essentials of the associated Fernández-Ferrari-García dynamics as follows:
\begin{enumerate}
\item [$\bullet$] At rate $e^{-\Delta E}$ the birth of new animals is proposed with intensity given by $\nu$.
\item [$\bullet$] Each $\gamma_x$ proposed for birth will be effectively born with probability $e^{-(\Delta E_{\eta}(\gamma_x)-\Delta E)}$, where $\eta$ is the state of the system at the time in which the birth of $\gamma_x$ is proposed.
\item [$\bullet$] Every effectively born animal has an independent random exponential lifetime of parameter 1.
\item [$\bullet$] After its lifetime has expired, each animal dies and vanishes from the configuration.
\end{enumerate}
Our aim in this section is to make this description rigorous as well as to study some of the basic properties enjoyed by these dynamics. We shall begin by introducing the dynamics restricted to a finite volume and then treat the more delicate scenario of infinite volume. All processes defined below shall be subsets of the product space $\mathcal{C}= (S \times G) \times \R \times \R^+$. The elements of $\mathcal{C}$ shall be called \textit{cylinders} since any $(\gamma_x, t, s) \in \mathcal{C}$ can be seen as a cylinder on $S \times \R$ of axis $\{x\}\times [t,t+s]$ and diameter $\gamma$. However, we shall prefer to describe each cylinder $C=(\gamma_x,t,s)\in \mathcal{C}$ in terms of its \textit{basis} $\gamma_x$, its \textit{time of birth} $t$ and its \textit{lifespan} $s$. We shall denote these three features of $C$ by $basis(C)$, $b_C$ and $l_C$, respectively.

Following this line of thought, we can identify any random element $\mathcal{V} \in \C$ with a birth and death process on $S \times G$ by means of its time sections, i.e. if for each $t \in \R$ we define the random particle configuration $\mathcal{V}_t \in \mathcal{N}(S \times G)$ by the formula
$$
\mathcal{V}_t ( \{ \gamma_x\})  = \#\{ C \in \mathcal{V} : basis(C)=\gamma_x \text{ and }b_C \leq t < b_C + l_C \}
$$ for every $\gamma_x \in S \times G$, then $(\mathcal{V}_t)_{t \in \R}$ constitutes a birth and death process on $S \times G$. \mbox{From this point of view} we thus interpret any cylinder $(\gamma_x,t,s)$ as a particle $\gamma$ being born at time $t$ on location $x$ which lives on for a period of length $s$.

\section{Local dynamics}\label{localdynamics}

We begin our formal introduction of the Fernández-Ferrari-García dynamics (referred to as FFG dynamics from now on) by fixing a bounded set $\Lambda \in \mathcal{B}^0_{S}$ and a particle configuration $\eta \in \mathcal{N}(S\times G)$, and defining the dynamics on the finite volume $\Lambda \times G$ with $\eta$ acting as a boundary condition. Though it may seem misleading at first, we choose to build these local dynamics from infinite-volume processes since it will provide a clear and direct way in which to couple all local dynamics together.

Consider a Poisson process $\Pi$ on $\mathcal{C}$ with intensity measure $\phi_\nu = \nu \times e^{-\Delta E} \mathcal{L} \times \mathcal{E}^1$, where $\mathcal{L}$ is the Lebesgue measure on $\R$ and $\mathcal{E}^1$ is the exponential distribution of parameter 1. We shall refer to $\Pi$ as the \textit{free process}, whose time evolution can be described as follows:
\begin{enumerate}
\item [$\bullet$] At rate $e^{-\Delta E}$ animals are born with intensity given by $\nu$, regardless of the impact preexisting animals may have upon them.
\item [$\bullet$] Each animal has an independent random exponential lifetime of parameter 1.
\item [$\bullet$] After its lifetime has expired, each animal dies and vanishes from the configuration.
\end{enumerate}
This free process constitutes, as its name suggests, a stationary non-interacting birth and death process whose invariant measure is $\pi^\nu$. To be able to define the dynamics we need to add an additional component to $\Pi$: to each cylinder in $\Pi$ we will attach an independent random variable uniformly distributed on $[0,1]$, which will be called its \textit{flag}. Each flag will be used to determine the success of the associated cylinder's attempted birth in the dynamics. One way in which to attach these flags would be to replace $\Pi$ with the marked Poisson process $\overline{\Pi}$ on $\mathcal{C} \times [0,1]$ with intensity measure $\overline{\phi}_\nu :=\phi_\nu \times \mathcal{L}_{[0,1]}$. Thus, elements of $\overline{\Pi}$ can be seen as cylinders in $\Pi$ together with their respective independent flags in $[0,1]$. For any given $(\gamma_x,t,s) \in \Pi$ we shall denote its corresponding flag by $F(\gamma_x,t,s)$.\footnote{Notice that for each $(\gamma_x,t,s) \in \Pi$ its flag is the unique $u \
in [0,1]$ such that $(\gamma_x,t,s,u) \in \overline{\Pi}$. Thus, there is no ambiguity in this choice of notation.} Finally, recalling the identification $\mathcal{N}(S \times G) = \mathcal{N}(\Lambda \times G) \times \mathcal{N}(\Lambda^c \times G)$ we define the local FFG process $\mathcal{K}^{\Lambda|\eta}$ on $\Lambda \times G$ with boundary condition $\eta$ by the formula
\begin{equation}\label{keptformula1}
\mathcal{K}^{\Lambda|\eta} = \{ (\gamma_x,t,s) \in \Pi_{\Lambda \times G} : F(\gamma_x,t,s) < M(\gamma_x | \mathcal{K}^{\Lambda|\eta}_{t^-}) \} \times \{ (\gamma_x,t,s) \in \mathcal{C}: \gamma_x \in \eta_{\Lambda^c \times G}\}
\end{equation} where for $\gamma_x \in S \times G$ and $\xi \in \mathcal{N}(S\times G)$ we use the notation $M(\gamma_x|\xi):=e^{-(\Delta E_{\xi}(\gamma_x)-\Delta E)}$ and
$\Pi_{\Lambda \times G}$ denotes the restriction of $\Pi$ to $(\Lambda \times G) \times \R \times \R^+$. In other words, $\mathcal{K}^{\Lambda|\eta}$ is the process obtained as a thinning of the free process inside $\Lambda$ given by the rule on \eqref{keptformula1} with the addition of a boundary condition $\eta$ outside $\Lambda$ which must be kept fixed for all times, i.e. $(\mathcal{K}^{\Lambda|\eta}_t)_{\Lambda^c \times G} = \eta_{\Lambda^c \times G}$ for every $t \in \R$. Notice that the self-referential nature of the thinning rule in \eqref{keptformula1} could lead to $\mathcal{K}^{\Lambda|\eta}$ not being well defined. Indeed, let us introduce some definitions that will help us give further details on this matter.

\begin{defi}\label{defiances}$\,$
\begin{enumerate}
\item [$\bullet$] Given $C, \tilde{C} \in \C$ we say that $\tilde{C}$ is a \textit{first generation ancestor} of $C$ and write $\tilde{C} \rightharpoonup C$ whenever
$$
basis(\tilde{C}) \rightharpoonup basis(C)\hspace{1.5cm} \text{ and }\hspace{1.5cm} b_{\tilde{C}} < b_C < b_{\tilde{C}} + l_{\tilde{C}}.
$$ We shall denote the set of all first generation ancestors of a given $C \in \C$ by $\mathcal{P}(C)$.
\item [$\bullet$] For $C \in \C$ we define $\A_1(C):=\Pi_{\mathcal{P}(C)}$ and for $n \in \N$ we set
$$
\A_{n+1}(C)= \bigcup_{\tilde{C} \in \mathcal{A}_n(C)} \A_1(\tilde{C}).
$$ We define the \textit{clan of ancestors} of $C$ in $\Pi$ as
$$
\A(C):= \bigcup_{n \in \N} \A_n(C).
$$ Furthermore, for $\Lambda \in \B^0_S$ and $n \in \N$ we define the $n$-th generation of ancestors of $C$ restricted to $\Lambda$ as
$$
\A_{n}^\Lambda(C) = \{ \tilde{C} \in \A_n(C) : basis(\tilde{C}) \in \Lambda \times G\}.
$$ We define $\A^\Lambda(C)$, the clan of ancestors of $C$ restricted to $\Lambda$, in an analogous fashion.
\item [$\bullet$] For $t \in \R$ and $\Lambda \in \B^0_S$ let us define the \textit{clan of ancestors of $\Lambda \times G$ at time $t$} as
$$
\mathcal{A}^t(\Lambda \times G) := \bigcup_{n \in \N_0} \mathcal{A}^t_{n}(\Lambda \times G)
$$ where $\mathcal{A}^t_0(\Lambda \times G) := \{ C \in \Pi : basis(C) \in \Lambda \times G \text{ , } b_C \leq t < b_C + l_C \}$ and for $n \in \N$
$$
\mathcal{A}^t_n(\Lambda \times G) := \bigcup_{C \in \mathcal{A}^t_0 (\Lambda \times G)} \mathcal{A}_n(C).
$$ For any $\Delta \in \B^0_S$ such that $\Lambda \subseteq \Delta$ we define $\A^{t,\Delta}(\Lambda \times G)$, the clan of ancestors of $\Lambda \times G$ at time $t$ restricted to $\Delta \times G$, in the same manner as above.
\end{enumerate}
\end{defi}

Having defined the notion of ancestors in the dynamics, we now return to discuss the good definition of $\mathcal{K}^{\Lambda|\eta}$. Notice that if we wish to determine whether a given cylinder $C=(\gamma_x,t,s) \in \Pi_{\Lambda \times G}$ belongs to $\mathcal{K}^{\Lambda|\eta}$ or not then first we need to specify the configuration $\mathcal{K}^{\Lambda|\eta}_{t^-}$ in order to evaluate whether the condition on \eqref{keptformula1} is satisfied. To be more accurate, due to Assumptions \ref{assump} we will only need to specify $\mathcal{K}^{\Lambda|\eta}_{t^-}$ inside the set $I( \{\gamma_x\})$. However, since $\mathcal{K}^{\Lambda|\eta}$ is known outside $\Lambda$ as it coincides with $\eta$ for all times, it remains to specify $\mathcal{K}^{\Lambda|\eta}_{t^-}$ inside $I( \{\gamma_x\}) \cap (\Lambda \times G)$. Therefore, recalling Definition \ref{defiances}, we see that to determine the fate of $C$ we must first determine the fate of all its first generation ancestors with bases in $\Lambda \times G$, i.
e. cylinders in $\A^\Lambda_1(C)$. But this task itself involves determining the fate of a second generation of ancestors of $C$, those cylinders with bases in $\Lambda \times G$ being ancestors to cylinders in $\A^\Lambda_1(C)$. In general, to determine if $C$ belongs to $\mathcal{K}^{\Lambda|\eta}$ we must study the fate of every cylinder in $\A^{\Lambda}(C)$, the clan of ancestors of $C$ restricted to $\Lambda$. If $\A^{\Lambda}(C)$ were to span over an infinite number of generations then it may be impossible to decide whether to keep $C$ or not and, therefore, $\mathcal{K}^{\Lambda|\eta}$ may not be well defined in this situation. On the other hand, if we were able to guarantee that for every cylinder $C \in \Pi_{\Lambda \times G}$ the restricted clan $\A^\Lambda(C)$ spans only over a finite number of generations then $\mathcal{K}^{\Lambda|\eta}$ would be well defined. Indeed, since $M(\gamma_x | \mathcal{K}^\sigma_{t^-}) = M(\gamma_x | \eta_{\Lambda^c \times G})$ for any cylinder $(\gamma_x, t,s) \in \
Pi_{\Lambda \times G}$ with no ancestors preceding it, we have that the fate of every cylinder in the last generation of ancestors restricted to $\Lambda$ of a given cylinder $C$ can be decided upon inspecting their respective flags (and nothing else) and thus it will also be possible to determine the fate of all their descendants, including $C$. More precisely, take $C \in \Pi_{\Lambda \times G}$ and let $N$ be a nonnegative integer such that $\mathcal{A}^\Lambda_n(C) = \emptyset$ for every $n > N$. If we set
$$
K_N^\Lambda(C):= \{ (\tilde{\gamma}_y,r,l) \in \mathcal{A}^\Lambda_N(C) : F(\tilde{\gamma}_y,r)< M(\tilde{\gamma}_y|\eta_{\Lambda^c \times G})\}
$$ and for $1 \leq i \leq N-1$ inductively define
$$
K_i^\Lambda(C) = K_{i+1}^\Lambda(C) \cup \{ (\tilde{\gamma}_y,r,l) \in \mathcal{A}^\Lambda_i(C) : F(\tilde{\gamma}_y,r,l)< M(\tilde{\gamma}_y|K^\Lambda_{i+1}(C))\}
$$ then the cylinder $C \in \Pi_{\Lambda \times G}$ will be kept if and only if
$$
F(C) < M(\tilde{\gamma}_y|K^\Lambda_1(C)).
$$
In other words, to decide if a cylinder $C \in \Pi_{\Lambda \times G}$ is kept one could conduct the following procedure:
\begin{enumerate}
\item [i.] If $C$ has no first generation ancestors, i.e. $\mathcal{A}^\Lambda_1(C)=\emptyset$, then the value of its flag $u$ alone will determine whether $C$ is kept or not. Otherwise, the value of $u$ will decide if $C$ is kept once we determine the fate of all the first generation ancestors of $C$.
\item [ii.] To decide whether any given first generation ancestor $\tilde{C} \in \mathcal{A}^\Lambda_1(C)$ is kept, one must repeat step (i) for $\tilde{C}$ instead of $C$.
\item [iii.] Since the clan of ancestors of $C$ restricted to $\Lambda$ possesses only a finite number of generations, one must go backwards in time and examine a previous generation of ancestors only a finite amount of times (at most $N$ times) so that it is ultimately possible to determine whether $C$ is kept.
\end{enumerate} Therefore, we are left to answer the question of under which conditions do the clans of ancestors restricted to $\Lambda$ possess only a finite amount of generations. Fortunately, under Assumptions \ref{assump} this is always the case. This is the content of the following proposition.

\begin{prop}\label{localfinit} For every $\Lambda \in \B^0_S$ we have that $\A^{t,\Lambda}(\Lambda \times G)$ is finite for all $t \in \R$.
\end{prop}

\begin{proof}Since $\Pi$ is a stationary process whose invariant measure $\pi^\nu$ satisfies
$$
\pi^\nu( \{ \xi \in \mathcal{N}(S \times G) : \xi( \Lambda \times G) = 0 \} ) = e^{-\nu(\Lambda \times G)} > 0
$$ we have that the entrance times $(t_i(\Lambda))_{i \in \Z}$ to the set $\{ \xi \in \mathcal{N}(S \times G) : \xi( \Lambda \times G) = 0 \}$ are well defined \mbox{(i.e. finite almost surely)} and satisfy $t_i(\Lambda) \rightarrow \pm \infty$ as $i \rightarrow \pm \infty$. In particular, for every $t \in \R$ there exists $i_0 \in \Z$ such that $t_{i_0-1}(\Lambda) \leq t < t_{i_0}(\Lambda)$. Since $\Pi_{t_i(\Lambda)}(\Lambda \times G) = 0$ for each $i \in \Z$ by definition, this implies that there exist (random) $k < r \in \Z$ such that
\begin{equation}\label{contan}
\A^{t,\Lambda}(\Lambda \times G) \subseteq \Pi_{(\Lambda \times G) \times [t,t_{i_0}] \times \R^+} \subseteq \Pi_{(\Lambda \times G) \times [k,r] \times \R^+}.
\end{equation} Since for every $k < r \in \Z$ we have $\phi((\Lambda \times G) \times [k,r] \times \R^+) = (r-k)\nu(\Lambda \times G) < +\infty$ by Assumptions \ref{assump}, with probability one we have that for every $k < r \in \Z$ the random particle configurations $\Pi_{(\Lambda \times G) \times [k,r] \times \R^+}$ are all finite. By \eqref{contan} this concludes the proof.
\end{proof}

By the discussion above, Proposition \ref{localfinit} yields that for $\Lambda \in \B^0_S$ and $\eta \in \mathcal{N}(S\times G)$ the process $\mathcal{K}^{\Lambda|\eta}$ is well defined and constitutes an interacting birth and death process. Moreover, $\mathcal{K}^{\Lambda|\eta}$ is stationary due to the time translational invariance of its construction and that of $\Pi$.

\section{Infinite-volume dynamics}

\subsection{Stationary dynamics}

As stated before, some complications arise when lifting the restriction of finite volume in the dynamics. The procedure to define the unrestricted FFG process in the entire space $S \times G$ is completely analogous to that of the finite volume case: it suffices to take $\Lambda=S$ everywhere in the previous section. Thus, the FFG process $\mathcal{K}$ on the whole space $S \times G$ is defined by the formula
\begin{equation}\label{keptformula2}
\mathcal{K} = \{ (\gamma_x,t,s) \in \Pi : F(\gamma_x,t,s) \leq M(\gamma_x | \mathcal{K}_{t^-}) \}.
\end{equation}
Following the analysis of the previous section in this context, we see that in order for $\mathcal{K}$ to be well defined we must guarantee that for every cylinder $C \in \Pi$ its clan of ancestors $\A(C)$ spans only over a finite number of generations. Notice that the argument used in Proposition \ref{localfinit} will not go through this time as in general we have $\nu(S \times G)=+\infty$. Therefore, we will need to impose additional conditions on both $\nu$ and $H$ besides those on Assumptions \ref{assump} to guarantee that $\mathcal{K}$ is well defined in this case. This is the content of the next proposition.

\begin{prop}\label{finit1} If there exists a measurable function $q: S \times G \to \R$ satisfying $\inf_{\gamma_x \in S \times G} q(\gamma_x) \geq 1$ and such that
\begin{equation}\label{hdiluted}
\alpha_q:=\sup_{\gamma_x \in S\times G} \left[\frac{e^{-\Delta E}}{q(\gamma_x)} \int_{I(\{\gamma_x\})} q(\tilde{\gamma}_y)d\nu(\tilde{\gamma}_y)\right] < 1
\end{equation} then $\mathcal{A}^t(\Lambda \times G)$ is finite for every $t \in \R$ and $\Lambda \in \B^0_S$ almost surely.
\end{prop}

\begin{defi} Whenever \eqref{hdiluted} holds we say that $\nu$ and $H$ satisfy the (F1)-diluteness condition with size function $q$ and that the associated model is \textit{heavily diluted}.
\end{defi}

Thus, whenever dealing with a heavily diluted model we have that $\mathcal{K}$ is well defined and constitutes a stationary interacting birth and death process on the entire space $S \times G$. We postpone the proof of Proposition \ref{finit1} until Section \ref{ancestors}.

\subsection{Dynamics on a bounded time window}\label{forwarddynamics}

One can wonder whether it is possible that, upon relaxing the conditions on Proposition \ref{finit1}, the FFG process remains well defined on the infinite volume but perhaps only for a bounded time window,
i.e. if given $t_1 < t_2 \in \R$ we replace $\phi$ in the construction above by
$$
\phi_{[t_1,t_2]}= \nu \times e^{-\Delta E} \mathcal{L}_{[t_1,t_2]} \times \mathcal{E}^1.
$$
This will occur if for every $\Lambda \in \B^0_{S}$ one has that
\begin{equation}\label{clanbounded}
\A^{[t_1,t_2]}(\Lambda \times G) := \{ C \in \A^{t_2}(\Lambda \times G) : birth(C) \geq t_1 \}
\end{equation} spans only over a finite number of generations. The next proposition shows that this is the case whenever the coefficient $\alpha_q$ defined on \eqref{hdiluted} is finite.

\begin{prop}\label{finit2} If there exists a measurable function $q: S \times G \to \R$ satisfying $\inf_{\gamma_x \in S \times G} q(\gamma_x) \geq 1$ and such that
\begin{equation}\label{wdiluted}
\alpha_q:=\sup_{\gamma_x \in S\times G} \left[\frac{e^{-\Delta E}}{q(\gamma_x)} \int_{I(\{\gamma_x\})} q(\tilde{\gamma}_y)d\nu(\tilde{\gamma}_y)\right] < +\infty
\end{equation} then $\mathcal{A}^{[t_1,t_2]}(\Lambda \times G)$ is finite for every $t_1 < t_2 \in \R$ and $\Lambda \in \B^0_{S}$ almost surely.
\end{prop}

\begin{defi} Whenever \eqref{wdiluted} holds we say that $\nu$ and $H$ satisfy the (F2)-diluteness condition with size function $q$ and that the associated model is \textit{well diluted}.
\end{defi}

Let us observe that due to Proposition \ref{finit2} we have that whenever a model is well diluted it is possible to define the FFG dynamics as a
forward time evolution on $\R^+$ for any initial condition $\sigma \in \mathcal{N}(S \times G)$. Indeed, given any particle configuration $\sigma \in \mathcal{N}(S \times G)$ and a family $(L_{(\gamma_x,i)})_{(\gamma_x,i)\in [\sigma]}$
of i.i.d. exponential random variables of parameter 1 independent of $\Pi$ we may set
$$
\overline{\Pi}^\sigma = \overline{\Pi} \cup \{(\gamma_x,0, L_{(\gamma_x,i)},0)  : (\gamma_x,i) \in [\sigma] \}
$$ and define $(\mathcal{K}^\sigma_t)_{t \geq 0}$ by the formula
\begin{equation}\label{keptformula3}
\mathcal{K}^\sigma = \{ (\gamma_x,t,s) \in \Pi^\sigma : F(\gamma_x,t,s) \leq M(\gamma_x | \mathcal{K}_{t^-}) \}
\end{equation}
where $\Pi^\sigma$ denotes the projection of $\overline{\Pi}^\sigma$ onto $\mathcal{C}^+:= (S\times G) \times \R^+ \times \R^+$ and by convention we set $\mathcal{K}_{0^-}\equiv \emptyset$. Notice that, even though by Proposition \ref{finit2} we have that
$\A^{[0,t]}(\Lambda \times G)$ is finite for every $t \geq 0$ and $\Lambda \in \B^0_S$ almost surely, the clan of ancestors associated to these forward dynamics contains also cylinders corresponding to the initial configuration
$\sigma$ and therefore it may not be finite (unless $\sigma$ has a finite local interaction range). Nonetheless, under the (F2)-diluteness condition it will always span over a finite number of generations and so $\mathcal{K}^\sigma$ is ultimately well defined.
Furthermore, since we have assigned a $0$ flag value to every particle in the initial condition $\sigma$, we get that the initial condition is always kept in $\mathcal{K}^\sigma$ even if $\sigma$ is a particle configuration forbidden by $H$.
We prove Proposition \ref{finit2} in the next section.

\section{Finiteness criteria for the clan of ancestors}\label{ancestors}

The aim of this section is to give the proofs of Propositions \ref{finit1} and \ref{finit2}, and to investigate some of their consequences as well. In both proofs we shall make use of the crucial fact that each
clan of ancestors can be contained in the offspring of some branching process. This is the content of the following lemma.

\begin{lema}[\textbf{Domination lemma}]\label{domilema} Given $\Lambda \in \B^0_S$ and $t \in \R$ there exists a family of random sets $(\B_n)_{n \in \N_0} \subseteq \mathcal{C}$ such that
\begin{enumerate}
\item [i.] $\A^t_0(\Lambda \times G) = \B_0$
\item [ii.] $\displaystyle{\bigcup_{i=0}^{n} \A^t_{i}(\Lambda \times G) \subseteq \bigcup_{i=0}^{n} \B_i}$ for every $n \in \N$
\item [iii.] Conditional on $(\B_i)_{0\leq i \leq n}$, $\B_{n+1}$ is a Poisson process with intensity measure $\displaystyle{\sum_{C \in \B_n} \phi_{\mathcal{P}(C)}.}$
\end{enumerate}
\end{lema}

\begin{proof} Consider the space $\M_t$ of particle configurations $\zeta$ on $\mathcal{C}_t:=(S\times G)\times (-\infty,t] \times \R^+$ such that
\begin{enumerate}
\item [$\bullet$] $\zeta$ is finite
\item [$\bullet$] No two cylinders in $\zeta$ have the same time of birth.
\end{enumerate} For $\zeta \in \M_t$ we shall set $\A_0(\zeta):=\zeta$ and for $n \in \N$ write
$$
\displaystyle{\A_n(\zeta) := \bigcup_{C \in \zeta} \A_n(C)}.
$$ Furthermore, suppose that we have ordered the elements of $\zeta$ in some particular way. Then, if $C_1 \preceq \dots \preceq C_k$ denote the ordered elements of $\zeta$, for each $i=1,\dots,k$ we define
$$
\mathcal{P}_\zeta(C_i) = \mathcal{P}(C_i) - \bigcup_{j=1}^{i-1} [C_j \cup\mathcal{P}(C_j)].
$$To define the family $(\B_n)_{n \in \N_0}$ first we shall fix $\zeta \in \M_t$ and construct a collection of sets $(\B_n(\zeta))_{n \in \N_0}$ such that for every $n \in \N_0$ the following properties are satisfied:
\begin{enumerate}
\item [i.] $\B_n(\zeta)$ belongs to $\M_t$ almost surely.
\item [ii.] Conditional on $\B_0(\zeta),\dots,\B_n(\zeta)$, the random set $\B_{n+1}(\zeta)$ is a Poisson process on $\C_t$ with intensity measure $\sum_{C \in \B_n(\zeta)} \phi_{\mathcal{P}(C)}.$
\item [iii.]$\displaystyle{\bigcup_{i=0}^{n} \A_i(\zeta) \subseteq \bigcup_{i=0}^{n} \B_i(\zeta).}$
\end{enumerate}
We start by setting $\B_0(\zeta):=\zeta$ and now proceed with the construction of the set $\B_1(\zeta)$. First we order the elements of $\zeta$ according to their times of birth, i.e. $\zeta=\{C_1,\dots,C_k\}$
where $0 \leq b_{C_1} < \dots < b_{C_k} \leq t$. Then continue by considering a collection $\Pi^{(1,1)},\dots,\Pi^{(1,k)}$ of independent Poisson processes on $\C$
of intensity measure $\phi$ such that $\Pi^{(1,1)}=\Pi$ and defining for each $i=1,\dots,k$
$$
\B_\zeta(C_i) := \Pi^{(1,i)}_{\mathcal{P}(C_i) - \mathcal{P}_\zeta(C_i)} \cup \Pi_{\mathcal{P}_\zeta(C_i)}.
$$ If we set $\displaystyle{\B_1(\zeta):=\bigcup_{i=1}^k \B_\zeta(C_i)}$ then $\B_1(\zeta)$ satisfies the properties stated above. Indeed:

\begin{enumerate}
\item [(1)] Each $\B_\zeta(C_i)$ is a Poisson process with intensity measure $\phi_{\mathcal{P}(C_i)}$ by virtue of the \mbox{independence of }$\Pi^{(1,1)},\dots,\Pi^{(1,k)}$ and the disjointness of $\mathcal{P}(C_i) - \mathcal{P}_\zeta(C_i)$ and $\mathcal{P}_\zeta(C_i)$.
\item [(2)] The independence of $\B_\zeta(C_1),\dots,\B_\zeta(C_k)$ follows from the independence of the $\Pi^{(1,i)}$ and the fact that $\mathcal{P}_\zeta(C_i) \cap \mathcal{P}_\zeta(C_j)=\emptyset$ for $i \neq j$. Together with (1) this gives (ii).
\item [(3)] Property (iii) follows upon noticing that for $i=1,\dots,k$
$$
\A_1(C_i) - \bigcup_{j=1}^{i-1} \A_1(C_j) \subseteq \Pi_{\mathcal{P}_\zeta(C_i)}.
$$
\item [(4)] Property (i) is also a consequence of (1) and (2) since for each $i=1,\dots,k$
$$
\phi_\nu(\mathcal{P}(C_i)) = e^{-\Delta E} \int_{I(basis(C_i))} \int_{-\infty}^{b_{C_i}} \int_{b_{C_i} - t}^{+\infty} e^{-s}ds dt d\nu < +\infty.
$$
\end{enumerate}
Having constructed $\B_1(\zeta)$, we now define the next generations in an inductive manner. For this we shall need to consider an array of $\N \times \N$ Poisson processes on $\C$ such that:
\begin{enumerate}
\item [$\bullet$] $\Pi^{(n,k)}$ is a Poisson process with intensity measure $\phi$ for every $n,k \in \N$
\item [$\bullet$] $\Pi^{(n,1)}=\Pi$ for every $n \in \N$
\item [$\bullet$] The processes $\{ \Pi^{(n,k)} : n \in \N , k \geq 2\}$ are independent of each other and also of $\Pi$.
\end{enumerate}
Suppose now that we have constructed the first $n$ generation of sets $\B_1(\zeta),\dots,\B_n(\zeta)$ and let us construct the next generation, $\B_{n+1}(\zeta)$. Order each of the constructed generations separately by their times of birth and write for each $j=1,\dots,n$
$$
\B_j(\zeta) = \{ C^{(j,1)}, \dots, C^{(j,k_j)} \}.
$$ Now let us consider the joint configuration $\zeta^{(n)}=\{ C^{(j,i)} : 1 \leq j \leq n \text{ and }1 \leq i \leq k_j \}$ ordered by the dictionary order, i.e. $C^{(j,i)} \preceq C^{(j',i')}$ if either $j < j'$ or $j=j'$ and $i \leq i'$. We then define
$$
\B_{n+1}(\zeta): = \bigcup_{i=1}^{k_n} \B_{\zeta^ {(n)}}(C^{(n,i)})
$$ where
$$
\B_{\zeta^{(n)}}(C^{(n,i)}) = \Pi^{(n+1,i)}_{\mathcal{P}(C^{(n,i)}) - \mathcal{P}_\zeta(C^{(n,i)})} \cup \Pi_{\mathcal{P}_\zeta(C^{(n,i)})}.
$$ Following a similar argument to the one given above it is possible show by inductive hypothesis that $\B_{n+1}(\zeta)$ satisfies properties (i), (ii) and (iii). Finally, having defined the collection $(\B_n(\zeta))_{n \in \N_0}$ for each $\zeta \in \M_t$, for each $n \in \N_0$ we set
$$
\B_n := \B_n\left( \mathcal{A}^t_0(\Lambda \times G)\right).
$$ One can check that, by construction of $(\B_n(\zeta))_{n \in \N_0}$, the family $(\B_n)_{n \in \N}$ satisfies all the desired properties. This concludes the proof.
\end{proof}

\begin{proof}[Proof of Proposition \ref{finit1}] Let us first fix $\Lambda \in \B^0_S$ and $t \in \R$ and consider the family of random sets $(\B_n)_{n \in \N_0}$ satisfying the conditions in the Domination lemma. By condition (ii) of this lemma we see that if we wish to show that $\A^t(\Lambda \times G)$ is almost surely finite it will suffice to prove that $\sum_{n \in \N_0} |\B_n| <+ \infty$ almost surely. But this will follow immediately once we show that for every $n \in \N_0$
\begin{equation}\label{cotabranching}
\E\left( \sum_{C \in \B_n} q(basis(C)) \bigg| \B_0 \right) \leq \left(\sum_{C \in \B_0} q(basis(C))\right)\alpha^n_q.
\end{equation}Indeed, if \eqref{cotabranching} holds then since $\inf_{\gamma_x \in S \times G} q(\gamma_x) \geq 1$ we have
\begin{align*}
P\left( \sum_{n \in \N_0} |\B_n| = +\infty \bigg| \B_0 \right) & \leq \lim_{k \rightarrow +\infty} P\left( \sum_{n \in \N_0} \sum_{C \in \B_n} q(basis(C)) > k \bigg| \B_0 \right) \\
\\
& \leq \lim_{k \rightarrow +\infty} \frac{\sum_{n \in \N_0} \E( \sum_{C \in \B_n} q(basis(C)) | \B_0 )}{k} \\
\\
& \leq \lim_{k \rightarrow +\infty} \frac{\sum_{C \in \B_0} q(basis(C))}{k(1-\alpha_q)} = 0
\end{align*}
since $P( |\B_0| < +\infty ) = 1$. From this we get that the unconditional probability is also null. Thus, in order to obtain \eqref{cotabranching} we first notice that by (iii) and \eqref{poisson} a simple calculation yields for every $n \in \N_0$
\begin{equation}\label{tower}
 \E\left( \sum_{C \in \B_{n+1}} q(basis(C)) \Bigg| \B_n,\dots,\B_0\right) \leq \left(\sum_{C \in \B_{n}} q(basis(C))\right) \alpha_q.
\end{equation} Now, since \eqref{cotabranching} clearly holds for $n=0$, the validity for every $n \in \N_0$ follows upon induction by applying \eqref{tower} and the tower property of conditional expectation. Finally, to show that with probability one this holds for every $\Lambda \in \B^0_S$ and $t \in \R$ simultaneously, we take $(\Lambda_n)_{n \in \N} \subseteq \B^0_S$ such that $\Lambda_n \nearrow S$ and observe that (with probability one) given $\Lambda \in \B^0_S$ and $t \in \R$ there exists $n_0 \in \N$ and $r \in \mathbb{Q}$ such that
$$
\A^t(\Lambda \times G) \subseteq \A^r(\Lambda_{n_0} \times G).
$$ Since there are only countable possibilities for $n_0$ and $r$ and we have shown that for every fixed pair $n,r$ the set $\A^r(\Lambda_{n} \times G)$ is finite almost surely, this yields the result.
\end{proof}

\begin{proof}[Proof of Proposition \ref{finit2}] By the same reasoning as in the proof of Proposition \ref{finit1} it will suffice to show that for each $r < l \in \mathbb{Q}$ and $n \in \N$ the random set $\A^{[r,l]}(\Lambda_n \times G)$ is almost surely finite, where $(\Lambda_n)_{n \in \N} \subseteq \B^0_S$ is such that $\Lambda_n \nearrow S$. But we can show this by performing an inductive procedure once we manage to prove the following two facts:
\begin{enumerate}
\item There exists $\delta > 0$ such that if $0 < t-s < \delta$ then $\A^{[s,t]}(\Lambda \times G)$ is \mbox{finite almost surely.}
\item If $|h| < \delta$ and $\A^{[s,t]}(\Lambda \times G)$ is finite almost surely then $\A^{[s-h,t]}(\Lambda \times G)$ is also finite almost surely.
\end{enumerate} To establish (1) we fix $\Lambda \in \B^0_S$, $s < t$ and construct similarly to the Domination lemma a family of random sets $(\B_n)_{n \in \N}$ satisfying
\begin{enumerate}
\item [i.] $\A^{[s,t]}_0(\Lambda \times G) = \B_0$
\item [ii.] $\displaystyle{\bigcup_{i=0}^{n} \A^{[s,t]}_{i}(\Lambda \times G) \subseteq \bigcup_{i=0}^{n} \B_i}$ for every $n \in \N$
\item [iii.] Conditional on $(\B_i)_{0 \leq i \leq n}$, $\B_{n+1}$ is a Poisson process with intensity $\sum_{C \in \B_n} (\phi_{[s,t]})_{\mathcal{P}(C)}$
\end{enumerate}where for $n \in \N$ we set $\A^{[s,t]}_{n}(\Lambda \times G) = \A^{t}_{n}(\Lambda \times G) \cap \A^{[s,t]}(\Lambda \times G)$. By performing a similar computation to the one yielding \eqref{tower} we obtain
\begin{equation}\label{tower2}
 \E\left( \sum_{C \in \B_{n+1}} q(basis(C)) \Bigg| \B_n,\dots,\B_0\right) \leq \left(\sum_{C \in \B_{n}} q(basis(C))\right) \tilde{\alpha}_q.
\end{equation}where $\tilde{\alpha}_q := \alpha_q (1- e^{-(t-s)})$. Since $\alpha_q < +\infty$, we may take $t-s$ small enough so as to guarantee that $\tilde{\alpha}_q < 1$. From this one obtains (1) by proceeding as in Proposition \ref{finit1}.

To see (2) we first notice that if $\A^{[s,t]}(\Lambda \times G)$ is finite then there exists a (random) set $\Lambda' \in \B^0_S$ such that the basis of every cylinder in $\A^{[s,t]}(\Lambda \times G)$ belongs to $\Lambda' \times G$. Furthermore, since $\nu(I(\Lambda'\times G)) < +\infty$ we have that there exists another (random) set $\Lambda'' \in \B^0_S$ such that the basis of every cylinder in $\Pi_{I(\Lambda'\times G) \times [s-h,s)\times \R^+}$ belongs to $\Lambda'' \times G$. Then, it is not hard to see that
\begin{equation}\label{ancesinduc}
\A^{[s-h,t]}(\Lambda \times G) \subseteq A^{[s,t]}(\Lambda \times G) \cup A^{[s-h,s]}(\Lambda'' \times G)
\end{equation}
Together with (1) (for $\Lambda''$ instead of $\Lambda$) , \eqref{ancesinduc} implies (2), which concludes the proof.
\end{proof}

When dealing with a heavily diluted model, the finiteness of every clan of ancestors forces the FFG dynamics to exhibit a loss of memory property. In particular, we obtain the convergence of the forward dynamics to the invariant measure of the stationary dynamics. More precisely, we have the following proposition.

\begin{prop}\label{uni}
Let us suppose that $\nu$ and $H$ satisfy the (F1)-diluteness condition. Then for any initial particle configuration $\sigma \in \mathcal{N}_H(S\times G)$ as $t \rightarrow +\infty$ we have
$$
\mathcal{K}^\sigma_t \overset{loc}{\longrightarrow} \mathcal{K}_0.
$$
\end{prop}

\begin{proof} Given a particle configuration $\sigma \in \mathcal{N}_H(S\times G)$ the idea is to construct a coupling of $\mathcal{K}^\sigma_t$ for each $t \geq 0$ together with $\mathcal{K}_0$ where the local convergence can be easily verified. To achieve this, recall that the forward FFG dynamics are built from a set $(L_{\gamma_x,i})_{(\gamma_x,i) \in [\sigma]}$ of exponential random variables of parameter 1 and a Poisson process $\overline{\Pi}$ on $\mathcal{C} \times [0,1]$ with intensity measure $\nu \times \mathcal{L}\times \mathcal{E}_1 \times \mathcal{L}_{[0,1]}$. More precisely, for each $\Lambda \in \B^0_S$ there exists a measurable function $\psi_\Lambda$ such that for each $t > 0$
$$
(\mathcal{K}^\sigma_t)_{\Lambda \times G} = \psi_\Lambda \left( \mathcal{A}^{[0,t]}_\sigma(\Lambda \times G), F(\mathcal{A}^{[0,t]}_\sigma(\Lambda \times G))\right)
$$ where $\mathcal{A}^{[0,t]}_\sigma(\Lambda \times G)$ denotes the clan of ancestors of $\Lambda \times G$ defined as in \eqref{clanbounded} but using $\Pi^{\sigma}$ instead of $\Pi$ and $F(\mathcal{A}^{[0,t]}_\sigma(\Lambda \times G))$ denotes its corresponding set of flags. Furthermore, both the clan of ancestors and its flags are determined by the evolution of the process $\overline{\Pi}^\sigma$ in the time interval $[0,t]$, i.e., there exists a second measurable function $\theta$ such that
$$
(\mathcal{K}^\sigma_t)_{\Lambda \times G}= \psi_\Lambda \circ \theta \left(\left(\overline{\Pi}^{\sigma}_s\right)_{s \in [0,t]}\right)
$$
Similarly, the stationary FFG process is defined for each $t \in \R$ and $\Lambda \in \B^0_S$ by the formula
\begin{equation}\label{statproc}
(\mathcal{K}_t)_{\Lambda \times G} = \psi_\Lambda \circ \theta \left( \left(\overline{\Pi}_s\right)_{s \in (-\infty,t]}\right).
\end{equation} We shall construct our coupling by taking the Poisson process $\overline{\Pi}$ together with the family $(L_{\gamma_x,i})_{(\gamma_x,i) \in [\sigma]}$ of independent exponential random variables of parameter 1 and for each $t > 0$ defining the $t$-shifted free process $\overline{\Pi}^{\sigma,(t)}$ with initial condition $\sigma$ by the formula $\overline{\Pi}^{\sigma,\,(t)} = \overline{\Pi}^{\sigma}_0 \cup \overline{\Pi}^{\,(t)}$ where
$$
\overline{\Pi}^{\,(t)} = \{ (\gamma_x, r + t , s, u) \in \mathcal{C} \times [0,1] : (\gamma_x, r , s, u) \in \overline{\Pi} \,,\, r > - t \}.
$$ We then define for each $t > 0$ the random particle configuration $X_t$ by the formula
$$
(X_t)_{\Lambda \times G} := \psi_\Lambda \circ \theta \left( \left(\overline{\Pi}^{\,\sigma,\,(t)}_s\right)_{s \in [0,t]} \right)
$$ for every $\Lambda \in \B^0_S$ and set $X_\infty := \mathcal{K}_0$, where $\mathcal{K}$ is defined exactly as in \eqref{statproc}, \mbox{i.e. using $\overline{\Pi}$.}
In other words, $X_t$ is the current state of the FFG process started at time $-t$ with initial condition $\sigma$ and underlying free process $\overline{\Pi}$, after having evolved for a time period of length $t$. Let us observe that each $X_t$ has the same distribution as $\mathcal{K}^\sigma_t$ by the time translational invariance of $\overline{\Pi}$. Moreover, this construction possesses a crucial property: the free processes $\overline{\Pi}^{\,\sigma,\,(t)}$ are ``coupled backwards'' with $\overline{\Pi}$, i.e. for $t > 0$ we have
\begin{equation}\label{coupleback}
\overline{\Pi}^{\,\sigma,\,(t)}_{s} = \overline{\Pi}_{s-t}.
\end{equation}
Using this property we shall prove that for any given $\Lambda \in \mathcal{B}^0_{S}$ one has $(X_t)_{\Lambda \times G} = (X_\infty)_{\Lambda \times G}$ for every $t$ sufficiently large, a fact from which we immediately obtain the validity of (i). Indeed, since the model is heavily diluted we know that $\mathcal{A}^0(\Lambda \times G)$ is finite almost surely.
In particular, there exist (random) $t^*$ and $\Lambda' \in \mathcal{B}^0_{S}$ such that
$$
\mathcal{A}^0(\Lambda \times G) \subseteq (\Lambda' \times G) \times (-t^*,0] \times \R^+.
$$Moreover, since the initial condition $\sigma$ has a locally finite interaction range with respect to $H$ we have that $\sigma(I(\Lambda'\times G)) < +\infty$ and so $t^\sigma_\Lambda := \sup_{(\gamma_x,i) \in [\sigma|_{I(B_\Lambda)}]} L_{\gamma_x,i} < +\infty$ as well. Therefore, if $t > t^* + t^\sigma_\Lambda$ then by \eqref{coupleback} we have that
$$
\mathcal{A}^{[0,t]}_\sigma(\Lambda \times G) = \{ (\gamma_x,r+t,s) : (\gamma_x,r,s) \in \mathcal{A}^{0}(\Lambda \times G)\}.
$$ i.e., all the cylinders in the initial configuration $\sigma$ which could possibly interact with the clan of ancestors of the cylinders in $(\Pi^{\,\sigma,\,(t)}_t)_{\Lambda \times G}$ in the forward dynamics die out before reaching it and, as a consequence, this clan coincides with the ancestors of cylinders in $(\Pi_0)_{\Lambda \times G}$ in the stationary dynamics modulo some appropriate time shift. In particular, we get that $(X_t)_{\Lambda \times G} = (X_\infty)_{\Lambda \times G}$ if $t > t^*+t^\sigma_\Lambda$ as we wanted to show.
\end{proof}

\section{Reversible measures for the FFG dynamics}

The purpose of this section is to study the relationship between invariant measures for the FFG dynamics and Gibbs measures of the associated diluted model. More precisely, we will show that Gibbs measures are reversible for the corresponding FFG dynamics. Together with Proposition \ref{uni}, this will imply the existence of a unique Gibbs measure in all heavily diluted models. We begin our task by introducing the global and local evolution semigroups for the dynamics. However, in order to define the global semigroup we need the infinite-volume forward dynamics to be well defined. Hence, for the rest of this section we assume that the model under consideration is well diluted.

\begin{defi} Given $t > 0$, a bounded subset $\Lambda \in \mathcal{B}^0_{S}$ and a particle configuration $\eta \in \mathcal{N}(S \times G)$ we define the operators $S_t$ and $S^{\Lambda|\eta}_t$ on the class of bounded functions $f:\mathcal{N}(S \times G) \to \R$ by the formulas
$$
S_t(f) (\sigma) = \E\left(f(\mathcal{K}_t^\sigma)\right) \hspace{2cm}\text{ and }\hspace{2cm}S_t^{\Lambda|\eta} (f) (\sigma) = \E ( f(\mathcal{K}^{\Lambda,\,\sigma,\,\eta}_t) )
$$ where $\mathcal{K}^\sigma$ denotes the infinite-volume forward process with initial condition $\sigma$ and $\mathcal{K}^{\Lambda,\,\sigma,\,\eta}$ is the forward process on $\Lambda$ with boundary condition $\eta$ and initial condition $\sigma_{\Lambda \times G} \cdot \eta_{\Lambda^c \times G}$.\footnote{The forward dynamics on a finite volume are defined in the natural manner, following the approach of Section \ref{forwarddynamics}.}

The families of operators $(S_t)_{t \geq 0}$ and $(S_t^{\Lambda|\eta})_{t \geq 0}$ are called the \textit{global evolution semigroup} and \textit{local evolution semigroup on $\Lambda$ with boundary condition $\eta$}, respectively.
\end{defi}

\begin{obs} Both families $(S_t)_{t \geq 0}$ and $(S_t^{\Lambda|\eta})_{t \geq 0}$ satisfy the semigroup property, i.e.
$$
S_{t} \circ S_s = S_{t+s} \hspace{2cm}\text{ and }\hspace{2cm}S_t^{\Lambda|\eta} \circ S_s^{\Lambda|\eta} = S_{t+s}^{\Lambda|\eta}
$$ for every $t,s \geq 0$, and hence their name.
\end{obs}

\begin{defi} Let $\mu$ be a measure on $\mathcal{N}(S \times G)$ and $\Lambda \in \mathcal{B}^0_{S}$.
\begin{enumerate}
\item [$\bullet$] We say that $\mu$ is \textit{invariant} for the global FFG dynamics if for every $t \geq 0$ we have
$$
\int_{\mathcal{N}(S \times G)} S_t(f)(\sigma) d\mu(\sigma) = \int_{\mathcal{N}(S \times G)} f(\sigma) d\mu(\sigma)
$$ for every bounded local function $f: \mathcal{N}(S \times G) \to \R$.
\item [$\bullet$] We say that $\mu$ is \textit{invariant} for the local FFG dynamics on $\Lambda$ with boundary condition $\eta \in \mathcal{N}(S \times G)$ if
\begin{enumerate}
\item [i.] $\mu( \{ \xi \in \mathcal{N}(\R^d \times G) : \xi_{\Lambda^c \times G}= \eta_{\Lambda^c \times G}\} ) = 1$
\item [ii.] For every $t \geq 0$ we have
$$
\int_{\mathcal{N}(S \times G)} S_t^{\Lambda|\eta}(f)(\sigma) d\mu(\sigma) = \int_{\mathcal{N}(S \times G)} f(\sigma) d\mu(\sigma)
$$ for every bounded $\mathcal{F}_{\Lambda \times G}$-measurable function $f: \mathcal{N}(S \times G) \to \R$.
\end{enumerate}
\item [$\bullet$] We say that $\mu$ is \textit{reversible} for the global FFG dynamics if it also satisfies
\begin{equation}\label{rever}
\int_{\mathcal{N}(S \times G)} g(\sigma) S_t(f)(\sigma) d\mu(\sigma) = \int_{\mathcal{N}(S \times G)} S_t(g)(\sigma) f(\sigma) d\mu(\sigma)
\end{equation} for every $t \geq 0$ and bounded local functions $f,g: \mathcal{N}(S \times G) \to \R$.
\item [$\bullet$] We say that $\mu$ is \textit{reversible} for the local FFG dynamics on $\Lambda$ with boundary condition $\eta \in \mathcal{N}(S \times G)$ if it is invariant and also satisfies for all $t \geq 0$
$$
\int_{\mathcal{N}(S \times G)} g(\sigma) S_t^{\Lambda|\eta}(f)(\sigma) d\mu(\sigma) = \int_{\mathcal{N}(S \times G)} S_t^{\Lambda|\eta}(g)(\sigma)f(\sigma) d\mu(\sigma)
$$ for every $t \geq 0$ and bounded $\mathcal{F}_{\Lambda \times G}$-measurable functions $f,g: \mathcal{N}(S \times G) \to \R$.
\end{enumerate}

\end{defi}

Our first step in studying the invariant measures for the FFG dynamics will be to derive an explicit formula for the generator of the local evolution semigroup valid for a sufficiently wide class of functions. Recall that, in general, the generator of a semigroup $(T_t)_{t \geq 0}$ is defined as
$$
L(f)(\sigma)= \frac{ dT_t(f)(\sigma) }{dt}\bigg|_{t=0} = \lim_{h \rightarrow 0^+} \frac{ T_{h}(f)(\sigma) - f(\sigma) }{h}
$$ whenever $f$ is such that the limit exists for every particle configuration $\sigma \in \mathcal{N}(S \times G)$.

\begin{prop}\label{generador} For every $\Lambda \subseteq \mathcal{B}^0_{S}$ and $\eta \in \mathcal{N}(S \times G)$ the local evolution semigroup $(S^{\Lambda|\eta}_t)_{t \geq 0}$ has generator $L_{\Lambda|\eta}$ defined for any bounded $\F_{\Lambda \times G}$-measurable $f: \mathcal{N}(S \times G) \to \R$ by the formula
$$
L_{\Lambda|\eta}(f)(\sigma) = D_{\Lambda|\eta}(f)(\sigma) + B_{\Lambda|\eta}(f)(\sigma)
$$ where
$$
D_{\Lambda|\eta}(f)(\sigma)=\sum_{ \gamma_x \in \langle \sigma_{\Lambda \times G} \rangle } \sigma(\gamma_x) \left(f( \sigma - \delta_{\gamma_x}) - f(\sigma)\right)
$$ and
$$
B_{\Lambda|\eta}(f)(\sigma)=\int_{\Lambda \times G} e^{-\Delta E_{\sigma_{\Lambda \times G} \,\cdot\, \eta_{\Lambda^c \times G}}(\gamma_x)}\left(f( \sigma + \delta_{\gamma_x}) - f(\sigma)\right) d\nu(\gamma_x).
$$
\end{prop}

\begin{proof} Given a bounded $\F_{\Lambda \times G}$-measurable function $f: \mathcal{N}(S \times G) \to \R$ we must show that for every particle configuration $\sigma \in \mathcal{N}(S \times G)$ we have
$$
\lim_{h \rightarrow 0^+} \frac{ \E( f(\mathcal{K}^{\Lambda,\,\sigma,\,\eta}_h) - f(\sigma)) }{h} = L(f)(\sigma).
$$ If we write $B = \Lambda \times G$ then notice that we have the decomposition
$$
\frac{ \E( f(\mathcal{K}^{\Lambda,\,\sigma,\,\eta}_h) -f(\sigma))}{h}= \sum_{k=0}^\infty \frac{ \E(( f(\mathcal{K}^{\Lambda,\,\sigma,\,\eta}_h) -f(\sigma))\mathbbm{1}_{B_k})}{h}
$$ where $B_k = \{ \Pi( B \times (0,h] \times \R_+) = k \}$ for each $k \in \N$. We shall deal with each of these terms separately. If for every $j=1,\dots,\sigma(B)$ we write $L^{(j)}_B$ for the $j$-th order statistic of the family $(L_{(\gamma_x,i)})_{(\gamma_x,i)\in [\sigma_B]}$ then the term with $k=0$ can be decomposed into two parts
$$
\frac{ \E(( f(\mathcal{K}^{\Lambda,\,\sigma,\,\eta}_h) -f(\sigma))\mathbbm{1}_{B_0})}{h} = \frac{ \E( (f(\mathcal{K}^{\Lambda,\,\sigma,\,\eta}_h) -f(\sigma))(\mathbbm{1}_{\{ L^{(1)}_B \leq h < L^{(2)}_B\} \cap B_0}+ \mathbbm{1}_{\{ L^{(2)}_B \leq h \}\cap B_0}))}{h}
$$ since $f(\mathcal{K}^{\Lambda,\,\sigma,\,\eta}_h)=f(\sigma)$ on $\{ L^{(1)}_B > h\}\cap B_0$. Let us observe that due to the independence between $(L_{(\gamma_x,i)})_{(\gamma_x,i)\in [\sigma_B]}$ and $\Pi$, the first term in the right hand side can be rewritten as
$$
\sum_{(\gamma_x,i) \in [\sigma_B]} (f(\sigma - \delta_{\gamma_x}) - f(\sigma)) \frac{(1-e^{-h})e^{-(\sigma(B)-1 + \nu(B))h}}{h}
$$ where $\nu(B) < +\infty$ by hypothesis. Thus, we obtain that
$$
\lim_{h \rightarrow 0^+} \frac{ \E(( f(\mathcal{K}^{\Lambda,\,\sigma,\,\eta}_h) -f(\sigma))\mathbbm{1}_{\{ L^{(1)}_B \leq h < L^{(2)}_B\} \cap B_0})}{h} = D_{\Lambda|\eta}(f)(\sigma).
$$ On the other hand, for the second term in the right hand side we have that
$$
\left| \frac{ \E(( f(\mathcal{K}^{\Lambda,\,\sigma,\,\eta}_h) -f(\sigma))\mathbbm{1}_{\{ L^{(2)}_B \leq h \}\cap B_0})}{h}\right| \leq \frac{2\|f\|_\infty \binom{\sigma(B)}{2} (1-e^{-h})^2}{h} \longrightarrow 0
$$ which establishes the case $k=0$. Now, to deal with the case $k=1$ notice that
$$
\frac{ \E(( f(\mathcal{K}^{\Lambda,\,\sigma,\,\eta}_h) -f(\sigma))\mathbbm{1}_{B_1})}{h} = \frac{ \E(( f(\mathcal{K}^{\Lambda,\,\sigma,\,\eta}_h) -f(\sigma))(\mathbbm{1}_{\{ L^{(1)}_B \leq h \} \cap B_1}+ \mathbbm{1}_{\{ L^{(1)}_B > h \}\cap B_1}))}{h}
$$ where the first term in the right hand side satisfies
$$
\Bigg|\frac{ \E(( f(\mathcal{K}^{\Lambda,\,\sigma,\,\eta}_h) -f(\sigma))\mathbbm{1}_{\{ L^{(1)}_B \leq h \} \cap B_1})}{h}\Bigg| \leq 2 \|f\|_\infty (1-e^{-\sigma(B)h})e^{-\nu(B) h}\nu(B) \longrightarrow 0
$$ and the second one equals
\begin{equation}\label{generaequ}
\frac{ \E(( f( \sigma + (\Pi_h)_{B}) -f(\sigma))\mathbbm{1}_{\{ L^{(1)}_B > h , \Pi( C(B,h,\,\sigma|_{\Lambda \times G} \,\cdot\, \eta|_{\Lambda^c \times G}) ) = 1 \} \cap B_1})}{h}
\end{equation} where for $\xi \in \mathcal{N}(S \times G)$ we set
$$
C(B,h,\xi) = \{ (\gamma_x,t,s) \in \mathcal{C} : \gamma_x \in B , F(\gamma_x,t,s) \leq M(\gamma_x | \xi) , 0 < t \leq h , t+s > h \}.
$$ By \eqref{poisson} the expression on \eqref{generaequ} can be rewritten as
$$
\frac{e^{-(\sigma(B)+\nu(B))h}}{h}  \int_0^h \int_{h-t}^\infty \left(\int_B e^{-\Delta E_{\sigma_{\Lambda \times G} \,\cdot\, \eta_{\Lambda^c \times G}}(\gamma_x)}\left(f( \sigma + \delta_{\gamma_x}) - f(\sigma)\right) d\nu(\gamma_x)\right) e^{-s} ds dt
$$ from where a simple calculation yields
$$
\lim_{h \rightarrow 0^+} \frac{ \E(( f(\mathcal{K}^{\Lambda,\,\sigma,\,\eta}_h) -f(\sigma))\mathbbm{1}_{\{ L^{(1)}_B > h \} \cap B_1})}{h} = B_{\Lambda|\eta}(f)(\sigma).
$$ Finally, to deal with the cases when $k > 1$ let us observe that
$$
\left|\sum_{k=2}^\infty \frac{ \E(( f(\mathcal{K}^{\Lambda,\,\sigma,\,\eta}_h) -f(\sigma))\mathbbm{1}_{B_k})}{h}\right| \leq \frac{2 \|f\|_\infty P( \Pi( B \times (0,h] \times \R_+) \geq 2 )}{h} \longrightarrow 0.
$$ Together with the previous cases, this allows us to conclude the proof.
\end{proof}

Having established a proper formula for the generator of the local evolution semigroups we now show that the Boltzmann-Gibbs distributions are reversible for the local dynamics. We will do so with the aid of the following two lemmas.

\begin{lema}\label{inv0} Let $\Lambda \in \B^0_S$ and $\eta \in \mathcal{N}(S \times G)$. Then for for every pair $f,g$ of bounded $\F_{\Lambda \times G}$-measurable functions the Boltzmann-Gibbs distribution $\mu_{\Lambda|\eta}$ satisfies
\begin{equation}\label{genera}
\int g L_{\Lambda|\eta}(f) d\mu_{\Lambda|\eta} = \int f L_{\Lambda|\eta}(g) d\mu_{\Lambda|\eta}.
\end{equation}
\end{lema}

\begin{proof} By Proposition \ref{generador} and symmetry it suffices to show that the two integrals
\begin{equation}\label{inta}
\int_{\mathcal{N}(S \times G)}\left[\sum_{ \gamma_x \in \langle \sigma_{\Lambda \times G} \rangle} \sigma(\gamma_x)g(\sigma) f( \sigma - \delta_{\gamma_x})\right]d\mu_{\Lambda|\eta}(\sigma)
\end{equation} and
\begin{equation}\label{intb}
\int_{\mathcal{N}(S\times G)} \left[\int_{\Lambda \times G} e^{-\Delta E_{\sigma_{\Lambda \times G} \,\cdot\, \eta_{\Lambda^c \times G}}(\gamma_x)}g(\sigma + \delta_{\gamma_x})f(\sigma)d\nu(\gamma_x)\right]d\mu_{\Lambda|\eta}(\sigma)
\end{equation} coincide.

A simple calculation using \eqref{poisson} yields that \eqref{inta} equals
$$
\frac{1}{Z_{\Lambda|\eta}}\sum_{n=1}^\infty \frac{e^{-\nu(B)}}{(n-1)!} \int_{B^n} g(\sigma^{(n)})f\left(\sigma^{(n)} - \delta_{\gamma^1_x}\right) e^{-H_{\Lambda|\eta}(\sigma^{(n)})} d\nu^n\left(\gamma_x^{(n)}\right)
$$ and \eqref{intb} equals
$$
\frac{1}{Z_{\Lambda|\eta}}\sum_{n=0}^\infty \frac{e^{-\nu(B)}}{n!} \int_{B^n}e^{-H_{\Lambda|\eta}(\sigma^{(n)})} f(\sigma^{(n)})\left(\int_{B} e^{-\Delta E_{\sigma_{B} \,\cdot\, \eta_{B^c}}(\tilde{\gamma}_y)}g(\sigma^{(n)} + \delta_{\tilde{\gamma}_y})d\nu\left(\tilde{\gamma}_y\right)\right) d\nu^{n}\left(\gamma_x^{(n)}\right)
$$ where we write $B=\Lambda \times G$, $\gamma_x^{(n)}=(\gamma_x^1,\dots,\gamma_x^n)$ and $\sigma^{(n)}= \sum_{i=1}^n \delta_{\gamma^i_x}$. The equality between \eqref{inta} and \eqref{intb} now follows upon a change of index in \eqref{inta} as a consequence of the Fubini-Tonelli theorem and (ii) in the consistent Hamiltonian property.
\end{proof}

\begin{lema}\label{inv1} Let $f:\mathcal{N}(S \times G) \to \R$ be a bounded local function. Then for each particle configuration $\eta \in \mathcal{N}_H(S \times G)$ and $t \geq 0$ we have
$$
S^{\Lambda|\eta}_t(f) \longrightarrow S_t(f)
$$ pointwise on $\mathcal{N}_H(S \times G)$ as $\Lambda \nearrow S$ for each $t \geq 0$.
\end{lema}

\begin{proof}
For $\Lambda \in \mathcal{B}^0_{S}$ consider the coupling of $\mathcal{K}^\sigma_t$ and $\mathcal{K}^{\Lambda,\,\sigma,\,\eta}_t$ obtained by constructing these random particle configurations using the same Poisson process $\overline{\Pi}$ and exponential lifetimes $(L_{(\gamma_x,i)})_{(\gamma_x,i) \in [\sigma_{\Lambda \times G}]}$ for particles in $\sigma_{\Lambda \times G}$. Recall that by Proposition \ref{finit2} we know that if $\sigma$ has a finite local interaction range then $\mathcal{A}^{[0,t]}_\sigma(\Lambda_f \times G)$ is almost surely finite and thus there exists $\Lambda' \in \B^0_S$ such that the basis of every cylinder in $\mathcal{A}^{[0,t]}_\sigma (\Lambda_f \times G)$ is contained in $\Lambda' \times G$. Furthermore, since $\eta$ is also of finite local interaction range, by taking any $\Lambda \in \mathcal{B}^0_{S}$ sufficiently large so that $\Lambda' \subseteq \Lambda$ and $\eta( I(\Lambda' \times G) \cap (\Lambda^c \times G) ) = 0$ we have that $(\mathcal{K}^{\Lambda,\,\sigma,\,\eta}_t)_
{\Lambda_f \times G} = (\mathcal{K}^{\sigma}_t)_{\Lambda_f \times G}$ and thus $f(\mathcal{K}^{\Lambda,\,\sigma,\,\eta}_t) = f(\mathcal{K}^{\sigma}_t)$. In particular, we conclude that $f(\mathcal{K}^{\Lambda,\,\sigma,\,\eta}_t)$ converges almost surely to $f(\mathcal{K}^{\sigma}_t)$ as $\Lambda \nearrow S$. The assertion now follows at once from the dominated convergence theorem since $f$ is bounded.
\end{proof}

We are now ready to show the reversibility of Gibbs measures for the FFG dynamics.

\begin{teo}\label{teoreversibilidad} Any Gibbs measure $\mu$ for the diluted model $(\nu,H)$ is reversible for the global dynamics.
\end{teo}

\begin{proof} Let $f$ and $g$ be bounded local functions and consider $\Lambda_0 \in \mathcal{B}^0_{S}$ such that $f$ and $g$ are both $\F_{\Lambda_0 \times G}$-measurable. Notice that by (i) in Lemma \ref{inv1}, Proposition \ref{lfir} and the dominated convergence theorem we have that \eqref{rever} will hold if we show that
\begin{equation}\label{rever1}
\int_{\mathcal{N}(S \times G)} g(\sigma) S^{\Lambda,\,\sigma}_t(f)(\sigma) d\mu(\sigma) = \int_{\mathcal{N}(S \times G)} f(\sigma) S^{\Lambda,\,\sigma}_t(g)(\sigma) d\mu(\sigma)
\end{equation} is satisfied for every $\Lambda \in \mathcal{B}^0_{S}$ with $\Lambda_0 \subseteq \Lambda$. Moreover, since $\mu$ is a Gibbs measure then we can rewrite \eqref{rever1} as
$$
\int \int g(\sigma) S_t^{\Lambda|\eta}(f)(\sigma) d\mu_{\Lambda|\eta}(\sigma) d\mu(\eta)= \int \int f(\sigma) S_t^{\Lambda|\eta}(g)(\sigma) d\mu_{\Lambda|\eta}(\sigma) d\mu(\eta)
$$ so that \eqref{rever} will follow if we prove that for each $\eta \in \mathcal{N}(S \times G)$
\begin{equation}\label{rever2}
\int_{\mathcal{N}(S \times G)}  g(\sigma) S_t^{\Lambda|\eta}(f)(\sigma) d\mu_{\Lambda|\eta}(\sigma) = \int_{\mathcal{N}(S \times G)}  f(\sigma) S_t^{\Lambda|\eta}(g)(\sigma) d\mu_{\Lambda|\eta}(\sigma)
\end{equation} holds for every $t \geq 0$ and $\Lambda \in \mathcal{B}^0_{S}$ containing $\Lambda_0$. In order to simplify the notation ahead, we shall fix $\Lambda \in \mathcal{B}^0_{S}$ containing $\Lambda_0$, $\eta \in \mathcal{N}(S\times G)$ and through the rest of the proof write $S_t := S^{\Lambda,\eta}_t$ for each $t \geq 0$ and $L:=L_{\Lambda | \eta}$. Now, to establish \eqref{rever2}, given $0 < h \leq t$ let us begin by rewriting the left-hand side of \eqref{rever2} as
\begin{equation}\label{rever3}
h \int  g L(S_{t-h}(f))d\mu_{\Lambda|\eta} + \int  g\left[R_h(S_{t-h}(f)) + S_{t-h}(f)\right] d\mu_{\Lambda|\eta}
\end{equation} where for any bounded $\F_{\Lambda \times G}$-measurable function $u:\mathcal{N}(S \times G) \to \R$ the \textit{error term} $R_h(u): \mathcal{N}(S \times G) \to \R$ is defined by the formula
$$
R_h(u)(\sigma) = S_h(u)(\sigma) - u(\sigma) - L(u)(\sigma)h.
$$ By Lemma \ref{inv0} we obtain that \eqref{rever3} equals
\begin{equation}\label{rever4}
h \int  L(g) S_{t-h}(f)d\mu_{\Lambda|\eta} + \int  g\left[R_h(S_{t-h}(f)) + S_{t-h}(f)\right] d\mu_{\Lambda|\eta}
\end{equation}which, upon performing computations analogous to those made to obtain \eqref{rever3} but in reverse order, turns into
\begin{equation}\label{rever5}
\int  S_h(g) S_{t-h}(f)d\mu_{\Lambda|\eta} + \int  \left[ gR_h(S_{t-h}(f)) - R_h(g) S_{t-h}(f)\right]d\mu_{\Lambda|\eta}.
\end{equation}
By iterating this procedure we ultimately obtain
\begin{equation}\label{reverfinal}
\int  g S_t(f) d\mu_{\Lambda|\eta} = \int S_t(g) f d\mu_{\Lambda|\eta} + \int R_{t,h}(g,f)d\mu_{\Lambda|\eta}
\end{equation} where
$$
R_{t,h}(g,f) = \sum_{k=1}^{\lceil\frac{t}{h}\rceil} \left( S_{(k-1)h} (g) R_h(S_{t-kh}(f)) - R_h(S_{(k-1)h} (g)) S_{t-kh}(f)\right).\footnote{There is a slight abuse of notation in the last term of the sum. The term corresponding to $k=\lceil\frac{t}{h}\rceil$ is actually $\left[S_{[\frac{t}{h}]h} (g) R_{t-[\frac{t}{h}]h}(f) - R_{t-[\frac{t}{h}]h}(S_{[\frac{t}{h}]h} (g)) f\right]$.}
$$
Now, let us observe that from the proof of Proposition \ref{generador} we get that for each bounded $\F_{\Lambda \times G}$-measurable function $u:\mathcal{N}(S \times G) \to \R$ and $\sigma \in \mathcal{N}(S \times G)$ there exists a positive constant $C$ depending only on $\nu(\Lambda \times G)$ such that for any $\eta \in \mathcal{N}(S \times G)$ and $0 < h < 1$
$$
|R_h(u)(\sigma)| \leq C\|u\|_\infty (1+\sigma^2(\Lambda \times G))  h^2.
$$ But since
\begin{align*}
\int \sigma^2(\Lambda \times G) d\mu_{\Lambda|\eta}(\sigma) &\leq \frac{1}{Z_{\Lambda|\eta}}\int [\sigma(\Lambda \times G)]^2 d\pi^\nu_{\Lambda \times G}(\sigma) \\
\\
& \leq \frac{ \nu(\Lambda \times G) + \nu^2(\Lambda \times G)}{\pi^\nu ( N_{\Lambda \times G} = 0)} = e^{\nu(\Lambda \times G)}(\nu(\Lambda \times G) + \left(\nu(\Lambda \times G)\right)^2) < +\infty
\end{align*} we obtain that there exists another positive constant $\hat{C}_t$, this time depending only on $\nu(\Lambda \times G)$ and $t$, such that for $0 < h < \min\{t,1\}$ we have
$$
\left|\int R_{t,h}(g,f)d\mu_{\Lambda|\eta}\right| \leq \hat{C}_t \|g\|_\infty \|f\|_\infty h \underset{h \to \,0}{\longrightarrow} 0.
$$ Since the left-hand side of \eqref{reverfinal} does not depend on $h$, by letting $h \rightarrow 0$ we conclude the proof.
\end{proof}

As a consequence of Theorem \ref{teoreversibilidad} we obtain the following important result.

\begin{teo}[Uniqueness of Gibbs measures in heavily diluted models] \label{teounigibbs} $\\$
Let $(\nu,H)$ be a heavily diluted model on $\mathcal{N}(S \times G)$ and let $\mu$ be the invariant measure of the associated stationary global dynamics. Then the following holds:
\begin{enumerate}
\item [i.] For each $\Lambda \in \mathcal{B}^0_{S}$ and $\eta \in \mathcal{N}(S \times G)$ the Boltzmann-Gibbs distribution $\mu_{\Lambda|\eta}$ is the unique invariant measure of the local dynamics on $\Lambda$ with boundary condition $\eta$.
\item [ii.] For any $\eta \in \mathcal{N}_H(S \times G)$ we have $\mu_{\Lambda|\eta} \overset{loc}{\longrightarrow} \mu$ as $\Lambda \nearrow S$.
\item [iii.] $\mu$ is the unique Gibbs measure for the model.
\end{enumerate}
\end{teo}

\begin{obs} Statements (i) and (ii) in Theorem \ref{teounigibbs} were obtained in \cite{FFG1} in \mbox{the case of} the Ising contours model, while (iii) is a new result which we present here. Notice that, in general, (iii) is not a direct consequence of (ii) since the diluted model might not be of bounded local interaction range (see Proposition \ref{limitegibbs}).
\end{obs}

\begin{proof} Notice that, since the clans of ancestors are always finite for the local dynamics over any bounded region $\Lambda \in \mathcal{B}^0_{S}$, we can mimic the proof of Proposition \ref{uni} to show that, for any $\eta \in \mathcal{N}(S \times G)$ and any bounded local function $f$, the local evolution $S^{\Lambda,\, \eta}_t (f)$ converges pointwise as $t \rightarrow +\infty$ to $\E( f( \mathcal{K}^{\Lambda|\eta}_0) )$, where $\mathcal{K}^{\Lambda|\eta}$ is the stationary local process on $\Lambda$ with boundary condition $\eta$. In particular, this implies that there is a unique invariant measure for the local dynamics on $\Lambda$ with boundary condition $\eta$ which coincides with the distribution of $\mathcal{K}^{\Lambda|\eta}$. Since by \eqref{rever2} in the proof of Theorem \ref{teoreversibilidad} we know that $\mu_{\Lambda|\eta}$ is invariant for these dynamics, we conclude (i).

To see (ii), let us fix $\eta \in \mathcal{N}_H(S \times G)$ and given $\Lambda \in \mathcal{B}^0_{S}$ consider the coupling of $\mu$ and $\mu_{\Lambda|\eta}$ obtained by constructing the stationary local process $\mathcal{K}^{\Lambda|\eta}$ and the stationary global process $\mathcal{K}$ (which is well defined due to the heavy diluteness condition) using the same underlying process $\overline{\Pi}$. Now, given some bounded local function $f$, Proposition \ref{finit1} gives that $\mathcal{A}^0(\Lambda_f \times G)$ is almost surely finite and thus there exists $\Lambda' \in \B^0_S$ such that the basis of every cylinder in $\mathcal{A}^0(\Lambda_f \times G)$ is contained in $\Lambda' \times G$. Furthermore, since $\eta$ is also of finite local interaction range, by taking any $\Lambda \in \mathcal{B}^0_{S}$ sufficiently large so that $\Lambda' \subseteq \Lambda$ and $\eta( I(\Lambda' \times G) \cap (\Lambda^c \times G) ) = 0$ we obtain that $(\mathcal{K}^{\Lambda|\eta}_0)_{\Lambda_f \times G} = (\mathcal{K}_0)_
{\Lambda_f \times G}$. Thus, we get that $f(\mathcal{K}^{\Lambda|\eta}_0)$ converges almost surely to $f(\mathcal{K}_0)$ as $\Lambda \nearrow S$ and, since $f$ is bounded, by the dominated convergence theorem and (i) we obtain (ii).

\mbox{To establish (iii)}, let us first show that $\mu$ is a Gibbs measure. Given $\Delta \in \B^0_S$ and a \mbox{local event} $A \in \F$, consider the mapping $g_{\Delta,A}:\mathcal{N}(S\times G) \rightarrow [0,1]$ given by the formula
$$
g_{\Delta,A}(\xi)=\mu_{\Delta|\xi}(A).
$$ By the proof of Proposition \ref{limitegibbs} we know that $g_{\Delta,A}$ is $\F_{(\Lambda_A \times G) \cup I(\Delta \times G)}$-measurable. Thus, to see that $\mu$ is a Gibbs measure it suffices to show that for any $\Lambda \in \B^0_S$ sufficiently large we have that $(\mathcal{K}^{\Lambda|\eta}_0)_{(\Lambda_A \times G) \cup I(\Delta \times G)} = (\mathcal{K}_0)_{(\Lambda_A \times G) \cup I(\Delta \times G)}$ holds for some $\eta \in \mathcal{N}_H(S \times G)$. Indeed, if this is the case then
$$
\lim_{\Lambda \nearrow S} g_{\Delta,A} (\mathcal{K}^{\Lambda|\eta}_0) = g_{\Delta,A}(\mathcal{K}_0)
$$ and so by (ii), the consistency of the Boltzmann-Gibbs distributions and the dominated convergence theorem we have that
$$
\mu(A) = \lim_{\Lambda \nearrow S} \mu_{\Lambda|\eta}(A) = \lim_{\Lambda \nearrow S} \E(g_{\Delta,A} (\mathcal{K}^{\Lambda|\eta}_0)) = \E(g_{\Delta,A}(\mathcal{K}_0)) = \int \mu_{\Delta|\xi}(A) d\mu(\xi)
$$
which proves that $\mu$ is a Gibbs measure. Now, since $\Pi_0( I(\Delta \times G) )< +\infty$ by the integrable local interaction range assumption, there exists $\Lambda^* \in \B^0_S$ such that $(\Pi_0)_{I(\Delta \times G)}$ is contained in $\Lambda^* \times G$. Hence, since $\mathcal{A}^0((\Lambda_A \cup \Lambda^*) \times G)$ is almost surely finite, by taking any $\Lambda \in \mathcal{B}^0_{S}$ sufficiently large so that the basis of every cylinder in $\mathcal{A}^0((\Lambda_A \cup \Lambda^*) \times G)$ is contained in $\Lambda \times G$ we conclude that $(\mathcal{K}^{\Lambda,\,\emptyset}_0)_{(\Lambda_A \times G) \cup I(\Delta \times G)} = (\mathcal{K}_0)_{(\Lambda_A \times G) \cup I(\Delta \times G)}$ as we wished to see. Having shown that $\mu$ is a Gibbs a measure, it only remains to show that it is unique. But if $\tilde{\mu}$ is a Gibbs measure for the model then by Proposition \ref{uni} and assertion (ii) of Proposition \ref{lfir} for any bounded local function $f$ we have that
$$
\lim_{t \rightarrow +\infty} \int_{\mathcal{N}(S \times G)} S_t (f) (\sigma) d\tilde{\mu} (\sigma) = \int_{\mathcal{N}(S \times G)} f(\sigma) d\mu(\sigma).
$$ Since $\tilde{\mu}$ must also be an invariant measure for the global dynamics by Theorem \ref{teoreversibilidad}, from this we obtain $\tilde{\mu}=\mu$ and thus we conclude (iii).
\end{proof}

\begin{obs} Notice that in Theorem \ref{teounigibbs} we have actually showed that $\mu$ is the only invariant measure for the global dynamics which is supported on $\mathcal{N}_H(S \times G)$.
\end{obs}

\section{Exponential mixing of Gibbs measures}

Our next goal is to study mixing properties for Gibbs measures of translation invariant heavily diluted models. Let us begin by settling what we understand by mixing properties and translation invariant models in this context.
Throughout this section we assume that the allocation space is endowed with an operation $ + : S \times S \rightarrow S$ such that $(S,+)$ is a commutative group.

\begin{defi} We say that a diluted model $(\nu,H)$ is \textit{translation invariant} if it satisfies the following properties:
\begin{enumerate}
\item [i.] $\nu$ is translation invariant: for every $a \in S$ we have
\begin{equation}\label{tidm}
\nu = \nu \circ \tau_a^{-1}
\end{equation} where $\tau_a : S \times G \rightarrow S \times G$ is defined as $\tau_a(\gamma_x) = \gamma_{x+a}$.
\item [ii.] $H$ is translation invariant: for every $\Lambda \in \B^0_{S}$, $\eta \in \mathcal{N}(S \times G)$ and $a \in S$
$$
H_{\Lambda|\eta}( \tau_{-a} (\sigma )) = H_{\Lambda + a| \tau_a(\eta)}(\sigma)
$$ for all $\sigma \in \mathcal{N}\left((\Lambda + a) \times G\right)$ with $\Lambda+a:=\{ x + a : x \in \Lambda\}$, where for any given particle configuration $\xi$ we define $\tau_a ( \xi)$ through the standard representation
$$
\tau_a(\xi) = \sum_{\gamma_x \in Q_\xi} m_\xi(\gamma_x) \delta_{\tau_a(\gamma_x)}.
$$
\end{enumerate}
\end{defi}
Notice that as a straightforward consequence of Theorem \ref{teounigibbs} we conclude that if $(\nu,H)$ is a translation invariant heavily diluted model then its unique Gibbs measure is translation invariant in the sense of \eqref{tidm}. Also, observe that all models introduced in Section \ref{examples} are translation invariant.

\begin{defi} A measure $\mu$ is said to satisfy the \textit{exponential mixing property} when there exist $c_1,c_2 > 0$ such that for any pair of bounded local functions $f,g : \mathcal{N}(S \times G) \rightarrow \R$ with $d_S(\Lambda_f,\Lambda_g)$ sufficiently large (depending only on $c_1$) one has that
\begin{equation}\label{eqmixing}
\left|\int f(\eta)g(\eta) d\mu(\eta) - \int f(\eta) d\mu(\eta)\int g(\eta) d\mu(\eta)\right| \leq \| f\|_\infty \|g\|_\infty e^{-c_1 d_S(\Lambda_f, \Lambda_g ) + c_2(\nu(\Lambda_f) + \nu(\Lambda_g))}.
\end{equation}
\end{defi}

Our purpose in this section is to show that Gibbs measures of translation invariant heavily diluted models satisfy the exponential mixing property. To achieve this, however, the size function $q$ will need to satisfy some further conditions. The necessary requirements are contained in Definition \ref{defigoodsize} below.

\begin{defi} Consider the space $\M$ of particle configurations $\zeta$ on $\mathcal{C}$ such that
\begin{enumerate}
\item [$\bullet$] $\zeta$ is finite
\item [$\bullet$] No two cylinders in $\zeta$ have the same time of birth.
\end{enumerate} Let us order the cylinders of $\zeta$ by time of birth, $\zeta = \{ C_1 , \dots, C_k \}$ with $b_{C_k} < \dots < b_{C_1}$. We say that $\zeta$ is an \textit{ancestor family} if for every $j=2,\dots,k$ there exists $i < j$ such that $C_j$ is a first generation ancestor of $C_i$. The cylinder $C_1$ shall be referred to \mbox{as the \textit{root} of $\zeta$:} all cylinders of $\zeta$ are ancestors of $C_1$.
\end{defi}

\begin{defi}\label{defigoodsize} Given a diluted model on $\mathcal{N}(S \times G)$ we say that \mbox{$q : S \times G \rightarrow \R$} is a \textit{good size function} for the model if it satisfies the following properties:
\begin{enumerate}
\item [i.] $\inf_{\gamma_x \in S \times G} q(\gamma_x) \geq 1$.
\item [ii.] Given two ancestor families $\zeta, \zeta' \in \M$ if
$$
\sum_{C \in \zeta} q(basis(C)) + \sum_{C' \in \zeta'} q(basis(C')) < d_S\left( \pi_S\left(basis(C_1)\right),\pi_S\left(basis(C'_1)\right)\right)
$$ then none of the bases of $\zeta$ have an impact on any of the bases of $\zeta'$ and viceversa.
\mbox{Here $d_S$} denotes the metric in $S$ and $\pi_S : S \times G \rightarrow S$ is the projection onto $S$.
\item [iii.] There exist $b_1,b_2 > 0$ such that for any $\Lambda \in \B_S^0$
$$
\E \left( e^{b_1 \sum_{C \in \mathcal{A}^0(\Lambda \times G)} q(C)} \right) \leq e^{b_2 \nu(\Lambda \times G)}.
$$
\end{enumerate}
\end{defi}

The idea behind property (ii) is that, if $q$ is a good size function for the model, given an ancestor family $\zeta \in \M$ the quantity $\sum_{C \in \zeta} q(basis(C))$ should represent in some way the ``reach'' (or size) of the family. Hence, it is natural to ask that whenever the combined sizes of two ancestor families cannot overcome the distance between their roots then neither of the families has an impact on the other. On the other hand, notice that properties (i) and (iii) imply that the model under consideration admits exactly one Gibbs measure (see proof of Proposition \ref{finit1}).

\begin{teo}\label{teomixing} If $(\nu,H)$ is a translation invariant heavily diluted model with \mbox{respect to} a good size function $q$ then its unique Gibbs measure is exponentially mixing.
\end{teo}

\begin{proof} Recall that for any $\Lambda \in \B^0_S$ there exists a measurable function $\psi_\Lambda$ such that
$$
\left(\mathcal{K}_0\right)_{\Lambda \times G} = \psi_\Lambda\left( \mathcal{A}^0_F(\Lambda) \right)
$$ where we use the notation $\mathcal{A}^0_F(\Lambda):= \left( \mathcal{A}^0(\Lambda \times G), F(  \mathcal{A}^0(\Lambda \times G) ) \right).$ Keeping this in mind, the left hand side of \eqref{eqmixing} can be rewritten as
\begin{equation}\label{eqmixing2}
\left| \E( f( \psi_{\Lambda_f}( \mathcal{A}^0_F(\Lambda_f))) g( \psi_{\Lambda_g}(\mathcal{A}^0_F(\Lambda_g))) - f( \psi_{\Lambda_f}( \mathcal{A}^0_F(\Lambda_f)))g( \psi_{\Lambda_g}(\tilde{\mathcal{A}}^0_F(\Lambda_g))) ) \right|
\end{equation} for $\tilde{\mathcal{A}}^0_F(\Lambda_g)$ carrying the same distribution as $\mathcal{A}^0_F(\Lambda_g)$ while being independent of $\mathcal{A}^0_F(\Lambda_f)$.
Furthermore, if we construct the triple so as to also verify that
$$
\mathcal{A}^0_F(\Lambda_f) \sim \mathcal{A}^0_F(\Lambda_g) \Longrightarrow \mathcal{A}^0_F(\Lambda_g) = \tilde{\mathcal{A}}^0_F(\Lambda_g),
$$ where $\mathcal{A}^0_F(\Lambda_f) \sim \mathcal{A}^0_F(\Lambda_g)$ means that none of the bases of $\mathcal{A}^0_F(\Lambda_f)$ have \mbox{an impact on any of} the bases of $\mathcal{A}^0_F(\Lambda_g)$ and viceversa, then we obtain the bound
\begin{equation}\label{eqmixing3}
\left|\int f(\eta)g(\eta) d\mu(\eta) -\int f(\eta) d\mu(\eta)\int g(\eta) d\mu(\eta)\right| \leq 2 \| f \|_\infty \| g \|_\infty P( \mathcal{A}^0_F(\Lambda_f) \not \sim \mathcal{A}^0_F(\Lambda_g)).
\end{equation} The construction of this triple is similar in spirit to the one in the \mbox{Domination Lemma}. The idea is to construct the families $\mathcal{A}^0_F(\Lambda_f)$ and $\mathcal{A}^0_F(\Lambda_g)$ using the same free process $\Pi$ and then to obtain $\tilde{\mathcal{A}}^0_F(\Lambda_g)$ by replacing those cylinders in $\mathcal{A}^0_F(\Lambda_g)$ which have an impact on $\mathcal{A}^0_F(\Lambda_f)$ (or receive an impact from $\mathcal{A}^0_F(\Lambda_f)$) and their ancestors with cylinders belonging to an independent \mbox{free process $\Pi'$.} We refer to \cite{FFG1}.

Therefore, it suffices to produce a suitable bound for the right hand side of \eqref{eqmixing3}.
Now, since $q$ is a good size function we have
\begin{align*}
P( \mathcal{A}^0_F(\Lambda_f) \not \sim \mathcal{A}^0_F(\Lambda_g) ) & \leq P \left(  \sum_{C \in \mathcal{A}^0(\Lambda_f \times G)} q(C)  + \sum_{C \in \mathcal{A}^0(\Lambda_g \times G)} q(C) \geq d_S( \Lambda_f, \Lambda_g ) \right)\\
\\
& \leq P\left( \sum_{C \in \mathcal{A}^0(\Lambda_f \times G)} q(C) \geq \frac{d_{f,g}}{2} \right) + P\left( \sum_{C \in \mathcal{A}^0(\Lambda_g \times G)} q(C)\geq \frac{d_{f,g}}{2} \right)
\end{align*} where we use the notation $d_{f,g}:= d_S(\Lambda_f,\Lambda_g)$. Now, since $q$ is a good size function, by the exponential Tchebychev inequality we have that there exist $b_1,b_2 > 0$ such that for any bounded local function $h:\mathcal{N}(S \times G)$ one has the estimate
\begin{equation}\label{eqmixing4}
P\left(\sum_{C \in \mathcal{A}^0(\Lambda_h \times G)} q(C) \geq r \right) \leq e^{- b_1 r  + b_2 \nu(\Lambda_h)}
\end{equation} for every $r > 0$, which yields the bound
$$
P( \mathcal{A}^0_F(\Lambda_f)\not \sim \mathcal{A}^0_F(\Lambda_g) ) \leq e^{- \frac{b_1}{2} d_S(\Lambda_f,\Lambda_g)}( e^{b_2 \nu(\Lambda_f)} + e^{b_2 \nu(\Lambda_g)}) \leq 2 e^{- \frac{b_1}{2} d_S(\Lambda_f,\Lambda_g) + b_2(\nu(\Lambda_f)+\nu(\Lambda_g))}
$$ Thus, by taking $c_1=\frac{b_1}{3}$, $c_2=b_2$ and $d_S(\Lambda_f,\Lambda_g)$ sufficiently large we conclude \eqref{eqmixing}.
\end{proof}

\begin{obs}\label{obsmixing}Even though Theorem \ref{teounigibbs} ensures that finding any size function $q$ which satisfies the (F1)-diluteness condition will be enough to conclude that the corresponding model is heavily diluted, if one wishes to obtain further properties of the unique Gibbs measure such as \eqref{eqmixing}, then it is important for the size function $q$ to be of geometrical relevance within the context of the model, for example as (ii) in Definition \ref{defigoodsize} suggests.
\end{obs}
We would like to point out that whenever checking if a certain size function is indeed a good size function for a \mbox{given model}, condition (iii) will be in general the \mbox{hardest to verify.} The following result \mbox{proved in \cite[Theorem~2.1]{NSS}} will be of much aid to us in this matter.

\begin{defi} A \textit{Galton-Watson} process is a family $Z=(Z_n)_{n \in \N_0}$ of random variables taking values in $\N_0$ which satisfy for every $n \in \N_0$ the recurrence formula
$$
Z_{n+1} = \sum_{i=1}^{Z_n} X^{(n+1)}_i
$$ where $(X^{(n+1)}_i)_{(i,n) \in \N \times \N_0}$ is a sequence of i.i.d. random variables taking values in $\N_0$. \mbox{The distribution of} $Z_0$ is called the \textit{initial distribution} of the Galton-Watson process $Z$, while the distribution of the random variables $X^{(n)}_i$ is called its \textit{offspring distribution}.
\end{defi}

\begin{teo}\label{expbranching} Let $Z$ be a Galton-Watson process with some offspring distribution $X$. If $Z_0 \equiv 1$ and $\E(X) < 1$ then there exists $b > 0$ which satisfies $\E( e^{b \sum_{n \in \N_0} Z_n} ) < +\infty$ \mbox{if and only if} \mbox{there exists $s > 0$} such that $\E( e^{s X} ) < +\infty$.
\end{teo}
The following corollary illustrates the use of Theorem \ref{expbranching} in this context.

\begin{cor}\label{corbranching} Let $(\nu,H)$ be a diluted model satisfying the following properties:
\begin{enumerate}
\item [$\bullet$] $\nu(I(\{ \gamma_x \})) = \nu(I(\{ \tilde{\gamma}_y \}))$ for every pair $\gamma_x,\tilde{\gamma}_y \in S \times G$.
\item [$\bullet \bullet$] $\nu(I(\{ \gamma_x\}) < e^{\Delta E}$ for every $\gamma_x \in S \times G$.
\end{enumerate} Then there exist $b_1,b_2 > 0$ such that for any $\Lambda \in \B^0_S$
\begin{equation}\label{corbranching1}
\E \left( e^{b_1 \#(\mathcal{A}^0(\Lambda \times G))} \right) \leq e^{b_2 \nu(\Lambda \times G)}.
\end{equation}In particular, any bounded size function for $(\nu,H)$ satisfies (iii) in Definition \ref{defigoodsize}.
\end{cor}
\begin{proof} Given $\Lambda \in \B^0_S$, let us consider the family $\B$ constructed in the Domination lemma. By definition of $\B=(\B_n)_{n \in \N_0}$, for any $b_1 > 0$ we have that
$$
\E \left( e^{b_1 \#(\mathcal{A}^0(\Lambda \times G))} \right) \leq \E \left( e^{b_1 \# \B} \right) = \E \left( \prod_{C \in \B_0} \E( e^{b_1 \#\B(C)} ) \right).
$$ Now, by the assumptions on the pair $(\nu,H)$, we have that for every $C \in \mathcal{C}$ the family $Z^C=(Z_n^C)_{n \in \N_0}$ defined by the formula $Z_n^C = \# \B_n(C)$ is a Galton-Watson process whose distribution does not depend on $C$. Indeed, $Z^C$ has initial value 1 for all $C \in \mathcal{C}$ and Poisson offspring distribution with mean $e^{- \Delta E}\nu( I(\{basis(C)\}) )$ which does not depend on $C$.
Furthermore, since $\nu(I(\{ basis(C)\}) < e^{\Delta E}$ by assumptions and the Poisson distribution has well defined exponential moments, for $b_1 > 0$ sufficiently small Theorem \ref{expbranching} yields the existence of $\tilde{b}_2 > 1$ such that for all $C \in \C$ we have
$$
\tilde{b}_2:=\E( e^{b_1 \# \B(C)} ) < +\infty.
$$
\newpage
Hence, we obtain that
$$
\E \left( e^{b_1 \#(\mathcal{A}^0(\Lambda \times G))} \right) \leq \E( \tilde{b}_2^{\# \B_0} ) = e^{(\tilde{b}_2 - 1) e^{-\Delta E}\nu(\Lambda \times G)}
$$ since $\# \B_0$ has Poisson distribution with mean $e^{-\Delta E}\nu(\Lambda \times G)$. By taking $b_2 = (\tilde{b}_2 - 1)e^{-\Delta E}$ we conclude the proof.
\end{proof}

The hypotheses in Corollary \ref{corbranching} may seem restrictive, but in fact they can be relaxed: conditions ($\bullet$) and ($\bullet \bullet$) can be replaced by the weaker condition
\begin{enumerate}
\item [($*$)] \begin{center}$\sup_{\gamma_x \in S \times G} \nu(I(\{\gamma_x\})) < e^{\Delta E}.$\end{center}
\end{enumerate}
Indeed, if $(*)$ holds then one can obtain \eqref{corbranching1} by enlarging $\B$ so that each individual in the enlarged  process has Poisson offspring distribution with mean $e^{-\Delta E}\sup_{\gamma_x \in S \times G} \nu( I(\{\gamma_x\}))$. We leave the details to the reader.

\section{Applications}\label{teounigibssexamples}

In the following we investigate which conditions are implied by Theorems \ref{teounigibbs} and \ref{teomixing} for the existence of a unique Gibbs measure and its exponential mixing property in some of the models introduced in Section \ref{examples}. Bearing Remark \ref{obsmixing} in mind, in each of the examples we proceed as follows: first we propose a function $q$ satisfying (i) and (ii) in Definition \ref{defigoodsize} and then we investigate under which choice of parameters in the model are the (F1)-diluteness condition and (iii) in Definition \ref{defigoodsize} also satisfied.

\subsection{Widom-Rowlinson model with generalized interactions}

Let us suppose that we have supp($h$)$=[0,m_h]$ and supp($j$)$=[0,m_j]$ for some $m_h,m_j > 0$. Then a proper choice of size function for this model would be $q(\gamma_x) = \max \{ 1 , m_h, m_j\}$. Now, by \eqref{wru} and the fact that $h$ and $j$ are both nonnegative we obtain that $\Delta E = 0$, which implies that
$$
\alpha_{WR}(\lambda^{\pm}, h,j) = 2^d \max \{ \lambda_+ m_j^d + \lambda_- \max\{ m_h^d , m_j^d \} , \lambda_- m_j^d + \lambda_+ \max\{ m_h^d , m_j^d \} \}.
$$ Furthermore, since $q$ is constant we see that if $\alpha_{WR}(\lambda^{\pm}, h,j) < 1$ then ($*$) holds so that $q$ is a good size function for the model. Thus, we arrive at the following result.

\begin{teo}\label{teounigibbswrm}For $\alpha_{WR}(\lambda^{\pm}, h,j) < 1$ the Widom-Rowlinson model with \mbox{fugacities $\lambda^\pm$} and generalized interactions given by the pair $(h,j)$ admits a unique Gibbs measure. Furthermore, this Gibbs measure satisfies the exponential mixing property.
\end{teo}

Let us notice that for the original Widom-Rowlinson model with equal fugacities \mbox{we get} a simpler expression for $\alpha_{WR}$:
\begin{equation}\label{wrdc1}
\alpha_{WR}(\lambda,r)= \lambda (2r)^d.
\end{equation} For the discrete Widom-Rowlinson model we obtain an analogue of Theorem \ref{teounigibbswrm}, but the (F1)-diluteness condition in this context is given by the coefficient
\begin{equation}\label{wrdc2}
\alpha_{WR}^{discrete}(\lambda,r) = \lambda (2r)^d + \lambda,
\end{equation} where the extra term is due to the exclusion of equal-type particles with the same position.

\subsection{Thin rods model}

The size function in this model is given by the length of the rods, $q(\gamma_{x}):=\max\{1,2l\}$.
Once again, the structure of the Hamiltonian implies that $\Delta E = 0$ and, since for every $\gamma_x \in \R^2 \times S^1_*$ we have
$$
I(\gamma_x) \subseteq \{ \tilde{\gamma}_y \in \R^2 \times S^1_* : \|x-y\|_2 \leq 2l \},
$$ we obtain the bound
$$
\alpha_{TR}(\lambda,l) \leq 4 \lambda l^2 \sigma_2
$$ where $\sigma_2$ denotes the Lebesgue measure of the $2$-dimensional unit ball in the $\| \cdot \|_2$ norm.
Let us notice that, although this bound is valid for any choice of orientation measure $\rho$, it can be improved considerably provided that one has some knowledge on $\rho$. In any case, regardless of the particular choice of $\rho$ one may have, since $q$ is constant we have that ($*$) holds whenever $4\lambda l^2 < \frac{1}{\sigma_2}$. Thus, we obtain the following result.

\begin{teo}\label{teounigibbstrm} If $4 \lambda l^2 < \frac{1}{\sigma_2}$ then the thin rods model with fugacity $\lambda$ (and arbitrary orientation measure $\rho$) admits a unique Gibbs measure. Furthermore, this Gibbs measure satisfies the exponential mixing property.
\end{teo}

\subsection{Ising contours model}\label{ffgisingc1}

Since the Ising contours model does not satisfy Assumptions \ref{assump}, one needs to be careful when defining its FFG dynamics. We proceed as follows:
\begin{enumerate}
\item [i.] We replace $\Delta E$ in the construction of the local and global dynamics \mbox{respectively with}
$$
\Delta E^* := \inf_{\substack{ \eta \in \mathcal{N}(S \times G) \\ \gamma_x \in S \times G }} \Delta E^*_{\eta} (\gamma_x) \hspace{1cm}\text{ and }\hspace{1cm} \Delta E_{\Lambda} := \inf_{\substack{ \eta \in \mathcal{N}(S \times G) \\ \gamma_x \in S \times G }} \Delta E_{\Lambda|\eta} (\gamma_x).
$$
\item [ii.] We define the local dynamics on $\Lambda$ by replacing $M(\gamma_x|\xi)$ in the construction with
$$
M_\Lambda(\gamma_x|\xi):=e^{- (\Delta E_{\Lambda|\xi}(\gamma_x)-\Delta E_{\Lambda})}.
$$
\item [iii.] We define the global FFG dynamics by replacing $M(\gamma_x|\xi)$ in the construction with
$$
M^*(\gamma|\xi):=e^{-(\Delta E^*_{\xi}(\gamma_x)-\Delta E^*)}.
$$
\end{enumerate} One can easily see that, by replacing the original energy leap functions with their localized versions, one recovers (ii) in the consistent Hamiltonian property of Assumptions \ref{assump}. Furthermore, from the proof of Lemma \ref{inv0} it is clear that this property alone is enough to show (i) of Theorem \ref{teounigibbs}. Finally, from this and \eqref{elising} it is straightforward to also obtain (ii) of Theorem \ref{teounigibbs} in this particular context. Thus, the (F1)-diluteness condition for the Ising contours model implies the existence of an infinite-volume limit measure.
Since the natural choice of size function in this model is the size of \mbox{a contour, $q(\gamma_x) := |\gamma_x|$}, then \eqref{imic} yields
$$
\alpha_{IC}(\beta)= \sup_{\gamma_x \in (\Z^d)^* \times G} \left[ \frac{1}{|\gamma_x|} \sum_{\tilde{\gamma}_y : \tilde{\gamma}_y \not \sim \gamma_x} |\tilde{\gamma}_y| e^{-\beta|\tilde{\gamma}_y|}\right].
$$
In particular, for $\alpha_{IC}< 1$ there exists a probability measure $\mu$ on $\mathcal{N}((\Z^d)^* \times G)$ which is the infinite-volume limit of Boltzmann-Gibbs distributions with empty boundary condition, i.e.
$$
\mu = \lim_{\Lambda \nearrow (\Z^d)^*} \mu_{\Lambda|\emptyset}.
$$
In fact, Theorem \ref{teounigibbs} produces a coupling $( (\mathcal{K}^{\Lambda|\emptyset}_0)_{\Lambda \subseteq (\Z^d)^*}, \mathcal{K}_0)$ of these measures satisfying the following property: for any $\Lambda_0 \in \B^0_{(\Z^d)^*}$ there exists (a random) $\Delta \in \B^0_{(\Z^d)^*}$ such that for all $\Lambda \subseteq (\Z^d)^*$ with $\Delta \subseteq \Lambda$ one has
\begin{equation}\label{isconvc}
(\mathcal{K}^{\Lambda|\emptyset}_0)_{\Lambda_0 \times G} = \left(\mathcal{K}_0\right)_{\Lambda_0 \times G}.
\end{equation}
Now, let us observe that the condition $\alpha_{IC} < 1$ implies that almost surely there exist only finitely many contours in $\Pi_0$ surrounding each point in the lattice $\Z^d$. Indeed, if $\gamma_x$ is a contour surrounding a given point in the lattice, for example the origin $0$, then $\gamma_x$ contains a plaquette which intersects the $x_1$-axis on some negative value $l(\gamma_x)$ which is at a distance not \mbox{greater than $|\gamma_x|$} from the origin.  With this, a straightforward argument using the translational invariance of the model gives the bound
$$
\sum_{\gamma_x : 0 \in \text{Int}(\gamma_x)} P( \gamma_x \in \Pi_0 )  \leq \sum_{\gamma_x : 0 \in \text{Int}(\gamma_x)} e^{-\beta |\gamma_x|} \leq \sum_{\gamma_x : \,p_0 \not \sim \,\gamma_x} |\gamma_x|e^{-\beta|\gamma_x|}=: \alpha^0_{IC}(\beta)
$$ where the expression $0 \in \text{Int}(\gamma_x)$ means that $\gamma_x$ is a contour surrounding $0$, $p_0$ denotes a fixed plaquette in the dual lattice and the expression $p_0 \not \sim \gamma_x$ means that $p_0$ is adjacent to some plaquette in $\gamma_x$. Notice that the value of $\alpha^0_{IC}$ does not depend on the \mbox{choice of $p_0$.} It is not hard to check that for any $\beta > 0$ one has the inequalities
$$
\alpha_{IC}(\beta) \leq \alpha^0_{IC}(\beta) \leq 2d \alpha_{IC}(\beta).
$$ Thus, whenever $\alpha_{IC}(\beta) < 1$ we see that $\alpha^0_{IC}(\beta)$ is finite so that the Borel-Cantelli Lemma implies that almost surely $\Pi_0$ has only finitely many contours surrounding the origin. \mbox{By translational invariance} we conclude that the same conclusion must hold almost surely for \textit{all} sites in the lattice $\Z^d$ simultaneously. Observe that whenever this holds it is possible to conduct an infinite-volume $(+)$-alignment $\sigma^+_{\mathcal{K}_0}$ of $\mathcal{K}_0$ as explained in Section \ref{exampleisingc}. This spin configuration $\sigma^+_{\mathcal{K}_0}$ satisfies what is known as the $(+)$-\textit{sea with islands picture}: there are always finitely many contours around each point in $\Z^d$ and thus there is no percolation of the minority spin $(-)$. In accordance to the discussion in the Introduction of Part II, we see that $\sigma^+_{\mathcal{K}_0}$ can thus be regarded as a small random perturbation of the constant $(+)$-configuration, where this small
perturbation consists of finite islands on which $\sigma^+_{\mathcal{K}_0}$ disagrees with the $(+)$-configuration. Together with \eqref{isconvc}, the $(+)$-sea with islands picture implies that for any sequence $(\Lambda_n)_{n \in \N} \subseteq \B^0_{\Z^d}$ of simply connected sets with $\Lambda_n \nearrow \Z^d$ one has the $(+)$-alignment of the finite-volume dynamics $\mathcal{K}^{\Lambda_n^*|\emptyset}_0$ converging to the infinite-volume $(+)$-alignment, i.e.
$$
\sigma^+_{\mathcal{K}^{\Lambda_n^*|\emptyset}_0} \overset{loc}{\rightarrow} \sigma^+_{\mathcal{K}_0}.\footnote{Here local convergence is defined in analogy with Definition \ref{localconvergence}, i.e. convergence of the expectation of bounded local functions, where a function $f: \{+,-\}^{\Z^d} \rightarrow \R$ is said to be \textit{local} if it depends only on the spin values inside some bounded set $\Lambda \in \B^0_{\Z^d}$.}
$$ Thus, by \eqref{dualidadising} there exists a probability measure $\mu^+$ on $\{+,-\}^{\Z^d}$ \mbox{(the distribution of $\sigma^+_{\mathcal{K}_0}$)} such that
$$
\mu^+_{\Lambda_n} \overset{loc}{\rightarrow} \mu^+.
$$ Using Proposition \ref{limitegibbs} we conclude that $\mu^+$ is a Gibbs measure for the Ising model on $\Z^d$ (the diluteness condition is not required throughout the proof). By symmetry we obtain that there exists another Gibbs measure $\mu^-$, which is realized as the \mbox{$(-)$-alignment of $\mathcal{K}_0$}, i.e. $\sigma^-_{\mathcal{K}_0} := - \sigma^+_{\mathcal{K}_0}$. However, since both ``sea with islands'' pictures cannot be satisfied simultaneously, we must have $\mu^+ \neq \mu^-$. Thus, we see that whenever $\beta$ is sufficiently large so as to guarantee that $\alpha_{IC}(\beta) < 1$ then the Ising model exhibits a phase transition.

Now, with respect to establishing mixing properties for these measures, the exponential mixing property for the Ising contours model was studied in \cite{FFG1}. However, the true objects of interest here are the aligned measures $\mu^+$ and $\mu^-$, whose mixing properties cannot be immediately deduced from those of their underlying contour model. Indeed, local information on the original Ising model such as the spin at a given site in the lattice depends on the total amount of contours surrounding this site, which is highly non-local information in terms of contours. In general, given a bounded set $\Lambda \in \B^0_{\Z^d}$, the spin values of the configuration $\sigma^+_{\mathcal{K}_0}$ inside $\Lambda$ depend on the contour configuration $\mathcal{K}_0$ inside the set
$$
r^*(\Lambda) = \{ \gamma_x \in (\Z^d)^* \times G : p(\Lambda^*) \cap \text{V}(\gamma_x) \neq \emptyset\}
$$ where $p(\Lambda^*)$ denotes the set of plaquettes with vertices in $\Lambda^*$ and $\text{V}(\gamma_x)$ denotes the \textit{volume} of the contour $\gamma_x$, i.e. the set of points in $(\Z^d)^*$ lying outside the only infinite component of $(\Z^d)^* - \text{supp}(\gamma_x)$. In more precise terms, there exists a $\F_{r^*(\Lambda)}$-measurable function $\phi_{\Lambda}^+ : \mathcal{N}((\Z^d)^*\times G) \rightarrow \{+,-\}^{\Z^d}$ such that
$$
(\sigma^+_{\mathcal{K}_0})_{\Lambda} = \phi_{\Lambda}^+ (\mathcal{K}_0).
$$ Thus, for every pair of bounded local functions $f,g : \{+,-\}^{\Z^d} \rightarrow \R$ we have that
\begin{equation}\label{eqmixingicm}
\left|\mu^+(fg) - \mu^+(f)\mu^+(g)\right| = \left| \mu( (f\circ \phi_{\Lambda_f}^+)(g \circ \phi_{\Lambda_g}^+) ) -\mu( f\circ \phi_{\Lambda_f}^+)\mu(g \circ \phi_{\Lambda_g}^+)\right|
\end{equation} where we have used the standard integral notation
$$
\vartheta(f):= \int f(\eta)d\vartheta(\eta).
$$ The problem with \eqref{eqmixingicm} is that the functions on its right hand side are not local as functions on $\mathcal{N}((\Z^d)^*\times G)$. Nevertheless, if for each $\Lambda \in \B^0_{\Z^d}$ we had $\nu(r^*(\Lambda))< +\infty$ then we could construct a coupling between $\mathcal{A}^0_F(r^*(\Lambda_f))$ and $\mathcal{A}^0_F(r^*(\Lambda_g))$ as explained in the proof of Theorem \ref{teomixing} and use it to obtain the bound
\begin{equation}\label{eqmixingicm2}
\left|\mu^+(fg) - \mu^+(f)\mu^+(g)\right| \leq 2\|f\|_\infty\|g\|_\infty P(\mathcal{A}^0_F(r^*(\Lambda_f)) \not \sim \mathcal{A}^0_F(r^*(\Lambda_g))).
\end{equation} Fortunately, when $\alpha_{IC}(\beta) < 1$ this happens to be the case. Indeed, one has the estimate
$$
\nu(r^*(\Lambda)) \leq (\#\Lambda + 2d\#\p \Lambda) \alpha_{IC}^0
$$ where $\p \Lambda := \{ x \in \Lambda : d(x,\Lambda^c) = 1 \}$, which follows upon noticing that if $\gamma_x \in r^*(\Lambda)$ then $\gamma_x$ is either surrounding a point in $\Lambda$ or containing a plaquette whose associated bond connects $\Lambda$ with $\Lambda^c$. Now, to estimate the probability in the right hand side of \eqref{eqmixingicm2}, by definition of $r^*(\Lambda)$ one can check that
$$
P( \mathcal{A}^0_F(r^*(\Lambda_f)) \not \sim \mathcal{A}^0_F(r^*(\Lambda_g)) ) \leq P \left(  \sum_{C \in \mathcal{A}^0(r^*(\Lambda_f))} q(C)  + \sum_{C \in \mathcal{A}^0(r^*(\Lambda_g))} q(C) \geq d_S( \Lambda_f, \Lambda_g ) \right).
$$ However, the right hand side cannot be bounded as in the proof of Theorem \ref{teomixing} by using Corollary \ref{corbranching} directly: neither the size function $q$ is bounded nor is condition ($*$) satisfied. To fix this problem, for each $\Lambda \in \B^0_S$ we consider the branching process $\B$ dominating $\mathcal{A}^0(r^*(\Lambda))$ which can be constructed as in the proof of the \mbox{Domination lemma.} In this process $\B$, the individuals are the different contours, each having an independent number of offspring which has Poisson distribution with mean proportional to their size. We wish to enlarge this branching process $\B$, so that the enlarged process $\overline{\B}$ satisfies:
\begin{enumerate}
\item [i.] $\sum_{C \in \mathcal{A}^0(r^*(\Lambda))} q(C) \leq \# \overline{\B}$,
\item [ii.] All individuals in $\overline{\B}$ have the same offspring distribution.
\end{enumerate}If we manage to do this, then we can proceed as in the proof of Corollary \ref{corbranching} to bound the right hand side of \eqref{eqmixingicm2}, provided that the offspring distribution in $\overline{\B}$ has mean less than one. The way in which to achieve this is to enlarge $\B$ by considering plaquettes as individuals instead of whole contours. More precisely, the initial individuals in $\overline{\B}$ will be those plaquettes conforming the initial contours in $\B$ and for a given plaquette $p$ we define its offspring as follows: we first draw an independent number of contours containing $p$ with Poisson distribution of mean $\alpha_{IC}^0(\beta)$ and then regard all the plaquettes which constitute these contours as the offspring of $p$. The detailed construction of this enlarged process is similar to the one in the Domination Lemma, so we omit the details. Thus, a straightforward computation using Theorem \ref{expbranching} yields that, whenever $\alpha^0_{IC}(\beta) < 1$, \mbox{there
exists} $b_1 > 0$ sufficiently small (depending on $\beta$) such that
$$
\E( e^{b_1 \sum_{C \in \mathcal{A}^0(r^*(\Lambda))} q(C)} ) \leq e^{ b_1 \# \overline{B} } = e^{\nu^{\tilde{\beta}}(r^*(\Lambda)) - \nu^\beta(r^*(\Lambda))}
$$ where we write the dependence of $\nu$ on the inverse temperature explicitly, and furthermore set $\tilde{\beta}:= \beta  - \log \tilde{b}_2$ for a certain constant
$$
\tilde{b}_2:= \E( e^{b_1 \# \overline{\B}(p)} ) < +\infty
$$ which is strictly larger than one and depends only on $b_1$ and $\beta$. Let us notice that $\tilde{\beta} < \beta$ and also that:
\begin{enumerate}
\item [$\bullet$] $\tilde{\beta}$ is a increasing function of $\beta$ satisfying $\tilde{\beta} \rightarrow +\infty$ as $\beta \rightarrow +\infty$,
\item [$\bullet$] $\tilde{\beta}$ is a increasing function of $b_1$ satisfying $\beta - \tilde{\beta} \rightarrow 0$ as $b_1 \rightarrow 0$.
\end{enumerate} Thus, if we set
\begin{equation}\label{betas}
\beta^*= \inf\{ \beta > 0 : \alpha^0_{IC}(\beta) < 1 \}\hspace{1cm}\text{ and }\hspace{1cm}\beta^{**}=\inf\{ \beta > 0 : \alpha^0_{IC}(\beta) < +\infty \}
\end{equation} then if $\beta > \beta^*$ we have that for each $\tilde{\beta} \in (\beta^{**},\beta)$ there exists a constant $c > 0$ such that
\begin{equation}\label{eqmixingicm3}
\left|\mu^+(fg) - \mu^+(f)\mu^+(g)\right| \leq \|f\|_\infty\|g\|_\infty e^{- c d_S(\Lambda_f,\Lambda_g) + \nu^{\tilde{\beta}}(r^*(\Lambda_f)) + \nu^{\tilde{\beta}}(r^*(\Lambda_g))}.
\end{equation}for $d_S(\Lambda_f,\Lambda_g)$ sufficiently large (depending only on $c$). Furthermore, if we take $\tilde{\beta} \in (\beta^*,\beta)$ then we obtain the simpler (yet weaker) formula
\begin{equation}\label{eqmixingicm4}
\left|\mu^+(fg) - \mu^+(f)\mu^+(g)\right| \leq \|f\|_\infty\|g\|_\infty e^{- c d_S(\Lambda_f,\Lambda_g) + (2d+1)(\# \Lambda_f + \#\Lambda_g)}.
\end{equation}

Clearly, the analogous conclusion also remains valid for the other Gibbs measure, $\mu^-$. \mbox{We summarize our} analysis in the following theorem.

\begin{teo}\label{teounigibbsising} If $\beta > 0$ is sufficiently large so as to satisfy $\alpha_{IC}(\beta) < 1$ then:
\begin{enumerate}
\item [i.] The Ising model on $\Z^d$ admits two distinct Gibbs measures, $\mu^+$ and $\mu^-$.
\item [ii.] The measures $\mu^{+}$ and $\mu^{-}$ can be obtained as the local limits
$$
\mu^{+} := \lim_{n \rightarrow +\infty} \mu^{+}_{\Lambda_n} \hspace{2cm}\text{ and }\hspace{2cm}\mu^{-} := \lim_{n \rightarrow +\infty} \mu^{-}_{\Lambda_n}
$$ for any sequence $(\Lambda_n)_{n \in \N} \subseteq \B^0_{\Z^d}$ of simply connected sets with $\Lambda_n \nearrow \Z^d$.
\item [iii.] $\mu^+$ and $\mu^-$ satisfy the sea with islands picture for the $(+)$ and $(-)$ spins, respectively.
\item [iv.] If also $\beta > \beta^*$ where $\beta^*$ is defined in \eqref{betas}, then both $\mu^+$ and $\mu^-$ also satisfy the exponential mixing property in the sense of \eqref{eqmixingicm3} and \eqref{eqmixingicm4}.
\end{enumerate}
\end{teo}

\begin{obs} Much sharper conditions than the one obtained here are known for the occurrence of phase transition in the Ising model. However, the argument presented here is of considerable relevance, since it shall be repeated when studying the applications of the FFG dynamics to the Pirogov-Sinai theory, where in general sharper conditions than the one given by Theorem \ref{teounigibbs} are not known. As for the exponential mixing property, the range of validity provided by standard cluster expansion methods is strictly smaller than the one obtained here: these methods show that the exponential mixing property is satisfied as soon as $\beta > \beta'$, where
$$
\beta' = \inf\left\{ \beta > 0 : \sum_{\gamma_x : \,p_0 \not \sim \,\gamma_x} e^{q(\gamma_x)} e^{- \beta|\gamma_x|} < 1\right\}.
$$ The coefficient $\beta'$ can be improved (see \cite{LM} for example), although these methods are not capable of getting rid of the exponential dependence in $q$.
\end{obs}

\section{A remark on perfect simulation of Gibbs measures}\label{perfectsimulation}

One of the advantages of the FFG dynamics is that not only do they yield a criterion for the existence and uniqueness of the equilibrium measure, but they also provide a way \mbox{in which to perfectly sample from it.}
Indeed, if given a certain diluted model one wishes to obtain a perfect sample of its unique Gibbs measure on a finite volume $\Lambda \in \B^0_S$, then all one has to do is to obtain a perfect sample of the clan of ancestors $\mathcal{A}^0(\Lambda \times G)$ and afterwards perform the deleting procedure discussed in Section \ref{localdynamics}. But by the definition of $\mathcal{A}^0(\Lambda \times G)$, one has that:
\begin{enumerate}
\item [$\bullet$] $\mathcal{A}^0_0(\Lambda \times G)$ is a Poisson process on $\mathcal{C}$ with intensity measure $\phi_{\mathcal{P}(\Lambda \times G)}$, where
$$
\mathcal{P}(\Lambda \times G):=\{ C \in \mathcal{C} : basic(C) \in \Lambda \times G, b_C \leq 0 < b_C + l_C \},
$$
\item [$\bullet$] Conditional on $\mathcal{A}^0_0(\Lambda \times G), \dots, \mathcal{A}^0_{n}(\Lambda \times G)$, the set $\mathcal{A}^0_{n+1}(\Lambda \times G) - \bigcup_{i=0}^n \mathcal{A}^0_{i}(\Lambda \times G)$ is a Poisson process on $\mathcal{C}$ with intensity measure $\phi$ restricted to the set
    $$
    \bigcup_{C \in \mathcal{A}^0_{n}(\Lambda \times G)} \mathcal{P}(C) - \left( \mathcal{P}(\Lambda \times G) \cup \bigcup_{i=1}^{n-1} \bigcup_{C \in \mathcal{A}^0_{i}(\Lambda \times G)} \mathcal{P}(C)\right).
    $$
\end{enumerate}
Thus, to obtain a perfect sample of $\mathcal{A}^0(\Lambda \times G)$ one may proceed as follows:
\begin{enumerate}
\item [i.] Sample $\mathcal{A}^0_0(\Lambda \times G)$ from a Poisson process on $\mathcal{P}(\Lambda \times G)$ with intensity measure $\phi$.
\item [ii.] Having sampled $\mathcal{A}^0_0(\Lambda \times G), \dots, \mathcal{A}^0_{n}(\Lambda \times G)$, obtain $\mathcal{A}^0_{n+1}(\Lambda \times G)$ by sampling from a Poisson process on $\bigcup_{C \in \mathcal{A}^0_{n}(\Lambda \times G)} \mathcal{P}(C)$ and discarding all those cylinders which are ancestors of cylinders in generations lesser than $n$.
\end{enumerate}
Since we are dealing with a diluted model, by the results on Section \ref{ancestors} we have that eventually we will reach a step in which no new ancestors are added. Once that happens, the algorithm stops and the ancestor family constructed until that moment constitutes the perfect sample of $\mathcal{A}^0(\Lambda \times G)$. Upon conducting the deleting procedure on it as explained in Section \ref{localdynamics}, all the kept cylinders with bases in $\Lambda \times G$ which are alive at time $0$ constitute a perfect sample of the unique Gibbs measure on the finite volume $\Lambda$. We refer to \cite{FFG1,FFG2} for alternative sampling algorithms, results on speed of convergence and comments on the user-impatience bias in the simulation scheme.

\newpage
\section{Resumen del Capítulo 9}

Aquí introducimos la dinámica de Fernández-Ferrari-Garcia desarrollada originalmente en \cite{FFG1} para el modelo de contornos de Ising y mostramos que se encuentra bien definida para la clase más general de modelos diluidos. Dado un par $(\nu,H)$ satisfaciendo las condiciones en \ref{assump}, podríamos resumir la dinámica de Fernández-Ferrari-Garcia asociada como sigue:
\begin{enumerate}
\item [$\bullet$] A tasa $e^{-\Delta E}$ se propone el nacimiento de nuevas partículas con intensidad $\nu$.
\item [$\bullet$] Cada partícula $\gamma_x$ propuesta para nacer lo hará efectivamente con probabilidad $e^{-(\Delta E_{\eta}(\gamma_x)-\Delta E)}$, donde $\eta$ es el estado del sistema al momento en que el nacimiento de la partícula $\gamma_x$ es propuesto.
\item [$\bullet$] Cada partícula que ha nacido efectivamente tiene un tiempo de vida aleatorio con distribución exponencial de parámetro 1.
\item [$\bullet$] Luego de que su tiempo de vida haya expirado, cada partícula muere y desaparece de la configuración.
\end{enumerate}
Estudiamos primero la dinámica sobre volúmenes finitos para luego dar condiciones que garanticen la existencia de la dinámica en el volumen infinito. Bajo estas condiciones, mostramos que toda medida de Gibbs para el modelo dado por $(\nu,H)$ es una medida invariante para la dinámica. Más aún, verificamos que bajo una condición adicional (la generalización de la originalmente propuesta en \cite{FFG1} para el modelo de contornos de Ising) la dinámica en volumen infinito posee una única medida invariante, de donde se deduce que bajo dicha condición existe una única medida de Gibbs y que ésta coincide con la medida invariante de la dinámica.

Para probar esto, primero mostramos que la unicidad de medida invariante está garantizada por la ausencia de percolación en un proceso particular de percolación orientada dependiente. Luego, mostramos que dicho proceso puede ser dominado por otro de percolación independiente y que, bajo la condición propuesta en \cite{FFG1}, éste resulta ser subcrítico. A partir de esto se concluye que hay ausencia de percolación en el proceso original, lo cual implica la unicidad buscada. Por último, probamos que, bajo la condición de unicidad, la medida invariante es límite local de distribuciones de Boltzmann-Gibbs y, por lo tanto, resulta ser también una medida de Gibbs para el modelo. Por el razonamiento anterior es, además, la única que existe.

Estudiamos también cómo se traduce dicha condición de unicidad a algunos de los modelos introducidos en el Capítulo 8. Además, mostramos cómo este tipo de resultados pueden utilizarse en el modelo de contornos de Ising para probar la existencia de múltiples medidas de Gibbs en el modelo de Ising original a baja temperatura.

También investigamos bajo qué condiciones adicionales la única medida invariante posee la propiedad de mixing exponencial; en el caso del modelo de contornos de Ising, adaptamos este análisis para obtener la propiedad de mixing exponencial para cada una de las medidas de Gibbs extremales en el régimen de baja temperatura. Para terminar, damos un algoritmo de simulación perfecta para la medida invariante (bajo la condición de unicidad) basado en la construcción hacia el pasado de la dinámica.

\chapter{Continuity of Gibbs measures}\label{chapterconvabs}

In this chapter we show the continuity of the Gibbs measures for heavily diluted models with respect to their intensity measure and Hamiltonian in the \mbox{absolutely continuous case.} This scenario typically includes continuity with respect to the parameters of the model such as fugacity of particles, inverse temperature and interaction range among others.
The main result is contained in Theorem \ref{convabs} below. One important aspect to point out is that we not only obtain the local convergence of the corresponding Gibbs measures, but in fact in the proof of \mbox{Theorem \ref{convabs}} we construct a coupling between these measures in which a rather strong form of almost sure convergence takes place: given a finite volume $\Lambda \in \B^0_S$, all realizations of these measures are \textit{identical} on $\Lambda \times G$ for parameter values which are sufficiently close to the limit values. This is a distinctive feature of our approach since, in general, other methods used to establish these type of results (i.e. cluster expansion or disagreement percolation methods) are unable to obtain such a strong form of convergence, at least in the continuum setting.

\section{A general continuity result}

\begin{teo}\label{convabs} Let $(\nu^\varepsilon,H^\varepsilon)_{\varepsilon \geq 0}$ be a family of diluted models such that
\begin{enumerate}
\item [$\bullet$] There exists a heavily diluted model $(\nu,H)$ satisfying
\begin{enumerate}
\item [i.] For every $\varepsilon \geq 0$ the intensity measure $\nu^\varepsilon$ is absolutely continuous with respect to $\nu$ with density $\frac{d \nu^\varepsilon}{d \nu}$ such that
\begin{equation}\label{controldensidad}
0 \leq \frac{d \nu^\varepsilon}{d \nu} \leq 1.
\end{equation}
\item [ii.] For every $\varepsilon \geq 0$ and $\gamma_x \in S \times G$ we have $I^{H^\varepsilon}(\{\gamma_x\}) \subseteq I^H (\{\gamma_x\})$.
\item [iii.]
\begin{equation}
\Delta E^H \leq \inf_{\varepsilon \geq 0} \left[ \inf_{\substack{ \eta \in \mathcal{N}(S \times G) \\ \gamma_x \in S \times G }} \Delta \tilde{E}^{H^\varepsilon}_\eta (\gamma_x)  \right]
\end{equation}\label{controldensidad2} where for each $\eta \in \mathcal{N}(S \times G)$ we define
$$
\Delta \tilde{E}^{H^\varepsilon}_\eta := \Delta E^{H^\varepsilon}_\eta - \log \left( \frac{d\nu^\varepsilon}{d \nu}\right).
$$
\end{enumerate}
\item [$\bullet \bullet$] $\lim_{\varepsilon \rightarrow 0^+} \Delta \tilde{E}^{H^\varepsilon}_\eta (\gamma_x) = \Delta \tilde{E}^{H^0}_\eta (\gamma_x)$ for every $\eta \in \mathcal{N}(S \times G)$ and $\gamma_x \in S \times G$.
\end{enumerate} Then for each diluted model $(\nu^\varepsilon,H^\varepsilon)$ admits exactly one Gibbs measure $\mu^\varepsilon$ and as $\varepsilon \rightarrow 0^+$
$$
\mu^{\varepsilon} \overset{loc}{\longrightarrow} \mu^0.
$$ The model $(\nu,H)$ is called a \textit{majorant model} for $(\nu^\varepsilon,H^\varepsilon)_{\varepsilon \geq 0}$.
\end{teo}

\begin{proof} Let us start by noticing that, since $\nu^\varepsilon \ll \nu$, for every $\Lambda \in \B^0_S$ we have that $\pi^{\nu^\varepsilon}_\Lambda \ll \pi^{\nu}_\Lambda$ with density given by
\begin{equation}\label{density}
\frac{d\pi^{\nu^\varepsilon}_\Lambda}{d\pi^{\nu}_\Lambda} (\sigma) = e^{-( \nu^\varepsilon(\Lambda \times G) - \nu(\Lambda \times G))} \prod_{\gamma_x \in [\sigma]} \frac{d\nu^\varepsilon}{d \nu}(\gamma_x),
\end{equation} a fact which can be deduced from \eqref{poisson}. In particular, the Boltzmann-Gibbs distributions $\mu^\varepsilon_{\Lambda |\eta}$ specified by the pair $(\nu^\varepsilon,H^\varepsilon)$ are also be specified by $(\nu,\tilde{H}^\varepsilon)$, where $\tilde{H}^\varepsilon$ is given by the formula
$$
\tilde{H}^\varepsilon_{\Lambda|\eta} (\sigma) = H^\varepsilon_{\Lambda|\eta}(\sigma) - \sum_{\gamma_x \in [\sigma]} \log\left( \frac{d \nu^\varepsilon}{d \nu}(\gamma_x)\right).
$$ It is not difficult to check that for each $\varepsilon \geq 0$ the pair $(\nu,\tilde{H}^\varepsilon)$ satisfies Assumptions \ref{assump}. Moreover, since for every $\gamma_x \in S \times G$ it is possible to verify that $I^{\tilde{H}^\varepsilon}(\{\gamma_x\}) = I^{H^\varepsilon}(\{\gamma_x\})$, we have that the pair $(\nu,\tilde{H}^\varepsilon)$ also satisfies the (F1)-diluteness condition for every $\varepsilon \geq 0$, which guarantees that each diluted model $(\nu^\varepsilon, H^\varepsilon)$ admits exactly one Gibbs measure.

To establish the local convergence we shall couple all the measures $\mu^\varepsilon$ simultaneously. For this purpose we consider the infinite-volume stationary FFG processes $\mathcal{K}^\varepsilon$ constructed by taking a Poisson process $\overline{\Pi}$ with intensity measure $\nu \times e^{-\Delta E^H} \mathcal{L} \times \mathcal{L}_{\R^+} \times \mathcal{U}[0,1]$ and setting
\begin{equation}\label{keptscaled}
\mathcal{K}^\varepsilon = \{ (\gamma_x,t,s) \in \Pi : F(\gamma_x,t,s) < \tilde{M}^\varepsilon(\gamma_x | \mathcal{K}^\varepsilon_{t^-}) \}
\end{equation} where for each $\gamma_x \in S \times G$ and $\xi \in \mathcal{N}(S \times G)$ we define
$$
\tilde{M}^\varepsilon (\gamma_x |\xi) := e^{- (\Delta \tilde{E}^{H^\varepsilon}_\xi(\gamma_x) - \Delta E^H)}.
$$ By the arguments given in the previous sections we see that for each $\varepsilon \geq 0$ the process $\mathcal{K}^\varepsilon$ is stationary with invariant measure $\mu^\varepsilon$. Thus it will suffice to show that as $\varepsilon \rightarrow 0^+$
$$
\mathcal{K}^\varepsilon_0 \overset{loc}{\longrightarrow} \mathcal{K}^0_0.
$$ Let us take then $\Lambda \in \B^0_S$ and consider the clan of ancestors of $\Lambda \times G$ ancestors at time 0 $\A^{0,H} ( \Lambda \times G)$ with respect to $H$. Notice that, following this notation, for every $\varepsilon \geq 0$ we have the inclusion
\begin{equation}\label{ancesinc0}
A^{0,H^\varepsilon}(\Lambda \times G) \subseteq \A^{0,H} ( \Lambda \times G).
\end{equation} Furthermore, recall that $\A^{0,H}( \Lambda \times G)$ is finite almost surely since $(\nu,H)$ is a heavily diluted model. Now, since $\lim_{\varepsilon \rightarrow 0^+} \Delta \tilde{E}^{H^\varepsilon}_\eta (\gamma_x) = \Delta \tilde{E}^{H^0}_\eta (\gamma_x)$ for every $\eta \in \mathcal{N}(S \times G)$ and $\gamma_x \in S \times G$, it follows that there exists (random) $\varepsilon_0 > 0$ such that if $0 \leq \varepsilon < \varepsilon_0$ then
\begin{equation}\label{ancesinc1}
\mathcal{K}^\varepsilon_{\A^{0,H} ( \Lambda \times G)} = \mathcal{K}^0_{\A^{0,H} ( \Lambda \times G)}.
\end{equation} Indeed, if (random) $N \in \N$ is such that $\A^{0,H}_n(\Lambda \times G) = \emptyset$ for every $n > N$ almost surely then for every cylinder $C \in \A^{0,H}_N(\Lambda \times G)$ and $\varepsilon \geq 0$ we have that
$$
C \in \mathcal{K}^\varepsilon \Longleftrightarrow F(C) < \tilde{M}^\varepsilon( basis(C) | \emptyset )
$$ from which we immediately obtain that for $\varepsilon$ (randomly) small enough
$$
\mathcal{K}^\varepsilon_{\A^{0,H}_N ( \Lambda \times G)} = \mathcal{K}^0_{\A^{0,H}_N( \Lambda \times G)}
$$ and one may proceed with the succeeding generations by induction. But \eqref{ancesinc0} and \eqref{ancesinc1} together imply that for $0 \leq \varepsilon < \varepsilon_0$ we have
$$
(\mathcal{K}^\varepsilon_0)_{\Lambda \times G} = (\mathcal{K}^0_0)_{\Lambda \times G}
$$ which establishes the local convergence and concludes the proof.
\end{proof}

We would like to point out that although the condition of the existence of a majorant model may seem restrictive at first, in practice all heavily diluted models admit such a majorant. Indeed, as we will see on Section \ref{examplescont} below, most majorant models can be obtained by slightly increasing the fugacity (or decreasing the inverse temperature) and/or the interaction range of the limit model $(\nu^0,H^0)$. Since this limit model is heavily diluted by assumption, performing such an operation will yield once again a heavily diluted model.

Finally, let us notice that the hypothesis in Theorem \ref{convabs} can be relaxed a little bit. Indeed, \eqref{controldensidad} needs to hold $\nu$-almost surely since under this condition we can always choose a version of $\frac{d\nu^\varepsilon}{d\nu}$ satisfying \eqref{controldensidad} for every $\gamma_x \in S \times G$. Similarly, the convergence of the energy leap functions $\Delta \tilde{E}^{H^\varepsilon}_\eta$ for every $\eta \in \mathcal{N}(S \times G)$ can also be somewhat relaxed. The next definition explains how to do so.

\begin{defi}\label{diset} A measurable set $N \subseteq \mathcal{N}(S \times G)$ is said to be \textit{dynamically impossible} for an intensity measure $\nu$ on $S \times G$ if it satisfies the following properties:
\begin{enumerate}
\item [i.]$\pi^\nu(N)=0$
\item [ii.] $\eta \in N^c \Longrightarrow \xi \in N^c$ for every $\xi \preceq \eta$, i.e. for every $\xi \in \mathcal{N}(S \times G)$ such that its standard representation satisfies $Q_\xi \subseteq Q_\eta$ and $m_\xi(\gamma_x) \leq m_\eta(\gamma_x)$ for every $\gamma_x \in Q_\xi$.
\item [iii.]If $X \in \mathcal{N}( \mathcal{C} )$ satisfies $X_t \in N$ for some $t \in \R$ then there exists $h > 0$ such that $X_s \in N$ for every $s \in [t,t+h)$.
\end{enumerate}
\end{defi}

Let us notice that if $N$ is a dynamically impossible set for the intensity measure $\nu$ then the corresponding Poisson process $\Pi^{\phi_\nu}$ on $\mathcal{C}$ satisfies
$$
P( \Pi_t^{\phi_\nu} \in N \text{ for some }t \in \R )=0.
$$
Indeed, we have that
\begin{align*}
P( \Pi_t^{\phi_\nu} \in N \text{ for some }t \in \R ) & = P( \Pi_r^{\phi_\nu} \in N \text{ for some }r \in \mathbb{Q} )\\
\\
& \leq \sum_{r \in \mathbb{Q}} P( \Pi_r^{\phi_\nu} \in N ) = \sum_{r \in \mathbb{Q}} \pi^{\nu}(N) = 0.
\end{align*} If follows from (ii) in Definition \ref{diset} and the proof of Theorem \ref{convabs} that condition ($\bullet \bullet$) in the statement of Theorem \ref{convabs} may be replaced with the following weaker condition:
\begin{enumerate}
\item [($\bullet \bullet^*$)] There exists a dynamically impossible set $N$ for the intensity measure $\nu$ such that
$$
\lim_{\varepsilon \rightarrow 0^+} \Delta \tilde{E}^{H^\varepsilon}_\eta (\gamma_x) = \Delta \tilde{E}^{H^0}_\eta (\gamma_x)
$$ for all $\eta \in \mathcal{N}(S \times G)$ and $\gamma_x \in S \times G$ satisfying $\eta, \eta + \delta_{\gamma_x} \in N^c$.
\end{enumerate}

\section{Applications}\label{examplescont}

In the following we show how Theorem \ref{convabs} may be applied to the models in Section \ref{examples}.

\subsection{Continuity in the inverse temperature for the Ising model}

Consider $\beta_0 > 0$ such that $\alpha_{IC}(\beta_0) < 1$ and a sequence $(\beta_\varepsilon)_{\varepsilon > 0} \subseteq \R^+$ \mbox{converging to $\beta_0$.} For each $\varepsilon \geq 0$ we may consider the Ising contours model with inverse temperature $\beta_\varepsilon$, which is specified by the intensity measure $\nu^{\beta_\varepsilon}$ as in \eqref{imic} and Hamiltonian $H^\varepsilon$ as in \eqref{hic}, the latter being independent of $\varepsilon$.
By Theorem \ref{convabs} we have that for $\varepsilon \geq 0$ sufficiently small the corresponding Ising contours model admits an infinite-volume Boltzmann-Gibbs distribution $\mu^\varepsilon$ and, furthermore, that
$$
\mu^\varepsilon \overset{loc}{\rightarrow} \mu^0.
$$ Indeed, it suffices to check that these models satisfy the hypothesis of Theorem \ref{convabs} \mbox{(with the energy leap} functions $\Delta E_\eta$ replaced by their modified versions $\Delta E^*_\eta$ defined in \eqref{elising} and the interaction ranges $I(\{\gamma_x\})$ given by incompatibility, see Section \ref{ffgisingc1}). But notice that if for $0 < \beta^* < \beta_0$ such that $\alpha_{IC}(\beta^*) < 1$ we consider the Ising contours model with inverse temperature $\beta^*$ given by the pair $(\nu^{\beta^*},H)$ then we see that:
\begin{enumerate}
\item [$\bullet$] For $\varepsilon \geq 0$ such that $\beta^* < \beta_\varepsilon$ we have that $\nu^\varepsilon \ll \nu^{\beta^*}$ with density given by
\begin{equation}\label{controldensidad3}
\frac{d \nu^\varepsilon}{d \nu^{\beta^*}} (\gamma_x) = e^{- (\beta_\varepsilon - \beta^*)|\gamma_x|}
\end{equation} which satisfies $0 \leq \frac{d \nu^\varepsilon}{d \nu^{\beta^*}} \leq 1$.
\item [$\bullet$] The validity of (ii) and (iii) in the hypothesis of the theorem follows at once from the fact that the Hamiltonian is the same in all the contour models under consideration. Furthermore, since $\beta_\varepsilon \rightarrow \beta_0$ it also follows that $\lim_{\varepsilon \rightarrow 0^+} \Delta \tilde{E^*}^{H^\varepsilon}_\eta (\gamma_x) = \Delta \tilde{E^*}^{H^0}_\eta (\gamma_x)$ \mbox{for every} $\eta \in \mathcal{N}(\Z^d \times G)$ and $\gamma_x \in \Z^d \times G$, so that $(\nu^{\beta^*},H)$ \mbox{acts as a majorant model.}
\end{enumerate}

Finally, let us notice that since we are always under the heavily diluted regime, all the FFG processes under consideration have well defined $(+)$-alignments and $(-)$-alignments. As a direct consequence we obtain the following result.

\begin{teo}\label{convabsic} For any $\beta_0 > 0$ such that $\alpha_{IC}(\beta_0) < 1$ we have that
$$
\lim_{\beta \rightarrow \beta_0} \mu^{+,\beta} = \mu^{+,\beta_0} \hspace{2cm}\text{ and }\hspace{2cm}\lim_{\beta \rightarrow \beta_0} \mu^{-,\beta} = \mu^{-,\beta_0}
$$ where for $\beta > 0$ the measures $\mu^{+,\beta}$ and $\mu^{-,\beta}$ are respectively defined as the weak limits
$$
\mu^{+,\beta} := \lim_{n \rightarrow +\infty} \mu^{+,\beta}_{\Lambda_n} \hspace{2cm}\text{ and }\hspace{2cm}\mu^{-,\beta} := \lim_{n \rightarrow +\infty} \mu^{-,\beta}_{\Lambda_n}
$$ for any increasing sequence $(\Lambda_n)_{n \in \N} \subseteq \B^0_{\Z^d}$ of simply connected sets with $\bigcup_{n \in \N} \Lambda_n = \Z^d$.
\end{teo}

\newpage
\begin{obs} Let us observe that this continuity result is well known (even for arbitrary values of $\beta_0$) and is a direct consequence of the monotonicity properties of the Ising model.
The advantage of the approach presented here is that it can be applied in the same manner to other contour models lacking these properties. We will do this in Chapter \ref{chapterpirogovsinai}.
\end{obs}

\subsection{Widom-Rowlinson model with generalized interactions}

For the Widom-Rowlinson model with generalized interactions we obtain the next result.

\begin{teo}\label{convabswr} Let us consider for $\varepsilon \geq 0$ the Widom-Rowlinson model with fugacities $\lambda^+_\varepsilon$ and $\lambda^-_\varepsilon$, interspecies repulsion function $h^\varepsilon$ and type-independent repulsion function $j^\varepsilon$. Let us assume that the following conditions hold:
\begin{enumerate}
\item [i.] $\lim_{\varepsilon \rightarrow 0^+} \lambda^+_\varepsilon = \lambda^+_0$ and $\lim_{\varepsilon \rightarrow 0^+} \lambda^-_\varepsilon = \lambda^-_0$.
\item [ii.] $\lim_{\varepsilon \rightarrow 0^+} h^\varepsilon(r) = h^0(r)$ and $\lim_{\varepsilon \rightarrow 0^+} j^\varepsilon(r) = j^0(r)$ for every $r \geq 0$.
\item [iii.] $\lim_{\varepsilon \rightarrow 0^+} m_{h^\varepsilon} = m_{h^0}$ and $\lim_{\varepsilon \rightarrow 0^+} m_{j^\varepsilon} = m_{j^0}$ where $m_{h^\varepsilon}$ and $m_{j^\varepsilon}$ are defined for each $\varepsilon \geq 0$ through the relation supp($h^\varepsilon$)$=[0,m_{h^\varepsilon}]$ and supp($j^\varepsilon$)$=[0,m_{j^\varepsilon}]$.
\item [iv.] $\alpha_{WR}(\lambda^+_0,\lambda^-_0,h^0,j^0) < 1$.
\end{enumerate} Then for $\varepsilon \geq 0$ sufficiently small there exists a unique Gibbs measure $\mu^\varepsilon$ of the associated Widom-Rowlinson model. Furthermore, we have the convergence
$$
\mu^\varepsilon \overset{loc}{\rightarrow} \mu^0.
$$
\end{teo}

\begin{proof}It suffices to see that this family of models is under the hypothesis of Theorem \ref{convabs}. For this purpose, consider $L > m_{h^0}$ and define the interspecies repulsion function
$$
h:= \sup_{0 \leq \varepsilon < \varepsilon_L} h^\varepsilon
$$ where $\varepsilon_L > 0$ is such that $m_{h^\varepsilon} < L$ for all $0 \leq \varepsilon < \varepsilon_L$. Next, take $K > m_{j^0}$ and define the type independent repulsion function $j$ in the analogous manner.
Let us observe that $h$ and $j$ are both monotone decreasing, have bounded support and also satisfy $h^\varepsilon \leq h$ and $j^\varepsilon \leq j$ for every $\varepsilon \geq 0$ sufficiently small. Furthermore, we may take $L$ and $K$ sufficiently close to $m_{h^0}$ and $m_{j^0}$ respectively so as to guarantee that there exist $\lambda^+ > \lambda^+_0$, $\lambda^- > \lambda^-_0$ such that
$$
\alpha_{WR}(\lambda^+,\lambda^-,h,j) < 1.
$$ Finally, if we consider the Widom-Rowlinson model with fugacities $\lambda^+$ and $\lambda^-$, interspecies repulsion function $h$ and type-independent repulsion function $j$ then for $\varepsilon \geq 0$ sufficiently small this model acts as a majorant. Indeed, in the notation of Theorem \ref{convabs} we have:
\begin{enumerate}
\item [$\bullet$] $\nu^\varepsilon \ll \nu$ with density given by
$$
\frac{d\nu^\varepsilon}{d\nu} (\gamma_x) = \frac{\lambda^+_\varepsilon}{\lambda^+}\mathbbm{1}_{\{\gamma = +\}} + \frac{\lambda^-_\varepsilon}{\lambda^-}\mathbbm{1}_{\{\gamma = -\}}
$$ which satisfies $0 \leq \frac{d\nu^\varepsilon}{d\nu} \leq 1$ for $\varepsilon \geq 0$ sufficiently small.
\item [$\bullet$] (ii) is a direct consequence from the fact that $h^\varepsilon \leq h$ and $j^\varepsilon \leq j$ for every $\varepsilon \geq 0$ sufficiently small.
\item [$\bullet$] (iii) follows from the fact that $\Delta E^H = 0$ and $\Delta E^{H^\varepsilon}=0$ for all $\varepsilon \geq 0$ since all the interactions are repulsive.
\item [$\bullet$] $\lim_{\varepsilon \rightarrow 0^+} \Delta \tilde{E}^{H^\varepsilon}_\eta (\gamma_x) = \Delta \tilde{E}^{H^0}_\eta (\gamma_x)$ for all $\eta \in \mathcal{N}( \R^d \times \{+,-\})$ and $\gamma_x \in  \R^d \times \{+,-\}$ such that both $\eta$ and $\eta + \delta_{\gamma_x}$ are outside the dynamically impossible set
    $$
    N=\{ \xi \in \mathcal{N}(\R^d \times \{+,-\}) : \exists\,\, \gamma_x \neq \tilde{\gamma}_y \in \langle \xi \rangle \text{ such that } \|x-y\|_\infty \in \{ m_h^0 , m_j^0 \} \}
    $$ by assumptions (i), (ii) and (iii) in the statement of the theorem.
\end{enumerate}By Theorem \ref{convabs} this concludes the proof.
\end{proof}

Observe that Theorem \ref{convabswr} shows, in the particular example of the heavily diluted Widom-Rowlinson model, that Gibbs measures of softcore \mbox{models} converge as the \mbox{repulsion} force tends to infinity towards the Gibbs measure of its corresponding hardcore analogue. It is clear from the proof of Theorem \ref{convabs} that this behavior also holds for other systems in a similar situation.

\subsection{Thin rods model}

For the thin rods model Theorem \ref{convabs} yields the following result.

\begin{teo}\label{convabstr} Let us consider for each $\varepsilon \geq 0$ the thin rods model with fugacity $\lambda^\varepsilon$, rod length $2l^\varepsilon$ and orientation measure $\rho^\varepsilon$.
Assume that the following conditions hold:
\begin{enumerate}
\item [i.] $\lim_{\varepsilon \rightarrow 0^+} \lambda^\varepsilon = \lambda^0$.
\item [ii.] $\lim_{\varepsilon \rightarrow 0^+} l^\varepsilon = l^0$.
\item [iii.] There exists a probability measure $\rho$ on the circle such that $\rho^\varepsilon \ll \rho$ for every $\varepsilon \geq 0$ with density $\frac{d \rho^\varepsilon}{d \rho}$ satisfying $0     \leq \frac{d \rho^\varepsilon}{d \rho} \leq 1$.
\item [iv.] $4\lambda^0 (l^0)^2 \sigma_2 < 1$.
\end{enumerate} Then for $\varepsilon \geq 0$ sufficiently small there exists a unique Gibbs measure $\mu^\varepsilon$ of the associated thin rods model. Furthermore, we have the convergence
$$
\mu^\varepsilon \overset{loc}{\rightarrow} \mu^0.
$$
\end{teo}

The proof of this result is similar to that of Theorems \ref{convabsic} and \ref{convabswr} so we omit it here. Nonetheless, we would like to point out that, just as in Section \ref{teounigibssexamples}, condition (iv) in the statement of Theorem \ref{convabstr} may be relaxed provided that we have further knowledge on the measures $(\rho^\varepsilon)_{\varepsilon \geq 0}$.

\newpage
\section{Resumen del Capítulo 10}

En este capítulo mostramos la continuidad de las medidas de Gibbs de modelos altamente diluidos, i.e. bajo la condición de unicidad dada en el Capítulo 9, con respecto a su medida de intensidad $\nu$ y Hamiltoniano $H$ en el caso absolutamente continuo. Este escenario típicamente incluye continuidad con respecto a los parámetros del modelo como pueden ser la densidad de partículas, la temperatura inversa y el rango de interacción entre otros.

El resultado principal está contenido en el Teorema \ref{convabs} arriba. Esencialmente, éste garantiza continuidad con respecto a pequeños cambios en $(\nu,H)$ que sean absolutamente continuos en $\nu$ bajo la existencia de un modelo mayorante (ver Teorema \ref{convabs}) que sea altamente diluido. La demostración consiste en acoplar de manera conveniente las medidas de Gibbs del modelo original y el modificado mediante la construcción hacia el pasado de la dinámica de Fernández-Ferrari-Garcia. Es para poder construir este acoplamiento de manera exitosa que se requiere la existencia de un modelo mayorante.

Un aspecto importante a destacar de este resultado es que prueba la continuidad con respecto a la convergencia local de medidas de probabilidad, mientras que hasta ahora sólo era conocida la continuidad con respecto a la convergencia en distribución que es, al menos en el contexto continuo, más débil que la local.

Por último, discutimos algunas aplicaciones de este resultado. Mostramos que en la práctica los modelos mayorantes siempre existen bajo la condición de unicidad, y que típicamente se pueden obtener mediante un ligero incremento en la densidad de partículas (o disminución de la temperatura inversa) y/o del rango de interacción del modelo $(\nu,H)$. Obtenemos así, entre otros resultados, la continuidad de la medida de Gibbs con respecto a la densidad de partículas y rango de exclusión en el modelo de Widom-Rowlinson (tanto continuo como discreto), continuidad con respecto a la temperatura inversa para cada medida de Gibbs extremal en el modelo de Ising a baja temperatura y con respecto a la medida de orientación en el modelo de las varas finas.

\chapter{Discretization of Gibbs measures}\label{secdis}

In this chapter we establish the continuity of Gibbs measures in heavily diluted models with respect to discretization procedures. We begin by introducing a formal definition of discretization and then move on to establish a general continuity result in this scenario. Finally, we conclude with some examples and applications.

\section{A general discretization result}

\begin{defi} A metric space $X$ is called \textit{absolutely locally compact} if for any $\delta > 0$ the closed balls $\overline{B}(x,\delta)$ are compact for all $x \in X$.
\end{defi}

\begin{obs}$\,$
\begin{itemize}
\item If $X$ is an absolutely locally compact metric space then for any $\delta > 0$ and compact set $K \subseteq G$ the $\delta$-neighborhood of $K$ denoted by $K^{(\delta)}$ has compact closure.
\item If $X$ and $Y$ are absolutely locally compact metric spaces then so is $X \times Y$.
\end{itemize}
\end{obs}

\begin{defi} A metric space $X$ is called a \textit{discretizable} if it is complete, separable and absolutely locally compact.
\end{defi}

\begin{obs} If $X$ and $Y$ are discretizable metric spaces then so is the product $X \times Y$.
\end{obs}

Throughout the rest of the section we shall with diluted models where both $S$ and $G$ are discretizable metric spaces so that $S \times G$ remains a discretizable metric space under the product metric $d_{S\times G}:= d_S + d_G$.

\begin{defi} \label{discdefi}
A family $(D_\varepsilon)_{\varepsilon > 0}$ of measurable applications $D_\varepsilon : S \times G \to S\times G$ is called a \textit{discretization family} if for every $\varepsilon > 0$ and $\gamma_x \in S \times G$ one has
\begin{equation}\label{rho}
d_{S \times G}( D_\varepsilon(\gamma_x), \gamma_x ) \leq \varepsilon.
\end{equation} The application $D_\varepsilon$ shall be called the $\varepsilon$-discretization operator.
\end{defi}

\newpage
\begin{ejem}$\,$
\begin{enumerate}
\item [$\bullet$] \textit{Spatial discretization}. Let us consider $S = \R^d$ and $G=\{1,\dots,q\}$ for some $q \in \N$. For each $\varepsilon > 0$ we define $\varepsilon$-discretization operator $D_\varepsilon$ by the formula
$$
D_\varepsilon ( x, \gamma ) = ( x_\varepsilon , \gamma )
$$ where if $x=(x_1,\dots,x_d) \in \R^d$ we set
\begin{equation}\label{spatialdiscret}
x_\varepsilon := \left( \varepsilon \left[\frac{x_1}{\varepsilon}\right],\dots,\varepsilon \left[\frac{x_d}{\varepsilon}\right]\right).
\end{equation}
\item [$\bullet$] \textit{Spin discretization}. Let us consider $S = \R^2$ and $G = S^1_* := [0,\pi)$.
For each $\varepsilon > 0$ we define $\varepsilon$-discretization operator $D_\varepsilon$ by the formula
\begin{equation}\label{spindiscret}
D_\varepsilon ( x, \gamma ) = \left( x , \varepsilon \left[\frac{\theta}{\varepsilon}\right] \right)
\end{equation}
\end{enumerate}
\end{ejem}

Notice that, in order to remain faithful to the idea of discretization, in Definition \ref{discdefi} it would be natural to also require the image of $D_\varepsilon$ to be countable for every $\varepsilon > 0$.
However, this extra assumption is not needed for our results and so we leave \mbox{it out of the definition.} From now onwards, to simplify the notation we shall write $\gamma_x^\varepsilon$ instead of $D_\varepsilon (\gamma_x)$.

Now, let us consider some fixed discretization family $(D_\varepsilon)_{\varepsilon > 0}$ on the space $S \times G$. Given a Poisson process $\overline{\Pi}$ on $\mathcal{C} \times [0,1]$ with intensity measure $\overline{\phi}_\nu$ we may define for each $\varepsilon > 0$ the $\varepsilon$-discretized process $\overline{\Pi}^\varepsilon$ (or simply $\varepsilon$-process) by the formula
\begin{equation}\label{piepsilon}
\overline{\Pi}^\varepsilon := \{ ( \gamma_x^\varepsilon, t, s, u ) \in \mathcal{C} \times [0,1] : (\gamma_x, t, s, u ) \in \overline{\Pi}\}.
\end{equation}Let us observe that $\overline{\Pi}^\varepsilon$ is a Poisson process on $\mathcal{C} \times [0,1]$ with intensity measure $\overline{\phi}_{\nu_\varepsilon}$ where $\nu_\varepsilon$ denotes the $\varepsilon$-\textit{discretized intensity measure} defined by the formula
$$
\nu_\varepsilon := \nu \circ D_\varepsilon^{-1}.
$$ Furthermore, \eqref{piepsilon} establishes a one-to-one correspondence between cylinders of \mbox{$\Pi$ and $\Pi^\varepsilon$.} With this in mind, we shall write $C_\varepsilon$ to denote the $\varepsilon$-cylinder in $\Pi^\varepsilon$ which corresponds to the cylinder $C \in \Pi$, i.e.,
if $C=(\gamma_x,t,s)$ then we shall set $C_\varepsilon = (\gamma_x^\varepsilon,t,s )$.

\begin{prop}\label{convaga} For each $t \in \R$ we have $\Pi^\varepsilon_t \overset{as}{\longrightarrow} \Pi_t$ as $\varepsilon \rightarrow 0^+$ with the vague topology.
\end{prop}

\begin{proof} Straightforward consequence of the following lemma.
\end{proof}

\begin{lema}\label{lemaconvaga} Let $\xi \in \mathcal{N}(S \times G)$ and for each $\varepsilon > 0$ consider the configuration $\xi^\varepsilon$ defined by the standard representation
$$
\xi^\varepsilon = \sum_{\gamma_x \in Q_\xi} m(\gamma_x)\delta_{\gamma_x^\varepsilon}.\footnote{The fact that for every $\varepsilon > 0$ the configuration $\xi^\varepsilon$ is indeed \textit{locally finite} follows from the absolute local compactness of $S \times G$.}
$$ Then with respect to the vague topology in $\mathcal{N}(S \times G)$ we have $\lim_{\varepsilon \rightarrow 0^+} \xi^\varepsilon = \xi.$
\end{lema}

\begin{proof} It suffices to show that for each compact set $K \subseteq S \times G$ and $\delta > 0$ there exists $\varepsilon_0 > 0$ small enough such that $\xi^\varepsilon \in (\xi)_{K,\delta}$ for all $0 <\varepsilon < \varepsilon_0$.
Notice that if we take
$$
\varepsilon_0 := \left(\frac{1}{2}\min\{ d_{S \times G}\left(\gamma_x,\tilde{\gamma}_y\right) : \gamma_x\neq \tilde{\gamma}_y \in [\xi_{K_\delta}]\}\right)\wedge \delta > 0
$$ where $K_\delta = \{ \gamma_x \in S \times G : d_{S \times G} (\gamma_x, K) < \delta \}$ is the $\delta$-neighborhood of $K$ then for every $0 < \varepsilon < \varepsilon_0$ we have that $\xi^\varepsilon \in (\xi)_{K,\delta}$ since:
    \begin{enumerate}
    \item [i.] By the mere definition of $D_\varepsilon$ to each point of $\xi^\varepsilon_K$ we can assign at least one point of $\xi$ at a distance smaller than $\varepsilon$ (without any regard for their respective multiplicities). Moreover, since $\varepsilon < \varepsilon_0$ we have that there is at most one point of $\xi$ in these conditions so that the multiplicity must be preserved. Thus we may define $p: [\xi^\varepsilon_K] \to [\xi]$ by the formula
        $$
        p(\gamma_x^\varepsilon, i)=\left(D_\varepsilon^{-1}(\gamma_x^\varepsilon),i\right)
        $$ which is clearly injective. This shows that $\xi^\varepsilon_K \preceq_\delta \xi$.
    \item [ii.] Once again, by definition of $D_\varepsilon$ to each point of $\xi_K$ we can assign a point of $\xi^\varepsilon$ at a distance smaller than $\varepsilon$ (without any regard for their respective multiplicities). Moreover, since $0 < \varepsilon < \varepsilon_0$ this assignation is injective. Hence, if $p:[\xi_K] \to [\xi^\varepsilon]$ is defined by the formula
    $$
    p(\gamma_x, i)=\left(\gamma_x^\varepsilon,i\right).
    $$ then $p$ is injective for $0 < \varepsilon < \varepsilon_0$, which shows that $\xi_K \preceq_\delta \xi^\varepsilon$.
\end{enumerate}
\end{proof}

\begin{teo}\label{convdis} Let $(D_\varepsilon)_{\varepsilon > 0}$ be a discretization family and consider a family of diluted models $(\nu^\varepsilon,H^\varepsilon)_{\varepsilon \geq 0}$ such that
\begin{enumerate}
\item [$\bullet$] There exists a heavily diluted model $(\nu,H)$ satisfying
\begin{enumerate}
\item [i.] For every $\varepsilon \geq 0$ the intensity measure $\nu^\varepsilon$ satisfies
$$
\nu^\varepsilon = \nu \circ D_\varepsilon^{-1}
$$ where $D_0$ is set as the identity operator, i.e. $\nu^0 = \nu$.
\item [ii.] For every $\varepsilon \geq 0$ and $\gamma_x \in S \times G$ we have that $D_\varepsilon^{-1}\left(I^{H^\varepsilon}(\{\gamma^\varepsilon_x\})\right) \subseteq I^H (\{\gamma_x\})$, i.e. if $\tilde{\gamma}^\varepsilon_y \rightharpoonup_{H^\varepsilon} \gamma_x^\varepsilon$ for some $\varepsilon \geq 0$ then $\tilde{\gamma}_y \rightharpoonup_{H} \gamma_x$.
\item [iii.] $\Delta E^H \leq \inf_{\varepsilon \geq 0} \Delta E^{H^\varepsilon}.$
\end{enumerate}
\item [$\bullet \bullet$] $\lim_{\varepsilon \rightarrow 0^+} \Delta E^{H^\varepsilon}_{\eta^\varepsilon} (\gamma^\varepsilon_x) = \Delta E^{H^0}_\eta (\gamma_x)$ for every $\eta \in \mathcal{N}(S \times G)$ and $\gamma_x \in S \times G$.
\end{enumerate} Then each diluted model $(\nu^\varepsilon,H^\varepsilon)$ admits exactly one Gibbs measure $\mu^\varepsilon$ and as $\varepsilon \rightarrow 0^+$
$$
\mu^{\varepsilon} \overset{d}{\longrightarrow} \mu^0.
$$ The model $(\nu,H)$ is called a \textit{majorant model} for $(\nu^\varepsilon,H^\varepsilon)_{\varepsilon \geq 0}$.
\end{teo}

\begin{proof} Let us start by showing that each model $(\nu^\varepsilon,H^\varepsilon)$ admits exactly one Gibbs measure for every $\varepsilon \geq 0$. To do this let us consider a Poisson process $\overline{\Pi}$ on $\mathcal{C} \times [0,1]$ with intensity measure $\nu \times e^{-\Delta E^H} \mathcal{L} \times \mathcal{L}_{\R^+} \times \mathcal{U}[0,1]$ and and its corresponding discretizations $(\overline{\Pi}^\varepsilon)_{\varepsilon > 0}$. By the the proof of Theorem \ref{teounigibbs} we see that it suffices to show that for each $\varepsilon \geq 0$ and $\Lambda \in \B^0_S$ the clan of ancestors at time 0 with respect to the Hamiltonian $H^\varepsilon$ and underlying free process $\Pi^\varepsilon$ is finite almost surely. But this follows from the heavy diluteness of $(\nu,H)$ since
\begin{equation}\label{ancesincdis}
\mathcal{A}^{0,H^\varepsilon}(\Lambda \times G) \subseteq D_\varepsilon \left( \A^{0,H}(\overline{\Lambda_{\varepsilon}} \times G)\right)
\end{equation} where $\overline{\Lambda_{\varepsilon}}$ denotes the closed $\varepsilon$-neighborhood of $\Lambda$ and for $\Gamma \subseteq \mathcal{C}$ we set
$$
D_\varepsilon (\Gamma) = \{ (\gamma^\varepsilon_x , t,s ) \in \mathcal{C} : (\gamma_x,t,s) \in \Gamma \}.
$$ This settles the first statement.

To establish the local convergence, we shall proceed as in the proof of Theorem \ref{convabs}. We couple all measures $\mu^\varepsilon$ simultaneously by considering the infinite-volume stationary FFG processes $\mathcal{K}^\varepsilon$ defined as
\begin{equation}\label{keptscaled2}
\mathcal{K}^\varepsilon = \{ (\gamma_x^\varepsilon,t,s) \in \Pi : F(\gamma_x^\varepsilon,t,s) < M^\varepsilon(\gamma_x^\varepsilon | \mathcal{K}^\varepsilon_{t^-}) \}
\end{equation} where for each $\gamma_x \in S \times G$ and $\xi \in \mathcal{N}(S \times G)$ we define
$$
M^\varepsilon (\gamma_x |\xi) := e^{- (\Delta E^{H^\varepsilon}_\xi(\gamma_x) - \Delta E^H)}.
$$ Just as in the proof of Theorem \ref{convaga}, for each $\varepsilon \geq 0$ the process $\mathcal{K}^\varepsilon$ is stationary with invariant measure $\mu^\varepsilon$ and thus it will suffice to show that as $\varepsilon \rightarrow 0^+$
$$
\mathcal{K}^\varepsilon_0 \overset{as}{\longrightarrow} \mathcal{K}^0_0.
$$ Let us take then a compact set $K \in \B^0_{S\times G}$ and $\Lambda \in \B^0_S$ such that $K \subseteq \Lambda \times G$. Now, since $\lim_{\varepsilon \rightarrow 0^+} \Delta E^{H^\varepsilon}_{\eta^\varepsilon} (\gamma^\varepsilon_x) = \Delta E^{H^0}_\eta (\gamma_x)$ for every $\eta \in \mathcal{N}(S \times G)$ and $\gamma_x \in S \times G$, it follows that there exists (random) $\varepsilon_0 > 0$ such that if $0 \leq \varepsilon < \varepsilon_0$ then
\begin{equation}\label{ancesincdisc1}
\mathcal{K}^\varepsilon_{D_\varepsilon(\A^{0,H} ( \Lambda_1 \times G))} = D_\varepsilon \left(\mathcal{K}^0_{\A^{0,H} ( \Lambda_1 \times G)}\right),
\end{equation} where $\Lambda_1$ denotes the $1$-neighborhood of $\Lambda$. Indeed, if (random) $N \in \N$ is such that $\A^{0,H}_n(\Lambda_1 \times G) = \emptyset$ for every $n > N$ almost surely then for every cylinder $C \in \A^{0,H}_N(\Lambda_1 \times G)$ and $\varepsilon \geq 0$ we have that
$$
C^\varepsilon \in \mathcal{K}^\varepsilon \Longleftrightarrow F(C) < M^\varepsilon( basis(C^\varepsilon) | \emptyset )
$$ from which we immediately obtain that for $\varepsilon$ (randomly) small enough
$$
\mathcal{K}^\varepsilon_{D_\varepsilon(\A^{0,H}_N ( \Lambda_1 \times G))} = D_\varepsilon \left(\mathcal{K}^0_{\A^{0,H}_N( \Lambda_1 \times G)}\right)
$$ and one may proceed with the succeeding generations by induction using inclusion \eqref{ancesincdis}.
From \eqref{ancesincdisc1} and the inclusion
$$
D^{-1}_\varepsilon (K) \subseteq \Lambda_1 \times G
$$ valid for every $0 <\varepsilon < 1$ one can show as in the proof of Lemma \ref{lemaconvaga} that given $\delta > 0$ for $\varepsilon$ (randomly) small enough we have $\mathcal{K}^\varepsilon_0 \in (\mathcal{K}^0_0)_{K,\delta}$. This establishes the almost sure convergence and thus concludes the proof.
\end{proof}

Just as it was the case for Theorem \ref{convabs}, majorant models in this context are fairly easy to obtain, and one can generally do so by slightly ``inflating'' the interaction of the corresponding limit model in some appropriate sense. Also, let us notice that the proof of \mbox{Theorem \ref{convdis}} does not only yield convergence in distribution but in fact provides a coupling between the corresponding Gibbs measures in which the convergence takes place in the stronger almost sure sense. Once again, we stress this fact since other methods used to obtain these type of results (i.e. cluster expansion or disagreement percolation methods) in general cannot produce such a coupling. Finally, condition ($\bullet \bullet$) in the statement of the theorem may be replaced by the following weaker condition:
\begin{enumerate}
\item [($\bullet \bullet^*$)] There exists a dynamically impossible set $N$ for the intensity measure $\nu$ such that
$$
\lim_{\varepsilon \rightarrow 0^+} \Delta E^{H^\varepsilon}_{\eta^\varepsilon} (\gamma^\varepsilon_x) = \Delta E^{H^0}_\eta (\gamma_x)
$$ for all $\eta \in \mathcal{N}(S \times G)$ and $\gamma_x \in S \times G$ satisfying $\eta, \eta + \delta_{\gamma_x} \in N^c$.
\end{enumerate}

\section{Applications}

We now discuss two applications of Theorem \ref{convdis} related to models in Section \ref{examples}.

\subsection{Widom-Rowlinson model}

As a direct consequence of Theorem \ref{convdis} one obtains that Gibbs measures in discrete heavily diluted models converge, when properly rescaled, towards the Gibbs measure of the analogous continuum model.
As an example we study the particular case of the Widom-Rowlinson model. Other models may be handled in the same fashion.

\begin{teo}\label{diswr} For $\lambda^0,r^0 > 0$ such that $\lambda^0 (2r^0)^d < 1$ we have the following:
\begin{enumerate}
\item [i.] The continuum Widom-Rowlinson model on $\mathcal{N}(\R^d \times \{+,-\})$ with fugacity $\lambda^0$ and exclusion radius $r^0$ admits exactly one Gibbs measure, which we shall denote by $\mu^0$.
\item [ii.] For $ 0 < \varepsilon < \sqrt[d]{\frac{1}{\lambda^0} - (2r^0)^d}$ the discrete Widom-Rowlinson model on $\mathcal{N}(\Z^d \times \{+,-\})$ with fugacity $\varepsilon^d\lambda^0$
and exclusion radius $\frac{r^0}{\varepsilon}$ admits exactly one Gibbs measure $\tilde{\mu}^\varepsilon$.
\item [iii.] Provided $0 < \varepsilon < \sqrt[d]{\frac{1}{\lambda^0} - (2r^0)^d}$ as $\varepsilon \rightarrow 0^+$ we have
$$
\tilde{\mu}^\varepsilon \circ i_\varepsilon^{-1} \overset{d}{\longrightarrow} \mu^0.
$$ where for each $\varepsilon > 0$ we define the \textit{shrinking map} $i_\varepsilon : \Z^d \times \{+,-\} \to \R^d \times \{+,-\}$ by the formula
$$
i_\varepsilon ( x , \gamma ) = (\varepsilon\cdot x, \gamma).
$$
\end{enumerate}
\end{teo}

\begin{proof} The first two statements are a direct consequence of \mbox{Theorem \ref{teounigibbs}, \eqref{wrdc1} and \eqref{wrdc2}.} To show (iii), first we consider the spatial discretization family $(D_\varepsilon)_{\varepsilon \geq 0}$ given by \eqref{spatialdiscret} and for each $\varepsilon \geq 0$ set the intensity measure $\nu^\varepsilon$ as
$$
\nu^\varepsilon := \nu \circ D_\varepsilon^{-1}
$$ where
$$
\nu := \left( \lambda \mathcal{L}^d \times \delta_+ \right) +  \left(\lambda \mathcal{L}^d \times \delta_-\right).
$$ Then we set the Hamiltonian $H^0$ as in Section \ref{examples}, i.e.
$$
H_{\Lambda|\eta}(\sigma)= \sum_{(\gamma_x ,\tilde{\gamma}_y) \in e_{\Lambda}(\sigma|\eta)} U( \gamma_x , \tilde{\gamma}_y )
$$ where
\begin{equation}\label{wrud}
U(\gamma_x,\tilde{\gamma}_y) := \left\{ \begin{array}{ll} +\infty &\text{if }\gamma \neq \tilde{\gamma}\text{ and }\|x-y\|_\infty \leq r^0\\ 0 &\text{otherwise.}\end{array}\right.
\end{equation}
Finally, for every $\varepsilon > 0$ we consider the Hamiltonian $H^\varepsilon$ defined for each $\Lambda \in \B^0_{\R^d}$ and \mbox{$\eta \in \mathcal{N}(\R^d \times \{+,-\})$} by the formula
$$
H^\varepsilon_{\Lambda|\eta}(\sigma) := \sum_{(\gamma_x ,\tilde{\gamma}_y) \in e^\varepsilon_{\Lambda}(\sigma|\eta)} U( \gamma_x , \tilde{\gamma}_y ) + \sum_{x_\varepsilon \in \Lambda} V_{x_\varepsilon}(\sigma)
$$ where
$$
e^\varepsilon_{\Lambda}(\sigma|\eta) := \{ (\gamma_x ,\tilde{\gamma}_y) \in \langle \sigma^\varepsilon_{\Lambda \times G} \cdot \eta_{\Lambda^c \times G} \rangle^2 : x \in \Lambda \},
$$ the pair interaction $U$ is the same as in \eqref{wrud}, $x_\varepsilon$ is defined as in \eqref{spatialdiscret} and
$$
V_{x_\varepsilon}(\sigma) := \left\{ \begin{array}{ll} +\infty &\text{if } \sigma( \{x_\varepsilon\} \times \{+,-\} ) > 1 \\ 0 &\text{otherwise.}\end{array}\right.
$$ Now, the crucial observation is that for every $\varepsilon > 0$ the diluted model specified by the pair $(\nu^\varepsilon, H^\varepsilon)$ is essentially the shrunken version of the discrete Widom-Rowlinson model of fugacity $\varepsilon^d \lambda^0$ and exclusion radius $\frac{r^0}{\varepsilon}$. More precisely, for every $\Lambda \in \B^0_{\Z^d}$ and $\varepsilon > 0$ we have
\begin{equation}\label{igualdaddis}
\mu^\varepsilon_{i_\varepsilon(\Lambda)|\emptyset} = \tilde{\mu}^\varepsilon_{\Lambda|\emptyset} \circ i_\varepsilon^{-1}
\end{equation} where $\tilde{\mu}^\varepsilon_{\Lambda|\emptyset}$ is the Boltzmann-Gibbs distribution with empty boundary condition associated to the discrete Widom-Rowlinson model whereas $\mu^\varepsilon_{\Lambda|\emptyset}$ is the one associated to $(\nu^\varepsilon,H^\varepsilon)$. Thus, by taking the limit as $\Lambda \nearrow \Z^d$, Theorem \ref{teounigibbs} yields for $0 < \varepsilon < \sqrt[d]{\frac{1}{\lambda^0} - (2r^0)^d}$
$$
\mu^\varepsilon = \tilde{\mu}^\varepsilon \circ i_{\varepsilon}^{-1}
$$ where $\mu^\varepsilon$ is the unique Gibbs measure of the diluted model given by the pair $(\nu^\varepsilon, H^\varepsilon)$. Hence, it suffices to show that the family $(\nu^\varepsilon,H^\varepsilon)$ is under the hypothesis of Theorem \ref{convdis}. But notice that if for $\delta > 0$ we define the Hamiltonian $H$ by the formula
$$
H_{\Lambda|\eta}(\sigma)= \sum_{(\gamma_x ,\tilde{\gamma}_y) \in e_{\Lambda}(\sigma|\eta)} U^\delta( \gamma_x , \tilde{\gamma}_y ) + \sum_{\gamma_x \in \Lambda} V_{x}^\delta( \sigma )
$$ where
$$
U^\delta(\gamma_x,\tilde{\gamma}_y) := \left\{ \begin{array}{ll} +\infty &\text{if }\gamma \neq \tilde{\gamma}\text{ and }\|x-y\|_\infty \leq r^0 + \delta\\ 0 &\text{otherwise.}\end{array}\right.
$$
$$
V_{x}^\delta( \sigma ) := \left\{ \begin{array}{ll} +\infty &\text{if } \sigma( \{y \in \R^d : \|x - y \|_\infty \leq \delta\} \times \{+,-\} ) > 1 \\ 0 &\text{otherwise}\end{array}\right.
$$ then for $\delta > 0$ small the diluted model $(\nu,H)$ acts a majorant for $\varepsilon \geq 0$ sufficiently small. Indeed, we have that
\begin{enumerate}
\item [$\bullet$] (i) holds trivially by the choice of measures $\nu^\varepsilon$.
\item [$\bullet$] (ii) holds for all $\varepsilon < \frac{\delta}{2}$ by definition of $D_\varepsilon$.
\item [$\bullet$] (iii) holds since $\Delta E^H = 0 = \inf_{\varepsilon \geq 0} \Delta E^{H^\varepsilon}$ due to the fact that all interactions considered are repulsive.
\item [$\bullet$] $\lim_{\varepsilon \rightarrow 0^+} \Delta E^{H^\varepsilon}_{\eta^\varepsilon} (\gamma_x^\varepsilon) = \Delta E^{H^0}_\eta (\gamma_x)$ for all $\eta \in \mathcal{N}( \R^d \times \{+,-\})$ and $\gamma_x \in  \R^d \times \{+,-\}$ such that both $\eta$ and $\eta + \delta_{\gamma_x}$ are outside the dynamically impossible set $N_1 \cup N_2$ where
    $$
    N_1=\{ \xi \in \mathcal{N}(\R^d \times \{+,-\}) : \exists\,\, \gamma_x \neq \tilde{\gamma}_y \in \langle \xi \rangle \text{ such that } \|x-y\|_\infty = r^0 \}
$$ and
$$
N_2 = \{ \xi \in \mathcal{N}(\R^d \times \{+,-\}) : \sigma( \{x\} \times \{+,-\} ) > 1 \text{ for some }x \in \R^d \}.
$$
\end{enumerate}
If we take $\delta > 0$ such that $\lambda^0(2(r^0+\delta))^d < 1$ then the model $(\nu,H)$ is heavily diluted. Thus, by Theorem \ref{convdis} we obtain the result.
\end{proof}

\subsection{Thin rods model}

Another application of Theorem \ref{convdis} is to study the limit of the thin rods model when the number of possible orientations tends to infinity. As expected, under the heavily diluted regime we have the following result.

\begin{teo} Given $\lambda, l > 0$ and a probability measure $\rho$ on $S^1_*$, for each $\varepsilon > 0$ consider the thin rods model on $\mathcal{N}(\R^2 \times S^1_*)$ with fugacity $\lambda$, rod length $2l$ and orientation measure
$$
\rho^\varepsilon = \sum_{i=0}^{n^\varepsilon} w^\varepsilon(i) \delta_{i \varepsilon}
$$ where $n^\varepsilon:= \left[\frac{\pi}{\varepsilon}\right]$ and $w^\varepsilon(i):= \rho( \{\theta \in S^1_* : \left[\frac{\pi}{\varepsilon}\right]=i \} )$. If $4\lambda l^2 \sigma_2 < 1$ then for every $\varepsilon > 0$ there exists a unique Gibbs measure $\mu^\varepsilon$ of the corresponding thin rods model. Furthermore, as $\varepsilon \rightarrow 0^+$ we have
$$
\mu^\varepsilon \overset{d}{\rightarrow} \mu^0
$$ where $\mu^0$ is the unique Gibbs measure of the thin rods model with fugacity $\lambda$, \mbox{rod length $2l$} and orientation measure $\rho$.
\end{teo}

We omit the proof of this result since it goes very much along the lines of \mbox{Theorem \ref{diswr}} but using the spin discretization family introduced in \eqref{spindiscret} instead of the spatial one.

\subsection{Some important remarks on discretization procedures}

Suppose that we have some continuum heavily diluted model and let $\mu$ denote its unique Gibbs measure. One could then ask what can be said about the discretized measures $\mu^\varepsilon := \mu \circ D_\varepsilon^{-1}$ for $\varepsilon > 0$. For example,
\begin{enumerate}
\item [i.] Is it true that $\mu^\varepsilon$ is a Gibbs measure for the corresponding discrete model?
\item [ii.] If not, is it close to the actual Gibbs measure of the discrete system?
\end{enumerate}
The examples discussed above show that we cannot expect (i) to be true. Indeed, for example in the Widom-Rowlinson model the discretized Gibbs measures $\mu^\varepsilon$ can assign positive weight to particle configurations in which particles of opposite type are within the exclusion radius; this is because certain allowed configurations in the \mbox{continuum system} may violate the exclusion radius restriction when discretized. Therefore, in general it is not enough to discretize the continuum Gibbs measure to obtain the Gibbs measure of the discrete system. What Theorem \ref{convdis} in fact shows is that obtaining the actual discrete Gibbs measure demands a more complicated procedure: one has to discretize the \mbox{continuum} free process and then do the deleting procedure all over again. Nevertheless, \eqref{ancesincdisc1} implies that $\mu^\varepsilon$ is indeed close to the Gibbs measure of the discrete system.

On a similar note, observe that when trying to simulate Gibbs measures of continuum systems using the FFG dynamics, practical limitations prevent the inclusion of all possible configurations in the simulation, and so one inevitably has to replace the original model by a discretized version of it. What Theorem \ref{convdis} also shows is that no problems arise by this replacement, since by \eqref{ancesincdisc1} the simulated discrete measure will be close to the original continuum one.

\newpage
\section{Resumen del Capítulo 11}

Mostramos aquí la continuidad de medidas de Gibbs en modelos altamente diluidos con respecto a procesos de discretización, i.e. la convergencia de modelos discretos a modelos continuos a nivel de las medidas de Gibbs correspondientes. En general, los modelos discretos a los que hacemos referencia pueden ser de dos tipos: con espacio de ubicaciones discreto (el modelo de Widom-Rowlinson discreto, por ejemplo) o con espacio de spines discreto (el modelo de varas finas con finitas orientaciones posibles).

El resultado principal está contenido en el Teorema \ref{convdis} arriba. Nuevamente, se obtiene la convergencia (en distribución) de modelos discretos a sus análogos continuos bajo la existencia un modelo mayorante altamente diluido. La demostración consiste una vez más en acoplar las medidas de Gibbs de los modelos discretos junto a la del modelo continuo límite mediante la construcción hacia el pasado de la dinámica FFG. Aquí es donde se vuelve evidente la necesidad de un marco teórico que nos permita encarar por igual la dinámica tanto en modelos discretos como continuos.

Cabe destacar que el problema de la continuidad con respecto a discretizaciones no ha sido muy estudiado hasta ahora, y que muchas de las técnicas de mayor influencia dentro de la mecánica estadística (como por ejemplo la teoría de Pirogov-Sinai) no se encuentran, en principio, preparadas para lidiar con este tipo de problemas (especialmente para discretizaciones en el espacio de spins). No obstante, en nuestro contexto este tipo de problemas pueden plantearse y resolverse de manera natural.

Luego, a manera de aplicación mostramos que bajo el régimen de unicidad el modelo de Widom-Rowlinson discreto apropiadamente escalado converge, cuando la densidad de partículas tiende a cero y el radio de exclusión tiende a infinito, al correspondiente modelo de Widom-Rowlinson continuo. También mostramos que el modelo de varas finas con $n$ orientaciones converge cuando $n \rightarrow +\infty$ al modelo con un continuo de orientaciones, nuevamente bajo el régimen de unicidad.

Por último, discutimos algunas conclusiones que pueden sacarse a partir del resultado probado. En primer lugar, la demostración del Teorema \ref{convdis} muestra que al discretizar una medida de Gibbs en un modelo continuo no se obtiene, en general, una medida de Gibbs del correspondiente modelo discreto pero que, sin embargo, el resultado se encuentra razonablemente próximo de esta última. Por otro lado, el Teorema \ref{convdis} también garantiza que para simular numéricamente medidas de Gibbs de modelos continuos bajo el régimen de unicidad es razonable simular medidas de Gibbs para modelos que sean aproximaciones discretas de los mismos, ya que éstas serán una buena aproximación de las verdaderas medidas de interés.

\chapter{Applications to Pirogov-Sinai theory}\label{chapterpirogovsinai}

In this final chapter we combine the ideas of previous chapters with the framework of Pirogov-Sinai theory to show that some of the typical results in this theory can be obtained without the use of cluster expansions.
Moreover, we show that this allows us to enlarge the traditional range of validity of the theory in some cases. This constitutes a step towards completing the approach first proposed in \cite{FFG1}.
As a byproduct, we obtain a perfect simulation algorithm for systems at the low temperature or high density regime. For simplicity, we shall discuss the framework of Pirogov-Sinai theory and its applications only in some particular cases, but the experienced reader will understand how to extend these ideas to the general setting. We follow the presentation of this theory given in \cite{Z1}.

\section{Discrete $q$-Potts model of interaction range $r$}

Consider the discrete model on $\{1,\dots,q\}^{\Z^d}$ defined in the traditional manner through the Boltzmann-Gibbs distributions given for each $\Lambda \in \B^0_{\Z^d}$ and $\eta \in \{1,\dots,q\}^{\Z^d}$ by
$$
\mu_{\Lambda}^\eta(\sigma) = \frac{\mathbbm{1}_{\{\sigma_{\Lambda^c} \equiv \eta_{\Lambda^c}\}}}{Z_{\Lambda}^\eta} e^{- \beta \sum_{B : B \cap \Lambda \neq \emptyset} \Phi_B(\sigma_{B})}
$$ with $\beta > 0$ denoting the inverse temperature, $Z_{\Lambda}^\eta$ being the normalizing constant and
\begin{equation}\label{hpotts}
\Phi_B (\sigma_B):= \left\{\begin{array}{ll} \mathbbm{1}_{\{ \sigma(x) \neq \sigma(y)\,,\,\|x-y\|_1 \leq r\}} & \text{ if $B=\{x,y\}$} \\ \\ 0 & \text{ otherwise.}\end{array}\right.
\end{equation}
This model, known as the $q$-Potts model of interaction range $r$, can be interpreted as a direct generalization of the Ising model presented in Section \ref{exampleisingc}.

We define the set $\mathcal{R}:=\{ \eta_1,\dots,\eta_q\}$ of reference configurations, where for $i=1,\dots,q$ the configuration $\eta_i$ corresponds to the $i$-aligned configuration, i.e. \mbox{$\eta_i(x) \equiv i$ for all $x \in \Z^d$.} For convenience purposes, for every $i=1,\dots,q$ we shall denote the corresponding Boltzmann-Gibbs distribution on $\Lambda \in \B^0_{\Z^d}$ simply by $\mu^i_{\Lambda}$ instead of $\mu^{\eta_i}_\Lambda$ \mbox{as we usually do.} \mbox{These reference configurations} shall be of particular interest to us since, as we will see, each of them will represent a different equilibrium state of the system at low temperature. To show this fact we shall need to introduce as in the original Ising model the notion of contour with respect to each of these reference configurations. We proceed as follows.

For $i=1,\dots,q$ we shall say that the site $x \in \Z^d$ is $i$-correct for a given configuration $\sigma \in \{1,\dots,q\}^{\Z^d}$ whenever $\sigma(y)=i$ for every $y \in \Z^d$ such that $\| x - y\|_1 \leq r$.
We label a site as \textit{incorrect} with respect to $\sigma$ if it fails to be $i$-correct for all $i=1,\dots,q$. We define the \textit{defect set} $D_\sigma$ of the configuration $\sigma$ as the set of all incorrect sites with respect to $\sigma$. The restriction of $\sigma$ to any one of the finite connected components of $D_\sigma$ will be called a \textit{contour} of $\sigma$ and the corresponding component will be called the \textit{support} of this contour.

Given a contour $\gamma$, the space $\Z^d - \text{supp}(\gamma)$ is divided into a finite number of connected components, only one of which is infinite. We call one of this components a $i$-component if its neighboring spins in $\gamma$ all have the value $i$. Notice that every component of $\Z^d - \text{supp}(\gamma)$ is an $i$-component for some $i \in \{1,\dots,q\}$. If the infinite component in $\Z^d - \text{supp}(\gamma)$ is a $i$-component we say that $\gamma$ is a $i$-contour and denote this fact by $\gamma^i$.
For any $j \in \{1,\dots,q\}$ we denote by $\text{Int}_{j}(\gamma)$ the union of all finite $j$-components of $\Z^d - \text{supp}\gamma$. Then we set
$$
\text{Int}(\gamma) := \bigcup_{j=1}^q \text{Int}_j(\gamma)\hspace{1cm}\text{V}(\gamma):= \text{supp}(\gamma) \cup \text{Int}(\gamma) \hspace{1cm}\text{Ext}(\gamma):= \Z^d - \text{V}(\gamma).
$$
See Figure \ref{fig4} for a possible configuration of Pirogov-Sinai contours in the Ising model. Finally, if $\gamma$ is a contour of a configuration $\sigma \in \{1,\dots,q\}^{\Z^d}$ then define the energy of $\gamma$ as

\begin{figure}
	\centering
	\includegraphics[width=8cm]{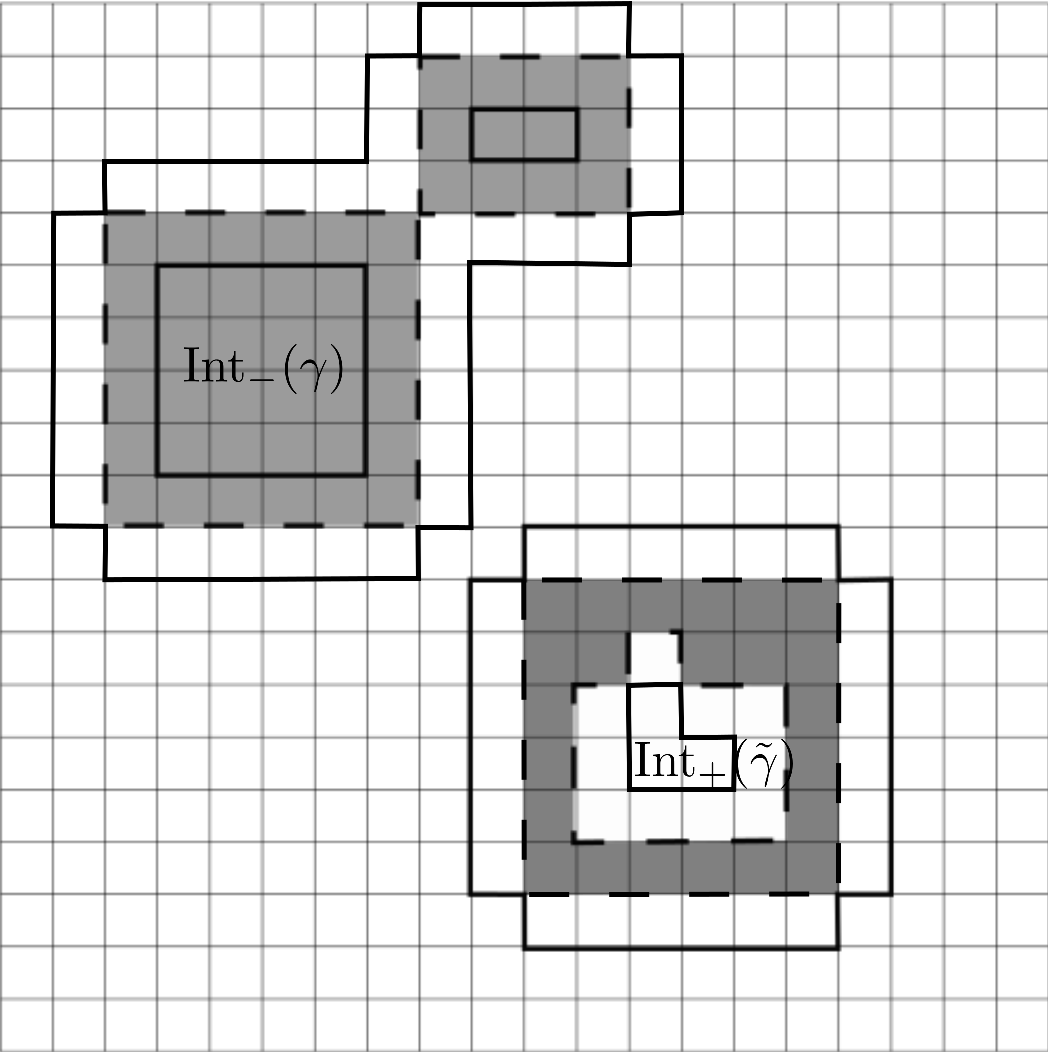}
	\caption{Pirogov-Sinai contours for the Ising model on $\Z^2$.}
	\label{fig4}
\end{figure}

$$
\Phi(\gamma) := \sum_{B \subseteq \Z^d} \frac{|B \cap \text{supp}(\gamma)|}{|B|} \Phi_B(\sigma_B).
$$ Notice that this value does not depend on the choice of $\sigma$, only on $\gamma$.

Now, notice that each configuration $\sigma \in \{1,\dots,q\}^{\Z^d}$ with a finite defect set $D_\sigma$ defines a unique family of contours $\Gamma_\sigma$ from which it can be completely recovered. Furthermore, we have the following result.

\begin{prop}\label{proprepresentacionps} For $\Lambda \in \B^0_{\Z^d}$ let us consider a configuration $\sigma$ such that $D_\sigma \subseteq \Lambda$ and $\sigma(x)=i$ for every $x \in \Z^d$ lying in the unique infinite connected component of $\Z^d - D_\sigma$. If $\Gamma_{\sigma}$ denotes the family of contours associated to $\sigma$, then we have
\begin{equation}\label{ps1}
\mu^{i}_\Lambda (\sigma) = \frac{1}{Z^{i}_\Lambda} e^{-\beta \sum_{\gamma \in \Gamma_\sigma} \Phi(\gamma)}.
\end{equation}
\end{prop}

\begin{proof}For each finite $B \subseteq \Z^d$ let us write
$$
\Phi_B (\sigma_B) = \sum_{x \in B} \frac{1}{|B|} \Phi_B(\sigma_B) = \sum_{\gamma \in \Gamma_\sigma} \frac{|B \cap \text{supp}(\gamma)|}{|B|} \Phi_B(\sigma_B) + \frac{|B \cap (\Z^d - D_\sigma)|}{|B|} \Phi_B(\sigma_B).
$$ Notice that $\sigma$ is necessarily constant on any $B=\{x,y\}$ such that $B \cap (\Z^d - D_\sigma) \neq \emptyset$, which implies that $\Phi_B(\sigma_B) = 0$ for such $B$. Thus, summing over all $B$ such that $B \cap \Lambda \neq \emptyset$ we immediately obtain \eqref{ps1}.
\end{proof}

Looking at \eqref{ps1}, one might be tempted to proceed as for the Ising contours model on Section \ref{exampleisingc}. However, the situation now is more complicated than it was before: it is no longer true that each family of contours with disjoint supports corresponds to some configuration in $\{1,\dots,q\}^{\Z^d}$. Indeed, besides having disjoint supports, nested contours must have matching internal and external labels for the whole family to correspond to some configuration. In particular, if one defines the $q$-Potts contour model by analogy with what was done on Section  \ref{exampleisingc}, the resulting model will violate the bounded energy loss condition, i.e. $\Delta E^* = -\infty$, so that one cannot associate an FFG dynamics to it. Indeed, there exist contour configurations which are forbidden because they carry nested contours with mismatched labels on them, but that can be turned into admissible configurations by adding a suitable contour in between. The energy leap function associated to an
addition of this sort is thus $-\infty$.
Nonetheless, the following procedure by Minlos and Sinai \cite{MS1,MS2} will help solve this problem.

Let us fix $i=1,\dots,q$ and consider the contour model on $\mathcal{N}(\Z^d \times G^i)$, where $G^i$ denotes the space of all $i$-contour shapes, given by the intensity measure
\begin{equation}\label{impottsc}
\nu^i(\gamma^i_x):= e^{-\beta \Phi(\gamma^i_x)}
\end{equation} and the Hamiltonian
\begin{equation}\label{hpottsc}
H^i_{\Lambda|\Gamma'} (\Gamma) = \left\{ \begin{array}{ll} +\infty & \text{ if either $\Gamma$ is incompatible, $\Gamma \not \sim \Gamma'_{\Lambda^c \times G}$ or $\Gamma \not \subset \Lambda$}\\ \\ 0 & \text{ otherwise.}\end{array}\right.
\end{equation}where we say that two contours $\gamma,\gamma'$ are incompatible whenever $d_1( \text{supp}(\gamma), \text{supp}(\gamma') ) \leq 1$ and the expression $\Gamma \subset \Lambda$ indicates that $d_1(V(\gamma^i), \Lambda^c) > 1 \text{ for every contour } \gamma^i \in \Gamma$.
Notice that, as it happened in the Ising contours model, this contour model will also fail to satisfy Assumptions \ref{assump}. Thus, when working with this model we will have to take the necessary precautions already described for the Ising contours model, we omit them here. Also, observe that in this model we still have that compatible families of contours will not, in general, correspond to actual configurations in $\{1,\dots,q\}^{\Z^d}$. Nonetheless, this artificial contour model no longer violates the bounded energy loss condition since its interactions are given only by intersections and \mbox{it also preserves the distribution of \textit{exterior contours}.}

\begin{defi} Given a family $\Gamma$ of contours with disjoint support we say that $\gamma \in \Gamma$ is an \textit{exterior contour} of $\Gamma$ if $\gamma$ is not contained in the interior of any other contour of $\Gamma$. The set of all exterior contours of $\Gamma$ shall be denoted by $\text{Ext}(\Gamma)$. Furthermore, given a configuration $\sigma \in \{1,\dots,q\}^{\Z^d}$ with a finite defect set $D_\sigma$, we denote the set of all exterior contours of $\Gamma_\sigma$ by $\text{Ext}(\sigma)$.
\end{defi}

\begin{prop}\label{propps} Let $\Gamma$ be a finite family of $i$-contours which are pairwise compatible. Then for any $\Lambda \in B^0_{\Z^d}$ such that $d_1(\text{V}(\gamma), \Lambda^c) > 1$ for all $\gamma \in \Gamma$ we have that
$$
\frac{ \mu^{i}_{\Lambda}( \{\sigma \in \{1,\dots,q\}^{\Z^d} : \text{Ext}(\sigma)= \Gamma \} ) }{ \mu^{i}_\Lambda ( \{ \sigma : d_1(\text{V}(\gamma), \Lambda^c) > 1 \text{ for all }\gamma \in \Gamma_\sigma\} )}= \mu_{\Lambda|\emptyset}(\{ \Gamma' \in \mathcal{N}(\Lambda \times G^i) : \text{Ext}(\Gamma') = \Gamma \})
$$ where $\mu_{\Lambda|\emptyset}$ is the Boltzmann-Gibbs distribution of the model $(\nu^i,H^i)$.
\end{prop}

\begin{proof} We can assume that all contours in $\Gamma$ are exterior. Thus, we have to check that
\begin{equation}\label{ps2}
\frac{\sum_{ \sigma : \text{Ext}(\sigma) = \Gamma } e^{-\beta \sum_{\gamma \in \Gamma_\sigma} \Phi(\gamma)}}{ \sum_{\sigma \in C^i(\Lambda)} e^{-\beta \sum_{\gamma \in \Gamma_\sigma} \Phi(\gamma)}}  = \frac{\sum_{ \Gamma' : \text{Ext}(\Gamma') = \Gamma } e^{-\beta \sum_{\gamma^i \in \Gamma'} \Phi(\gamma^i) }}{ \sum_{\Gamma' \in D^i(\Lambda)} e^{-\beta \sum_{\gamma^i \in \Gamma'} \Phi(\gamma^i) }}
\end{equation} where the sums in the left hand side are only over configurations $\sigma$ in the support of $\mu^{i}_{\Lambda}$, those in the right hand side are only over families $\Gamma'$ of compatible $i$-contours and, finally, where we have set
$$
C^i(\Lambda):=\{ \sigma : d_1(\text{V}(\gamma), \Lambda^c) > 1 \text{ for all }\gamma \in \Gamma_\sigma\}
$$ and
$$
D^i(\Lambda):=\{ \Gamma' : d_1(V(\gamma^i), \Lambda^c) > 1 \text{ for all } \gamma^i \in \Gamma' \}.
$$ We show that both numerators and both denominators in \eqref{ps2} are respectively identical.
We shall proceed by induction. Given a contour $\gamma$ we define its \textit{level} as the maximum $n \in \N_0$ such that there exists a sequence of contours $\gamma_0, \dots, \gamma_n$ such that $\gamma = \gamma_0$ and $\text{supp}(\gamma_{i}) \subseteq \text{Int}(\gamma_{i-1})$ for all $i=1,\dots,n$. Furthermore, define the level of a family $\Gamma$ of contours as the maximum level of any contour $\gamma \in \Gamma$ and also define the level of $\Lambda \in \B^0_{\Z^d}$ as the maximum level of any contour $\gamma$ such that $V(\gamma) \subseteq \Lambda$. Thus, we shall proceed by induction on the level of both $\Gamma$ and $\Lambda$, respectively.

If $\Gamma$ has level zero then we have that
$$
\text{Ext}(\sigma)=\Gamma \Longleftrightarrow \Gamma_\sigma = \Gamma \hspace{2cm}\text{ and }\hspace{2cm} \text{Ext}(\Gamma')= \Gamma \Longleftrightarrow \Gamma' = \Gamma.
$$ which immediately implies that both numerators in \eqref{ps2} are identical in this case. We would like to point out the importance in the previous argument of the fact that $d_1(\text{V}(\gamma), \Lambda^c) > 1$ for all $\gamma \in \Gamma$. Indeed, notice that the sum in the numerator of the left hand side of \eqref{ps2} is over all $\sigma$ in the support of $\mu^{i}_\Lambda$ and that, for arbitrary $\Lambda$, it could very well happen that there are no configurations $\sigma$ in the support of $\mu^{i}_\Lambda$ such that $\Gamma_\sigma=\Gamma$. \mbox{The latter happens whenever} there exists $\gamma \in \Gamma$ such that \mbox{$\text{Int}_j(\gamma) \cap \Lambda^c \neq \emptyset$ for some $j\neq i$}, that is, when $\Lambda$ is not simply connected and $\Gamma$ has a contour whose interior labels come into conflict with the boundary configuration $i$. If this were to be the case, then both numerators would not coincide. However, due to the assumption that $d_1(\text{V}(\gamma), \Lambda^c) > 1$ for all $\gamma \in \Gamma$, we can rule out this possibility and thus conclude as we have.

\mbox{Similarly,} if $\Lambda$ has level zero we have that
$$
\sigma \in C^i(\Lambda) \Longrightarrow \Gamma_\sigma \in D^i(\Lambda) \hspace{0.5cm}\text{ and }\hspace{0.5cm}\Gamma' \in D^i(\Lambda) \Longrightarrow \exists\,\, \sigma \in C^i(\Lambda) \text{ such that }\Gamma'=\Gamma_\sigma
$$ since for $\sigma$ in the support of $\mu^{i}_\Lambda$ we have that all contours in $\Gamma_\sigma$ must be exterior $i$-contours. Thus we conclude that both denominators in \eqref{ps2} are equal as well.

Next, if $\Gamma$ has level $n+1$ then, by reordering the sum in the left hand side by summing independently over configurations in the interior of each $\gamma \in \Gamma$, we obtain that
$$
\sum_{ \sigma : \text{Ext}(\sigma) = \Gamma } e^{-\beta \sum_{\gamma \in \Gamma_\sigma} \Phi(\gamma)} =  \prod_{\gamma^i \in \Gamma}\left[e^{-\beta \Phi(\gamma^i) } \prod_{j=1}^q \left( \sum_{\sigma \in C^j(\text{Int}_j(\gamma^i))} e^{-\beta \sum_{\gamma \in \Gamma_\sigma} \Phi(\gamma)}\right)\right]
$$ where we have used the fact that all contours involved are compatible and, furthermore, that $\text{Int}_j(\gamma)$ has simply connected components for every contour $\gamma \in \Gamma$ and $j=1,\dots,q$. Now, since $\text{Int}_j(\gamma)$ is of level no greater than $n$ for every contour $\gamma \in \Gamma$ and $j=1,\dots,q$, by inductive hypothesis we conclude that
\begin{equation}\label{ps3}
\sum_{ \sigma : \text{Ext}(\sigma) = \Gamma } e^{-\beta \sum_{\gamma \in \Gamma_\sigma} \Phi(\gamma)} = \prod_{\gamma^i \in \Gamma}\left[e^{-\beta \Phi(\gamma^i) } \prod_{j=1}^q \left( \sum_{\Gamma^j \in D^j(\text{Int}_j(\gamma^i))} e^{-\beta \sum_{\gamma^j \in \Gamma^j} \Phi(\gamma^j)}\right)\right]
\end{equation}Furthermore, by the symmetry between spins in the $q$-Potts model, for each $j=1,\dots,q$ we have that
\begin{equation}\label{ps4}
\sum_{\Gamma^j \in D^j(\text{Int}_j(\gamma^i))} e^{-\beta \sum_{\gamma^j \in \Gamma^j} \Phi(\gamma^j)} = \sum_{\Gamma^i \in D^i(\text{Int}_j(\gamma^i))} e^{-\beta \sum_{\gamma^i \in \Gamma^i} \Phi(\gamma^i)}
\end{equation} so that the sums in the right hand of \eqref{ps3} become only over $i$-contours. Then, by reversing the summation order, we obtain the numerator in the right hand side of \eqref{ps2}.

Finally, if $\Lambda$ has level $n+1$ then we decompose the sum in the denominator of the left hand side of \eqref{ps2} over all compatible families of exterior $i$-contours and use the fact that all these families have level no greater than $n+1$, so that for each of them the numerators in \eqref{ps2} coincide by the argument given above.
\end{proof}

As a consequence of Proposition \ref{propps} we have that for every \mbox{simply connected $\Lambda \in \B^0_{\Z^d}$} the measure $\mu^{i}_{\Lambda}$ can be obtained by first sampling the external contours and then sampling the spin configuration in the interior of each contour \mbox{with the corresponding} finite-volume Boltzmann-Gibbs distributions. The precise statement is given in Proposition \ref{corps} below.

\begin{defi} Given $\Lambda \in \B^0_{\Z^d}$ we define its \textit{interior boundary} $\p \Lambda$ as
$$
\p \Lambda := \{ x \in \Lambda : d_1(x,\Lambda^c) = 1\}
$$ and its $r$-\textit{interior} $\Lambda^\circ$ as
$$
\Lambda^\circ := \{x \in \Lambda : d_1(x,\p \Lambda) > r \}.
$$
\end{defi}

\begin{prop}\label{corps} For simply connected $\Lambda \in \B^0_{\Z^d}$ and $\sigma \in \{1,\dots,q\}^{\Z^d}$ with $\mu^i_{\Lambda^\circ}(\sigma) > 0$ we have
\begin{equation}\label{corpsequation}
\mu^i_{\Lambda^\circ} (\sigma) = \mu_{\Lambda|\emptyset}( \{\Gamma' : \text{Ext}(\Gamma') = \text{Ext}(\Gamma_\sigma)\}) \prod_{\gamma^i \in \text{Ext}(\Gamma_\sigma)} \prod_{j=1}^q \mu^j_{\text{Int}^\circ_j(\gamma)}(\sigma^j_{\text{Int}^\circ_j(\gamma)})
\end{equation} where, for each $j=1,\dots,q$ and $\Delta \in \B^0_{\Z^d}$, the configuration $\sigma^j_{\Delta} \in \{1,\dots,q\}^{\Z^d}$ is defined by the formula
$$
\sigma^j_{\Delta}(x) := \left\{\begin{array}{ll} \sigma(x) & \text{ if $x \in \Delta$ }\\ \\ j & \text{ if $x \notin \Delta$.}\end{array}\right.
$$
\end{prop}

\begin{proof} Let us notice the following equivalences:
$$
\mu^i_{\Lambda^\circ}(\sigma) > 0 \Longleftrightarrow \sigma(x)=i \text{ for all $x \in \Z^d - \Lambda^\circ$} \Longleftrightarrow \text{$d(\text{supp}(\gamma^i),\Lambda^c) > 1$ for all $\gamma^i \in \Gamma_\sigma.$}
$$ Since $\Lambda$ is simply connected, this implies in fact that $d(\text{V}(\gamma^i),\Lambda^c) > 1$ for all $\gamma^i \in \Gamma_\sigma$. Thus, by the consistency of Boltzmann-Gibbs distributions and Proposition \ref{propps} we have
$$
\mu^i_{\Lambda^\circ} (\sigma) = \frac{\mu^i_{\Lambda}(\sigma)}{\mu^i_{\Lambda}(C^i(\Lambda))}= \mu_{\Lambda|\emptyset}(\{ \Gamma' : \text{Ext}(\Gamma') = \text{Ext}(\sigma) \})\mu^i_{\Lambda}(\sigma | \{ \sigma' : \text{Ext}(\sigma') = \text{Ext}(\sigma)\}).
$$ Now, notice that if a configuration $\sigma'$ in the support of $\mu^i_{\Lambda}$ is such that $\text{Ext}(\sigma') = \text{Ext}(\sigma)$ then the spin values of $\sigma'$ outside $\text{Int}^\circ(\text{Ext}(\sigma)):=\bigcup_{\gamma^i \in \text{Ext}(\sigma)} \text{Int}^\circ(\gamma^i)$ are fully determined. More precisely, we have the following identity of events
$$
\{ \sigma' : \text{Ext}(\sigma') = \text{Ext}(\sigma)\} = \{ \sigma' : \sigma'_{\Z^d - \text{Int}^\circ(\text{Ext}(\sigma))} = \sigma_{\Z^d - \text{Int}^\circ(\text{Ext}(\sigma))}\}.
$$ This, combined with the fact that spins in different interiors never interact (either because they are too far apart or they have the same value), gives \eqref{corpsequation}.
\end{proof}

As a consequence of Proposition \ref{corps}, we obtain the following \mbox{simulation scheme for $\mu^i_\Lambda$.}

\begin{cor}\label{corposta} Given a simply connected set $\Lambda \in \B^0_{\Z^d}$, let $Y$ be a random $i$-contour collection distributed according to $\mu_{\Lambda|\emptyset}$ and $X=\{ X^j_{\Delta} : \Delta \subseteq \B^0_{\Z^d}, j=1,\dots,q\}$ be a family of random spin configurations satisfying:
\begin{enumerate}
\item [$\bullet$] $X$ is independent of $Y$.
\item [$\bullet$] The random elements $X^j_\Delta$ are all independent and with distribution $\mu^j_\Delta$, respectively.
\end{enumerate} If we take the random spin configuration $Z$ with external contours matching those of $Y$ and with internal spin configuration given by the corresponding random \mbox{configurations in $X$,} i.e. the random spin configuration $Z$ defined by the formula
$$
Z(x)=\left\{\begin{array}{ll} i & \text{ if }x \in \bigcap_{\gamma^i \in \text{Ext}(Y)} \text{Ext}(\gamma^i)\\ \\ Y_\Lambda(x) & \text{ if }x \in \bigcup_{\gamma^i \in \text{Ext}(Y)} \text{supp}(\gamma^i) \\ \\ j & \text{ if }x \in \text{Int}_j(\gamma^i)-\text{Int}^\circ_j(\gamma^i) \text{ for }\gamma^i \in \text{Ext}(Y) \text{ and }j=1,\dots,q \\ \\ X^j_{\text{Int}^\circ_j(\gamma^i)}(x) & \text{ if }x \in \text{Int}^\circ_j(\gamma^i) \text{ for }\gamma^i \in \text{Ext}(Y),\end{array}\right.
$$ then $Z$ is distributed according to $\mu^i_{\Lambda^\circ}$. We call $Z$ the \mbox{$i$-alignment of $Y$ with respect to $X$.}
\end{cor}

The key advantage of this simulation scheme is that it can also be carried out in an infinite volume with the help of the FFG dynamics. Indeed, consider the coefficient
$$
\alpha_{q\text{-Potts}}(\beta) := \sup_{\gamma^i_x \in \Z^d \times G^i} \left[ \frac{1}{|\gamma^i_x|} \sum_{\tilde{\gamma}^i_y \not \sim \gamma^i_x} |\tilde{\gamma}^i_y| e^{-\beta \Phi(\tilde{\gamma}^i_y)} \right].
$$ where, for a given contour $\gamma$, we denote the cardinal of its support by  $|\gamma|$. Notice that, by the symmetry of the $q$-Potts model, the coefficient $\alpha$ does not depend on the choice of $i$.
Now, if $\alpha_{q\text{-Potts}}< 1$ then $(\nu^i, H^i)$ admits an infinite-volume Gibbs distribution $\mu$ obtained as the stationary measure of the corresponding FFG dynamics $\mathcal{K}^i$. Furthermore, it follows as in the Ising contours model that the free process at time $0$, $\Pi^i_0$, has only finitely many contours surrounding any point in $\Z^d$ and, thus, external contours in $\mathcal{K}_0^i$ are well defined. Hence, if we consider a family $X$ independent of $\mathcal{K}^i$ as in Corollary \ref{corposta}, then it is possible to conduct the $i$-alignment of $\mathcal{K}_0^i$ with respect to $X$. By repeating a similar analysis to the one carried out for the Ising contours model, one can verify that the distribution of this $i$-alignment is a Gibbs measure for the $q$-Potts model. Furthermore, by following the ideas discussed in the proof of Theorem \ref{teomixing}, it is possible to construct for any pair $f,g$ of bounded local functions a triple
$$
\left\{ \left(\mathcal{A}^0_F(r(\Lambda_f)), (X^j_\Delta)_{\Delta \not \subseteq \Lambda_f^c}\right),\left(\mathcal{A}^0_F(r(\Lambda_g)), (X^j_\Delta)_{\Delta \not \subseteq \Lambda_g^c}\right), \left(\tilde{\mathcal{A}}^0_F(r(\Lambda_g)), (\tilde{X}^j_\Delta)_{\Delta \not \subseteq \Lambda_g^c}\right) \right\}
$$ where for $\Lambda \in \B^0_{\Z^d}$ we define $r(\Lambda):=\{ \gamma_x^i \in \Z^d \times G^i : \text{V}(\gamma_x^i) \cap \Lambda \neq \emptyset\}$, such that
\begin{enumerate}
\item [$\bullet$] $\left(\mathcal{A}^0_F(r(\Lambda_f)), (X^j_\Delta)_{\Delta \not \subseteq \Lambda_f^c}\right)$ and $\left(\tilde{\mathcal{A}}^0_F(r(\Lambda_g)), (\tilde{X}^j_\Delta)_{\Delta \not \subseteq \Lambda_g^c}\right)$ are independent,
\item [$\bullet$] $\left(\mathcal{A}^0_F(r(\Lambda_g)), (X^j_\Delta)_{\Delta \not \subseteq \Lambda_g^c}\right)$ and $\left(\tilde{\mathcal{A}}^0_F(r(\Lambda_g)), (\tilde{X}^j_\Delta)_{\Delta \not \subseteq \Lambda_g^c}\right)$ have the same distribution,
\item [$\bullet$] $\mathcal{A}^0(r(\Lambda_f) \sim \mathcal{A}^0(r(\Lambda_g)) \Longrightarrow \left(\mathcal{A}^0_F(r(\Lambda_g)), (X^j_\Delta)_{\Delta \not \subseteq \Lambda_g^c}\right) =\left(\tilde{\mathcal{A}}^0_F(r(\Lambda_g)), (\tilde{X}^j_\Delta)_{\Delta \not \subseteq \Lambda_g^c}\right).$
\end{enumerate} By combining these elements as in the proof of Theorem \ref{teounigibbsising}, we get the following result.

\begin{teo}\label{teopotts}If $\beta > 0$ is sufficiently large so as to satisfy $\alpha_{q\text{-Potts}}(\beta) < 1$ then:
\begin{enumerate}
\item [i.] The $q$-Potts model on $\Z^d$ of interaction range $r$ admits $q$ distinct Gibbs measures, which we denote by $\mu^i$ for $i=1,\dots,q$.
\item [ii.] For each $i=1,\dots,q$ the measure $\mu^{i}$ can be obtained as the local limit
$$
\mu^{i} := \lim_{n \rightarrow +\infty} \mu^{i}_{\Lambda_n^\circ}
$$ for any sequence $(\Lambda_n)_{n \in \N} \subseteq \B^0_{\Z^d}$ of simply connected sets with $\Lambda_n \nearrow \Z^d$.
\item [iii.] For each $i=1,\dots,q$ the measure $\mu^i$ satisfies the $i$-sea with islands picture.
\item [iv.] If also $\beta > \beta^*$ where
$$
\beta^* := \inf \left \{ \beta > 0 \,:\, \sum_{\gamma^i_x : d_1(0, \text{supp}(\gamma^i_x)) \leq 1} |\gamma^i_x| e^{-\beta \Phi(\gamma^i_x)} < 1 \right\}
$$ then each $\mu^i$ is exponentially mixing in the sense of \eqref{eqmixingicm3} and \eqref{eqmixingicm4}.
\end{enumerate}
\end{teo}

\begin{obs} In principle it is not clear why for $\beta > 0$ sufficiently large one should have $\alpha_{q\text{-Potts}}(\beta) < 1$ or even $\alpha_{q\text{-Potts}}(\beta) < +\infty$ for that matter. Indeed, this will depend on the behavior of contour energies $\Phi(\gamma)$ in the limit as $|\gamma| \rightarrow +\infty$. Fortunately, it is not hard to see that all contours $\gamma$ in the $q$-Potts model satisfy the \textit{Peierls bound}
$$
\Phi(\gamma) \geq \frac{|\gamma|}{2}.
$$ This bound guarantees that for $\beta > 0$ sufficiently large one effectively has $\alpha_{q\text{-Potts}}(\beta) < 1$, so that the results of Theorem \ref{teopotts} have true meaning.
\end{obs}

One can also adapt the arguments featured in Chapter \ref{chapterconvabs} to this context and obtain the following continuity result.

\begin{teo}\label{teopotts2} For any $\beta_0 > 0$ such that $\alpha_{q\text{-Potts}}(\beta_0) < 1$ and any $i=1,\dots,q$ we have the local convergence
$$
\lim_{\beta \rightarrow \beta_0} \mu^{i,\beta} = \mu^{i,\beta_0}.
$$
\end{teo}

\begin{proof} Let us consider a sequence $(X^\beta)_{\beta > 0}$ of families as in Corollary \ref{corposta} (where we write the dependence in the inverse temperature explicitly) such that for each $j=1,\dots,q$ and $\Delta \subseteq \B^0_{\Z^d}$ we have the almost sure convergence
$$
\lim_{\beta \rightarrow \beta_0} X^{j,\beta}_\Delta = X^{j,\beta_0}_{\Delta}.
$$ Such a sequence can be constructed using the standard coupling from the past methods (see e.g. \cite{HA}). Notice then that, by Corollary \ref{corposta} and the discussion following it, \mbox{in order to} obtain the result it will suffice to couple the FFG dynamics $\mathcal{K}^{i,\beta}$ independently of $(X^\beta)_{\beta > 0}$ so that as $\beta \rightarrow \beta_0$ we have
$$
\mathcal{K}_0^{i,\beta} \overset{loc}{\longrightarrow} \mathcal{K}_0^{i,\beta_0}.
$$ This, however, can be done as in Chapter \ref{chapterconvabs}.
\end{proof}

Once again, we would like to point out that these results can also be obtained through standard Pirogov-Sinai theory. Nevertheless, the range of $\beta > 0$ for which these results hold under the standard theory is strictly smaller than the one obtained \mbox{with our approach:} standard methods give these results for $\beta > \beta'$ where
$$
\beta' = \inf\left\{ \beta > 0 : \sum_{\gamma^i_x : d(0, \text{supp}(\gamma^i_x)) \leq 1} e^{|\gamma^i_x|} e^{- \beta \Phi(\gamma^i_x)} < 1\right\}.
$$ Due to the fact that traditional Pirogov-Sinai theory relies on the convergence of certain \mbox{cluster expansions,} once again we obtain in the threshold value an exponential dependence in the size of contours which is only linear for our approach. Furthermore, another strength of our approach is that it provides us with a perfect simulation scheme for \mbox{the measures $\mu^i$.} We believe that such a scheme has not been developed before. It is essentially contained in Corollary \ref{corposta} and the discussion following it (see also the discussion on Section \ref{perfectsimulation}).

\section{Discrete Widom-Rowlinson model}

In this section we show how the ideas discussed for the $q$-Potts model on $\Z^d$ can be adapted to establish a phase transition in the discrete Widom-Rowlinson model on $\Z^d$ with fugacity $\lambda > 0$ and interaction range $r \in \N$. Recall that this model is traditionally defined on the configuration space $\{+,0,-\}^{\Z^d}$ through the Boltzmann-Gibbs distributions given for each $\Lambda \in \B^0_{\Z^d}$ and $\eta \in \{+,0,-\}^{\Z^d}$ by the formula
\begin{equation}\label{wrdbgdps}
\mu_{\Lambda}^\eta(\sigma) = \frac{\mathbbm{1}_{\{\sigma_{\Lambda^c} \equiv \eta_{\Lambda^c}\}}}{Z_{\Lambda}^\eta} e^{-\sum_{B : B \cap \Lambda \neq \emptyset} \Phi_B(\sigma)}
\end{equation} where for each $B \subseteq \Z^d$ the interaction $\Phi_B$ is given by
\begin{equation}\label{wrdht2}
\Phi_B(\sigma) = \left\{ \begin{array}{ll} (+\infty) \mathbbm{1}_{\{ \sigma(x) \times \sigma(y)= - \}} & \text{ if $B=\{x,y\}$ with $\|x-y\|_\infty \leq r$} \\ \\ - \mathbbm{1}_{\{\sigma(x)= \pm\}} \log \lambda & \text{ if $B=\{x\}$} \\ \\ 0 & \text{ otherwise.}\end{array}\right.
\end{equation}
We wish to show that for sufficiently large values of $\lambda > 0$ the model admits at least two different Gibbs measures. The idea is to follow the procedure for the $q$-Potts model.
Once again we fix a set of reference configurations $\mathcal{R}:=\{ \eta^+,\eta^0,\eta^- \}$, where for each $i \in \{+,0,-\}$ we let $\eta^i$ denote the configuration on $\{+,0,-\}^{\Z^d}$ with constant \mbox{spin value $i$.} Also, for each $\Lambda \in \B^0_{\Z^d}$ we let $\mu^i_\Lambda$ denote the Boltzmann-Gibbs distribution on $\Lambda$ with boundary configuration $\eta^i$. The notion of contour for this model is defined in exactly the same way as for the $q$-Potts model, except for the fact that in the definition of correct points here we use the supremum distance $d_\infty$ instead of $d_1$. However, since in this model \mbox{it is} no longer true that $\Phi_B(\eta^i)=0$ for every $\eta^i \in \mathcal{R}$, the expression for the energy of \mbox{a contour} becomes slightly different: for an $i$-contour $\gamma^i$ belonging to some \mbox{configuration $\sigma$} we define its energy $\Phi(\gamma^i)$ by the formula
$$
\Phi(\gamma^i):= \sum_{B \subseteq \Z^d} \frac{|B \cap \text{supp}(\gamma^i)|}{|B|} \Phi_B(\sigma_B) - e_i|\gamma^i|
$$ where $e_i$ is the \textit{mean energy} of the configuration $\eta^i$ defined as
$$
e_i := \sum_{B \subseteq \Z^d : 0 \in B} \frac{1}{|B|} \Phi_B(\eta^i_B)= \left\{ \begin{array}{ll} - \log \lambda & \text{ if $i=\pm$} \\ \\ 0 & \text{ if $i=0$.}\end{array}\right.
$$ By proceeding as in the proof of Proposition \ref{proprepresentacionps} one can show the following.

\begin{prop}For any $\Lambda \in \B^0_{\Z^d}$ let us consider a configuration $\sigma$ such that $D_\sigma \subseteq \Lambda$ and $\sigma(x)=i$ for every $x \in \Z^d$ lying in the unique infinite connected component of $\Z^d - D_\sigma$. If $\Gamma_{\sigma}$ denotes the family of contours associated to $\sigma$, then we have
\begin{equation}\label{wrps1}
\mu^{i}_\Lambda (\sigma) = \frac{1}{Z^{i}_\Lambda} e^{- \left(\sum_{\gamma \in \Gamma_\sigma} \Phi(\gamma) + \sum_{u \in \{+,0,-\}} e_u |\Lambda_u|\right)}
\end{equation} where for each $u \in \{+,0,-\}$ we let $\Lambda_u$ denote the set of all points in $\Lambda$ which are either $u$-correct or belong to an $u$-contour of $\sigma$.
\end{prop}

The next step is to fix a spin $i \in \{+,0,-\}$ and define a contour model on $\mathcal{N}(\Z^d \times G^i)$, where $G^i$ is the space of all $i$-contour shapes, satisfying the bounded energy loss condition and preserving the distribution of exterior contours. With this purpose in mind, we first define for $\Delta \in \B^0_{\Z^d}$ and $j \in \{+,0,-\}$ the \textit{diluted partition function}
$$
Z^j(\Delta) := \sum_{\sigma \in C^j(\Delta)} e^{-\sum_{B : B \cap \Lambda} \Phi_B(\sigma)}
$$ where $C^j(\Delta):=\{ \sigma \in \text{supp}(\mu^j_\Delta) : d_1(\text{V}(\gamma), \Lambda^c) > 1 \text{ for all }\gamma \in \Gamma_\sigma\}$, and then consider the contour model with intensity measure $\nu^i$ given by the formula
\begin{equation}\label{wrimpottsc}
\nu^i(\gamma^i_x):= \left[\prod_{j \neq i} \frac{Z^j(\text{Int}_j(\gamma^i_x))}{Z^i(\text{Int}_j(\gamma^i_x))}\right]e^{-\Phi(\gamma^i_x)}
\end{equation} and Hamiltonian $H^i$ defined as
\begin{equation}\label{hpottsc2}
H^i_{\Lambda|\Gamma'} (\Gamma) = \left\{ \begin{array}{ll} +\infty & \text{ if either $\Gamma$ is incompatible, $\Gamma \not \sim \Gamma'_{\Lambda^c \times G}$ or $\Gamma \not \subset \Lambda$}\\ \\ 0 & \text{ otherwise.}\end{array}\right.
\end{equation} Notice that the intensity measure carries an additional product of partition functions which was missing in the $q$-Potts model. This is again attributed to the fact that not all  reference configurations in this model have zero mean energy. By proceeding as in the previous section we obtain for this context the analogues of Propositions \ref{propps}, \ref{corps} and Corollary \ref{corposta}. Thus, if for $i \in \{+,0,-\}$ we define the coefficient
\begin{equation}\label{dcwrps007}
\alpha_i (\lambda) = \sup_{\gamma^i_x \in \Z^d \times G^i} \left[ \frac{1}{|\gamma^i_x|} \sum_{\tilde{\gamma}^i_y \not \sim \gamma^i_x} |\tilde{\gamma}^i_y| \left[\prod_{j \neq i} \frac{Z^j(\text{Int}_j(\gamma^i_x))}{Z^i(\text{Int}_j(\gamma^i_x))}\right] e^{-\Phi_\lambda(\tilde{\gamma}^i_y)} \right]
\end{equation} then the condition $\alpha_i(\lambda) < 1$ implies the existence of a Gibbs measure $\mu^i$ associated \mbox{to $\eta^i$} which fulfills the description in Theorems \ref{teopotts} and \ref{teopotts2}. In particular, since $\alpha_+ = \alpha_-$ \mbox{by symmetry}, the condition $\alpha_+(\lambda) < 1$ already implies a phase transition for the model.
However, it is not clear why this condition should be fulfilled by any value of $\lambda$. Indeed, this will not only depend on the behavior of the energies $\Phi(\gamma^+_x)$ as $|\gamma^+_x| \rightarrow +\infty$ on \eqref{dcwrps007} but also on the growth of the additional factor
\begin{equation}\label{pwrps007}
\prod_{j \neq +} \frac{Z^j(\text{Int}_j(\gamma^+_x))}{Z^+(\text{Int}_j(\gamma^+_x))} = \frac{Z^0(\text{Int}_0(\gamma^+_x))}{Z^+(\text{Int}_0(\gamma^+_x))}.
\end{equation}
Fortunately, it can be seen that for $\lambda > 0$ sufficiently large the condition $\alpha_+(\lambda) < 1$ is satisfied. Indeed, first let us notice that a straightforward combinatorial argument yields, for any $+$-contour $\gamma^+_x$, the Peierls bound
\begin{equation}\label{wrpeierls}
\Phi(\gamma^+_x) = (+\infty)\mathbbm{1}_{\{\Phi(\gamma^+_x) = +\infty\}} + \#\{ y \in \text{supp}(\gamma^+_x) : \sigma(y) = 0\} \log \lambda \mathbbm{1}_{\{\Phi(\gamma^+_x)< +\infty\}} \geq \frac{|\gamma_x^+|}{(2r)^d}\log \lambda.
\end{equation} On the other hand, the following lemma shows that the additional factor in \eqref{pwrps007} remains bounded for large values of $\lambda$.

\begin{lema}\label{lemacontrol} There exist constants $c_1,c_2 > 0$, which depend only on $d$ and $r$, \mbox{such that} for any $\Delta \in \B^0_{\Z^d}$
\begin{equation}\label{lemacontroleq}
\frac{Z^0(\Delta)}{Z^+(\Delta)} \leq \left( \frac{2^{c_1}}{\lambda^{c_2}}\right)^{\# \p_e \Delta},
\end{equation} where $\p_e \Delta:= \{ y \in \Delta^c : d_1(y,\Delta) =1 \}$.
\end{lema}

\begin{proof} For $\Delta \in \B^0_{\Z^d}$ we define its boundary and $r$-interior by the respective formulas
$$
\p \Delta = \{ x \in \Delta : d_1(x, \Delta^c) = 1\} \hspace{1cm}\text{ and }\hspace{1cm}\Delta^\circ = \{ x \in \Delta : d_\infty(x, \p \Delta) > r \}.
$$ Now, notice that we have the identities
$$
Z^0(\Delta) = Z^0_{\Delta^\circ} \hspace{2cm}\text{ and }\hspace{2cm}Z^+(\Delta)= \lambda^{|\Delta| - |\Delta^\circ|} Z^+_{\Delta^\circ}
$$ where $Z^0_{\Delta^\circ}$ and $Z^+_{\Delta^\circ}$ are the normalizing constants given in \eqref{wrdbgdps}. Furthermore, we have
$$
Z^0_{\Delta^\circ} = Z^+_{\Delta^\circ} + \sum_{\sigma \in A^0_+(\Delta^\circ)} e^{- \sum_{B: B \cap \Delta^\circ \neq \emptyset} \Phi_B(\sigma_B)}
$$
where $A^0_+(\Delta^\circ)$ is the set of all configurations $\sigma$ in the support of $\mu^0_{\Delta^\circ}$ such that $\sigma_{\Delta^\circ} \cdot \eta^+_{\Z^d - \Delta^\circ}$ does not belong to the support of $\mu^+_{\Delta^\circ}$. Observe that a configuration $\sigma$ in the support of $\mu^0_{\Delta^\circ}$ belongs to $A^0_+(\Delta^\circ)$ if and only if $\sigma(x)=-$ for some $x \in \Delta^\circ$ with $d_\infty(x,\p_e (\Delta^\circ)) \leq r$.

Now, for a fixed configuration $\sigma$ in the support of $\mu^0_{\Delta^\circ}$, we say that two given sites $x,y \in \text{supp}(\sigma):=\{ z \in \Delta^\circ : \sigma(z) \neq 0 \}$ are $\sigma$-\textit{connected} if $\|x-y\|_\infty \leq r$. Notice that if \mbox{$x$ and $y$} are $\sigma$-connected then $\sigma$ must assign both sites $x$ and $y$ the same spin value. \mbox{The maximal connected} components of $\text{supp}(\sigma)$ with respect to this notion of connection will be called $\sigma$-\textit{components}. Next, assign to $\sigma$ the configuration $\sigma^+$ defined by the formula
$$
\sigma^+(x):=\left\{ \begin{array}{ll} + & \text{ if either $x \in \Z^d - \Delta^\circ$ or $x$ belongs a $\sigma$-component $C \in C_\sigma(\p_e (\Delta^\circ))$ } \\ \\ \sigma(x) & \text{ otherwise,}\end{array}\right.
$$ where $C_\sigma(\p_e (\Delta^\circ))$ is the set $\sigma$-components $C$ of $\text{supp}(\sigma)$ satisfying $d_\infty(C,\p_e(\Delta^\circ))\leq r$. \mbox{In other terms,} $\sigma^+$ is obtained from $\sigma$ by flipping to $+$ the spin of all $\sigma$-components which interact with the boundary $\p_e(\Delta^\circ)$ and also the spin of sites outside $\Delta^\circ$. Notice that $\sigma^+$ belongs to the support of $\mu^+_{\Delta^\circ}$ and that both $\sigma$ and $\sigma^+$ have the same energy in $\Delta^\circ$, i.e.
$$
\sum_{B: B \cap \Delta^\circ \neq \emptyset} \Phi_B(\sigma_B) = \sum_{B: B \cap \Delta^\circ \neq \emptyset} \Phi_B(\sigma^+_B).
$$ Hence, since there are at most $2^{\# C_{\sigma^+}(\p_e (\Delta^\circ))}$ different configurations $\sigma'$ in \mbox{the support of $\mu^0_{\Delta^\circ}$} which can be assigned the same configuration $\sigma^+$ and there exists $\tilde{c}_1 > 0$ such that for any configuration $\sigma$ there can be at most $\tilde{c}_1 (\# \p_e(\Delta^\circ))$ $\sigma$-components in $C_\sigma(\p_e (\Delta^\circ))$, \mbox{we obtain the bound}
$$
\sum_{\sigma \in A^0_+(\Delta^\circ)} e^{- \sum_{B: B \cap \Delta^\circ \neq \emptyset} \Phi_B(\sigma_B)} \leq 2^{\tilde{c}_1 (\# \p_e (\Delta^\circ))} Z^+_{\Delta^\circ}
$$ which implies that
$$
\frac{Z^0(\Delta)}{Z^+(\Delta)} \leq \frac{1+2^{\tilde{c}_1 (\# \p_e (\Delta^\circ))}}{ \lambda^{|\Delta| - |\Delta^\circ|}}.
$$ From here a straightforward calculation allows us to conclude the result.
\end{proof}

It follows from the Peierls bound on the energy of contours and the previous lemma that for $\lambda > 0$ sufficiently large one has $\alpha_+(\lambda) < 1$. Thus, we obtain the following result.

\begin{teo}\label{teowrpt}If $\lambda > 0$ is sufficiently large so as to satisfy $\alpha_{+}(\lambda) < 1$ then:
\begin{enumerate}
\item [i.] The discrete Widom-Rowlinson model on $\Z^d$ with exclusion radius $r$ admits two distinct Gibbs measures, $\mu^+$ and $\mu^-$.
\item [ii.] The measures $\mu^{+}$ and $\mu^-$ can be obtained as the local limits
$$
\mu^{+}= \lim_{n \rightarrow +\infty} \mu^{+}_{\Lambda_n^\circ}\hspace{1cm}\text{ and }\hspace{1cm}\mu^{-}= \lim_{n \rightarrow +\infty} \mu^{-}_{\Lambda_n^\circ}
$$ for any sequence $(\Lambda_n)_{n \in \N} \subseteq \B^0_{\Z^d}$ of simply connected sets with $\Lambda_n \nearrow \Z^d$.
\item [iii.] The measures $\mu^+$ and $\mu^-$ satisfy the sea with islands picture for the $+$ and $-$ spins, respectively.
\item [iv.] The measures $\mu^+$ and $\mu^-$ are continuous in $\lambda$, i.e. if $\lambda_0$ is such that $\alpha_{+}(\lambda_0) < 1$ then we have the local convergence
$$
\lim_{\lambda \rightarrow \lambda_0} \mu^{\pm,\lambda} = \mu^{\pm,\lambda_0}.
$$
 \item [v.] If also $\lambda > \lambda^*$, where
$$
\lambda^* := \inf \left \{ \lambda > 0 \,:\, \sum_{\gamma^+_x : d_1(0, \text{supp}(\gamma^+_x)) \leq 1} |\gamma^+_x| \left[\frac{Z^0(\text{Int}_0(\gamma^+_x))}{Z^+(\text{Int}_0(\gamma^+_x))}\right]e^{-\Phi_\lambda(\gamma^+_x)} < 1 \right\},
$$ then both $\mu^+$ and $\mu^-$ are exponentially mixing in the sense of \eqref{eqmixingicm3} and \eqref{eqmixingicm4}.
\end{enumerate}
\end{teo}

Even though we have established the occurrence of phase transition for large values of the fugacity $\lambda$, one could still wonder what can be said about the remaining $0$-spin. Similar combinatorial arguments to the ones given above yield, for any $0$-contour $\gamma^0_x$, the inverse Peierls bound
$$
\Phi(\gamma_x^0) \leq - \frac{|\gamma_x^0|}{(2r)^d} \log \lambda
$$ and the lower bound
$$
\frac{Z^\pm(\Delta)}{Z^0(\Delta)} \geq \left( \frac{\lambda^{c_2}}{2^{c_1}}\right)^{\# \p_e \Delta}
$$ for any $\Delta \in \B^0_{\Z^d}$, which implies that for every $\lambda > 0$ sufficiently large one actually has $\alpha_0(\lambda) = +\infty$ and thus our entire argument breaks down for the $0$-spin.
This suggests that the reference configuration $\eta^0$ is \textit{unstable} for large fugacities, in the sense that there is no Gibbs measure of the model satisfying the $0$-sea with islands picture.
The other reference configurations $\eta^+$ and $\eta^-$ are thus regarded as \textit{stable}.

As a final remark, we would like to point out that the occurrence of a phase transition in the discrete Widom-Rowlinson model can be established by other methods beside cluster expansion, such as random cluster representations (see e.g. \cite{GHM}). However, the use of such representations usually leads to a less complete picture than that of Theorem \ref{teowrpt} for very large fugacities: these cannot be used to show neither the $\pm$-sea with islands picture nor the continuity in $\lambda$, and also these representations alone are not enough to establish exponential mixing properties. Furthermore, random cluster representations in general are not robust, in the sense that they may work very well for certain models but then fail to work whenever these models are subject to slight perturbations. As an example we have the $k$-tolerant Widom-Rowlinson model, for which one expects a similar behavior to the original Widom-Rowlinson model for small values of $k$, but whose random cluster representations become
extremely more difficult to handle: each cluster of particles may not have a unique $+$ or $-$ spin value inside them, and the number of spin changes allowed inside each cluster will vary depending on its own geometry. However, these considerations do not interfere with the general argument given in this section, so that the Peierls bound in \eqref{wrpeierls} and the upper bound in Lemma \ref{lemacontrol} still remain valid for the $k$-tolerant Widom-Rowlinson model provided that $k$ is small enough to guarantee that for any incorrect point of an admissible configuration there exists at least one empty site at a distance no greater than $r$ from it. We can therefore conclude the following result without any additional difficulties.

\begin{teo}\label{tolerantpt} For $\lambda > 0$ sufficiently large and $k,r \in \N$ such that $k < (r+1)^d - 1$, \mbox{the $k$-tolerant} Widom-Rowlinson model with fugacity $\lambda$ and exclusion radius $r$ admits two distinct Gibbs measures, $\mu^+$ and $\mu^-$, each satisfying the description of Theorem \ref{teowrpt}.
\end{teo}

\section{Continuum Widom-Rowlinson model}

The ideas discussed in the previous section for the discrete Widom-Rowlinson model can be adapted to the continuum setting to obtain analogous results. In this section we comment briefly on how to perform such adaptation.

Given $\lambda, r > 0$, recall that the Widom-Rowlinson model on $\R^d$ with fugacity $\lambda$ and exclusion radius $r > 0$ is defined as the diluted model on $\mathcal{N}(\R^d \times \{+,-\})$ given by $(\nu,H)$, where
$$
\nu = \lambda \mathcal{L}^d \times ( \delta_+ + \delta_- )
$$ and
$$
H_{\Lambda|\eta}(\sigma)= \sum_{(\gamma_x ,\tilde{\gamma}_y) \in e_{\Lambda}(\sigma|\eta)} U( \gamma_x , \tilde{\gamma}_y )
$$ with
\begin{equation}\label{wrups}
U(\gamma_x,\tilde{\gamma}_y) := \left\{ \begin{array}{ll} +\infty &\text{if }\gamma \neq \tilde{\gamma}\text{ and }\|x-y\|_\infty \leq r\\ 0 &\text{otherwise.}\end{array}\right.
\end{equation}Now, consider the tiling $\mathcal{T}$ of $\R^d$ given by the cubic cells $(\mathcal{Q}_x)_{x \in \Z^d}$, where \mbox{for each $x \in \Z^d$}
$$
\mathcal{Q}_x := l \cdot (x + [0,1)^d ),
$$ where $l > 0$ is sufficiently small so that particles in any two adjacent cells are \mbox{forbidden} by the Hamiltonian $H$ to be of opposite type. Here we understand adjacent cells \mbox{in the sense} of the graph $\Z^d$ with the supremum distance, i.e. we say that two cells $\mathcal{Q}_x, \mathcal{Q}_y$ are adjacent if $d_\infty(\mathcal{Q}_x,\mathcal{Q}_y)=0$ where $d_\infty$ is the supremum distance on $\R^d$.

Given a configuration $\sigma \in \mathcal{N}(\R^d \times G)$ we assign it the configuration $s_\sigma \in \{+,0,-,*\}^{\Z^d}$ defined by the formula
$$
s_\sigma (x) = \left\{\begin{array}{ll}+ & \text{ if there are only particles of type $(+)$ in $\mathcal{Q}_x$}\\ \\
- & \text{ if there are only particles of type $(-)$ in $\mathcal{Q}_x$}\\ \\
0 & \text{ if there are no particles in $\mathcal{Q}_x$}\\ \\
* & \text{ if there are particles of both types in $\mathcal{Q}_x$.}\end{array}\right.
$$ We call $s_\sigma$ the \textit{skeleton} of $\sigma$. Notice that for configurations $\sigma$ which are not forbidden \mbox{by $H$} we have in fact that $s_\sigma(x) \neq *$ for all $x \in \Z^d$. Next, fix a set of reference configurations
$$
\mathcal{R}=\{\eta^+, \eta^0,\eta^-\} \subseteq \mathcal{N}(\R^d \times \{+,-\})
$$ such that for each $i \in \{+,0,-\}$ the configuration $\eta^i$ satisfies $s_{\eta^i}(x)=i$ for all $x \in \Z^d$.
Then, for $i \in \{+,0,-\}$ say that a site $x \in \Z^d$ is $i$-correct with respect to the configuration $s_\sigma \in \{+,0,-,i\}^{\Z^d}$ if $s_\sigma(y)=i$ for all $y \in \Z^d$ with $\|x-y\|_\infty \leq \frac{r+1}{l}$ and label a site as incorrect if it is not $i$-correct for any $i \in \{+,0,-\}$. Finally, we proceed to define the contours of the configuration $s_\sigma$ by analogy with the previous sections. Observe that contours are not defined for particle configurations on $\mathcal{N}(\R^d \times \{+,-\})$ but rather for their skeletons. Now, given an $i$-contour $\gamma^i$ belonging to some skeleton configuration, \mbox{we define its} energy $\Phi(\gamma^i)$ \mbox{by the formula}
$$
\Phi(\gamma^i) := \left(- \log \int_{\mathcal{N}(\text{supp}(\gamma^i) \times \{+,-\})} \mathbbm{1}_{\{s_\sigma = \gamma^i\}} e^{-H_{\text{supp}(\gamma^i)|\emptyset}(\sigma)} d\pi^\nu_{\text{supp}(\gamma^i)}(\sigma)\right) - e_i |\text{supp}(\gamma^i)|
$$ where
$$
e_i = - \log \int_{\mathcal{N}(\mathcal{Q}_0 \times \{+,-\})} \mathbbm{1}_{\{s_\sigma = i \}} d\pi^\nu_{\text{supp}(\gamma^i)}(\sigma) = \left\{\begin{array}{ll} - \log (1-e^{-\lambda l^d}) + \lambda l^d & \text{ if $i= \pm$} \\ \\ 2\lambda l^d & \text{ if $i=0$.}\end{array}\right.
$$ We can write Boltzmann-Gibbs distributions in terms of the energy of contours as follows.

\begin{prop} For any $\Lambda \in \B^0_{\Z^d}$ let us consider a skeleton configuration $s'$ such that $D_{s'} \subseteq \Lambda$ and $s'(x)=i$ for every $x \in \Z^d$ lying in the infinite component of $\Z^d - D_{s'}$. If $\Gamma_{s'}$ denotes the family of contours associated to $s'$, then we have
\begin{equation}\label{wrcps1}
\mu^{\eta^i}_{\Lambda_l} (\{ \sigma \in \mathcal{N}(\R^d \times \{+,-\}) : s_\sigma = s' \}) = \frac{1}{Z_{\Lambda|\eta^i}} e^{- \left(\sum_{\gamma \in \Gamma_{s'}} \Phi(\gamma) + \sum_{u \in \{+,0,-\}} e_u |\Lambda(u)|\right)}
\end{equation} where
$$
\Lambda_l :=\bigcup_{x \in \Lambda} \mathcal{Q}_x
$$ and for each $u \in \{+,0,-\}$ we let $\Lambda(u)$ denote the set of all points in $\Lambda$ which are either $u$-correct or belong to an $u$-contour of $s^*$.
\end{prop}

\begin{proof} This follows from the consistent Hamiltonian property on Assumptions \ref{assump} by noticing that particles lying inside correct cells do not interact with particles \mbox{in other cells} and also that particles inside incorrect cells do not interact with particles in correct cells. We omit the details.
\end{proof}

We then fix a spin $i \in \{+,0,-\}$ and consider the contour model on $\mathcal{N}(\Z^d \times G^i)$, where $G^i$ is the space of $i$-contour shapes, with Hamiltonian as in \eqref{hpottsc2} and intensity measure
$$
\nu^i(\gamma^i_x):= \left[\prod_{j \neq i} \frac{Z^j(\text{Int}_j(\gamma^i_x))}{Z^i(\text{Int}_j(\gamma^i_x))}\right]e^{-\Phi(\gamma^i_x)}
$$ where for each $\Delta \in \B^0_{\Z^d}$ and $j \in \{+,0,-\}$ we set
$$
Z^j(\Delta) := \sum_{s_\sigma \in C^j(\Delta)} e^{- \left(\sum_{\gamma \in \Gamma_{s_\sigma}} \Phi(\gamma) + \sum_{u \in \{+,0,-\}} e_u |\Delta (u)|\right)}
$$ for $C^j(\Delta):=\{ s_\sigma : s_\sigma(x) = j \text{ for all $x \in \Z^d - \Delta$ and } d_1(\text{V}(\gamma), \Delta^c) > 1 \text{ for all }\gamma \in \Gamma_{s_\sigma}\}$. It is easy to verify that this contour model satisfies the bounded energy loss condition and that it also preserves the distribution of exterior contours in skeleton configurations.
Thus, in order to repeat the analysis of the previous section, it only remains to show an analogue of Corollary \ref{corposta}, i.e. we need to show how to conduct an $(i)$-alignment of contour configurations such that, whenever these are distributed according to $(\nu^i,H^i)$, the resulting particle configurations carry the correct Boltzmann-Gibbs distributions.
\mbox{Given a simply connected} set $\Lambda \in \B^0_{\Z^d}$ and $i \in \{+,-\}$, we proceed as follows:

\begin{enumerate}
\item [i.] Consider a contour configuration $Y$ with distribution $\mu_{\Lambda|\emptyset}$ specified by $(\nu^i,H^i)$.
\item [ii.] Place an $i$-particle independently on each cell belonging to the set $\Lambda^\circ \cap \bigcap_{\gamma^i \in \text{Ext}(Y)}$, so that each of these cells carries at least one $(i)$-particle. Here we let $\Lambda^\circ$ denote the $\frac{r+1}{l}$-interior of $\Lambda$, as in the previous section. Then, on top of these particles place a Poisson process of $(i)$-particles on $(\Lambda^\circ \cap \bigcap_{\gamma^i \in \text{Ext}(Y)})_l$ with fugacity $\lambda$, which will account for the remaining $(i)$-particles in each of these cells.
\item [iii.] Proceed as in (ii) for cells in $\text{Int}_j(\gamma^i)-\text{Int}^\circ_j(\gamma^i) \text{ for }\gamma^i \in \text{Ext}(Y) \text{ and } j \in \{+,0,-\}$.
\item [iv.] For every contour $\gamma^i \in \text{Ext}(Y)$ place a particle independently on each cell belonging to $\text{supp}(\gamma^i)$ of the type $\gamma^i$ indicates by the acceptance-rejection method, so that the resulting configuration of particles inside the support of $\gamma^i$ is not forbidden by $H$. \mbox{If the contour} $\gamma^i$ assigns a $0$-spin to some cell then place no particle in that cell. Then, place the remaining particles of the cells belonging to $\text{supp}(\gamma^i)$ according to the FFG dynamics on the finite volume $\left(\text{supp}(\gamma^i)\right)_l$ with empty boundary condition, but keeping the particles of the acceptance-rejection method alive for all times $t \in \R$.
\item [v.] For each contour $\gamma^i \in \text{Ext}(Y)$ place the particles inside the cells belonging to $\text{Int}(\gamma^i)$ with the FFG dynamics on the finite volume $\left(\text{Int}(\gamma^i)\right)_l$ with boundary condition given by the portion of particle configuration constructed in (iii).
\item [vi.] Finally, place particles in the remaining cells according to $\eta^i$.
\end{enumerate} It can be seen that this configuration carries the distribution $\mu_{(\Lambda^\circ)_l|\eta^i}$ specified by $(\nu,H)$. We may then proceed as in the previous section to show a phase transition for large $\lambda$. The Peierls bound for the energy of contours and Lemma \ref{lemacontrol} still hold in this setting; the latter being a consequence of the argument given in Lemma \ref{lemacontrol} combined with \eqref{poisson} and the Fubini theorem for the \mbox{product measure $\nu$.} Thus, we obtain the following result.

\begin{teo}\label{wrptps} Given $\lambda, r > 0$ let us suppose that $\lambda > 0$ is sufficiently large so that
$$
\alpha_+ (\lambda,r) = \sup_{\gamma^+_x \in \Z^d \times G^+} \left[ \frac{1}{|\gamma^+_x|} \sum_{\tilde{\gamma}^+_y \not \sim \gamma^+_x} |\tilde{\gamma}^+_y| \left[\frac{Z^0(\text{Int}_0(\gamma^+_x))}{Z^+(\text{Int}_0(\gamma^+_x))}\right] e^{-\Phi_{\lambda,r}(\tilde{\gamma}^+_y)} \right] < 1.
$$Then:
\begin{enumerate}
\item [i.] The continuum Widom-Rowlinson model on $\R^d$ with exclusion radius $r$ admits two distinct Gibbs measures, $\mu^+$ and $\mu^-$.
\item [ii.] The measures $\mu^{+}$ and $\mu^-$ can be obtained as the local limits
$$
\mu^{+}= \lim_{n \rightarrow +\infty} \mu_{(\Lambda_n^\circ)_l| \eta^+}\hspace{1cm}\text{ and }\hspace{1cm}\mu^{-}= \lim_{n \rightarrow +\infty} \mu_{(\Lambda_n^\circ)_l|\eta^-}
$$ for any sequence $(\Lambda_n)_{n \in \N} \subseteq \B^0_{\Z^d}$ of simply connected sets with $\Lambda_n \nearrow \Z^d$.
\item [iii.] The measures $\mu^+$ and $\mu^-$ satisfy the sea with islands picture for the $+$ and $-$ spins, respectively.
\item [iv.] The measures $\mu^+$ and $\mu^-$ are continuous in $\lambda$ and $r$, i.e. if $\lambda_0,r_0 > 0$ are such that $\alpha_{+}(\lambda_0,r_0) < 1$ then we have the local convergence
$$
\lim_{(\lambda,r) \rightarrow (\lambda_0,r_0)} \mu^{\pm,\lambda,r} = \mu^{\pm,\lambda_0,r_0}.
$$
 \item [v.] If also $\lambda > \lambda^*$, where
$$
\lambda^* := \inf \left \{ \lambda > 0 \,:\, \sum_{\gamma^+_x : d_1(0, \text{supp}(\gamma^i_x)) \leq 1} |\gamma^+_x| \left[\frac{Z^0(\text{Int}_0(\gamma^+_x))}{Z^+(\text{Int}_0(\gamma^+_x))}\right]e^{-\Phi_{\lambda,r}(\gamma^+_x)} < 1 \right\},
$$ then both $\mu^+$ and $\mu^-$ are exponentially mixing in the sense of \eqref{eqmixingicm3} and \eqref{eqmixingicm4}.
\end{enumerate}
\end{teo}

Notice that Theorem \ref{wrptps} gives also continuity with respect to the exclusion radius. This is due to the fact that given $r_0 > 0$ one can choose the size $l$ of the tiling so that for any $r > 0$ sufficiently close to $r_0$ the definition of correct point in the skeleton configuration remains unaltered. From this observation, a straightforward argument analogous to the one in Chapter \ref{chapterconvabs} gives the result. Moreover, since the Peierls bound and the estimation in Lemma \ref{lemacontrol} still hold in this continuum setting, one can also obtain an analogue of Theorem \ref{tolerantpt} showing a phase transition in the $k$-tolerant Widom-Rowlinson model for sufficiently small values of $k$. Unfortunately, for the Widom-Rowlinson model with generalized interactions one can see that a more careful argument is needed than the one given here, we shall not pursue it here.
Finally, we observe that the exact same analysis given here carries over to show a phase transition in the thin rods model for evenly spaced angles. The condition of evenly spaced is important, since it gives the symmetry between orientations which was used for the Widom-Rowlinson model. The result is the following.

\begin{teo} Given $n \geq 2$ consider the thin rods model of fugacity $\lambda > 0$, rod length $2r > 0$ and orientation measure $\rho$ given by
$$
\rho = \frac{1}{n} \sum_{i=0}^{n-1} \delta_{\frac{i}{n} \pi}.
$$ Then, for $\lambda > 0$ sufficiently large (depending on $n$ and $r$) the model admits $n$ distinct Gibbs measures, $(\mu^i)_{i=1,\dots,n}$, each of which satisfies the description on Theorem \ref{wrptps}.
\end{teo}

As a final remark we point out that this same analysis can be carried out in the same manner for more general models. However, for this analysis to be fruitful, one needs to guarantee a Peierls bound on the energy contours
and a suitable control on the ratios of diluted partition functions like the one given in Lemma \ref{lemacontrol}. While in general the Peierls bound will not bring much problems, an estimate as in \eqref{lemacontroleq} is a priori not easy to obtain. Cluster expansion methods offer further tools to control these terms \mbox{(see \cite{Z1}),} although at the cost of narrowing the range of validity of the results. Nevertheless, even if we cannot control the ratios of diluted partition functions directly and we must rely on cluster expansion methods to obtain any results, the approach discussed here still guarantees that in that smaller range one can still perform perfect simulation of the corresponding equilibrium measures.

\newpage
\section{Resumen del Capítulo 12}

En este último capítulo combinamos las ideas en los capítulos anteriores con el marco de la teoría de Pirogov-Sinai para obtener algunos resultados bajo el régimen de no unicidad.
Concretamente, la teoría de Pirogov-Sinai nos provee de un procedimiento sistemático para trasladar el estudio de las propiedades macroscópicas (existencia y propiedades de medidas de Gibbs) en sistemas bajo el régimen de no unicidad, i.e. alta densidad de partículas o baja temperatura, al estudio de un modelo de contornos adecuado a baja densidad que puede ser tratado con las técnicas desarrolladas en los capítulos previos. De hecho, aprovechando estas técnicas nos es posible extender el rango de validez de algunos resultados con respecto al ofrecido por la teoría de Pirogov-Sinai clásica.

Por simplicidad, discutimos estas ideas primero para el modelo de Potts con $q$ spines ($q \geq 2$), en donde la simetría entre spines dada por el modelo facilita el análisis y nos evita algunas cuestiones técnicas.
Para este modelo, mostramos que a temperaturas bajas existe para cada spin $i=1,\dots,q$ una medida de Gibbs $\mu^i$ que verifica un escenario de mar con islas para el spin $i$, análogo al descrito en la Introducción para el modelo de Ising. En particular, las medidas $\mu^i$ son singulares entre sí, lo cual implica una transición de fase para el modelo a bajas temperaturas. Además, probamos que cada medida $\mu^i$ es exponencialmente mixing y continua como función de la temperatura (con la noción de convergencia local).

Luego, tratamos el caso del modelo de Widom-Rowlinson discreto, en donde la falta de simetría entre los spins introduce algunas complicaciones en el análisis. No obstante, conseguimos resultados análogos a los obtenidos para el modelo de Potts: si la densidad de partículas es suficientemente alta entonces existen dos medidas de Gibbs $\mu^+$ y $\mu^-$ con las propiedades descritas arriba (para los spins $+$ y $-$, respectivamente). Por otro lado, por la falta de simetría entre los spins $(\pm)$ y el $0$, puede verse que no existe una medida con estas características asociada al spin $0$ restante.

A continuación, mencionamos cómo el análisis realizado en los dos casos anteriores puede generalizarse a modelos discretos generales. También mostramos cómo pueden aplicarse estas ideas en modelos continuos (con espacio de spines finito y con simetría entre spins), obteniendo así los mismos resultados para el modelo de Widom-Rowlinson continuo y el modelo de varas finas con $n$ orientaciones equidistantes.

Por último, en todos estos casos mostramos cómo puede ser aprovechada la dinámica FFG para simular perfectamente las distintas medidas bajo consideración. Cabe destacar que hasta ahora sólo se conocían métodos de simulación perfecta a baja temperatura para el modelo de Ising original. Nuestro análisis extiende estos métodos a una gama mucho más amplia de modelos, tanto discretos como continuos.

\appendix

\chapter{}

\section{Comparison principle}

\begin{teo} Let $f$ and $g$ be Lipschitz functions on $\R$ and, given $u,v \in C([0,1])$, consider $U^u$ and $U^v$ the solutions of the equation
$$
\p_t U = \p^2_{xx} U + f(U) + g(U)\dot{W}
$$ with initial data $u$ and $v$, respectively, and boundary conditions satisfying
$$
P( U(t,\cdot)|_{\p [0,1]} \geq V(t,\cdot)|_{\p [0,1]} \text{ for all $t \geq 0$}) = 1.
$$ Then, if $u \geq v$ we have that
$$
P\left( U^u(t,x) \geq U^v(t,x) \text{ for all }t \geq 0, x \in [0,1] \right)=1.
$$
\end{teo}
A proof of this theorem can be found on \cite[p.~130]{DNKXM}. Notice that if $g \equiv 0$ one obtains a comparison principle for deterministic partial differential equations.

\section{Growth and regularity estimates}

\begin{prop}\label{G.1}Given a bounded set $B \subseteq C_D([0,1])$ there exists $t_B > 0$ such that
\begin{itemize}
\item [$\bullet$] $\tau^u > t_B$ for any $u \in B$
\item [$\bullet$] There exists $b: [0,t_B] \rightarrow \R^+$ such that $\lim_{t \rightarrow 0^+} b(t) = 0$ and for any $t \in [0,t_B]$
$$
\sup_{u \in B} \| U^u(t,\cdot) - u \|_\infty \leq b(t).
$$
\end{itemize}
\end{prop}

\begin{proof} Let $n \in \N$ be such that $B \subseteq B_{n -1}$. Notice that for $u \in B$ the truncated system $U^{(n),u}$ verifies
$$
\partial^2_{xx} U^{(n),u} - \| g_n \|_\infty \leq \p_t U^{(n),\varepsilon} \leq \partial^2_{xx} U^{(n),u} + \| g_n \|_\infty
$$ and so, by the comparison principle there exists $b : \R^+ \rightarrow \R^+$ such that $\lim_{t \rightarrow 0^+} b(t) = 0$ and for any $t > 0$
$$
\sup_{u \in B} \| U^{(n),u}(t,\cdot) - u \|_\infty \leq b(t).
$$ Now, if we take $t_B > 0$ such that $b(t) < 1$ for all $t \in [0,t_B]$ then for any such $t$ we have that $U^{(n),u}(t,\cdot) \in B_n^\circ$ and so $U^{(n),u}(t,\cdot)$ coincides with $U^{u}(t,\cdot)$. In particular, we see that $\tau^u > t$ and also that
$$
\sup_{u \in B} \| U^u(t,\cdot) - u \|_\infty \leq b(t)
$$ which concludes the proof.
\end{proof}

\begin{prop}\label{G.2}The following local and pointwise growth estimates hold:
\begin{enumerate}
\item [i.] Given a bounded set $B \subseteq C_D([0,1])$ there exist $C_B, t_B > 0$ such that
\begin{itemize}
\item [$\bullet$] $\tau^u > t_B$ for any $u \in B$
\item [$\bullet$] For any pair $u,v \in B$ and $t \in [0,t_B]$
$$
\|U^u(t,\cdot) - U^v(t,\cdot) \|_\infty \leq e^{C_B t} \| u - v\|_\infty.
$$
\end{itemize}
\item [ii.] Given $u \in C_D([0,1])$ and $t \in [0,\tau^u)$ there exist $C_{u,t}, \delta_{u,t} > 0$ such that
\begin{itemize}
\item [$\bullet$] $\tau^v > t$ for any $v \in B_{\delta_{u,t}}(u)$
\item [$\bullet$] For any $v \in B_{\delta_{u,t}}(u)$ and $s \in [0,t]$
$$
\|U^u(s,\cdot) - U^v(s,\cdot) \|_\infty \leq e^{C_{u,t} s} \|u - v\|_\infty.
$$
\end{itemize}
\end{enumerate}
\end{prop}

\begin{proof} Given a bounded set $B \subseteq C_D([0,1])$ let us take $n \in \N$ and $t_B > 0$ as in the proof of Proposition \ref{G.1}. Since $g_n$ globally Lipschitz, there exists a constant $C_n > 0$ such that
for any pair $u,v \in C_D([0,1])$ and $t \geq 0$
\begin{equation}\label{G.2eq1}
\|U^{(n),u}(t,\cdot) - U^{(n),v}(t,\cdot) \|_\infty \leq e^{C_n t} \| u - v\|_\infty.
\end{equation} In particular, since for every $u \in B$ both $U^{(n),u}(t,\cdot)$ and $U^{u}(t,\cdot)$ coincide for all $t \in [0,t_B]$, we see that if we take $C_B := C_n$ then from \eqref{G.2eq1} and Proposition \ref{G.1} we obtain (i).

On the other hand, given $u \in C_D([0,1])$ and $t < t^u$ there exists $n \in \N$ such that $\sup_{s \in [0,t]} \|U^u(s,\cdot)\|_\infty \leq n-1$ and a constant $C_n > 0$ such that for any pair $u,v \in C_D([0,1])$ and $s \geq 0$
\begin{equation}\label{G.2eq2}
\|U^{(n),u}(s,\cdot) - U^{(n),v}(s,\cdot) \|_\infty \leq e^{C_n s} \| u - v\|_\infty.
\end{equation} Then by taking $C_{u,t}:=C_n$ and $\delta_{u,t} < e^{-C_n t}$ from \eqref{G.2eq2} we see that for any $v \in B_{\delta_{u,t}}$ both $U^{(n),v}(s,\cdot)$ and $U^{v}(s,\cdot)$ coincide for all $s \in [0,t]$ and so (ii) immediately follows.
\end{proof}

\begin{prop}\label{A.1} If $u \in C_D([0,1])$ then $\partial^2_{xx} U^{u}$ exists for any $t \in (0,\tau^u)$. Furthermore, for any bounded set $B \subseteq C_D([0,1])$ there exists a time $t_B > 0$ such that
\begin{enumerate}
\item [$\bullet$] $\tau^u > t_B$ for any $u \in B$
\item [$\bullet$] For any $t \in (0,t_B)$ we have $\sup_{u \in B} \| \partial^2_{xx} U^{u} (t,\cdot) \|_{\infty} < + \infty.$
\end{enumerate}
\end{prop}

\begin{proof} Given $u \in C_D([0,1])$ and $t \in (0,\tau^u)$ let us take $n \in \N$ such that
$$
\sup_{s \in [0,t]} \|U^u(s,\cdot)\|_\infty \leq n.
$$ It follows from this choice of $n$ that $U^u$ and $U^{(n),u}$ coincide on the interval $[0,t]$, so that it suffices to show that $\partial^2_{xx} U^{(n),u}(t,\cdot)$ exists. For this purpose, recall that $U^{(n),u}$ satisfies the integral equation
\begin{equation}\label{A.1eq1}
U^{(n),u}(t,x) = \int_0^1 \Phi(t,x,y)u(y)dy + \int_0^{t} \int_0^1 \Phi(t-s,x,y)g_n(U^{(n),u})(s,y))dyds
\end{equation} where $\Phi$ denotes the fundamental solution of the heat equation with Dirichlet boundary conditions given by the formula
$$
\Phi(t,x,y) = \frac{1}{\sqrt{4\pi t}} \sum_{n \in \Z} \left[ \exp\left( - \frac{(2n+y -x)^2}{4t} \right) - \exp\left( - \frac{(2n+y +x)^2}{4t} \right)\right].
$$
The equation \eqref{A.1eq1} can be rewritten as
$$
U^{(n),u}(t,x) = \int_\R \frac{e^{-\frac{(y-x)^2}{4t}}}{\sqrt{4\pi t}} \overline{u}(y)dy + \int_0^{t} \int_\R \frac{e^{-\frac{(y-x)^2}{4(t-s)}}}{\sqrt{4\pi (t-s)}} \overline{g_n(U^{(n),u})}(s,y)dyds
$$ where, for $v \in C_D([0,1])$, $\overline{v}$ denotes its odd $2$-periodic extension to the whole real line and, for $\varphi \in C_D([0,T] \times [0,1])$, we set $\overline{\varphi}$ as $\overline{\varphi}(t,x) := \overline{\varphi(t,\cdot)}(x)$ for all $(t,x) \in [0,T]\times [0,1]$. Since $g_n$ and $u$ are both bounded it follows that the spatial derivative $\p_x U^{(n),u}(t,\cdot)$ exists and satisfies
\begin{equation}\label{A.1eq2}
\p_x U^{(n),u}(t,x) = \int_\R \frac{\xi e^{-\xi^2}}{\sqrt{\pi t}} \overline{u}(x + \sqrt{4t} \xi)d\xi + \int_0^{t} \int_\R \frac{\eta e^{-\eta^2}}{\sqrt{\pi (t-s)}}\overline{g_n(U^{(n),u})}(s,x+\sqrt{4(t-s)}\eta)d\eta ds
\end{equation} where we have performed the changes of variables
$$
\xi = \frac{y-x}{\sqrt{4t}}\hspace{2cm}\text{ and }\hspace{2cm}\eta = \frac{y-x}{\sqrt{4(t-s)}}.
$$ Consequently, we obtain the bound
$$
\| \p_x U^{(n),u}(t,\cdot)\|_\infty \leq \left( \frac{1}{\sqrt{\pi t}} + \frac{\sqrt{t}}{\sqrt{\pi}}\right)(\|u\|_\infty + \|g_n\|_\infty).
$$ Let us observe that if we set $\psi := U^{(n),u}(\frac{t}{2},\cdot)$ then $U^{(n),u}(t,\cdot) = U^{(n),\psi}(\frac{t}{2},\cdot)$ and by \eqref{A.1eq2} applied to the time $\frac{t}{2}$ we obtain that $\frac{d\psi}{dx}$ exists and satisfies
$$
\left\|\frac{d\psi}{dx}\right\|_\infty \leq \left( \frac{\sqrt{2}}{\sqrt{\pi t}} + \frac{\sqrt{t}}{\sqrt{2\pi}}\right)(\|u\|_\infty + \|g_n\|_\infty).
$$ Furthermore, since $g_n'$ is bounded and for any $s \in (0,t)$ we have
$$
\| \p_x U^{(n),u}(s,\cdot)\|_\infty \leq \left( \frac{1}{\sqrt{\pi s}} + \frac{\sqrt{s}}{\sqrt{\pi}}\right)(\|u\|_\infty + \|g_n\|_\infty)
$$ then it follows that $\p^2_{xx} U^{(n),u}(t,\cdot)$ exists and satisfies
\begin{align*}
\p^2_{xx} U^{(n),u}(t,x) &= \p^2_{xx} U^{(n),\psi}\left(\frac{t}{2},x\right) = \int_\R \frac{\sqrt{2}\xi e^{-\xi^2}}{\sqrt{\pi t}} \frac{d \overline{\psi}}{dx}(x + \sqrt{2t} \xi)d\xi\\
\\
&+ \int_0^{\frac{t}{2}} \int_\R \frac{\sqrt{2}\eta e^{-\eta^2}}{\sqrt{\pi (t-2s)}}\left(\overline{g_n'(U^{(n),u})}\cdot\overline{\p_x U^{(n),\psi}}\right)(s,x+\sqrt{2(t-2s)}\eta)d\eta ds.
\end{align*}Consequently, we obtain the bound
$$
\|\p^2_{xx} U^{(n),u}(t,\cdot)\|_\infty \leq \left( \left( \frac{2}{\pi t} + \frac{1}{\pi} \right) + \|g'_n\|_\infty \int_0^{\frac{t}{2}} \frac{\frac{1}{\sqrt{\pi s}} + \frac{\sqrt{s}}{\sqrt{\pi}}}{\sqrt{\pi(\frac{t}{2}-s)}} ds \right)(\|u\|_\infty + \|g_n\|_\infty).
$$ By observing that
$$
\int_0^{\frac{t}{2}} \frac{\frac{1}{\sqrt{s}} + \sqrt{s}}{\sqrt{\frac{t}{2}-s}} ds \leq 4 + t
$$ we conclude that
\begin{equation}\label{A.1eq3}
\|\p^2_{xx} U^{(n),u}(t,\cdot)\|_\infty \leq \frac{1}{\pi} \left( 1 + \frac{2}{t} + \|g'_n\|_\infty(4+t)\right)(\|u\|_\infty + \|g_n\|_\infty).
\end{equation} Finally, given a bounded set $B \subseteq C_D([0,1])$ then the uniformity of the bound in $B$ follows from \eqref{A.1eq3} since by Proposition \ref{G.1} we may take $n \in \N$ such that
$$
\sup_{u \in B, t \in [0,t_B]} \| U^{u}(t,\cdot)\|_\infty \leq n
$$ for some $t_B > 0$ such that $\tau^u > t_B$ for all $u \in B$.
\end{proof}

\begin{prop}\label{A.2} For any bounded set $B \subseteq C_D([0,1])$ there exists $t_B > 0$ such that
\begin{enumerate}
\item [$\bullet$] $\tau^u > t_B$ for any $u \in B$
\item [$\bullet$] For any $t \in (0,t_B)$ there exist $R_t , N_t > 0$ such that for every $u \in B$ the function $U^{u}(t,\cdot)$ belongs to the compact set
$$
\gamma_{R_t,N_t} = \{ v \in C_D([0,1]) : \|v\|_\infty \leq R_t \,,\, |v(x)-v(y)| \leq N_t |x-y| \text{ for all }x,y \in [0,1]\}.
$$
\end{enumerate}
\end{prop}

\begin{proof}By Proposition \ref{G.1} there exists $t_B > 0$ such that $\tau^u > t_B$ for every $u \in B$ and for each $t \in (0,t_B)$ there exists $R_t > 0$ such that
$$
\sup_{u \in B, s \in [0,t]} \| U^u(s,\cdot) \|_\infty \leq R_t.
$$ It then follows from the proof of Proposition \ref{A.1} that
$$
\sup_{u \in B} \| \p_x U^{u}(t,\cdot)\|_\infty \leq \left( \frac{1}{\sqrt{\pi t}} + \frac{\sqrt{t}}{\sqrt{\pi}}\right)( R_t + \|g_{R_t}\|_\infty):=N_t
$$ which by the mean value theorem implies that $U^{u}(t,\cdot) \in \gamma_{R_t,N_t}$ for all $u \in B$.
\end{proof}

\begin{prop}\label{A.7} The following local and pointwise growth estimates hold:
\begin{enumerate}
\item [i.] Given a bounded set $B \subseteq C_D([0,1])$ there exists $t_B > 0$ such that
\begin{itemize}
\item [$\bullet$] $\tau^u > t_B$ for any $u \in B$
\item [$\bullet$] For any $t \in (0,t_B)$ there exists $C_{t,B} > 0$ such that for all $u,v \in B$
$$
\| \p_x U^{u}(t,\cdot) - \p_x U^v (t,\cdot) \|_\infty \leq C_{t,B} \| u - v\|_\infty.
$$
\end{itemize}
\item [ii.] Given $u \in C_D([0,1])$ and $t \in (0,\tau^u)$ there exist $C_{u,t}, \delta_{u,t} > 0$ such that
\begin{itemize}
\item [$\bullet$] $\tau^v > t$ for any $v \in B_{\delta_{u,t}}(u)$
\item [$\bullet$] For any $v \in B_{\delta_{u,t}}(u)$
$$
\| \p_x U^{u}(t,\cdot) - \p_x U^v (t,\cdot) \|_\infty \leq C_{u,t} \| u - v\|_\infty.
$$
\end{itemize}
\end{enumerate}
\end{prop}

\begin{proof} Notice that if $t_B > 0$ is such that $\tau^u > t_B$ for every $u \in B$ and for each $t \in (0,t_B)$ there exists $R_{B,t} > 0$ such that
$$
\sup_{u \in B, s \in [0,t]} \| U^u(s,\cdot) \|_\infty \leq R_{B,t},
$$
then it follows from \eqref{A.1eq2} that
$$
\| \p_x U^{u}(t,\cdot) - \p_x U^v (t,\cdot) \|_\infty \leq \left[\left( \frac{1}{\sqrt{\pi t}} + \frac{\sqrt{t}}{\sqrt{\pi}}\right)(1 + \|g'_{R_t}\|_\infty)\right] \| u - v\|_\infty.
$$ which shows (i). Now, (ii) follows in the same way upon noticing that by Proposition \ref{G.2} given $t \in (0,\tau^u)$ there exist $R_{u,t}, \delta_{u,t} > 0$ such that $\tau^v > t$ for any $v \in B_{\delta_{u,t}}(u)$ and
$$
\sup_{v \in B_{\delta_{u,t}}(u), s \in [0,t]} \| U^v(s,\cdot) \|_\infty \leq R_{u,t}.
$$
\end{proof}

\begin{prop}\label{A.3} For any equilibrium point $w$ of the deterministic system let us consider its stable manifold $\mathcal{W}^w$ defined as
$$
\mathcal{W}^{w}:=\{ u \in C_D([0,1]) : U^{u} \text{ is globally defined and }U^{u}(t,\cdot) \underset{t \rightarrow +\infty}{\longrightarrow} w\}.
$$ Notice that $\mathcal{W}^{\mathbf{0}}=\mathcal{D}_{\mathbf{0}}$. Then for any bounded set $B \subseteq \mathcal{W}^w$ there exists $t_B > 0$ such that for any $t_0 \in [0,t_B]$ and $r>0$ we have
$$
\sup_{u \in B} \left[ \inf \{ t \geq t_0 : d( U^u(t,\cdot), w) \leq r \} \right] < +\infty
$$ whenever one of the following conditions hold:
\begin{enumerate}
\item [$\bullet$] $w \neq \mathbf{0}$
\item [$\bullet$] $w = \mathbf{0}$ and $B$ is at a positive distance from $\mathcal{W} := \bigcup_{n \in \Z - \{0\}} \mathcal{W}^{z^{(n)}}$.
\end{enumerate} Furthermore, if $B \subseteq \mathcal{D}_e$ is a bounded set at a positive distance from $\mathcal{W}$ then for any $n \in \N$ we have
$$
\sup_{u \in B} \tau^{(n),u} < +\infty.
$$
\end{prop}

\begin{proof} Let us suppose first that $w \neq \mathbf{0}$. Then, since $\mathcal{W}^w$ is a closed set, by Proposition \ref{A.2} we have that the family $\{ U^u(t_B ,\cdot) : u \in B \}$ is contained in a compact set $B' \subseteq \mathcal{W}^{w}$ for some suitably small $t_B > 0$. Hence, we obtain that
$$
\sup_{u \in B} \left[ \inf \{ t \geq t_0 : d( U^u(t,\cdot), w) \leq r \} \right] \leq t_B + \sup_{v \in B'} \left[ \inf \{ t \geq 0 : d( U^v(t,\cdot), w) < r \} \right]
$$ Since the application $v \mapsto \inf \{ t \geq 0 : d( U^v(t,\cdot), w) < r \}$ is upper semicontinuous and finite on $\mathcal{W}^w$, we conclude that the right hand side is finite and thus the result follows in this case.
Now, if $w = \mathbf{0}$ then once again by Proposition \ref{A.2} we have that the family $\{ U^u(t_B ,\cdot) : u \in B \}$ is contained in a compact set $B' \subseteq \mathcal{D}_{\mathbf{0}}$ but this time by Proposition \ref{G.1} we may choose $t_B > 0$ sufficiently small so as to guarantee that $B'$ is at a positive distance from $\mathcal{W}$. From here we conclude the proof as in the previous case. Finally, the last statement of the proposition is proved in a completely analogous fashion.
\end{proof}

\section{Properties of the potential $S$}

\begin{prop}\label{Lyapunov} The mapping $t \mapsto S( U^u(t,\cdot) )$ is monotone decreasing and continuous for any $u \in H^1_0((0,1))$.
\end{prop}

\begin{proof}First, observe that a direct calculation shows that for any $u \in C_D([0,1])$ and $t_0 > 0$
\begin{align*}
\frac{d S( U^u(t,\cdot) )}{dt}(t_0) &= \int_0^1 \left(\partial_x U^u(t_0,\cdot) \partial^2_{xt}  U^u(t_0,\cdot) - g(U^u(t_0,\cdot)) \partial_t U^u(t_0,\cdot)\right) \\
\\
& = - \int_0^1 \left( \partial^2_{xx} U^u(t_0,\cdot) + g(U^u(t_0,\cdot)) \right) \partial_t U^u(t_0,\cdot) \\
\\
&= - \int_0^1 \left(\partial_t U^u(t_0,\cdot)\right)^2 \leq 0.
\end{align*} On the other hand, it is well known (see \cite{QS} p.75) that the mapping $t \mapsto U^{u}{(t,\cdot)}$ is continuous on $H^1_0((0,1))$ whenever $u \in H^1_0((0,1))$ and on $L^{p+1}([0,1])$ when $u \in C_D([0,1])$. In particular, we see that $t \mapsto S( U^u(t,\cdot) )$ is continuous at $t_0=0$ and so, by the previous calculation, we conclude the result.
\end{proof}

\begin{prop}\label{A.4} The potential $S$ is lower semicontinuous.
\end{prop}

\begin{proof}Let $(v_k)_{k \in \N} \subseteq C_D([0,1])$ be a sequence converging to some limit $v_\infty \in C_D([0,1])$. We must check that
\begin{equation}\label{continferiors}
S(v_\infty) \leq \liminf_{k \rightarrow +\infty} S(v_k).
\end{equation} Notice that since $(v_k)_{k \in \N}$ is convergent in the supremum norm we have, in particular, that
\begin{equation}\label{cotalp}
\sup_{k \in \N} \| v_k \|_{L^{p+1}} < +\infty
\end{equation} and therefore that $\liminf_{k \rightarrow +\infty} S(v_k) > -\infty$. Hence, by passing to an subsequence if necessary, we may assume that the limit in \eqref{continferiors} exists and is finite so that, in particular, we have that the sequence $(S(v_k))_{k \in \N}$ remains bounded. This implies that $v_k$ is absolutely continuous for every $k \in \N$ and, furthermore, by \eqref{cotalp} we conclude that the sequence $(v_k)_{k \in \N}$ is bounded in $H^1_0((0,1))$. Hence, there exists a subsequence $(v_{k_j})_{j \in \N}$ which is weakly convergent in $H^1_0((0,1))$ and strongly convergent in $L^2([0,1])$ to some limit $v_\infty^*$. Notice that since $(v_k)_{k \in \N}$ converges in the supremum norm to $v_\infty$, it also converges in $L^q$ for every $q \geq 1$. In particular, we have that $v^*_\infty = v_\infty$ and thus, by the lower semicontinuity of the $H^1_0$-norm with respect in the weak topology we conclude that
$$
\| \p_x v_\infty \|_{L^2} \leq \liminf_{j \rightarrow +\infty} \| \p_x v_{k_j}\|_{L^2}.
$$ Finally, since $(v_k)_{k \in \N}$ converges to $v_\infty$ in $L^{p+1}$ and we have $S(u) = \frac{1}{2} \| \p_x u \|_{L^2}^2 - \frac{1}{p+1} \| u \|_{L^{p+1}}^{p+1}$ for all $u \in H^1_0$, we obtain \eqref{continferiors}.
\end{proof}

\begin{prop}\label{S.1} Given $u \in C_D([0,1])$ and $t \in (0,\tau^u)$ there exist $C_{u,t}, \delta_{u,t} > 0$ such that
\begin{itemize}
\item [$\bullet$] $\tau^v > t$ for any $v \in B_{\delta_{u,t}}(u)$
\item [$\bullet$] For any $v \in B_{\delta_{u,t}}(u)$ one has
$$
\| S(U^{u}(t,\cdot)) - S(U^v (t,\cdot) ) \|_\infty \leq C_{u,t} \| u - v\|_\infty.
$$
\end{itemize}
\end{prop}

\begin{proof} This is a direct consequence of Propositions \ref{A.7} and \ref{G.2}.
\end{proof}

\section{Properties of the quasipotential $V$}

\begin{prop}\label{A.5} The mapping $u \mapsto V(\mathbf{0},u)$ is lower semicontinuous on $C_D([0,1])$.
\end{prop}

\begin{proof} Let $(u_k)_{k \in \N} \subseteq C_D([0,1])$ be a sequence converging to some limit $u_\infty \in C_D([0,1])$. We must check that
\begin{equation}\label{continferiorv}
V(\mathbf{0},u_\infty) \leq \liminf_{k \rightarrow +\infty} V(\mathbf{0},v_k).
\end{equation} If $S(u_\infty)=+\infty$ then by Proposition \ref{costo} we see that $V(\mathbf{0},u_\infty)=+\infty$ and thus by the lower semicontinuity of $S$ we conclude that $\lim_{v \rightarrow u} V(\mathbf{0},v)=+\infty$ which establishes \eqref{continferiorv} in this particular case. Now, if $S(u_\infty)< +\infty$ then, by the lower semicontinuity of $S$ and the continuity in time of the solutions to \eqref{MainPDE}, given $\delta > 0$ there exists
$t_0 > 0$ sufficiently small such that $S(U^{u_\infty}(t_0,\cdot)) > S(u_\infty) - \frac{\delta}{2}$. Moreover, by Proposition \ref{G.2} we may even assume that $t_0$ is such that
$$
\| U^{u_k}(t_0,\cdot) - U^{u_\infty}(t_0,\cdot) \|_\infty \leq 2 \| u_k - u_\infty\|_\infty
$$ for any $k \in \N$ sufficiently large. Thus, given $k$ sufficiently large and a path $\varphi_k$ \mbox{from $\mathbf{0}$ to $u_k$} we construct a path $\varphi_{k,\infty}$ from $\mathbf{0}$ to $u_\infty$ by the following steps:
\begin{enumerate}
\item [i.] We start from $\mathbf{0}$ and follow $\varphi_k$ until we reach $u_k$.
\item [ii.] From $u_k$ we follow the deterministic flow $U^{u_k}$ until time $t_0$.
\item [iii.] We then join $U^{u_k}(t_0,\cdot)$ and $U^{u_\infty}(t_0,\cdot)$ by a linear interpolation of speed one.
\item [iv.] From $U^{u_\infty}(t_0,\cdot)$ we follow the reverse deterministic flow until we reach $u_\infty$.
\end{enumerate}
By the considerations made in the proof of Lemma \ref{cotasuplema0} it is not hard to see that there exists $C > 0$ such that for any $k \in \N$ sufficiently large we have
$$
I(\varphi_{k,\infty}) \leq I(\varphi_k) + C \| u_k - u_\infty\|_\infty + \delta
$$ so that we ultimately obtain
$$
V(\mathbf{0},u_\infty) \leq \liminf_{k \rightarrow +\infty} V(\mathbf{0},u_k) + \delta.
$$ Since $\delta > 0$ can be taken arbitrarily small we conclude \eqref{continferiorv}.
\end{proof}

\begin{prop}\label{A.6} For any $u,v \in C_D([0,1])$ the map $t \mapsto V\left(u,U^v(t,\cdot)\right)$ is decreasing.
\end{prop}

\begin{proof}Given $0 \leq s < t$ and a path $\varphi$ from $u$ to $U^{u}(s,\cdot)$ we may extend $\phi$ to a path $\tilde{\varphi}$ from $u$ to $U^{u}(t,\cdot)$ simply by following the deterministic flow afterwards.
It follows that
$$
V\left(u,U^v(t,\cdot)\right) \leq I(\tilde{\varphi}) = I(\varphi)
$$ which, by taking infimum over all paths from $u$ to $U^{u}(s,\cdot)$, yields the \mbox{desired monotonicity.}
\end{proof}

\bibliographystyle{plain}
\bibliography{bibliografia}

\def\cprime{$'$} \def\cprime{$'$}
\begin{thebibliography}{10}

\bibitem{AFBR}
Gabriel Acosta, Juli{\'a}n Fern{\'a}ndez~Bonder, and Julio~D. Rossi.
\newblock Stable manifold approximation for the heat equation with nonlinear
  boundary condition.
\newblock {\em J. Dynam. Differential Equations}, 12(3):557--578, 2000.

\bibitem{A}
R.~Azencott.
\newblock Grandes d\'eviations et applications.
\newblock In {\em Eighth {S}aint {F}lour {P}robability {S}ummer {S}chool---1978
  ({S}aint {F}lour, 1978)}, volume 774 of {\em Lecture Notes in Math.}, pages
  1--176. Springer, Berlin, 1980.

\bibitem{BB}
Catherine Bandle and Hermann Brunner.
\newblock Blowup in diffusion equations: a survey.
\newblock {\em J. Comput. Appl. Math.}, 97(1-2):3--22, 1998.

\bibitem{B1}
Stella Brassesco.
\newblock Some results on small random perturbations of an infinite-dimensional
  dynamical system.
\newblock {\em Stochastic Process. Appl.}, 38(1):33--53, 1991.

\bibitem{B2}
Stella Brassesco.
\newblock Unpredictability of an exit time.
\newblock {\em Stochastic Process. Appl.}, 63(1):55--65, 1996.

\bibitem{BKL}
Jean Bricmont, Koji Kuroda, and Joel~L. Lebowitz.
\newblock The structure of {G}ibbs states and phase coexistence for
  nonsymmetric continuum {W}idom-{R}owlinson models.
\newblock {\em Z. Wahrsch. Verw. Gebiete}, 67(2):121--138, 1984.

\bibitem{CGOV}
Marzio Cassandro, Antonio Galves, Enzo Olivieri, and Maria~Eul{\'a}lia Vares.
\newblock Metastable behavior of stochastic dynamics: a pathwise approach.
\newblock {\em J. Statist. Phys.}, 35(5-6):603--634, 1984.

\bibitem{CE1}
Carmen Cort{\'a}zar and Manuel Elgueta.
\newblock Large time behaviour of solutions of a nonlinear reaction-diffusion
  equation.
\newblock {\em Houston J. Math.}, 13(4):487--497, 1987.

\bibitem{CE2}
Carmen Cort{\'a}zar and Manuel Elgueta.
\newblock Unstability of the steady solution of a nonlinear reaction-diffusion
  equation.
\newblock {\em Houston J. Math.}, 17(2):149--155, 1991.

\bibitem{DNKXM}
Robert Dalang, Davar Khoshnevisan, Carl Mueller, David Nualart, and Yimin Xiao.
\newblock {\em A minicourse on stochastic partial differential equations},
  volume 1962 of {\em Lecture Notes in Mathematics}.
\newblock Springer-Verlag, Berlin, 2009.
\newblock Held at the University of Utah, Salt Lake City, UT, May 8--19, 2006,
  Edited by Khoshnevisan and Firas Rassoul-Agha.

\bibitem{DVJ1}
D.~J. Daley and D.~Vere-Jones.
\newblock {\em An introduction to the theory of point processes. {V}ol. {I}}.
\newblock Probability and its Applications (New York). Springer-Verlag, New
  York, second edition, 2003.
\newblock Elementary theory and methods.

\bibitem{DVJ2}
D.~J. Daley and D.~Vere-Jones.
\newblock {\em An introduction to the theory of point processes. {V}ol. {II}}.
\newblock Probability and its Applications (New York). Springer, New York,
  second edition, 2008.
\newblock General theory and structure.

\bibitem{DER}
R.~G. Duran, J.~I. Etcheverry, and J.~D. Rossi.
\newblock Numerical approximation of a parabolic problem with a nonlinear
  boundary condition.
\newblock {\em Discrete Contin. Dynam. Systems}, 4(3):497--506, 1998.

\bibitem{E}
Lawrence~C. Evans.
\newblock {\em Partial differential equations}, volume~19 of {\em Graduate
  Studies in Mathematics}.
\newblock American Mathematical Society, Providence, RI, second edition, 2010.

\bibitem{FJL}
William~G. Faris and Giovanni Jona-Lasinio.
\newblock Large fluctuations for a nonlinear heat equation with noise.
\newblock {\em J. Phys. A}, 15(10):3025--3055, 1982.

\bibitem{FFG1}
Roberto Fern{\'a}ndez, Pablo~A. Ferrari, and Nancy~L. Garcia.
\newblock Loss network representation of {P}eierls contours.
\newblock {\em Ann. Probab.}, 29(2):902--937, 2001.

\bibitem{FFG2}
Pablo~A. Ferrari, Roberto Fern{\'a}ndez, and Nancy~L. Garcia.
\newblock Perfect simulation for interacting point processes, loss networks and
  {I}sing models.
\newblock {\em Stochastic Process. Appl.}, 102(1):63--88, 2002.

\bibitem{FW}
Mark~I. Freidlin and Alexander~D. Wentzell.
\newblock {\em Random perturbations of dynamical systems}, volume 260 of {\em
  Grundlehren der Mathematischen Wissenschaften [Fundamental Principles of
  Mathematical Sciences]}.
\newblock Springer, Heidelberg, third edition, 2012.
\newblock Translated from the 1979 Russian original by Joseph Sz{\"u}cs.

\bibitem{GV}
Victor~A. Galaktionov and Juan~L. V{\'a}zquez.
\newblock The problem of blow-up in nonlinear parabolic equations.
\newblock {\em Discrete Contin. Dyn. Syst.}, 8(2):399--433, 2002.
\newblock Current developments in partial differential equations (Temuco,
  1999).

\bibitem{GOV}
Antonio Galves, Enzo Olivieri, and Maria~Eul{\'a}lia Vares.
\newblock Metastability for a class of dynamical systems subject to small
  random perturbations.
\newblock {\em Ann. Probab.}, 15(4):1288--1305, 1987.

\bibitem{GH}
H.-O. Georgii and O.~H{\"a}ggstr{\"o}m.
\newblock Phase transition in continuum {P}otts models.
\newblock {\em Comm. Math. Phys.}, 181(2):507--528, 1996.

\bibitem{G}
Hans-Otto Georgii.
\newblock {\em Gibbs measures and phase transitions}, volume~9 of {\em de
  Gruyter Studies in Mathematics}.
\newblock Walter de Gruyter \& Co., Berlin, second edition, 2011.

\bibitem{GHM}
Hans-Otto Georgii, Olle H{\"a}ggstr{\"o}m, and Christian Maes.
\newblock The random geometry of equilibrium phases.
\newblock In {\em Phase transitions and critical phenomena, {V}ol. 18},
  volume~18 of {\em Phase Transit. Crit. Phenom.}, pages 1--142. Academic
  Press, San Diego, CA, 2001.

\bibitem{GS}
Pablo Groisman and Santiago Saglietti.
\newblock Small random perturbations of a dynamical system with blow-up.
\newblock {\em J. Math. Anal. Appl.}, 385(1):150--166, 2012.

\bibitem{GSS}
Pablo Groisman, Santiago Saglietti, and Nicolas Saintier.
\newblock Metastability for small random perturbations of a pde with blow-up.
\newblock {\em arXiv preprint arXiv:1501.01724}, 2015.

\bibitem{HA}
Olle H{\"a}ggstr{\"o}m.
\newblock {\em Finite {M}arkov chains and algorithmic applications}, volume~52
  of {\em London Mathematical Society Student Texts}.
\newblock Cambridge University Press, Cambridge, 2002.

\bibitem{K}
Olav Kallenberg.
\newblock {\em Random measures}.
\newblock Akademie-Verlag, Berlin; Academic Press, Inc., London, fourth
  edition, 1986.

\bibitem{LM}
J.~L. Lebowitz and A.~E. Mazel.
\newblock Improved {P}eierls argument for high-dimensional {I}sing models.
\newblock {\em J. Statist. Phys.}, 90(3-4):1051--1059, 1998.

\bibitem{L}
JL~Lebowitz and G~Gallavotti.
\newblock Phase transitions in binary lattice gases.
\newblock {\em Journal of Mathematical Physics}, 12(7):1129--1133, 2003.

\bibitem{MOS}
Fabio Martinelli, Enzo Olivieri, and Elisabetta Scoppola.
\newblock Small random perturbations of finite- and infinite-dimensional
  dynamical systems: unpredictability of exit times.
\newblock {\em J. Statist. Phys.}, 55(3-4):477--504, 1989.

\bibitem{IM}
Henry~P. McKean.
\newblock {\em Stochastic integrals}.
\newblock AMS Chelsea Publishing, Providence, RI, 2005.
\newblock Reprint of the 1969 edition, with errata.

\bibitem{MS1}
R.~A. Minlos and Ja.~G. Sina{\u\i}.
\newblock The phenomenon of ``separation of phases'' at low temperatures in
  certain lattice models of a gas. {I}.
\newblock {\em Mat. Sb. (N.S.)}, 73 (115):375--448, 1967.

\bibitem{MS2}
R.~A. Minlos and Ja.~G. Sina{\u\i}.
\newblock The phenomenon of ``separation of phases'' at low temperatures in
  certain lattice models of a gas. {II}.
\newblock {\em Trudy Moskov. Mat. Ob\v s\v c.}, 19:113--178, 1968.

\bibitem{M2}
Jesper M{\o}ller and Rasmus~Plenge Waagepetersen.
\newblock {\em Statistical inference and simulation for spatial point
  processes}, volume 100 of {\em Monographs on Statistics and Applied
  Probability}.
\newblock Chapman \& Hall/CRC, Boca Raton, FL, 2004.

\bibitem{M}
Carl Mueller.
\newblock Coupling and invariant measures for the heat equation with noise.
\newblock {\em Ann. Probab.}, 21(4):2189--2199, 1993.

\bibitem{NSS}
Marvin~K. Nakayama, Perwez Shahabuddin, and Karl Sigman.
\newblock On finite exponential moments for branching processes and busy
  periods for queues.
\newblock {\em J. Appl. Probab.}, 41A:273--280, 2004.
\newblock Stochastic methods and their applications.

\bibitem{OV}
Enzo Olivieri and Maria~Eul{\'a}lia Vares.
\newblock {\em Large deviations and metastability}, volume 100 of {\em
  Encyclopedia of Mathematics and its Applications}.
\newblock Cambridge University Press, Cambridge, 2005.

\bibitem{P}
{\'E}.~Pardoux.
\newblock Stochastic partial differential equations, a review.
\newblock {\em Bull. Sci. Math.}, 117(1):29--47, 1993.

\bibitem{QS}
Pavol Quittner and Philippe Souplet.
\newblock {\em Superlinear parabolic problems}.
\newblock Birkh\"auser Advanced Texts: Basler Lehrb\"ucher. [Birkh\"auser
  Advanced Texts: Basel Textbooks]. Birkh\"auser Verlag, Basel, 2007.
\newblock Blow-up, global existence and steady states.

\bibitem{RY}
Daniel Revuz and Marc Yor.
\newblock {\em Continuous martingales and {B}rownian motion}, volume 293 of
  {\em Grundlehren der Mathematischen Wissenschaften [Fundamental Principles of
  Mathematical Sciences]}.
\newblock Springer-Verlag, Berlin, third edition, 1999.

\bibitem{SGKM}
Alexander~A. Samarskii, Victor~A. Galaktionov, Sergei~P. Kurdyumov, and
  Alexander~P. Mikhailov.
\newblock {\em Blow-up in quasilinear parabolic equations}, volume~19 of {\em
  de Gruyter Expositions in Mathematics}.
\newblock Walter de Gruyter \& Co., Berlin, 1995.
\newblock Translated from the 1987 Russian original by Michael Grinfeld and
  revised by the authors.

\bibitem{PS}
Ya.~G. Sina{\u\i}.
\newblock {\em Theory of phase transitions: rigorous results}, volume 108 of
  {\em International Series in Natural Philosophy}.
\newblock Pergamon Press, Oxford-Elmsford, N.Y., 1982.
\newblock Translated from the Russian by J. Fritz, A. Kr{\'a}mli, P. Major and
  D. Sz{\'a}sz.

\bibitem{SOW}
Richard~B. Sowers.
\newblock Large deviations for a reaction-diffusion equation with
  non-{G}aussian perturbations.
\newblock {\em Ann. Probab.}, 20(1):504--537, 1992.

\bibitem{W}
John~B. Walsh.
\newblock An introduction to stochastic partial differential equations.
\newblock In {\em \'{E}cole d'\'et\'e de probabilit\'es de {S}aint-{F}lour,
  {XIV}---1984}, volume 1180 of {\em Lecture Notes in Math.}, pages 265--439.
  Springer, Berlin, 1986.

\bibitem{Z2}
M.~Zahradn{\'{\i}}k.
\newblock A short course on the {P}irogov-{S}inai theory.
\newblock {\em Rend. Mat. Appl. (7)}, 18(3):411--486, 1998.

\bibitem{Z1}
Milo{\v{s}} Zahradn{\'{\i}}k.
\newblock An alternate version of {P}irogov-{S}ina\u\i\ theory.
\newblock {\em Comm. Math. Phys.}, 93(4):559--581, 1984.

\end{thebibliography}

\end{document}